\documentclass{amsart}
\usepackage{amssymb,latexsym}
\usepackage{amsrefs}
\usepackage{stackengine}
\usepackage{comment}
\usepackage{hyperref}
\usepackage{caption}
\usepackage{supertabular}
\usepackage{graphicx}
\graphicspath{ {./images/} }
\usepackage{footnote}
\usepackage{color}
\usepackage{tikz-cd}
\usepackage{xypic}
\usepackage{float}
\theoremstyle{plain}
\usepackage{changepage}
\newtheorem{theorem}{Theorem}[section]
\newtheorem{corollary}[theorem]{Corollary}
\newtheorem{proposition}[theorem]{Proposition}
\newtheorem{lemma}[theorem]{Lemma}
\theoremstyle{definition}
\newtheorem{definition}[theorem]{Definition}
\newtheorem{example}[theorem]{Example}

\newtheorem{remark}[theorem]{Remark}
\newtheorem{construction}[theorem]{Construction}
\newtheorem{question}[theorem]{Question}

\newcommand{\Pn}{{\mathbb P}}
\newcommand{\PP}{{\mathbb P}}

\newcommand{\CC}{{\mathbb C}}
\newcommand{\ZZ}{{\mathbb Z}}
\newcommand{\kk}{{\mathbb K}}
\newcommand{\VV}{{\mathbb V}}

\newcommand{\cO}{\mathcal{O}}
\newcommand{\cF}{\mathcal{F}}

\newcommand{\Gr}{\mathbf{Gr}}

\DeclareMathOperator{\rk}{rk}
\DeclareMathOperator{\rank}{rk}
\DeclareMathOperator{\codim}{codim}

\DeclareMathOperator{\type}{type}

\DeclareMathOperator{\Ext}{Ext}
\DeclareMathOperator{\Hom}{Hom}
\DeclareMathOperator{\sat}{sat} 

\DeclareMathOperator{\Sec}{Sec}
\DeclareMathOperator{\Sym}{Sym}
\DeclareMathOperator{\Gor}{Gor}
\DeclareMathOperator{\bGor}{{\bf Gor}}

\DeclareMathOperator{\Cat}{Cat}

\DeclareMathOperator{\Hilb}{Hilb}
\DeclareMathOperator{\reg}{reg}
\DeclareMathOperator{\Tor}{Tor}
\DeclareMathOperator{\Pf}{Pf}
\DeclareMathOperator{\gin}{gin}

\DeclareMathOperator{\rr}{r}
\title[Quaternary Quartic forms and Gorenstein rings]{Quaternary Quartic forms and Gorenstein rings}

\author[G. Kapustka]{Gregorz Kapustka}
\thanks{G. Kapustka supported by Narodowe Centrum Nauki 2018/30/E/ST1/00530.}
\address{G. Kapustka: Department of Mathematics and Informatics, Jagiellonian University, Krak\'ow, Poland}
\email{grzegorz.kapustka@uj.edu.pl}

\author[M. Kapustka]{Micha\l{} Kapustka}
\thanks{M. Kapustka supported by Narodowe Centrum Nauki 2018/31/B/ST1/02857}  
\address{M. Kapustka: Institute of Mathematics of the Polish Academy of Sciences, Warszawa, Poland}
\email{michal.kapustka@impan.pl}

\author[K. Ranestad]{Kristian Ranestad}
\address{Ranestad: Matematisk institutt, Universitetet i Oslo, Oslo, Norway}
\email{ranestad@math.uio.no}

\author[H. Schenck]{Hal Schenck}
\thanks{Schenck supported by NSF 2006410}
\address{Schenck: Mathematics Department, Auburn University, Auburn, USA}
\email{hks0015@auburn.edu}

\author[M. Stillman]{Mike Stillman}
\thanks{Stillman and Yuan supported by NSF 1502294}
\address{Stillman: Mathematics Department, Cornell University, Ithaca, USA}
\email{mike@math.cornell.edu}

\author[B. Yuan]{Beihui Yuan}
\address{Yuan: Mathematics Department, Cornell University, Ithaca, USA}
\email{by238@cornell.edu} 


\begin{document}
\maketitle
\enlargethispage{\baselineskip}
\begin{abstract} 

A quaternary quartic form, a quartic form in four variables, is the dual socle generator of an Artinian Gorenstein ring of codimension and regularity 4.
We present a classification of quartic forms in terms of rank and powersum decompositions which corresponds to the classification by the Betti tables of the corresponding Artinian Gorenstein rings. This gives a stratification of the space of quaternary quartic forms which  we compare with the Noether-Lefschetz stratification. We discuss various phenomena related to this stratification. We study the geometry of powersum varieties for a general form in each stratum. In particular, we show that the powersum variety $VSP(F,9)$ of a general quartic with singular middle catalecticant is again a quartic surface, thus giving a rational map between two divisors in the space of quartics. Finally, we  provide various explicit constructions of general Artinian Gorenstein rings corresponding to each stratum and discuss their lifting to higher dimension.  These provide constructions of codimension four varieties, which include canonical surfaces, Calabi-Yau threefolds and Fano fourfolds. 
In the particular case of quaternary quartics, our results yield answers to questions posed by Geramita in \cite{Geramita1}, Iarrobino-Kanev in \cite{IK}, and Reid in \cite{ReidGeneral}.
\end{abstract}

\tableofcontents
\section{Introduction}\label{intro}
\noindent 
Quaternary quartic forms are dual socle generators of Artinian Gorenstein rings.   In this paper we investigate two aspects of this construction. First, we study the relation between decompositions of the form and the resolution of the Artinian ring as a quotient of $\CC[x_0,\ldots,x_3]$. Second, we investigate graded quotients $S/I$ of a polynomial ring $S=\CC[x_0,\ldots,x_n]$ by a homogeneous ideal $I$ whose Artinian reduction is Gorenstein with dual socle generator a quaternary quartic form.
In particular, we consider quotients $S/I$ that are arithmetically Gorenstein (AG). These rings are defined by the property that if $I$ is of codimension $c$, then $\Ext^i_S(S/I,S)=0$ if $i \ne c$, and the canonical module $\omega_{S/I}$ of $S/I$ satisfies
\begin{equation}\label{CanModule}
\omega_{S/I} \simeq \Ext^c_S(S/I,S(-n-1))\simeq S(-n-1+\codim(I)+\reg(S/I))/I.
\end{equation}
The Artinian reduction has dual socle generator a quaternary quartic form when $S/I$ has codimension and (Castelnuovo-Mumford) regularity both $4$. 

The motivation for studying this situation is twofold. First, when $X=V(I)$ is a threefold, then $X$ is a Calabi-Yau (CY) variety; such varieties play a central role in string theory.  We extend constructions of Calabi-Yau threefolds that appear as AG varieties in $\Pn^7$ in  Section \ref{section_irreducible_liftings}.

Second, while there is a classical structure theory for AG rings of codimension at most $3$, in codimension $4$ there are many open questions. Recent work in \cite{SSY} shows that a codimension $=4=$ regularity AG ring has one of $16$ possible {\em Betti tables} (see \ref{notation}). These 16 possibilities are listed in tables \ref{tableCGKK} and \ref{tableremaining} in Appendix~\ref{Appendix}. 
Our study extends the characterization of these AG rings and their relation to known constructions of Gorenstein rings in codimension $4$.

Our main object of investigation throughout the paper is the quartic form, as a   dual socle generator of an Artinian Gorenstein ring.  The geometry of a homogeneous form $F\in R=\CC[y_0,\ldots,y_n]$ is most often thought of as that of the associated hypersurface $V(F)$ in projective space. However, there are also zero schemes $\Gamma$ associated to $F$, via the apolarity construction of Macaulay. We discuss {\em apolarity} in detail in Section \ref{section_preliminaries}. The starting point is the role of the form as a dual socle generator:

Elements of $S$ act on elements of $R$ by differentiation: 
\[
x_i(y_j) =\frac{\partial}{\partial(y_i)}(y_j) = \delta_{ij}.
\]
We work over $\CC$, but our results hold over any field of characteristic zero; the operation above is essentially the same as contraction.
Consider the annihilator of $F$ under this action
 $$F^\perp:=\operatorname{ann}_S(F)=\{G\in S|\ G(F)=0\}.$$
 We denote by $A_F:=S/F^\perp$ the quotient and call it the {\em apolar ring} of $F$.  It is clearly Artinian, in fact also Gorenstein.
Macaulay shows (see Lemma \ref{ApolarPairing} in the next section), that apolarity gives a bijection between homogeneous polynomials $F \in R_j$, up to rescaling, and graded Artinian Gorenstein rings $S/F^\perp$ with socle in degree $j$.

A geometric aspect of a form that is best understood via apolarity is the powersum presentation.
A representation of $F=l_1^d+ \ldots + l_s^d$ as a sum of powers of linear forms $l_i$ yields a zero scheme $\Gamma=\{[l_1],...,[l_s]\}\subset \Pn(R_1)$ of degree $s$.  Via apolarity, the ideal $I_\Gamma$ is naturally an ideal in $S$, such that the homogeneous coordinate ring $S/I_\Gamma$ of $\Gamma$ admits a surjection $S/I_\Gamma\to A_F$.  

We use minimal power sum presentations to analyze a stratification of the set of quartics.
 By \cite{SSY}, for a quaternary quartic form $F$ that is nondegenerate (not annihilated by a linear form), there are $16$ different Betti tables for $A_F$, and for $i \in \{1,2,3\}$, the {\em Betti numbers} 
\begin{equation}
    \dim_{\CC}\Tor^S_i(A_F,\CC)_{i+1} = b_{i,i+1}(A_F),
    \end{equation}
i.e.~the quadratic part of the Betti table,        determine the entire Betti table of $A_F$. We write $B_{[b_{12}b_{23}b_{34}]}$ to denote such a Betti table. One goal of this paper is to describe the loci of quartic forms in $\Pn^3$ with a specific Betti table:
\begin{definition}\label{FBdefinition} Let 
\[
{\mathcal F}_{[b_{12}b_{23}b_{34}]} = \{F \in \Pn(\Sym_4(\CC^4))=\Pn^{34} \mid S/F^\perp \mbox{ has Betti table  }  B_{[b_{12}b_{23}b_{34}]}\}.
\]
We say that $F$ is of type ${[b_{12}b_{23}b_{34}]}$ if $F\in {\mathcal F}_{[b_{12}b_{23}b_{34}]}$ to reflect this, or the equivalent condition $[F] \in {\mathcal F}_{[b_{12}b_{23}b_{34}]}$, and write ${\mathcal F}_B$ when the type is not fixed. 
\end{definition}
We identify the irreducible components of ${\mathcal F}_B$ for all $B$ arising as Betti tables of quaternary quartics. The main ingredient in our argument is an analysis of two different subideals of $F^\perp$. In Section \ref{bettitables} we consider the ideals of finite sets of points that give minimal power sum presentations of quaternary quartics, 
and in Section \ref{quadraticideals} we consider the subideal $Q_F$ defined by the quadrics in $F^\perp$. 

Since the Betti table of $A_F$ is determined by the quadratic part, we will describe the locus ${\mathcal F}_B$ in terms of the schemes defined by $Q_F$ for $F\in{\mathcal F}_B$.  We find, somewhat surprisingly, the quadratic ideals $Q_F$ are saturated, and identify  these ideals by their set of generators in the space of quadrics.

\begin{definition}\label{GbettiTable} Let 
\[
{\mathcal G}_{[b_{12}b_{23}b_{34}]} = \{G \in \Gr(b_{12},S_2) \mid \mbox{ the Betti table of }S/I_G = B^Q_{[b_{12}b_{23}b_{34}]}\},
\]
where $I_G$ is the ideal generated by $G$, seen as a subspace of the space of quadrics, and $B^Q_{[b_{12}b_{23}b_{34}]}$ is one of the possible Betti tables of quadratic ideals of the form $Q_F$ for some form $F$ of type $[b_{12}b_{23}b_{34}]$. These Betti tables are classified in Section \ref{quadraticideals}.   

As in Definition~\ref{FBdefinition}, we say that a quadratic ideal $Q_F$ is of type ${[b_{12}b_{23}b_{34}]}$ if $Q_F$ corresponds to a point in ${\mathcal G}_{[b_{12}b_{23}b_{34}]}$, and write ${\mathcal G}_B$ when the type is not fixed. 
\end{definition}
In Section \ref{stratification}, we identify the ${\mathcal F}_B$ strata in $\Pn^{34}$. While many of the ${\mathcal F}_B$ are irreducible, not all are, and the geometry of the ${\mathcal G}_B$ loci is key to unraveling this. 
The two cases which require additional analysis are
${\mathcal F}_{[300]}$ and ${\mathcal F}_{[441]}$.

We identify three irreducible subsets of ${\mathcal F}_{[300]}$, which we designate as $[300a],[300b]$ and $[300c]$. Proposition~\ref{CGKK4a,b,c} shows that the corresponding loci satisfy 
\[
{\mathcal F}_{[300c]} \subseteq \overline{{\mathcal F}_{[300b]}} \subseteq \overline{{\mathcal F}_{[300a]}}.
\]
The corresponding loci ${\mathcal G}_{[300a]}$ and ${\mathcal G}_{[300b]}$ in ${\mathcal G}_{[300]}$ are irreducible and coincide, so we denote them both by  ${\mathcal G}_{[300ab]}$.
The locus ${\mathcal G}_{[300c]}$ is also irreducible, and corresponds to a Hilbert scheme distinct from that of ${\mathcal G}_{[300ab]}$.

Similarly, we show that ${\mathcal F}_{[441]}$ has two irreducible components, which we designate as ${\mathcal F}_{[441a]}$ and ${\mathcal F}_{[441b]}$.  In this case the corresponding loci for the quadratic components ${\mathcal G}_{[441a]}$ and ${\mathcal G}_{[441b]}$ are distinct components of the same Hilbert scheme. Our first main result is the following:

\begin{theorem}\label{bettistrat}
There are $19$ irreducible Betti strata for nondegenerate quaternary quartics.  They satisfy the following closure relations below. 

\begin{enumerate}
    \item   ${\mathcal F}_{[562]}$  $\subset\overline{ \mathcal{F}_{[551]}}$  $\subset\overline{ \mathcal{F}_{[550]}}$ 
    \item  ${\mathcal F}_{[441a]}$   $\subset\overline{ \mathcal{F}_{[420]}}$ 
    \item  ${\mathcal F}_{[441b]}$ $\subset\overline{ \mathcal{F}_{[430]}}$ $\subset\overline{ \mathcal{F}_{[420]}}$
     \item   $\mathcal{F}_{[320]}$ $\subset\overline{ \mathcal{F}_{[300b]}}$  
    \item    $\mathcal{F}_{[300c]}$  $\subset\overline{ \mathcal{F}_{[300b]}}$  
    \item  $\mathcal{F}_{[331]}$  $\subset\overline{ \mathcal{F}_{[310]}}$  $\subset\overline{ \mathcal{F}_{[300b]}}$
    \item  $\mathcal{F}_{[210]}$  $\subset\overline{ \mathcal{F}_{[200]}}$ 
    \item $\mathcal{F}_{[683]}$ $\subset\overline{ \mathcal{F}_{[550]}}$ $\subset\overline{ \mathcal{F}_{[420]}}$ $\subset\overline{ \mathcal{F}_{[300b]}}$ $\subset\overline{ \mathcal{F}_{[300a]}}$ $\subset\overline{ \mathcal{F}_{[200]}}$ $\subset\overline{ \mathcal{F}_{[100]}}$
    \item $\mathcal{F}_{[683]}$ $\subset\overline{ \mathcal{F}_{[551]}}$ 
    \item $\mathcal{F}_{[550]}$ $\subset\overline{ \mathcal{F}_{[430]}}$ 
    \item $\mathcal{F}_{[683]}$ $\subset\overline{\mathcal{F}_{[562]}}$  $\subset\overline{ \mathcal{F}_{[430]}}$  $\subset\overline{ \mathcal{F}_{[300c]}}$ 
    \item $\mathcal{F}_{[562]}$  $\subset\overline{ \mathcal{F}_{[441b]}}$ 
    \item $\mathcal{F}_{[552]}$  $\subset\overline{\mathcal{F}_{[441a]}}$  $\subset\overline{ \mathcal{F}_{[310]}}$ $\subset\overline{ \mathcal{F}_{[021]}}$ 
    \item ${\mathcal{F}_{[441a]}}$  $\subset \overline{\mathcal{F}_{[331]}}$  $\subset\overline{ \mathcal{F}_{[210]}}$ 
    \item ${\mathcal{F}_{[420]}}$  $\subset \overline{\mathcal{F}_{[400]}}$  $\subset\overline{ \mathcal{F}_{[300a]}}$ 
\end{enumerate}
\end{theorem}
The proofs of these containments appear in Section \ref{stratification} where Figure 1 depicts this graphically. Non-containments appear in Propositions \ref{ci} and \ref{noncontainment}.
This answers questions of Geramita and Iarrobino-Kanev concerning Gorenstein parameter schemes and Catalecticant varieties in the case of quartics as explained in Section \ref{questions}, as well as questions of Reid about codimension $4$ Gorenstein subschemes, as discussed in Section \ref{subsection_history}. 

The Fermat quartic, $y_0^4+y_1^4+y_2^4+y_3^4$, is of type $[683]$.  It is a nonsingular quartic, so a corollary of Theorem \ref{bettistrat} is that the general quartic in every Betti stratum is nonsingular.
We compare in Section \ref{non-noether} the Betti stratification with the Noether-Lefschetz locus, and show the relation between Betti strata and so called Noether-Lefschetz varieties of quartic surfaces.

A crucial issue in the analysis of the 19 irreducible Betti strata is the geometry of the set of power sum presentations of a general form in each stratum.  In particular, the power sum presentations of degree equal to the {\em rank} of the form. 
\begin{definition}\label{rankDef}
The rank  of a form $F\in R_d=\CC[y_0,\ldots,y_n]_d$,  denoted $\rr(F)$, is the minimal $s$ such that 
 $$F = l_1^d+ \ldots + l_s^d$$
for some linear forms $l_i\in R_1$.
\end{definition}

Apolarity defines a duality between the linear forms $S_1\subset S$ and $R_1\subset R$.  The linear forms $l_i\in R_1$ define hyperplanes $V(l_i) \subset \Pn(S_1)$, and points $[l_i] \in \Pn(R_1)$. The set of points $\{[l_1],...,[l_s]\}$ is a point in the Hilbert scheme $Hilb_s(\Pn(R_1))$. We have the
\begin{definition}\label{VSPdefinition}
The {\bf v}ariety of {\bf s}ums of {\bf p}owers presenting $F$ is the closure
$$VSP(F,s) = \overline{\{ \{[l_1],\ldots,[l_s] \} \in Hilb_s(\Pn(R_1)) \mid \exists \lambda_i \in \CC
: F=\lambda_1 l_1^d+\ldots+\lambda_s l_s^d \} }$$ 
of power sums presenting $F$ in the Hilbert scheme.
\end{definition}
We will be particularly interested in the case when $s$ is the rank $\rr(F)$ of $F$, and $F$ is a general member of each type. In the special cases when $VSP(F,\rr(F))$ is a point, the form is called identifiable. Angelini-Chiantini-Vannieuwenhoven find criteria for identifiability of quartic forms in \cite{ACV}.  

Our second main result is Proposition \ref{vsp proposition} of Section \ref{stratification} which describes $VSP(F,\rr(F))$ for the general form in each stratum.  Our results appear in Tables \ref{tablesremaining} and \ref{tableCGKK} of Section \ref{stratification}, which we reproduce below:
\renewcommand{\arraystretch}{1.4}
\begin{center}
\begin{table}[H]
\begin{tabular}{cc}

\begin{tabular}{|c|c|c|}
\hline Betti table B & $r$ & VSP(F,r)  \\
\hline [683] & $4$       & one point \\
\hline [550] & $5$       & one point \\
\hline [400] & $8$       & $\Pn^3$ \\
\hline [320] & $7$       & $\Pn^1$\\
\hline [300a] & $8$       & one point\\
\hline [300b]& $7$       & one point  \\
\hline [300c] & $7$       &$\Pn^1$ \\
\hline [200] & $8$       & two points\\
\hline [100] & $9$       & K3 Surface \\
\hline [000] & $10$       & Fivefold\\
\hline
\end{tabular}
&
\begin{tabular}{|c|c|c|}
\hline Betti table B & $r$ & VSP(F,r)  \\
\hline [562] & $5$       & $\Pn^1$ \\
\hline [551] & $5$       & one point \\
\hline [441a] & $6$       & $\Pn^1$ \\
\hline [441b] & $6$       & $\Pn^1 \times \Pn^1$\\
\hline [430] & $6$       & $\Pn^1$\\
\hline [420] & $6$       & one point \\
\hline [331] & $7$       & $V_{22}$\\
\hline [310] & $7$       & $\Pn^1$ \\
\hline [210] & $8$       & $V_{22}$ $\cup\; \Pn^1\times\Pn^1$ \\
\hline &&\\
\hline
\end{tabular}\\
\end{tabular}
\vskip .2in

\end{table}
\vskip -.3in
\end{center}
\renewcommand{\arraystretch}{1.0}
Here $V_{22}$ stands for the Fano threefold of degree 22 constructed by Mukai as $VSP(F_3,6)$, the variety of sums of powers of a general ternary quartic $F_3$, see \cite{Mukai}. 

The description of $VSP(F,9)$ for a general quartic form of rank $9$ as a K3 surface was a long standing question.  It is tempting to compare it to the result by Mukai, in \cite{Mukai}, that $VSP(F,10)$ for a general ternary sextic form is a K3 surface.   We ask:
\begin{question} Is $VSP(F,\rr(F))$ a K3 surface, when $F$ is a general form in $H^0(S,-2K_S)$ on any toric surface $S$?
\end{question}
The results in the table leaves another open question:
\begin{question} What kind of variety is $VSP(F,10)$, when $F$ is a general quartic form?
\end{question}

Finally, we study AG projective varieties: varieties $X=V(I)$ such that $S/I$ is AG, whose homogeneous coordinate rings $A(X)$ are liftings of $A_F$, that is
\[
A_F\cong A(X)/(\ell_0,...,\ell_{{\rm dim}X}),
\]
where the $\ell_i$ are a regular sequence of linear forms. Thus $A(X)$ and $A_F$ have the same Betti table. In that case, $A_F$ will be called an {\em Artinian reduction} of $X$, and we shall call the variety $X$ an {\em AG-lifting} (or lifting) of $A_F$ or $F$. If $X$ is a cone over a variety $Y$, then $A(X)$ is trivially a lifting of $A(Y)$, so we consider only non-trivial liftings: those liftings $A(X)$ where $X$ is not a cone.

For example, an elliptic normal curve is an AG-lifting of a quadratic form, and a projectively normal $K3$-surface is an AG-lifting of a cubic form. A smooth threefold $X$ that is a lifting of a quartic form is a Calabi-Yau threefold, i.e.  $K_X \simeq \mathcal{O}_X$ and $H^i(\mathcal{O}_X)=0 \mbox{ if  }i \ne 0,3$.

The codimension $4$ case was studied systematically in \cite{B} and  \cite{CGKK}. Patience Ablett, cf. \cite{Ablett}, recently found liftings to curves in $\Pn^5$ of quartics $F$ whose Betti table is in Table~\ref{tableremaining} in Appendix~\ref{Appendix}. Our third main result complements the lifting results of \cites{CGKK,SSY,Ablett}. We prove in Corollary \ref{liftofCGKK} that a general Artinian Gorenstein ring $A_F$ with a Betti table from Table~\ref{tableCGKK} lifts to a smooth curve. It remains an open question in the case of many Betti tables considered in this paper to understand to which dimension we can lift them to smooth varieties.

We find liftings to (possibly reducible) Calabi-Yau threefolds $X$ of $A_F$ for each type of nondegenerate quartic $F$.  In some cases of irreducible liftings, the constructions allow liftings to higher dimensions to Fano varieties.

Our first approach, used to find an AG-lifting of a quaternary quartic form $F$ with a smooth quadric in $F^\perp$, is to restrict our attention to such a quadric.  More precisely, in this case the restriction $F_Q^\perp$ of $F^\perp$ to the quadric $Q$ is a codimension $3$ ideal, that that is related to a Pfaffian ideal  (see Section~\ref{Pfaffian}). We extend these constructions to higher dimensional quadrics using the spinor bundles and Lagrangian degeneration loci (see Section \ref{Tom}).
In particular, by describing the liftings of the generic $A_F$ with Betti table $[100]$, we extend and complete the results of \cites{SSY,CGKK}.

For instance, let $V$ be a $8$-dimensional vector space with a nondegenerate quadratic form $q:V\to\CC$ and let $Q_6=\{q=0\}$ be the $6$-dimensional quadric defined by $q$ in $\Pn(V)$.
Consider a Lagrangian subbundle $\mathcal{E}\subset V\otimes{\mathcal O}_{Q_6}$.  Let $W\subset V$ be a Lagrangian subspace such that $W\cap \mathcal{E}(x) $ is odd-dimensional for all $x\in Q_6$. Then, by \cite{EPW1}*{Thm 2.1},  for sufficiently general $W$,  $$X=\{x\in Q_6|\ {\rm dim}(W\cap \mathcal{E}(x))=3\}$$ is a smooth subvariety of $X$ of codimension $3$.

\begin{proposition}\label{degree 19 Lagrangian-intro}
Suppose $\mathcal{E}=3\mathcal{S}_1$ or $\mathcal{S}_1\oplus 2\mathcal{S}_2$, where $S_1$ and $S_2$ are spinor bundles on  $Q_6$. Then for each choice of $\mathcal{E}$
and for sufficiently general Lagrangian subspace $W\subset V$,  $$X=\{x\in Q_6|\ {\rm dim}(W\cap \mathcal{E}(x))=3\}$$
is a Calabi-Yau threefold of degree $19$.  Moreover, $X$ is AG in $\mathbb{P}^7$ and its coordinate ring has Betti table $[100]$.\end{proposition}

Note that the construction in Proposition \ref{degree 19 Lagrangian-intro} resembles the Tom and Jerry situation for del Pezzo varieties of degree $6$. A del Pezzo surface of degree $6$ lifts to a hyperplane section of $\Pn^2\times \Pn^2$ (Tom), but also of $\Pn^1\times\Pn^1\times\Pn^1$ (Jerry). 

A similar phenomenon occurs for Artinian Gorenstein rings $A_F$  already at the level of lifting to points.  For instance, for the general ring $A_F$ with Betti table of type [562] or [320], the family of liftings to sets of points $X\subset \Pn^4$ has at least two irreducible components. We give a geometric description of these loci in Sections \ref{section_irreducible_liftings} and \ref{section_reducible_liftings}.

A second approach to find liftings is the doubling construction, described in Section~\ref{DoubleConst}. We prove that in most cases (except $[300a]$) the Artinian Gorenstein ring is obtained as a doubling of a configuration of points computing its rank. Geometrically, in positive dimension, a doubling of an arithmetic Cohen Macaulay variety $Y$ is a variety $X$ obtained as a divisor in a twist of the anti-canonical system of $Y$. Note that,   anticanonical divisors in arithmetic Cohen Macaulay fourfolds are Calabi-Yau threefolds (see Proposition \ref{double}). For Artinian Gorenstein rings of type $[000]$  
we consider, in Proposition \ref{CY20}, a fourfold $Y\subset \Pn^7$ defined by the maximal minors of a $3\times 5$ matrix $A$ of linear forms.  When $A$ has rank $1$ at some point, then $Y$ has an effective anticanonical divisor. For a general $A$ with a rank $1$ point, the anticanonical divisor $X$ is unique and is a two nodal AG Calabi-Yau  threefold of degree $20$.  The same construction in $\Pn^5$ yields a halfcanonical curve $X$ of degree $20$, that is the lifting of a general quartic form $F$ of type $[000]$. Smooth Calabi-Yau threefolds of degree $20$ in $\Pn^7$ can be defined by $3\times 3$ minors of a $4\times 4$ matrix of linear forms.  These are liftings of special quartic forms $F$ of type $[000]$.

Bilinkage, or biliaison, as explained in Section \ref{unpr}, provides a third method of constructing AG-liftings. As an example we consider Calabi-Yau threefolds with Betti tables $[300a]$, $[300b]$ and $[300c]$. The first is obtained by elementary biliaison of height $2$ from $\mathbb{P}^3$ inside an irreducible intersection of three quadrics.  These threefolds cannot be constructed by a doubling in a complete intersection as discussed in Section \ref{eg_exceptionCGKK4}.
The second case occurs when the intersection of the quadrics is reducible: The union of a $\mathbb{P}^4$ and a fourfold $Y$ of degree $7$. In this case, a Calabi-Yau  threefold can be constructed by elementary biliaison of height 2 from the cubic threefold $\mathbb{P}^4\cap Y$ in the fourfold $Y$.  However, we show that a biliaison of height $2$ cannot be used to construct an example of type $[300c]$. Instead we can use a doubling construction, or a biliaison of height $1$, see Section \ref{section_irreducible_liftings}.

Constructions of liftings for quaternary quartic forms $F$ of each type to reduced, but not necessarily irreducible manifolds $X$  are summarised in Sections \ref{section_irreducible_liftings} and \ref{section_reducible_liftings}. Liftings of dimension three are particularly interesting as they are necessarily Calabi-Yau manifolds.
Some of these constructions allow further liftings to  dimensions higher than $3$, i.e.~to AG Fano varieties of dimensions at least $4$.  The following table presents the maximal dimensions of smooth liftings of general Artinian Gorenstein rings of the given type that we are able to construct,  see in particular Proposition \ref{liftofCGKK}, Corollary \ref{cor:lifting}. 
\begin{table}[h]
\begin{tabular}{|c|c|}
\hline $\stackrel{\mbox{Betti table}}{\mbox{B}}$ & $\stackrel{\mbox{dimension of possible}}{\mbox{smooth lifting}}$  \\
\hline [683] &  3        \\
\hline [550] &    6     \\
\hline [400] &     $\infty$    \\
\hline [320] &    4    \\
\hline [300a] &   3      \\
\hline [300b] &   2      \\
\hline [200] &    4    \\
\hline [100] &    1 (3)      \\
\hline [000] &  1 (4)   \\
\hline
\end{tabular}
\end{table}

For the types $[100]$ and $[000]$, there are special rings $A_F$ that lift to higher dimensions as marked in parenthesis in the table. Also, some rings $A_F$ lift to different families of higher dimensional varieties.  For details see Section \ref{section_irreducible_liftings}.

An analysis of Betti tables for Artinian Gorenstein rings $A_F$  of codimension $=4$ and regularity $=5$, as well as of codimension $=5$ and regularity $=4$, is a natural and interesting extension of this study, but is beyond the scope of this paper.

\subsection{Organization of the paper}
\begin{enumerate}
 \item In Section \ref{section_preliminaries} we give an overview of theory and constructions of Gorenstein rings, and comment on how the rings $A_F$ of regularity $4$ fit with and inform the general theory. We then move to a discussion of the geometric connections to catalecticant varieties and secant varieties to Veronese varieties. The section closes with a description of apolarity and powersum presentations of forms, and the key role played by syzygies.
\vskip .05in
\item In Section \ref{bettitables} we find the Betti tables of point sets $\Gamma$ that compute the rank for a general $F\in {\mathcal F}_B$, the set of forms with Betti table $B$.  
\vskip .05in
\item In Section \ref{quadraticideals} we find the Betti tables of the ideals $Q_F\subset F^\perp$, for any $F\in {\mathcal F}_B$. We show that, except for two Betti tables $B$, the set 
${\mathcal G}_B$ of these ideals form irreducible sets in the Hilbert scheme.
\vskip .05in
\item In Section \ref{section_VSP} we compute for each Betti table $B$, the rank $s=r(F)$ and the $VSP(F,s)$ for a general form $F\in{\mathcal F}_B$ whenever $s<9$.  We compute the dimension of  
${\mathcal F}_B$ and show that all but one are irreducible.
\vskip .05in
\item In Section \ref{stratification} we present the stratification (defined by the Betti tables) of the space $\Pn(\Sym(S_4)) \simeq \Pn^{34}$ of quartic forms in four variables.
\vskip .05in
\item In Section \ref{Pfaffian} we consider the construction of codimension $4$ AG rings supported on a quadric. We use Pfaffian and Lagrangian constructions of codimension $3$ varieties using vector bundles supported on this quadrics. We construct in this way liftings of some of the Artinian rings.
\vskip .05in
\item In Section \ref{section_irreducible_liftings} and Section \ref{section_reducible_liftings} we complete the task started in Section \ref{Pfaffian} and construct liftings of rings $A_F$ to families of AG threefolds for forms $F$ in each family ${\mathcal F}_B$. We use here the doubling and the bilinkage constructions.  In some cases we show that the Artinian reduction of the general AG threefold in the family is a general form in ${\mathcal F}_B$, while for other families of threefolds, all the Artinian reductions belong to a proper subset of ${\mathcal F}_B$.
\item In Appendix~\ref{Appendix} we give the 16 Betti tables for Artinian Gorenstein rings 
appearing in \cite{SSY}, and in Appendix~\ref{Appendix2} we describe the computational tools underlying this paper: Generic initial ideals, the Groebner family and Groebner strata, and the Schreyer resolution. We illustrate these concepts with a running example: the case of six points. The scripts which implement the relevant computations are part of a package in {\tt Macaulay2}, 
\cite{M2}*{Package QuaternaryQuartics}.
\end{enumerate}

\subsection{Notation.} \label{notation}
\medskip We present the numerical information of the minimal free resolution of a graded 
$S=\CC[x_0,\ldots,x_n]$-module 
$$ 0 \leftarrow M \leftarrow F_0  \leftarrow F_1 \leftarrow \ldots  \leftarrow F_{n+1}  \leftarrow  0$$
with $F_i = \bigoplus_{j \in \ZZ} S(-j)^{b_{ij}}$ in the form of a {\em Betti table} (see \cite{GeomSyz}*{\S 1B}): 
\[
\begin{matrix}
0:& b_{00} & b_{11} & b_{22} & \ldots & b_{n+1,n+1}& \cr
1:&b_{01} & b_{12} & b_{23} & \ldots & b_{n+1,n+2}& \cr
\vdots &\vdots & \vdots & \vdots & \ldots & \vdots & \cr
m: &b_{0m} & b_{1,m+1} & b_{2,m+2} &\ldots &b_{n+1,n+1+m}& \cr 
\end{matrix}
\]
The Castelnuovo-Mumford regularity $\operatorname{reg}(M)$ of $M$ is $m$, and we denote $b_{ij} = 0$ with a $-$. 
\begin{example}
For the ideal $I_C$ of the twisted cubic in $\Pn^3$, we have
\vskip .01in
\begin{small}
\[
0 \rightarrow S(-3)^2 \xrightarrow{\left[ \!
\begin{array}{cc}
-x_2 & x_3\\
x_1 & -x_2 \\
-x_0 & x_1
\end{array}\! \right]} S(-2)^3
\xrightarrow{\left[ \!\begin{array}{ccc}
x_1^2-x_0x_2& x_1x_2-x_0x_3& x_2^2-x_1x_3
\end{array}\! \right]}
 S \rightarrow S/I_{C} \rightarrow 0.
\]
\end{small}
\vskip .01in
\noindent The Betti table for this resolution is:
\begin{small}
$$
\vbox{\offinterlineskip 
\halign{\strut\hfil# \ \vrule\quad&# \ &# \ &# \ &# \ &# \ &# \
&# \ &# \ &# \ &# \ &# \ &# \ &# \
\cr
total&1&3&2\cr
\noalign {\hrule}
0&1 &--&--&\cr
1&--&3 &2 &\cr
\noalign{\bigskip}
\noalign{\smallskip}
}}
$$
\end{small}
\end{example}
\vskip -.1in
\noindent We work with arithmetically Gorenstein rings, and will call such rings AG or {\em Gorenstein}; CI will denote a complete intersection. Several of our results address, in the case of quaternary quartics, questions posed by Iarrobino-Kanev in \cite{IK}. In terms of notation, the rings we call $S$ and $R$ are denoted $R$ and $D$ in \cite{IK}; we use $H$ to denote an Artinian Hilbert function, while \cite{IK} use $T$. We use $A_F$ to denote $S/F^\perp$, where $F$ is a quartic in $\CC[x_0,\ldots,x_3]$. Our focus is on the case where $F$ is {\em nondegenerate}, which means that $F^\perp$ contains no linear form.
\vskip .05in

\section{Preliminaries}\label{section_preliminaries}

\subsection{Apolarity and syzygies}
\label{apolarity}
Consider the rings $R=\CC[y_0,\ldots,y_n]$ and $S=\CC[x_0,\ldots, x_n]$ appearing in the introduction, and the action of $S$ on $R$ by differentiation:
$$x^{\alpha}(y^{\beta}) =  \alpha!\binom{\beta}{\alpha } y^{\beta-\alpha}$$
if $\beta \geq \alpha$ and 0 otherwise. Here $\alpha$ and $\beta$ are multi-indices,
$\binom{\beta}{\alpha} = \prod \binom{\beta_i}{\alpha_i}$ and so on; this defines a natural duality between $R_d$ and $S_d$. We can interchange the roles of $R$ and $S$ by defining
$$y^{\beta}(x^{\alpha}) =  \beta!\binom{\alpha}{\beta} x^{\alpha-\beta}.$$
We set $\Pn^n=\Pn(R_1)$, identifying a point $a=(a_0:\ldots :a_n) \in \Pn^n$ with a linear form  $l_a =\sum a_i y_i\in R_1$ up to scalar.   Then  $$l_a^d(G)=G(l_a^d)=d!G(a),$$
for any $G\in S_d$.  In particular if $m \geq d$ $$G(l_a^m)=0 \iff G(a)=0 \eqno(1.1.1)$$ 
So the ideal of a subscheme $\Gamma\subset \Pn^n$ is a homogenous ideal in $S$. Similarly, if $P_a=\sum a_i x_i \in S_1$, then 
$$P_a^d(F)=F(P_a^d)=d!F(a).$$

\begin{definition}
The ideal of differential operators that annihilates $F \in R$, written $F^{\bot}=\{G\in S| G(F)=0\}$, is called the {\em apolar ideal} of $F$.
 \end{definition}
The ring  $$A_F = S/F^{\bot}$$
is Artinian and Gorenstein, with the socle of $A_F$ in degree $d=\deg(F)$. Indeed 
\[P_a(G(F)) = 0 \hskip3pt \forall P_a \in S_1 \iff G(F) = 0 \hskip3pt \mbox{ or } \hskip3pt G \in S_d. \]
In particular the socle of $A_F$ is 1-dimensional, and $A_F$ is Gorenstein; it is called the {\em apolar Artinian Gorenstein ring} or simply the {\em apolar ring} of $F$. 
 
Conversely for a graded Gorenstein ring $A = S/I$ with socle degree $d$, multiplication in $A$ induces a linear form $\phi\colon S_d\to \CC$ which can be identified with a
 homogeneous polynomial $F \in R$ of degree $d$. This is:

\begin{lemma}\label{ApolarPairing} {\rm [Macaulay 1916]}. 
The map $F \mapsto A_F$ is a bijection between homogeneous forms $F \in R_d$ and graded Artinian Gorenstein quotient rings $A_F = S/I$ of S with socle degree d.
\end{lemma}
The polynomial $F$ is sometimes called the {\em dual socle
generator} or the {\em dual polynomial} of $A_F$ (\cite{CAwvtAG}*{Thm. 21.6 and Exercise 21.7}). 

\begin{definition}
 A subscheme $\Gamma \subset \Pn^n=\Pn(R_1)$ is said to be {\em apolar} to $F$ if the homogeneous ideal $I_{\Gamma} \subset F^{\bot} \subset S$.
 \end{definition}
 By the duality between $R_d$ and $S_d$, this apolarity may be interpreted via the  Veronese map
 $$v_d: \Pn(R_1)\to\Pn(R_d);\quad [l]\mapsto [l^d].$$
 For $\Gamma \subset \Pn^n$ the image $v_d(\Gamma)\subset \Pn(R_d)$ spans a linear subspace defined by the space of forms  $(I_{\Gamma})_d\subset S_d$ of degree $d$ in the ideal $I_{\Gamma}$, i.e.~
 $$\langle v_d(\Gamma)\rangle=(I_{\Gamma})_d^\perp\subset \Pn(R_d).$$
 The fundamental lemma is:
 \begin{lemma}[\bf Apolarity]\label{Apolarity} Let $\Gamma\subset \Pn(R_1)$ be any subscheme, then  
$[F]\in \langle v_d(\Gamma)\rangle \subset \Pn(R_d)$   
if and only if $\Gamma$ is apolar to $F$.
\end{lemma}
In particular, the relationship between a form $F$ and its powersum presentations is given by:  
\begin{lemma}\label{ApolarityRank} 
$F = l_1^d+\ldots+l_s^d$ where $\Gamma = \{ [l_1],\ldots , [l_s]\} \subset \Pn(R_1)$
if and only if $\Gamma$ is apolar to $F$.
\end{lemma}
So the rank of a form (Definition~\ref{rankDef}) is the minimal length of an apolar scheme $\Gamma$ that is a reduced set of points. Note that there is a generalization of the notion of rank to the setting where $\Gamma$ is not reduced:
\begin{definition}\label{CactusRankDef} 
 The cactus rank $\operatorname{cr}(F)$ is the minimal length of a subscheme  apolar to $F$. \end{definition}  
 
 \begin{remark}\label{cactusrankversusrank}
  Clearly $\operatorname{cr}(F) \le \rr(F)$, and they often coincide when $d$ and $n$ or $\rr(F)$ are small. The cactus rank is computed by a scheme that is locally Gorenstein \cite{BB}*{Proposition 2.2, Lemma 2.3}. 
  In an irreducible family of forms, both the rank and the cactus rank attains its minimum in a general member.  Both of them may however increase in special members.
 \end{remark}
  
  To find the rank of a form is, in general, very hard.  In low degrees,  the syzygies of $A_F$ may sometimes be used to determine
  $\rr(F)$ and $VSP(F,s)$, as illustrated by our results.
 
\subsection{Apolarity in codimension $3$ or $4$ and regularity $3$}
For a general ternary cubic form $F\in R_3=\CC[y_0,y_1,y_2]_3$ the apolar ring is a complete intersection
$$A_F \cong S/(q_1,q_2,q_3) $$
with quadrics $q_i$.  The Betti table of  $A_F$ is
\[
 \begin{matrix}
 1 & - & - & - &  \cr
                                               - & 3 & - & - & \cr
            - &- & 3& - & \cr
              - & - & - & 1 &  \cr
              \end{matrix}.
              \]    Any pencil of quadrics $\langle q,q'\rangle\subset \langle q_1,q_2,q_3\rangle$ defines a scheme of length $4$, in general four points. No ideal of three points is contained in $F^{\bot}$, so the rank of $F$ is $4$ and $VSP(F,4)$ is a projective plane.

If $F=l_1^3+l_2^3+l_3^3$ has rank $3$, i.e.~$F$ is a Fermat cubic,  the Betti table of  $A_F$ is
\[
 \begin{matrix}
 1 & - & - & - &  \cr
                                               - & 3 & 2 & - & \cr
            - &2 & 3& - & \cr
              - & - & - & 1 &  \cr
              \end{matrix},
              \]
and the three quadric generators in $F^{\bot}$ are the $2\times 2$ minors of a $2\times 3$ matrix of linear forms which generate the ideal of three points  $$\{[l_1],[l_2],[l_3]\}\subset\Pn(R_1)=\Pn^2.$$
For example, changing variables we may assume $l_i = x_i$, and so it is easy to see that
\[
F^\perp = \langle x_0x_1,x_0x_2,x_1x_2,x_0^3-x_1^3, x_0^3-x_2^3\rangle.
\]
The set of ternary cubic forms that are nondegenerate (have no linear form in the apolar ideal) has a stratification defined by the Betti table. The forms with a Fermat Betti table lie in the closure of the forms with a complete intersection Betti table.

 For a general quaternary cubic $F\in R_3=\CC[y_0,y_1,y_2,y_3]_3$ the apolar ring $A_F$ has Betti table 
\[
 \begin{matrix}
 1 & - & - & - &  - &\cr
                                               - & 6 & 5 & - &  - &\cr
            - & - & 5 & 6 & - &\cr
              - & - & - & - & 1 & \cr
              \end{matrix}.
              \]
              \vskip .1in

Five of the six quadrics in the ideal are the $4\times 4$ Pfaffians of a skew-symmetric $5\times 5$ - matrix.  They define the  ideal of five points.  This decomposition of syzygies is familiar for a general canonical curve $C \subset \Pn^5$ of genus $6$, which is a complete intersection of a unique Del Pezzo surface of degree $5$ and a non unique hyperquadric, cf. [ACGH
1985], [Schreyer 1986], [Mukai 1988].
So $F$ has rank at most $5$.  If it has rank $4$, the ideal $F^{\bot}$  would contain  six quadrics that vanish on the four points.  But $F^{\bot}$ is generated by six quadrics so this is impossible.  So $F$ has a unique presentation as a sum of five cubes (Sylvesters Pentahedral Theorem). 
\vskip .1in
There are two more possible Betti tables of rings $A_F$ of codimension $4$ and regularity $3$. In the first case  the six quadrics define a scheme of length $4$, so if this scheme is smooth, $F$ has rank $4$ and is a Fermat cubic, with Betti table 

\[
 \begin{matrix}
  1 & - & - & - &  - &\cr
                                               - & 6 & 8 & 3 &  - &\cr
              - &3 & 8 & 6 & - &  \cr
              - & - & - & - & 1&\cr 
              \end{matrix}.  
                            \]
              In the second case three of  the six quadrics are reducible with a common linear factor.  So, assuming that these three quadrics are $x_0x_1, x_0x_2,x_0x_3$, the form $F$ decomposes into $F=y_0^3+F_0(y_1,y_2,y_3)$.  If the apolar ideal of the ternary form $F_0$ is determinantal, then it is apolar to a scheme of length $3$, and hence $F$ is apolar to a scheme of length $4$.  Then the apolar ring $A_F$  has the Betti table of a Fermat cubic.   If the apolar ideal of the ternary form $F_0$ is not determinantal, it is a complete intersection $(2,2,2)$, in which case $F_0$ has rank $4$ and $F$ has rank $5$.   The form $F_0$ has a projective plane of different decompositions as a sum of four cubes, so $F$ has a projective plane of different decompositions as a sum of five cubes. The Betti table for $A_F$ is: 
              \[
               \begin{matrix}1 & - & - & - &  - &\cr
                                               - & 6 & 6 & 1 &  - &\cr
             - & 1 & 6 & 6 & - &  \cr
              - & - & - & - & 1&\cr   \end{matrix} 
              \] 
       \vskip .1in
As in the ternary case, there is a stratification in the space of quaternary cubics of the set of nondegenerate cubic forms defined by the Betti tables of their apolar ring. 

Let ${\mathcal F}_{\operatorname{fe}}\subset \Pn^{19}$ be the set of cubic quaternary forms $F$ with the Betti table for $A_F$ of a  Fermat cubic, and similarly,  
let ${\mathcal F}_{\operatorname{pl}}$ be the set of forms with the Betti table of a general form that is a sum of a rank $1$ form and a ternary form. Finally, let  ${\mathcal F}_{\operatorname{gen}}$ be the set of forms with the Betti table of a general form.   Then ${\mathcal F}_{\operatorname{gen}}$ is dense in $\Pn^{19}$, 
the subset ${\mathcal F}_{\operatorname{pl}}$ is irreducible of dimension $16$, and the subset ${\mathcal F}_{\operatorname{fe}}$ is irreducible of dimension $15$.  Furthermore, there is a stratification of the form below:
$${\mathcal F}_{\operatorname{fe}}\subset \overline{{\mathcal F}_{\operatorname{pl}}}\subset \overline{{\mathcal F}_{\operatorname{gen}}}=\Pn(R_3)=\Pn^{19}.$$
Recent results on decompositions for cubics in the real setting appear in \cite{MM}.

\subsection{Geometry: Catalecticant and Secant varieties}
We now review classical connections between power sum decompositions, secant varieties to Veronese varieties and catalecticant matrices. In addition to the book of Iarrobino-Kanev \cite{IK}, other good references are Geramita's papers \cite{Geramita1}, \cite{Geramita2}. 
\subsubsection{Geometry of the catalecticant variety and the Gorenstein parameter scheme}
\begin{definition}\label{GorTDef}
 Following \cite{IK}, and the definitions of apolarity in \ref{apolarity}, we write $\mathcal{H}(d,n+1)$ for the set of Hilbert functions for an Artinian Gorenstein algebra $S/F^\perp$ of socle degree $d={\rm deg}\; F$ when $S=\CC[x_0,...,x_n]$, and define $\Gor(H)$ as the set of polynomials $F \in R_d=\CC[y_0,...,y_n]_d$ such that $S/F^\perp$ has Hilbert function $H$.  Then 
\begin{equation}
    R_d = \bigcup\limits_{H \in \mathcal{H}(d,n+1)}\Gor(H)
\end{equation}
\end{definition}

\begin{definition}\label{catalecticantDef}
Let $F \in R_d$ and $G \in S_{d-u}$, so that $G\circ F \in R_{u}$. If we write $F$ with indeterminate coefficients $\{a_1, \ldots, a_m\}$ with $m = \binom{n+d}{d}$, then the matrix defining the map 
\[
S_{d-u} \longrightarrow R_u \mbox{ via } G \mapsto G \circ F
\]
is called the {\em catalecticant}, and is written $\Cat(u,d-u,n+1)$. 
Let $I_{r}(\Cat(u,d-u,n+1))\subset \CC[a_1,..,a_m]$ be the ideal generated by the $(r\times r)$-minors of $\Cat(u,d-u,n+1)$.
Next, for a Hilbert function 
\[
H=(1=h_0,h_1,h_2,\ldots, h_d=1) \in \mathcal{H}(d,n+1),
\]
define the affine scheme 
\[
\bGor_{\le}(H) = \VV(\sum\limits_{u=1}^{d-1} I_{h_{u}+1}(\Cat(u,d-u,n+1))).
\]
The scheme $\bGor_{\le}(H)$ is a sum of determinantal ideals of catalecticants. It has an open subset $\bGor(H)$ consisting of $F$ such that the Hilbert function of $S/F^\perp$ is $H$; $\Gor(H)$ is the corresponding reduced affine scheme. 
\end{definition}
\noindent It follows from Definition~\ref{catalecticantDef} that the locus in $\Pn(R_d)$ where 
$\rk(\Cat(u,d-u,n+1)) = k$ corresponds to those $F \in R_d$ which have Hilbert function $\dim(S/F^\perp)_{d-u}=k$. 

\begin{example}\label{Cat224}
For the case of quartics in four variables, we are interested in the non-degenerate situation: that is, where no linear form annihilates $F$. Thus, the Hilbert function of $S/F^\perp$ satisfies
\[
H(S/F^\perp) = (1,4,h_2,4,1),\mbox{ with } h_2 \in \{4,\ldots,10\}.
\]
Therefore $10-h_2$ is the dimension of the vector space of quadrics apolar to $F$, i.e.~$10-h_2={\rm dim} F_2^\perp$, and the relevant catalecticant matrix is $\Cat(2,2,4)$, given by 
\vskip .1in
\begin{center}
$\left[\begin{matrix}
      a_{0}&a_{1}&a_{2}&a_{3}&a_{4}&a_{5}&a_{6}&a_{7}&a_{8}&a_{9}\\
      a_{1}&a_{4}&a_{5}&a_{6}&a_{10}&a_{11}&a_{12}&a_{13}&a_{14}&a_{15}\\
      a_{2}&a_{5}&a_{7}&a_{8}&a_{11}&a_{13}&a_{14}&a_{16}&a_{17}&a_{18}\\
      a_{3}&a_{6}&a_{8}&a_{9}&a_{12}&a_{14}&a_{15}&a_{17}&a_{18}&a_{19}\\
      a_{4}&a_{10}&a_{11}&a_{12}&a_{20}&a_{21}&a_{22}&a_{23}&a_{24}&a_{25}\\
      a_{5}&a_{11}&a_{13}&a_{14}&a_{21}&a_{23}&a_{24}&a_{26}&a_{27}&a_{28}\\
      a_{6}&a_{12}&a_{14}&a_{15}&a_{22}&a_{24}&a_{25}&a_{27}&a_{28}&a_{29}\\
      a_{7}&a_{13}&a_{16}&a_{17}&a_{23}&a_{26}&a_{27}&a_{30}&a_{31}&a_{32}\\
      a_{8}&a_{14}&a_{17}&a_{18}&a_{24}&a_{27}&a_{28}&a_{31}&a_{32}&a_{33}\\
      a_{9}&a_{15}&a_{18}&a_{19}&a_{25}&a_{28}&a_{29}&a_{32}&a_{33}&a_{34}\\
      \end{matrix}\right]$.
      \end{center}
      \vskip .1in
\noindent The catalecticant relates to the rank of $F$ as follows.  First,
\[
\begin{array}{ccc}
[F] \in \VV(I_{s+1}(\Cat(2,2,4))) & \leftrightarrow &h_2(F)=\dim(S/F^\perp)_2= \rk(\Cat(2,2,4)(F) \le \rr(F)\\
  & \leftrightarrow & \exists \ge 10-\rr(F) \mbox{ independent quadrics apolar to } F.
 \end{array}
\]
Secondly, if $h_2(F) = s$, then the space of quadrics in $S$ annihilating $F$ has dimension  $10-s$. 
But then any $0$-dimensional scheme $\Gamma$ that is apolar to $F$ contains at most $10-s$ independent quadrics.  The length of $\Gamma$ is therefore at least $s$.
In particular, by Lemma~\ref{ApolarityRank} and Remark \ref{cactusrankversusrank},
$$h_2(F)\leq \operatorname{cr}(F)\leq \rr(F).$$
Assume, for a moment, that the quadrics in $F^\perp$ define a $0$-dimensional scheme $\Gamma$. If the length of $\Gamma$ equals $h_2(F)$, then  $\Gamma$ computes the cactus rank, i.e.~$\operatorname{cr}(F)=h_2(F)$. If the length of $\Gamma$ is strictly larger than $h_2(F)$, then, in some cases $\Gamma$ computes the cactus rank, in other cases it does not.  Quaternary quartic forms $F$ of type [300a] are of the first kind, while forms $F$ of type [300b] are of the second kind, see Section \ref{section_VSP}.

 In any case, the scheme defined by the quadrics in $F^\perp$ plays a crucial role in our analysis of schemes apolar to $F$, see 
Section \ref{quadraticideals}.
\end{example}
\subsubsection{Motivating Questions}\label{questions}
We highlight some of the questions posed in \cite{Geramita1}, \cite{Geramita2}, and \cite{IK}:
\begin{enumerate}
    \item Problem A of the introduction to \cite{IK} asks how to determine the length of a power sum decomposition of a general polynomial, when there is a unique decomposition of minimal length, and how the closure of the power sum locus relates to catalecticant matrices. We obtain results for quaternary quartics in Section \ref{stratification}. \vskip .05in
    \item Problem B of the introduction to \cite{IK} asks when $\Gor(H)$ is irreducible, and for a description of the Zariski closure and dimension of $\Gor(H)$. We obtain results for quaternary quartics in Section \ref{bstratification}. \vskip .05in
    \item In \cite{Geramita1}, Geramita asks under what conditions the ideal of minors of the catalecticant is prime. 
    In  Proposition~\ref{SecantStrat} we use geometric reasoning to show that $I_7(\Cat(2,2,4))$ and $I_8(\Cat(2,2,4))$ are not prime. \vskip .05in
    \item Conjecture 3.20 of \cite{IK} gives a formula for the dimension of the tangent space to $\bGor(H)$ at a generic point $[F]$, when $\deg(F)$ is even. In the case where $n+1=4$ and $\deg(F)=4$, the conjectured affine dimensions are
   \begin{center}
 \vskip -.1in
\begin{table}[H]
\begin{tabular}{|c|c|c|c|c|c|c|c|}
\hline $h_2 = $& $4$ & $5$ & $6$ & $7$ & $8$ & $9$&$10$ \\
\hline $\dim T_{[F]}= $& $16$ & $20$ & $25$ & $29$ & $32$ & $34$&$35$ \\
\hline
\end{tabular}
\end{table}
 \vskip -.3in
\end{center}
We prove that the conjecture holds for $n+1 = 4 = \deg(F)$ in 
Proposition \ref{vsp proposition}.
\end{enumerate}
 \subsubsection{Secant varieties to the Veronese variety and power sum decompositions}

Let $n = \binom{m+d}{d}$. The Veronese map
\begin{equation}\label{VeroneseMap}
\Pn^m = \Pn(V^*) \stackrel{v_d}{\longrightarrow} \Pn^{n-1}= \Pn(\Sym_d(V^*))
\end{equation}
is defined by $v_d = [f_1:\cdots:f_n]$ with the $f_i$ a basis for $\Sym_d(V^*)$. The map can be thought of intrinsically as 
\[
[l] \in \Pn(V^*) \mapsto [l^d] \in \Pn(\Sym_d(V^*)).
\]
Hence, a point on the image of $v_d$ corresponds to a power of a linear form, a point on the secant line $\alpha[l_1^d]+\beta[l_2^d]$ corresponds to $F=\alpha l_1^d+\beta l_2^d \in \Sym_d(V^*)$ which can be written as a sum $F=(al_1)^d+(bl_2)^d$ of two powers of linear forms, where $a^d=\alpha, b^d=\beta$, and so on. Therefore, we may interpret the $s^{th}$ secant variety 
to the Veronese variety as
\begin{equation}\label{SecantToVeronese}
\Sec_s(v_d(\Pn(V^*))) = \overline{\{[F] \in \Pn(\Sym_d(V^*))\mid F = l_1^d+\cdots l_s^d, \;{\rm for \; some}\;l_1,...,l_s\in V^* \}}
\end{equation}
\begin{definition}\label{borderRank}
The border rank $\operatorname{br}(F)$ of $F \in \Sym_d(V^*)$ is the smallest $s$ such that $F$ lies in the closure of $\Sec_s(v_d(\Pn(V^*)))$. 
\end{definition}

The relationship between the catalecticant and the secant variety to the Veronese variety is classical: 
\begin{lemma}\label{SecVerCat}\cite{Raicu}*{Lemma 2.1}
For $1 \le i \le d$ and $s \ge 1$, 
\[
I_{s+1}(\Cat(i,d-i,n)) \subseteq I(\Sec_s(v_d(\Pn^{n-1}))).
\]
\end{lemma}
So in the situation of interest to us, $I_{s+1}(\Cat(2,2,4)) \subseteq I(\Sec_s(v_4(\Pn^{3}))).$
\begin{remark}\label{rankversuscactusrankversusborderrank}
The border rank is clearly a lower bound for rank.
If the cactus rank of $F$ is computed by a scheme that is a flat limit of smooth schemes, then 
the border rank is also at most the cactus rank.  In $\Pn^3$ any finite locally Gorenstein scheme is smoothable in this sense.  Finally, by Lemma \ref{SecVerCat}, the border rank is at least the catalecticant rank $h_2(F)$.  Therefore, if $F$ is a quaternary quartic,
then 
$$h_2(F)\leq \operatorname{br}(F)\leq \operatorname{cr}(F)\leq \rr(F).$$
\end{remark}

\subsection{Gorenstein rings of low codimension: History and Techniques}\label{subsection_history} For a Gorenstein ring $A=S/I$ of codimension $1$ or $2$, $I$ must be a complete intersection, and in codimension $3$ the Buchsbaum-Eisenbud theorem \cite{be1975generic} shows that $I$ is given by the Pfaffians of a skew-symmetric matrix. In particular, knowing $I$ is equivalent to knowing the matrix of first syzygies. This is obviously also true in codimension $4$: if $F_i$ are the modules in a minimal free resolution, since
\[
\sum_{i=0}^4(-1)^i \rk(F_i)=0,
\]
the minimal free resolution of a codimension $4$ Gorenstein ring $A=S/I$ will have the form
\[
0 \longrightarrow S^1 \stackrel{d_4}{\longrightarrow}F_3 \stackrel{d_3}{\longrightarrow}F_2 \stackrel{d_2}{\longrightarrow}F_1 \stackrel{d_1}{\longrightarrow}S^1 \longrightarrow S/I \longrightarrow 0,
\]
with $F_1 \simeq S^{k+1} \simeq F_3$ and $F_2 \simeq S^{2k}$. In \cite{KMmzeit}, Kustin-Miller show that there is a symmetric map 
\[
F_2 \stackrel{s}{\longrightarrow}F_2^*
\]
of the form 
\[
 \left[ \!
\begin{array}{cc}
  0 & I_{k}   \\
I_k  &0
 \end{array}\! \right]. 
\]
The map $s$ induces a non-degenerate symmetric quadratic form on $F_2$, giving it the structure of an even orthogonal module. This structure is further explored by Reid in \cite{ReidGeneral}, who introduces Spin-Hom varieties, and also by Celikbas-Laxmi-Weyman in \cite{celikbas2020family}. Reid makes several remarks at the conclusion of \cite{ReidGeneral} on the structure of codimension $4$ Gorenstein ideals for various values of $k$: 
\begin{enumerate}
\item $k=3$: Since $I$ is codimension $4$, $I$ must be a complete intersection.
\item $k=4$: Kunz shows in \cite{Ku} that there are no Gorenstein almost complete intersections.
\item $k=5$: Reid asks if any such ideal is a hypersurface section of Pfaffians. 
\item $k=6$: Reid asks if any such ideal is a Kustin-Miller unprojection.
\end{enumerate}
As shown by Vasconcelos-Villereal in \cite{VV}, for $k=5$ the answer to Reid's question is yes when $I$ is generically a complete intersection. Reid also asks if any Gorenstein codimension $4$ ideal with an even number of generators (so $k$ odd) is always a hypersurface section of Pfaffians.
Proposition~\ref{ReidNoAnswer} gives a negative answer to this question. The ideals which provide the answer appear as type $[310]$ in the appendix, and cannot be lifted to irreducible threefolds; we show that such ideals can be lifted to a reducible Gorenstein threefold with components of degrees 6 and 11. An irreducible threefold giving a negative answer is provided by ideals with Betti table of type $[200]$. Celikbas-Laxmi-Weyman give a negative answer to $(4)$ in \cite{laxmi2021spinor}.

A standard way to build Gorenstein rings is via Nagata's technique of {\em idealization}, which takes as input a ring $T$ and the canonical module $\omega_T$. Foxby, Gulliksen and Reiten show the idealization of a level algebra $T$ yields $T \ltimes \omega_T$ which is standard graded and Gorenstein. The idealization satisfies
\[
\codim(T \ltimes \omega_T) = \codim(T)+\type(\omega_T),
\]
so does not produce interesting objects in codimension $4$. However, there is a similar construction known as {\em doubling}, and in \cite{celikbas2020family} Celikbas-Laxmi-Weyman use doubling to construct many examples of codimension $4$ Gorenstein ideals with a small number of generators. We next review the doubling construction, as well as another way of constructing codimension $4$ Gorenstein rings: the bilinkage construction.

\subsection{The doubling construction}\label{DoubleConst}
Assume $S=\CC[x_0,\ldots,x_n]$ and $I\subseteq S$ is codimension $c$ with $S/I$ Cohen-Macaulay and a minimal free resolution $(G_{\bullet},d_{\bullet})$. By \cite{CAwvtAG}*{\S 21}, 
\begin{equation}\label{canonicalModule}
     \omega_{S/I}=\Ext^{c}_{S}(S/I,S(-n-1))
  \end{equation}
is the \emph{canonical module} of $S/I$. 
 Assume, furthermore, that $S/I$ is Gorenstein at all minimal primes (i.e.~satisfies the condition ${\rm G}_0$ cf.\cite{Mar}, \cite {Har}) and consider maps $$\psi:\omega_{S/I}(-\gamma) \to S/I.$$
 
If ${\rm im}(\psi)$ is an ideal $J$ of codimension $c+1$, then, by \cite{CMrings}*{Proposition 3.3.18}, $J$ is a Gorenstein ideal.
The ideal $J$ is then called a \emph{doubling} of $I$ with $\psi$.
By abuse of notation  we also say in this case that the divisor $\VV(J)$ is a doubling in $\VV(I)$. It is not true that every Artinian Gorenstein ideal of codimension $c+1$ is a doubling of some codimension $c$ ideal.             

Letting $M^*$ denote $\Hom_S(M,S(-n-1))$, $G_{\bullet}^{*}$ is a minimal free resolution of $\omega_{S/I}$. A map $\psi:\omega_{S/I}(-\gamma) \to S/I$ yields an induced map $\psi_{\bullet}:G_{\bullet}^{*}(-\gamma) \to G_{\bullet}$:

\vskip -.1in
\begin{equation}\label{doubleConstruction}
   \xymatrixrowsep{50pt}
\xymatrixcolsep{15pt}
\xymatrix{
G_{\bullet}^{*}(-\gamma):  &\omega_{S/I}(-\gamma) \ar[d]^{\psi} & G_c^{*}(-\gamma) \ar[d]^{\psi_0} \ar[l] & G_{c-1}^{*}(-\gamma) \ar[d]^{\psi_1} \ar[l]^{d_c^{*}} &\cdots\ar[l]^{d_{c-1}^{*}} & G_1^{*}(-\gamma) \ar[d]^{\psi_{c-1}} \ar[l]^{d_2^{*}} & G_0^{*}(-\gamma) \ar[d]^{\psi_{c}} \ar[l]^{d_1^{*}} \\
\hskip .25in G_{\bullet}:& S/I                 & G_0 \ar[l]                         & G_1 \ar[l]^{d_1}     &\cdots  \ar[l]^{d_{2}}                         & G_{c-1} \ar[l]^{d_{c-1}}                             & G_c \ar[l]^{d_c}       
} 
    \end{equation}
If $\psi$ is injective, the mapping cone of $\psi$ gives a resolution of  $S/(I+J)$.  If it is minimal, then one can read off the Betti table of $S/J$ from the Betti table of $S/I$.               

\begin{proposition}\label{double} Let $Y\subset \Pn^{n}$ be arithmetically Cohen-Macaulay subscheme of codimension $c$ and regularity $r$, and assume that $Y$ is generically Gorenstein. Assume a map 
$$\psi\in\Hom_Y(\omega_Y(-\gamma),{\mathcal O}_Y), \; {\rm with}\; \gamma \ge 2r+c-n$$ is an inclusion. Then the mapping cone of $\psi$ gives a minimal resolution of the divisor $X=\{\psi=0\}$ as a doubling in $Y$.  
\end{proposition}
\begin{proof} By \cite{CMrings}*{Proposition 3.3.18} a section
$\psi\in\Hom_Y(\omega_Y(-\gamma),{\mathcal O}_Y)$, where $\omega_Y(-\gamma)$ is invertible and not isomorphic to ${\mathcal O}_Y$, is equivalent to the identification of $\omega_Y(-\gamma)$ with a codimension $1$ ideal of ${\mathcal O}_Y$. Denote by $X$ the divisor defined by this ideal. Hence we may identify the doubling construction of Equation~\ref{doubleConstruction} using $\omega_Y$ with the standard mapping cone construction associated to the exact sequence $$0\to \omega_Y(-\gamma)\to \mathcal O_Y\to \mathcal O_X\to 0, $$
which gives a minimal resolution of the structure sheaf of $X$, assuming there is no cancellation in the mapping cone resolution. 
To see that there is no cancellation, let $G_\bullet$ and $G_\bullet^*(-\gamma)$ be the minimal free resolutions appearing in Equation~\ref{doubleConstruction}.
Let $-a_i$ be the maximal negative degree shift in $G_i$, and let $b_i^\vee$ be the maximal degree shift in $G_i^*(-\gamma)$. Hence a sufficient condition for there to be no cancellation in the mapping cone is that 
\[
\begin{array}{ccc}
b_0^\vee & < &-a_c \\
b_1^\vee & < &-a_{c-1} \\
\vdots & < &\vdots \\
b_{c-1}^\vee & < &-a_1 \\
b_c^\vee & <& -a_0
\end{array}
\]
Since $b_i^\vee = a_{i}-n-1-\gamma$, this means it is sufficient to show
\begin{equation}\label{biggestShift}
\gamma > a_i+a_{c-i} -n-1\mbox{ for all } i.
\end{equation}
Since the regularity of $\cO_Y$ is $r$, the maximum shift in $G_i$ is $r+i$. Thus, $a_i \le r+i$, and the maximum value for the right hand side in Equation~\ref{biggestShift} is $(r+i)+(r+c-i)-n-1 = 2r+c-n-1$. Therefore the condition that 
\begin{equation}\label{eqn_condition_exactnessOfMappingCone}
    \gamma \ge 2r+c+1-n-1=2r+c-n
\end{equation}
is sufficient to force exactness of the mapping cone. 
\end{proof}
\noindent We formulate this geometrically as  
\begin{corollary}\label{doublingWithRightGrading}
Let $Y\subset \Pn^{n}$ be arithmetically Cohen-Macaulay subscheme of regularity $r$, that is generically Gorenstein. 
Assume that a map $\psi \in \Hom_Y(\omega_Y(\dim(Y)-2r), \cO_Y)$ is an inclusion, then it induces a minimal resolution of the divisor $X=\{\psi=0\}$ as a doubling of $Y$.

\end{corollary}
\begin{example}
Consider the case of an ideal $I_\Gamma$ of six points. If $\Gamma$ does not contain four collinear points, then Proposition~\ref{prop_BettiGamma} shows that $S/I_\Gamma$ has regularity $2$ and is of type (5), (6), (7a) or (7b) in Table~\ref{tablepoints}. A computation shows that a general element of $\Hom(\omega_{S/I}(-4), S/I)$ is injective.
Therefore Corollary~\ref{doublingWithRightGrading} gives the corresponding doublings, which appear in Table~\ref{tableremaining} of Appendix~\ref{Appendix} as, respectively, Type 2.4, Type 2.3, Type 2.6.

Now suppose the six points span $\Pn^3$ and four are collinear; to cut out the four points $I_\Gamma$ must contain a quartic and $S/I_\Gamma$ has Betti table  
\[
    \begin{matrix}&&&&\cr
 1 & - & - & - &  \cr
 - & 5& 6 & 2 &  \cr
 - & -& - &- &  \cr
    - & 1& 2 & 1 &  \cr
\cr
    \end{matrix}. 
    \]
    A computation shows that a general element of $\Hom(\omega_{S/I}(-2), S/I)$ is not injective. 
    
    When $\gamma \ge 3$ a general element is injective. The four Betti tables below are the results of the doubling construction with a general element of $\Hom(\omega_{S/I}(-\gamma), S/I)$ for $\gamma \in \{3,4,5,6\}$. 
    \[
    \begin{matrix}&&&&&\cr
 1 & 1 & - & - & -& \cr
 - & 3& 5 & 2 &-&  \cr
 - & -& - &- &  -&\cr
   - & 2& 5 & 3 &-&  \cr 
   - & -& - & 1 &1&  \cr
\cr
    \end{matrix}   \;\;\;\;\;\; \begin{matrix}&&&&\cr
 
 1 & - & - & - & - \cr
 - & 6& 8 &3 &-  \cr
- & -& - &- &  -&\cr 
 -&3 & 8 & 6 & -  \cr
 -&- &- & -& 1   \cr
\cr
  \end{matrix}\;\;\;\;\;\; \begin{matrix}&&&&\cr
  1 & - & - & - & - \cr
 - & 5& 6 & 2 &-  \cr
 - &1 & 2& 1 &-  \cr
  -&1 &2 & 1 & -  \cr
 -&2 & 6& 5 & -   \cr
 -&- &- & -& 1   \cr
\cr  \end{matrix}\;\;\;\;\;\; 
\begin{matrix}&&&&\cr
  1 & - & - & - & - \cr
 - & 5& 6 & 2 &-  \cr
-&- & - & - & -  \cr
- &2 & 4& 2 &-  \cr
 -&- & - & - & -  \cr
  -&2 & 6& 5 & -   \cr
 -&- &- & -& 1   \cr
\cr
    \end{matrix}
\] 
For $\gamma = 3$, the leftmost Betti table shows that cancellation occurs in the mapping cone. For the two middle Betti tables, $\gamma \in \{4,5\}$ and no cancellation occurs in the mapping cone. However, the potential for cancellation exists. In the rightmost Betti table, $\gamma = 6 = 2r$ and there is no potential for cancellation. The mapping cone resolution is minimal, as guaranteed by Corollary~\ref{doublingWithRightGrading}.
\end{example}

\begin{construction}\label{doubling and component of complete intersection}

We can construct a doubling  of a finite scheme $\Gamma$ or more generally an an arithmetic Cohen-Macaulay scheme with ideal $I_\Gamma$ and codimension $c$. If 
$$\psi\in\Hom_{S/I_{\Gamma}}(\omega_{S/{I_\Gamma}}(-\gamma),S/I_{\Gamma})$$
is nonzero and $\gamma$ satisfies condition (\ref{eqn_condition_exactnessOfMappingCone}), then $\psi$ induces a doubling of $I_{\Gamma}$, so we suggest a way to find such maps $\psi$.

Let $J\subset I_{\Gamma}$ be a complete intersection of codimension $c$ and type $(d_1,...,d_c)$ , then
\begin{equation}
    \omega_{S/I_{\Gamma}}=\Ext^{c}_{S}(S/I_{\Gamma},S(-n-1))\simeq\Hom_{S}(S/I_{\Gamma},S/J(d_1+...+d_c-n-1)).
\end{equation}
On the other hand, let $I_{P}=J\colon I_{\Gamma}$. Then
\begin{equation}
    I_{P}\simeq\Hom_{S}(I_{\Gamma}, J).
\end{equation}
By \cite{CAwvtAG}*{Theorem 21.23},
\begin{equation}
    I_{P}\otimes S/J\simeq\Hom_{S/J}(S/I_{\Gamma}, S/J).
\end{equation}
Hence
\begin{equation}
    I_{P}\otimes S/I_{\Gamma}\simeq \omega_{S/I_{\Gamma}}(-d_1-...-d_c+n+1).
\end{equation}

Now consider maps $$\psi'\in \Hom_{S}(I_{P}(d_{1}+\dots+d_{c}-n-1-\gamma),S),$$
and their restrictions
$$\psi=\psi'|_{S/I_\Gamma}\in \Hom_{S/I_\Gamma}(I_{P}/I_\Gamma(d_{1}+\dots+d_{c}-n-1-\gamma),S/I_\Gamma).$$
If $\psi$ is nonzero, then  
$I_{\Gamma}+\mathrm{im}(\psi')$ is a doubling of $I_{\Gamma}$.

If $\dim \Gamma\geq 1$, then it makes sense to talk about divisors on $\Gamma$. We proceed as above and denote by $P$ the scheme defined by $J:I_{\Gamma}$. Note that in particular, this means $P$ is determined by the choice of $J$. In this formulation a doubling $X$ of $\Gamma$ is a divisor bilinked to $P_\Gamma:=P\cap \Gamma$ in $\Gamma$, 
$$X\in |P_\Gamma-(d_{1}+\dots+d_{c}-n-1-\gamma)H|$$
as discussed below in \ref{unpr}.
\end{construction}

\begin{proposition}\label{C4R4areDoubles}
For every Betti table $B$ for a codimension = 4 = regularity Gorenstein ideal, there is an $F \in {\mathcal F}_B$, such that $F^\perp$ can be obtained by a doubling.
\end{proposition}
\begin{proof}
Proposition~\ref{prop_BettiGamma} classifies reduced zero-schemes $\Gamma$ in $\Pn^3$ with $4 \le \deg(\Gamma) \le 9$ satisfying certain geometric constraints, and proves that in these situations $\Gamma$ has regularity $2$, except for the $(2,2,2)$ complete intersection. In the case of regularity $2$, we are in the setting of Corollary~\ref{doublingWithRightGrading}, and a computation with an element of each class shows there is a $\psi$ yielding an inclusion, cf. {\tt Macaulay2}, 
\cite{M2}*{Package QuaternaryQuartics}.  The complete intersection $(2,2,2,2)$ is trivially a doubling of the complete intersection $(2,2,2)$.  
\end{proof}

For a geometric perspective on this, in Corollary~\ref{liftofCGKK} we prove that a general Artinian Gorenstein ideal from a family of Betti table type in Table~\ref{tableCGKK}, with the exception of $[300a]$, can be obtained by doubling. For other recent work on doubling of ACI's, see \cite{laxmi}. 

Here, let us illustrate the interesting case of Betti table of type [300]. We will see that ideals with the same Betti table may admit a different behavior with respect to the doubling construction.
\begin{example}[Type {[300a]} is not a doubling]\label{eg_exceptionCGKK4}
A quartic $F$ of type [300a] has rank $8$ and $A_F$ has  Betti table
\[
 \begin{matrix}1 & - & - & - & -& \cr
       - & 3& - & - & -& \cr
       - &4 & 12& 4 & -& \cr
       - &- & -& 3 & -& \cr
       - &- & -& - & 1  &\cr
\end{matrix}.
\]
The quadric part $J$ of $F^{\perp}$ is a complete intersection that is a scheme of length $8$, which is a minimal apolar scheme to $F$. In that case the doubling of $J$ is always a complete intersection.  Hence, $F^{\perp}$ is not a doubling of $J$. There is however a family of rings $A_F$ with Betti table of type $[300]$ which is the doubling of an  almost complete intersection
ideal $I_{\Gamma}$ of codimension $3$. 
\end{example}

\begin{example}[Type {[300b]} is a doubling]\label{rank7}
 
This example shows that if $F$ is of type [300b] (see Definition \ref{FBdefinition}), then $F^{\perp}$ can be constructed as a doubling from its apolar ideal. Assume $I_\Gamma$ defines seven points in a complete intersection of three quadrics. Then $I_\Gamma$ is an almost complete intersection, and $S/I_{\Gamma}$ has Betti diagram
\begin{align*}
        \begin{matrix}1 & - & - & - &  \cr
               - & 3& - & - &  \cr
               - &1 & 6& 3 &  \cr
        \end{matrix}.
\end{align*}
Let $I_{P}=J:I_{\Gamma}$, where $J\subset I_{\Gamma}$ is a complete intersection $(2,2,2)$. Then $I_{P}$ defines one point $P$ in $\PP^{3}$. Note that $\reg S/I_{\Gamma}=2$, so by condition (\ref{eqn_condition_exactnessOfMappingCone}), $\gamma\geq 4$ guarantees that the mapping cone has no cancellation. If we want to obtain a doubling having regularity $4$, we choose $\gamma=4$. In particular, a nonzero map $\psi'\in \Hom_{S}(I_{P}(-2),S)$ defines a doubling of $I_{\Gamma}$.

Now assume $S/F^{\perp}$ has Betti table $[300b]$ and that $F^\perp$ contains the ideal $I_{\Gamma}$.
If $\psi'\in\Hom_{S}(I_{P}(-2),F^{\perp})$ is nonzero,  then $I_{\Gamma}+\mathrm{im}(\psi')\subseteq F^{\perp}$ is a doubling of $I_\Gamma$.  
Because $A_F$ and $S/(I_{\Gamma}+\mathrm{im}(\psi'))$ have the same Betti table,  they are actually equal. Hence, $F^{\perp}$ is a doubling of $I_{\Gamma}$.
\end{example}

\begin{example}[Type {[300c]} is a doubling]\label{Eg_300c}
Let $I_{\Gamma}$ be an ideal defining seven points with three on a line. Then $S/I_\Gamma$ has Betti table 
\begin{align*}
        \begin{matrix}1 & - & - & - &  \cr
               - & 3& - & - &  \cr
               - &1 & 6& 3 &  \cr
               \end{matrix},
\end{align*}
and $S/Q_{\Gamma}$, where $Q_{\Gamma}\subseteq I_{\Gamma}$ is the ideal generated by the quadrics in $I_{\Gamma}$, has Betti table 
\begin{equation}
    \begin{matrix}&&&&\cr
 1 & - & - & - &  \cr
 - & 3& - & - &  \cr
 - &- & 4& 2 &  \cr
\cr
    \end{matrix}.
\end{equation}
In this case $Q_{\Gamma}$ is the ideal of a line and four points.

 We may take $J\subseteq I_{\Gamma}$ to be a complete intersection $(2,2,3)$ and proceed as in the construction \ref{doubling and component of complete 
intersection}. Then $S/I_P$ where $I_{P}=J:I_{\Gamma}$, has Betti table 
\begin{equation}
    \begin{matrix}&&&&\cr
 1 & - & - & - &  \cr
 - & 5& 5 & - &  \cr
 - &- & -& 1 &  \cr
\cr
    \end{matrix}.
\end{equation}
\vskip -.1in
As in the previous example, $\gamma=4$ guarantees that we get a doubling of $I_{\Gamma}$ which has regularity $4$. Therefore, a nonzero map $\psi'\in\Hom_{S}(I_{P}(-1),S)$, defines an ideal $\mathrm{im}(\psi')+I_{\Gamma}$ that is a doubling of $I_\Gamma$ of type [300c].
\end{example}

\subsection{Bilinkage and unprojection}\label{unpr} Another classical construction of Gorenstein rings of codimension $4$ 
is by means 
of bilinkage.

\begin{definition} Let $R$ be a Gorenstein ring, and let $A$ and $B$ be ideals of R.
Then $A$ and $B$ are said to be linked by an ideal $I$ if $I\subset  A, B$ and if $A = ( I : B)$
and $B = ( I : A)$. We will similarly say that schemes defined by $A$ and $B$ are linked via the scheme defined by $I$
\end{definition}
Linkage is a very general construction that is usually narrowed down by adding assumptions on $I$. In this way one distinguishes different types of linkage.
When $I$ is a complete intersection it is the classical linkage described in \cite{PS85}.
The most general case is when $I$ is Cohen-Macaulay and generically Gorenstein \cite{CMlinkage}.
A composition of two linkages of the same type is called a bilinkage. This notion has been introduced because varieties related by an even number of linkages share more common properties with each other.

Linkages are helpful in the construction of varieties as they provide a way to relate two varieties (given by $A$ and $B$) of which one can be easier to construct than the other. We will consider only special types of bilinkages and use them to construct AG rings in Section \ref{section_reducible_liftings}.

\subsubsection{Biliaison} Building on the notion of bilinkage Hartshorne in \cite{Har2} introduces the notions of Gorenstein biliaison and generalised biliaison.

\begin{definition}Let $X$ and $Y$ be equidimensional closed subschemes of dimension $r$ of $\mathbb{P}^n$
We say that $X$ is obtained by a biliaison of height $h$ from $Y$ if there exists an ACM scheme $Z$ in $\mathbb{P}^n$, of dimension $r + 1$, containing $X$ and $Y$, and so that $X \in Y + hH$ as (generalized cf. \cite{Har}) divisors on $Z$, where $H$ denotes the hyperplane class. 
When $Z$ is arithmetically Gorenstein we call it a Gorenstein biliason (cf. \cite{Har2}).
\end{definition} 

\begin{remark} A biliaison between $X$ and $Y$ in $Z$ can be seen as a composition of two linkages in the following way. Let $I_X, I_Y, I_Z$ be the corresponding ideals. Choose a hypersurface $L_d$ of degree $d$ containing $X$. Then $I_X$ is linked to some $I_T$ in $I_Z+<L_d>$. Then consider a linkage of $I_T$ by $I_Z+L_{d+h}$ for  $L_{d+h}$ a hypersurface of degree $d+h$ containing $T$.  The latter yields an ideal $I_Y'$ with $Y'$ a divisor in $Z$ such that $Y'\simeq Y$. Choosing $d$ large enough we will be able to construct  $Y'=Y$.
\end{remark}

\begin{remark} \label{doubling via bilaison} Note that the construction of doubling in Example \ref{doubling and component of complete intersection} is in fact done by means of biliaison in $\Gamma$ of  $P\cap \Gamma$, where $P$ is linked to $\Gamma$ via a complete intersection. The biliaison for some choices of complete intersections might however be of height $0$, meaning that $P\cap \Gamma$ is already a doubling. 
\end{remark}

\subsubsection{Unprojections} A special case of biliaison is provided by the unprojection technique of Papadakis and Reid \cite{PR}. The unprojection in projective space is originally the inverse
of the central projection from a point on it. So having the data of the image of a projection and the proper image of the indeterminacy point, we can reconstruct the original variety.

Let $A$ be a local Gorenstein ring, and $I$ an ideal defining $X=\VV(I)\subset \operatorname{Spec}(A)$ a Gorenstein subscheme of codimension $1$. As in \cite{PR}*{Lem.~1.1} we have a natural injective map $s:I\to \omega_A$ to the canonical module.
Choose a set of generators $f_1,\dots, f_k$ of $I$ and write $s(f_i) = g_i\in \omega_A=A$ for the corresponding generators of $\omega_A$.
Then the unprojection ring (which is Gorenstein by \cite{PR}*{Thm.~1.5}) of $I$ in $A$ is
$$A[s]=A[x]/(xf_i-g_i),$$
where $x$ is an indeterminate; $A[s]$ is usually no longer local. This construction of unprojection can be globalised to the projective case \cite{PR}*{\S 2.4}.
Indeed, when $A=S/J$ with $J$ homogeneous (as in our case) we can consider the projective scheme $\operatorname{Proj}(A)$ and the subscheme defined by $I$. This allows us to lift local Gorenstein schemes to subvarietes in projective space.

The unprojection of $D\subset X \subset \mathbb{P}^N$ can be obtained as a Gorenstein biliaison \cite{PS74} of the cone $D_1\subset \mathbb{P}^{N+1}$ 
(cf.~\cite{KK1}*{\S 5}).

\subsubsection{The Kustin-Miller construction} Another special case is the biliason from a complete intersection of codimension 4 in a complete intersection of codimension 3. This construction has been performed in \cite{kustin1982structure}, were Kustin and Miller use biliaison to provide a generic family of codimension $4$ Gorenstein ideals. This construction has also later become the first model of unprojection studied by Papadakis and Reid (cf. \cite{PR}*{\S 2.8}).
A nice summary of Kustin-Miller's construction appears in  \cite{laxmi2021spinor}*{\S 6}, so our treatment here is terse. Let $\tilde{S}=\kk[x_i,a_{jk},v]$ be a polynomial ring, where $1 \le i,j \le 4$ and $1\le j \leq 3$. The Kustin-Miller family is associated to a $3\times 4$ matrix $M=\begin{pmatrix}
    a_{jk}
    \end{pmatrix}$, a $4$-vector $\mathbf{x}=\begin{pmatrix}
    x_i
    \end{pmatrix}$ and a variable $v$. Let
\begin{align*}
    \begin{pmatrix}
    q_1\\q_2\\q_3
    \end{pmatrix}=M\mathbf{x}, ~\begin{pmatrix}
    h_1\\h_2\\h_3\\h_4
    \end{pmatrix}=\mathbf{x}v+\wedge^3M .
\end{align*}
Then minimal free resolution for $I=\langle q_1,q_2,q_3,h_1,h_2,h_3,h_4\rangle$ is given by
\begin{equation*}
\begin{tikzcd}
\mathbb{F}_{\bullet}: & \tilde{S}/I  & \tilde{S} \arrow[l] & \tilde{S}^7 \arrow[l, "d_1"'] & \tilde{S}^{12} \arrow[l, "d_2"'] & \tilde{S}^{7} \arrow[l, "sd_2"'] & \tilde{S} \arrow[l, "d_1^{t}"'] & 0 \arrow[l]
\end{tikzcd},
\end{equation*}
\[
\mbox{ where }d_1=\begin{bmatrix}q_1,q_2,q_3,h_1,h_2,h_3,h_4\end{bmatrix}, s=\begin{bmatrix}0& I_6\\
I_6 & 0\end{bmatrix} \mbox{ and } d_2=\begin{bmatrix}
    Q& \wedge^2 M & -vI_{3}\\
    0 & \mathcal{K}_2 & M^{t}
    \end{bmatrix}.
\]
The blocks of $d_2$ are
\begin{align*}
    Q&=\begin{bmatrix}
    -q_2 & -q_3 & 0\\
    q_1 & 0 & -q_3\\
    0 & q_1 & q_2
    \end{bmatrix},\\
    \wedge^2M&=\begin{bmatrix}
    M_{23;34}  & M_{23;24}  & M_{23;23}  & M_{23;14}  & M_{23;13}  & M_{23;12}  \\
-M_{13;34} & -M_{13;24} & -M_{13;23} & -M_{13;14} & -M_{13;13} & -M_{13;12} \\
M_{12;34}  & M_{12;24}  & M_{12;23}  & M_{12;14}  & M_{12;13}  & M_{12;12}
    \end{bmatrix},\\
    \mathcal{K}_2&=\begin{bmatrix}
    -x_2       & x_3        & -x_4       & 0          & 0          & 0          \\
x_1        & 0          & 0          & -x_3       & x_4        & 0          \\
0          & -x_1       & 0          & x_2        & 0          & -x_4       \\
0          & 0          & x_1        & 0          & -x_2       & x_3
    \end{bmatrix}.
\end{align*}

\begin{remark}
Example \ref{eg_exceptionCGKK4} i.e.~example of type [300a] is a specialization of a Kustin-Miller construction. In this case,
\begin{align*}
    \mathbf{x}=\begin{pmatrix}
    x_0\\x_1\\x_2\\x_3
    \end{pmatrix},~M=\begin{bmatrix}
    x_0& 0 & 0& -x_3\\
    0& x_1 & 0& -4x_3\\
    0& 0 & x_2& -9x_3
    \end{bmatrix},~v=0.
\end{align*}
In particular, for these examples there is a biliaison in a complete intersection of three quadrics relating these ideals with ideals of a codimension $4$ complete intersection of linear forms.
\end{remark}
To deal both with doubling and bilinkage we use the adjunction relation between the canonical divisor on a union of varieties of the same dimension and the canonical divisor on each component.  
Algebraically, this is the following linkage relation in Gorenstein rings that we have already seen:
\begin{theorem}\label{antican}\cite{CAwvtAG}*{Thm. 21.23} Let $A$ be a local Gorenstein ring and $I\subset A$ an ideal of codimension $0$. Let $B=A/I$ and $J=(0:I)$. Then $J=Hom_A(B,A)$.  
If $B$ is Cohen-Macaulay, then $J$ is a canonical module for B. In particular,
$B$ is Gorenstein and $J$ is a canonical module for $B$ if and only if $J$ is a principal ideal in $A$.
\end{theorem}
If $X\cup X'$ is an arithmetic Gorenstein scheme, e.g.~a complete intersection, and the component $X$  is arithmetic Cohen-Macaulay. Then one may apply this theorem to the homogeneous coordinate rings $A$ and $B$ of $X\cup X'$ and $X$, respectively.

Geometrically we formulate this in terms of sheaves:
$$\omega_X\otimes \cO_{X}(X\cap X')\equiv \omega_{X\cup X'}\otimes \cO_{X},$$
or in terms of Cartier divisors on $X$, when $X$ is arithmetic Gorenstein
$$K_X+X'\cap X=K_{X\cup X'}|_X.$$

For generalizations to arithmetic Cohen-Macaulay varieties $Y$ and $X$ that are Gorenstein in codimension $0$, and generalized divisors see \cites {Har, Mar}). We do not use Theorem \ref{antican} in any proofs, but both in algebraic and in geometric formulation it is an important guide in constructions of reducible varieties with trivial canonical bundle:  Each component  intersects the union of the others in an anticanonical divisor.

 \subsection{AG varieties of codimension $3$ in quadrics}\label{Tom}
 In Section \ref{Pfaffian} we construct codimension $4$ AG varieties as varieties of codimension $3$ in a smooth quadric.  On the quadric these varieties are Pfaffian or Lagrangian degeneracy loci.
This allows us to
lift Artinian Gorenstein rings of codimension and regularity $4$ to Calabi-Yau varieties in quadrics. 

For example recall that Reid constructed codimension $4$ subvarieties by ``Tom and Jerry'' unprojection introduced \cite{Reid}. The unprojection data are a codimension $3$ scheme $Y$ defined by a $5 \times 5$ Pfaffian ideal, containing a codimension $4$ complete intersection. This data is strongly related to the lifting of a del Pezzo surface of degree $6$ to the Fano threefold $\mathbb{P}^2\times \mathbb{P}^2$ or $\mathbb{P}^1\times \mathbb{P}^1\times \mathbb{P}^1$. 
 
 In Section \ref{Pfaffian}, we  interpret these constructions by  describing embeddings of $\mathbb{P}^1\times \mathbb{P}^1\times \mathbb{P}^1$ and $\mathbb{P}^2\times \mathbb{P}^2$ in a $6$-dimensional quadric in $\mathbb{P}^7$. We use the Pfaffian and Lagrangian degeneration constructions \cite{EPW} involving spinor bundles on quadrics. 
 
 Recall from \cite{Ot} that a spinor bundle on a smooth quadric $Q_n$
is obtained as a pullback of the universal bundle by the natural maps
\[
Q_{2n+1}\to G( 2^{k},2^{k+1}),
\]
and two maps 
\[
Q_{2n}\to G( 2^{k-1},2^k)
\]
defined by maximal isotropic linear subspaces of the quadrics in \cite{Ot}*{Cor.~1.2}.
On a $(2n-1)$-dimensional quadric there is one spinor bundle, of rank $(2n-1)$, while for a $2n$-dimensional quadric there are two spinor bundles of rank $(2n-1)$. On $Q_1= \mathbb{P}^1$ it is $\mathcal{O}(1)$. On $Q_2= \mathbb{P}^1 \times \mathbb{P}^1$ they are $\mathcal{O}(1,0)$ and $\mathcal{O}(0,1)$. On $Q_3$, which is the Lagrangian Grassmannian $LG(2, 4)$, it is the tautological quotient bundle. On $Q_4= G(2, 4)$ they are the tautological quotient bundle and the dual of the tautological subbundle.
\section{Apolar curves and Betti tables of point sets}\label{bettitables}
To determine the rank and the variety of minimal power sum presentations of a quaternary quartic $F$, we will compare the Betti numbers of  $A_F$ and the Betti numbers of possible apolar sets of points.  If a form $F$ is apolar to a curve of small degree, then the rank of $F$ is often computed by a set of points on this curve. 
\begin{lemma}\label{curves} Assume a quaternary quartic form $F$ is apolar to a curve $C$.
\begin{enumerate}
    \item If $C$ is a line and $F$ is not apolar to any scheme of length less than $3$, then $F$ has rank $3$ and $VSP(F,3)=\Pn^1$. 
\vskip .05in
\item Let $C$ be a conic or two intersecting lines and assume $F$ is not apolar to any scheme of length less than $5$.  Then $F$ has cactus rank $5$, and if $C$ is smooth, $F$ has rank $5$ and $VSP(F,5)=\Pn^1$. 
\vskip .05in
\item Let  $C$ be a twisted cubic or a reduced curve defined by a determintal net of quadrics and assume $F$ is not apolar to any scheme of length less than $7$. Then $F$ has cactus rank $7$, and if $C$ is smooth, then $F$ has rank $7$ and $VSP(F,7)=\Pn^1$.
\vskip .05in
\item Assume that $C$ is an elliptic quartic curve or any reduced complete intersection of two quadrics, and that $F$ is not apolar to any scheme of length less than $8$.  Then $F$ is apolar to a scheme $Z$ of length $8$ on $C$.  \vskip .05in
\begin{enumerate}
    \item If $C$ is smooth and $2Z \in C$ is a quartic section of $C$, then $Z$ is the unique apolar scheme of length $8$ to $C$.  \vskip .05in
    \item If $C$ is smooth and $2Z \in C$ is not a quartic section of $C$, then $F$ is apolar to exactly two schemes of length $8$. \vskip .05in
\end{enumerate}
\end{enumerate}
\end{lemma}
\begin{proof} When the quartic form $F$ is apolar to a curve $C$, the rank of $F$ is the minimal cardinality of point sets on $C$ apolar to $F$.  
When $C$ is rational of degree $d$, then the restriction $F_C$ is a binary form of degree $4d$. Any binary form of degree $4d$ is apolar to a scheme of length $\leq 2d+1$, and if it is not apolar to a scheme of length strictly less, then it is apolar to a pencil of schemes of rank $2d+1$ of which the general is smooth,  cf. \cites{Sylvester,RS}.  
Furthermore, since the $(2d-1)$-secant variety of $C$ is a hypersurface,  in a pencil of forms $F$ there is at least one form $F_0$ of cactus rank $\leq 2d$. 

Let $C$ be a reducible curve $C=C_1\cup C_2$, whose components intersect in one point $p$ and assume $F$ is apolar to $C$, then $F$ has a pencil of decompositions $F=F_1+F_2$, where $F_1$ is apolar to $C_1$ and $F_2$ is apolar to $C_2$.  So if $C$ is a reducible conic or twisted cubic, we may assume $C_1$ is a line, and choose $F_1$ to have cactus rank $2$. In the conic case, we  conclude that $F_2$ has rank $3$, and therefore $F$ has cactus rank $5$.
In the twisted cubic case, we conclude similarly that $F_2$ has rank $5$ if $C_2$ is irreducible.  If $C_2$ is reducible, then $F_2$ is apolar to a reducible conic and therefore, as above, of cactus rank $7$.  Summing up, we see that $F$ has cactus rank $9$ also in this reducible case.

In the fourth case of the lemma, the curve $C$ is a complete intersection of two quadric surfaces.
When $C$ is nonsingular, it is an elliptic quartic curve, and the restriction $F_C$ is a form on the sections of a line bundle ${\mathcal O}_C(D)$ of degree $16$ on $C$. The linear system $|D|$ embeds $C\subset \Pn(H^0({\mathcal O}_C(D))^*)$, and $F_C$, up to scalar, is a point in this ambient space.

The line bundle ${\mathcal O}_C(D)$ decomposes into line bundles  ${\mathcal O}_C(B)\otimes {\mathcal O}_C(D-B)$ where $B$ has degree $8$. Each decomposition defines a determinantal hypersurface in $\Pn^{15}=\Pn(H^0({\mathcal O}_C(D))^*)$ defined by an $8\times 8$-determinant with linear entries, such that each row vanishes on a divisor in $|B|$ and each column vanishes on a divisor in $|D-B|$, cf. \cite{Fisher}*{Thm. 1.3, Lemma 2.9} and \cite{Room}.
The pencil of decompositions generates a pencil of hypersurfaces, one through each point in $\Pn^{15}$ outside the intersection of the hypersurfaces.  The latter is the variety of $7$-secant $\Pn^6$'s to $C$.
A general point on a determinantal hypersurface lies in the spans of a unique divisor in $|B|$ and in $|D-B|$ respectively. So $2B=D$ and the two divisors coincide.

If $C$ is an irreducible but singular complete intersection of degree $4$, then $C$ is the projection of a rational normal quartic curve.  In the $4$-tuple embedding $v_4(C)$, it is a projection of a rational normal curve of degree $16$, from a point on a secant.  The $7$-secant variety to this rational normal curve is a hypersurface that maps onto the linear span of $v_4(C)$, so $F$ must lie in a $7$-secant, and therefore have cactus rank at most $8$.
When $C$ is a reducible complete intersection of two quadrics, then the components are rational and any component intersect the unions of residual components in a scheme of length $2$.  So we may assume the $C=C_1\cup C_2$, where $C_1$, and $C_2$ are possibly reducible rational curves of degree $d_1$ and $d_2$ with $d_1+d_2=4$.  The form $F$ now has a net of decompositions $F=F_1+F_2$ such that the $F_i$ are apolar to $C_i$.  Now there is, as above, a finite set of decompostions such that $F_1$ has cactus rank $2d_1$ and $F_2$ has cactus rank $2d_2$. We conclude that $F$ has cactus rank $2(d_1+d_2)=8$.
\end{proof}

 In Section \ref{section_VSP} we will compute the rank of a quaternary quartic form $F$ under the assumption that it is general among the forms apolar to the space of quadrics in $F^{\perp}$. The rank will be the cardinality of a set of points $\Gamma$ contained in the locus defined by the quadrics in $F^{\perp}$. 
 
 When the quadrics in $F^{\perp}$ generate the ideal of a connected curve, this curve is one of the curves covered by Lemma \ref{curves} and $\Gamma$ is a set of points as described in that lemma. Otherwise, the set of points $\Gamma$ will satisfy certain generality conditions as listed in the next proposition.
 \vskip -.5cm
\begin{center}
\begin{footnotesize}
\begin{table}[H]
\begin{tabular}{ |c|c|c| } 
 \hline
  \begin{tabular}{c}
   
    $\begin{matrix}&&&&\cr
    1 & - & - & - &  \cr
                                               - & 6 & 8 & 3 &  \cr
                                               \cr
                                             \end{matrix}$ \\
    1) four points that span $\Pn^3$
    
  \end{tabular} &
                  \begin{tabular}{c}
 
 $\begin{matrix}&&&&\cr
 1 & - & - & - &  \cr
                                               - & 5& 5 & - &  \cr
                                               - &- & - & 1 &  \cr
                                               \cr
                                             \end{matrix}$ \\
                    
2) five points, no\\ four in a plane
\end{tabular} &
                                    \begin{tabular}{c}
 $ \begin{matrix}&&&&\cr
 1 & - & - & - &  \cr
                                               - & 5& 5& 1 &  \cr
                                               - &- & 1 &1 &  \cr
\cr
               \end{matrix}$ \\

     3) five points, four in a \\ plane, no three collinear                                
 \end{tabular}  \\
 \hline
  \begin{tabular}{c}
 
 $\begin{matrix}&&&&\cr
 1 & - & - & - &  \cr
                                               - & 5& 6& 2 &  \cr
                                               - &1& 2 &1 &  \cr
\cr
               \end{matrix}$ \\
    4) five points,\\ three in a line\\
 \end{tabular} &
                 \begin{tabular}{c}

 $\begin{matrix}&&&&\cr
 1 & - & - & - &  \cr
                                               - & 4& 2 & - &  \cr
                                               - &- & 3& 2 &  \cr
\cr
               \end{matrix}$ \\
                   5) six points,\\ no five in a plane\\               
   \end{tabular}  &
  \begin{tabular}{c}

      $\begin{matrix}&&&&\cr
      1 & - & - & - &  \cr
                                               - & 4& 3& - &  \cr
                                               - &1& 3& 2 &  \cr
\cr
               \end{matrix}$ \\
            6) six points,\\ three in a line\\             
 \end{tabular}  \\
 \hline
  \begin{tabular}{c}
   
$ \begin{matrix}&&&&\cr
1 & - & - & - &  \cr
                                               - & 4& 4 & 1 &  \cr
                                               - &2 & 4& 2 &  \cr
\cr
               \end{matrix}$ \\
     7a)  six points, five in a plane, or\\
     7b) three on each of two skew lines\\         
  \end{tabular}  &
\begin{tabular}{c}
         
 $\begin{matrix}&&&&\cr
 1 & - & - & - &  \cr
                                               - & 3& - & - &  \cr
                                               - &1 & 6& 3 &  \cr
\cr
               \end{matrix}$ \\

 8a)  seven points in a CI, or \\
 8b) seven points, three collinear
  \end{tabular}  &
 \begin{tabular}{c}
    
$ \begin{matrix}&&&&\cr
1 & - & - & - &  \cr
                                               - & 3& 1& - &  \cr
                                               - &2 & 6& 3 &  \cr
\cr
               \end{matrix}$ \\
               9) seven points,\\ five in a plane\\
                          
  \end{tabular}  \\
  \hline
 \begin{tabular}{c}
           
 $\begin{matrix}&&&&\cr
 1 & - & - & - &  \cr
                                               - & 3& 2& - &  \cr
                                               - &3& 6& 3 &  \cr
\cr
               \end{matrix}$ \\
           10) seven points on a\\determinantal cubic curve 

              \end{tabular}  &
         \begin{tabular}{c}
       
 $\begin{matrix}&&&&\cr
 1 & - & - & - &  \cr
                                               - & 3& 3& 1 &  \cr
                                               - &4& 7& 3 &  \cr
\cr
               \end{matrix}$ \\
               11) seven points,\\ six in a plane\\
           
         \end{tabular}  &
        \begin{tabular}{c}
   
 $\begin{matrix}&&&&\cr
 1 & - & - & - &  \cr
                                               - & 3& -& - &  \cr
                                               - &-& 3& - &  \cr
                                               - &-& -& 1 &  \cr
               \end{matrix}$ \\
               
     12) CI of eight points\\ 
 \end{tabular}  \\
\hline
  \begin{tabular}{c}
$ \begin{matrix}&&&&\cr
1 & - & - & - &  \cr
                                               - & 2& -& - &  \cr
                                               - &4& 9& 4 &  \cr
\cr
                                                              \end{matrix}$ \\
13)  eight    points, no six  \\in a plane, not a CI\\
                                                              
  \end{tabular}  &
           \begin{tabular}{c}
             $\begin{matrix}&&&&\cr
             1 & - & - & - &  \cr
                                               - & 2& 1& - &  \cr
                                               - &5& 9& 4 &  \cr
\cr
                                                            \end{matrix}$ \\
              14) eight points, six in a plane, or\\
               five in a plane and three in a line\\

                                                         \end{tabular} &
                \begin{tabular}{c}
               $ \begin{matrix}&&&&\cr
1 & - & - & - &  \cr
                      & 1& -& - &  \cr
                                               - &7& 12& 5 &  \cr
                                                    \cr         \end{matrix}$ \\
                                                          
                          15) nine points on \\ a unique quadric\\   
                                         
              \end{tabular}  \\
  \hline
  \end{tabular}
  \vskip .4cm
  \caption{Betti tables $B_\Gamma$ of certain point sets $\Gamma\subset \Pn^3$\\ CI denotes complete intersection } \label{tablepoints}
  \end{table}
                 \end{footnotesize}
                 \end{center}
                 \begin{proposition}\label{prop_BettiGamma} Let $\Gamma \subset \Pn^3$ be a set of $n$ distinct points,  $4 \leq n \leq 9$, that span $\Pn^3$.
Assume furthermore that no subset of four of the points lie on a line, no subset of six lie in a smooth or singular conic, no subset of seven lie in a plane, no subset of eight lie in a determinantal net of quadrics, and no subset of nine lie on two quadrics.
Then $S/I_\Gamma$ has one of the Betti tables in Table \ref{tablepoints}. For each table, we describe a set of points which realizes the table. 
                 \end{proposition}
\begin{proof} The proof has two parts. We first show that a configuration of reduced points satisfying the hypotheses has regularity $2$, with the obvious exception of \#12. In the second part, we show that
only the Betti tables in Table \ref{tablepoints} actually appear.

First part: since $\Gamma$ is Cohen-Macaulay, the Betti table is preserved for the Artinian reduction, the quotient of the coordinate ring of $\Gamma$ by a general linear form. The degree of $\Gamma$ is
the degree of the Artininan reduction $A_\Gamma$ of $S/I_\Gamma$. Since $Ext^i(A_\Gamma,
S)  \ne 0$ only at the last position, the regularity is attained at the
last step of the free resolution; if the largest shift at the last 
position is $S(-a)$, then the regularity is $a-3$. On the other hand,
the largest shift contributes a $t^a$ in the numerator of the Hilbert
series, which is divisible by $(1-t)^3$. Indeed, for any finite length quotient of a polynomial ring, the regularity is attained at the last step of the resolution, so is equal to the socle degree.

We write $\Delta H$ for the Hilbert function of the Artinian reduction $A_\Gamma$. As $\Gamma$ is Cohen-Macaulay and spans $\Pn^3$, $\Delta H$ begins $(1,3)$. With our hypotheses, the possible values for $\Delta H$  are listed in table \ref{tableAH}.
\begin{center}
\begin{table}[H]
\begin{tabular}{|c|c|c|c|c|c|}
\hline $\deg = 4$& $\deg = 5$ & $\deg = 6$ & $\deg = 7$ & $\deg = 8$ & $\deg = 9$ \\
\hline $(1,3)$ & $(1,3,1)$  & $(1,3,2)$ & $(1,3,3)$  & $(1,3,4)$ & $(1,3,5)$\\
\hline  &   & $(1,3,1,1)$ & $(1,3,2,1)$         & $(1,3,3,1)$ & $(1,3,4,1)$\\
\hline  &             &  & $(1,3,1,1,1)$     & $(1,3,2,2)$ & $(1,3,3,2)$\\
\hline  &             &               &   & $(1,3,2,1,1)$ & $(1,3,3,1,1)$\\
\hline  &             &                &                & $(1,3,1,1,1,1)$ & $(1,3,2,2,1)$\\
\hline  &             &                 &               &                  & $(1,3,2,1,1,1)$\\
\hline  &               &               &               &                  & $(1,3,1,1,1,1,1)$\\
\hline
\end{tabular}
\vskip .1in
\caption{Hilbert function $\Delta H$ of Artinian reductions}\label{tableAH}
\end{table}
\end{center}
\vskip -.4in
To see that these are the possible values for $\Delta H$, note that
if $h_2 \in \{1,2\}$, Macaulay growth says all higher $h_i$ have to be (respectively) $\le 1$ or $2$.
If $h_2 \in \{3,4\}$, there is no room due to degree constraints for any increase: If $h_2=4$, then $\deg(A_\Gamma) \ge 8$, so the only possibilities if $h_2 = 4$ are degree $8$ and $\Delta H =(1,3,4)$, or degree $9$ and $\Delta H = (1,3,4,1)$.
If $h_2=3$, then $\deg(A_\Gamma) \ge 7$. If $\deg(A_\Gamma) = 8$ then $\Delta H =(1,3,3,1)$, and if $\deg(A_\Gamma)=9$ then $\Delta H= (1,3,3,2)$ or $(1,3,3,1,1)$. We now bring Theorem 3.3 and 3.6 of Bigatti-Geramita-Migliore \cite{BGM} into play. \vskip .1in
\noindent {\bf Theorem 3.3 of \cite{BGM}}: Suppose $Y$ is a closed, nonempty reduced subscheme of $\Pn^{r+1}$, and for some $m$ and two consecutive values $d,d+1 \in \ZZ$ the Artinian reduction $\Delta H$ satisfies 
\[
\Delta H(Y, d) = \binom{d+m-1}{d}\mbox{  and }\Delta H(Y, d + 1) = 
\binom{d+m}{d+1}.
\]
Then $Y$ is the disjoint union of a scheme $Y_1$ and a scheme $Y_2$, where $Y_1$ lies in a $\Pn^m$ and $Y_2$ is a finite set of points. Furthermore
\[\Delta H(Y_1,t) = \Delta H(\Pn^m, t) \mbox{ if }t \le d+1, \mbox{ and }\Delta H(Y,t) \mbox{ if }t \ge d.
\]
\vskip .1in
\noindent {\bf Theorem 3.6 of \cite{BGM}}: Suppose $Y$ is a reduced scheme in $\Pn^{r+1}$
, and for some $d$, 
\[
\Delta H(Y, d) = \Delta H(Y, d + 1) = s, \mbox{ where }d \ge s.
\]
Then it must be the case that $\dim(Y) \le  1$, and $(I_Y)_{\le d}$ is the saturated ideal of a reduced curve of degree $s$ (not necessarily unmixed)
\vskip .1in
    
Applying Theorem 3.3 to any of the $\Delta H$ above which have two consecutive values $(1,1)$ means that there is a subscheme lying on a line 
containing $k$ points, where $k$ is the position of the last value $1$. So for example, $(1,3,1,1)$ forces $\Gamma$ to contain four collinear points, and $(1,3,1,1,1,1,1)$ forces $\Gamma$ to contain seven
collinear points. Hence the hypothesis that $\Gamma$ does not contain four collinear points means that we have reduced the possible Hilbert functions $\Delta H$ for an Artinian reduction of nine or fewer reduced points which span $\Pn^3$ to the set
\[
\begin{array}{cccccc}
(1,3) & (1,3,1) & (1,3,2) & (1,3,3)& (1,3,4) & (1,3,5)\\
(1,3,2,1) & (1,3,3,1)& (1,3,2,2)& (1,3,4,1)& (1,3,3,2)& (1,3,2,2,1). 
\end{array}
\]

By Theorem 3.6 of \cite{BGM}, if $\Delta H$ has consecutive values $(2,2)$, then the quadratic component of $I_\Gamma$ lies on a curve of degree $2$. So the hypothesis that $\Gamma$ has no subset of six points in a plane conic (smooth or singular) rules out $(1,3,2,2)$ and $(1,3,2,2,1)$. It remains to deal with the cases $(1,3,2,1)$, which has degree $7$, and $(1,3,4,1), (1,3,3,2)$, which are of degree $9$. 

\begin{description}
\item[{\bf $\Delta H = (1,3,3,2)$}] The Hilbert function of $\Gamma$ is $(1,4,7,9,9,\cdots)$, so $\dim(I_\Gamma)_2 =3$. The hypothesis that no nine points lie on an intersection of two quadrics precludes this case.  \vskip .1in
\item[{\bf $\Delta H = (1,3,4,1)$}]  The Hilbert function of $\Gamma$ is $(1,4,8,9,9,\cdots)$, so $\dim(I_\Gamma)_2 =2$. The hypothesis that no nine points lie on an intersection of two quadrics precludes this case. \vskip .1in

\item[{\bf $\Delta H = (1,3,2,1)$}] The seven points have Hilbert function $(1,4,6,7,7,\ldots)$, so $\dim(I_\Gamma)_2 =4$. In particular, the points do not impose independent conditions on quadrics. Checking the value of $h_3$ shows there must be at least three linear
syzygies on the quadrics. The codimension of $I_2$ cannot be $1$, for then $I_2 = L \cdot (x,y,z,w)$ and $I_\Gamma$ is not saturated. So $\codim(I_2) \ge 2$,
and two of the four quadrics yield a CI. If the codimension is exactly $2$, then the remaining quadrics drop the degree from $4$ down to at most $2$, yielding a plane conic or a line, which are ruled out by our hypotheses. If the $\codim(I_2)=3$, then three of the quadrics are a CI. Since $h_3 =7$, there are at least three linear syzygies on $I_2$. Let $Q$ denote the $4^{th}$ quadric and $J$ the three quadrics yielding a CI. Then $J:Q$ contains three linear forms, so the Betti table for the mapping cone of $I_2$ is
\begin{center}
$ \begin{matrix}&&&&&\cr
1 & - & - & - & -& \cr
                                               - & 4& 3 & 3 &1   \cr
                                               - & -& 3 & - & - \cr
                                               - & -& - & 1 & - \cr
                                              \cr
               \end{matrix}$. \\
\end{center}
This ideal is not saturated, so this case cannot arise with our hypotheses. 
\end{description}
We have shown that a set of points $\Gamma \subseteq \Pn^3$ satisfying the hypotheses of the lemma either has $\Delta H = (1,3,3,1)$, or has Artinian reduction with Hilbert function in the set
\begin{equation}\label{ArtinianHF}
\{(1,3), (1,3,1), (1,3,2), (1,3,3), (1,3,4), (1,3,5)\}.
\end{equation}

We now tackle the second part of the proof: showing that the Betti tables in Table \ref{tablepoints} are the only ones that occur. In what follows, we restrict our attention to the regularity $2$ cases, and deal with $\Delta H = (1,3,3,1)$ at the end. By Equation~\ref{ArtinianHF}, the Hilbert series of the Artinian reductions are all of the form
\[
1+3t+h_2t^2, \mbox{ with } h_2 \in \{0,1,2,3,4,5\}.
\]
As our reductions are all of regularity $2$, we can also read off the Hilbert series from the free resolutions, yielding 
\[
\frac{1-b_{12}t^2-b_{13}t^3+b_{23}t^3+b_{24}t^4-b_{34}t^4-b_{35}t^5}{(1-t)^3}.
\]
Comparing, we obtain the equalities
\begin{center}
    \[
\begin{array}{ccc}
6-b_{12} &=& h_2\\
8-b_{23}+b_{13}&=&3h_2\\
3-b_{34}+b_{24}&=&3h_2\\
b_{35}&=& h_2\; .
\end{array}
\]
\end{center}
Therefore the Betti tables of the Artinian reduction $A_\Gamma$ are of the form
\begin{center}
$ \begin{matrix}&&&&\cr
1 & - & - & -  \cr
                                               - & 6-h_2& b_{23} & b_{34}   \cr
                                               - & 3h_2-8+b_{23}& 3h_2-3+b_{34} &h_2\cr
                                              \cr
               \end{matrix}$. \\
\end{center}
We now examine the various possibilities for $h_2$. If $h_2 \in \{0,3,4,5\}$ the analysis is straightforward. For $h_2 \in \{1,2\}$, we will apply Macaulay's theorem to the degree $2$ piece $I_2$ of the ideal of the Artinian reduction.  

\begin{description}
    \item[{\bf Case }$h_2 \in \{0,3,4,5\}$]
    For $h_2=0$ the regularity of $A_\Gamma$ is $1$ and so the Hilbert function determines the Betti table.
For $h_2 \in \{4, 5\}$ the analysis is also straightforward: since $\dim (I_\Gamma)_2 \le 2$, $b_{23} \in \{0,1\}$.
For $h_2 =3 $ we find $b_{23} \le 3$, with $b_{23} = 3$ only if the ideal of the Artinian reduction is of the form $L \cdot(x,y,z)$ with $L$ a linear form. 
    \item[{\bf Case }$h_2 = 2$] We have $h^{\langle 2\rangle}_2 =2$, and therefore 
    \[
    h_3(I_2) = 10-4\cdot 3+b_{23} \le 2, \mbox{ so }b_{23} \le 4.
\]
From the Betti table, $b_{13} \ge 0$  forces $b_{23} \ge 2$, so for $h_2=2$, 
\[
b_{23} \in \{2,3,4\}
\]
When $b_{23} = 2$, there cannot be a second linear syzygy (it would have rank 2, which is impossible), so the top row of the Betti table must be $(4, 2, 0)$. When $b_{23} \in \{3, 4\}$, we consider the generic initial ideal (GIN). The GIN must be of the form (see Appendix \ref{Appendix2}).
\[
\langle x^2,xy,y^2,xz,yz^2, z^3\rangle,
\]
which has Betti table
\begin{center}
$ \begin{matrix}
1 & - & - & - & \cr
                                               - & 4& 4 & 1 &  \cr
                                               - & 2& 4 &2&\cr
                                              \cr
               \end{matrix}$ .\\
\end{center}
The Betti table of $I$ must be a subtable of this. We know $b_{23} \in \{2,3,4\}$, and we have dealt with the $b_{23} =2$ case above. To rule out top row $(4,3,1)$ note that the linear second syzygy would have to have rank $3$, so would force
$I_2$ to contain $L\cdot (x,y,z)$. Hence the three linear first syzygies are on the ideal $J = L\cdot(x,y,z)$, forcing the fourth
quadric $q$ to be a nonzero divisor on $J$, yielding a mapping cone resolution with Betti table
\begin{center}
    $ \begin{matrix}&&&&&\cr
1 & - & - & - &- \cr
                                               - & 4& 3 & 1 & -  \cr
                                               - & -& 3 &3 & 1\cr
                                              \cr
               \end{matrix}$,\\
\end{center}
which is clearly impossible. 

To rule out at top row of $(4,4,0)$, we continue to consider the rank of the linear syzygies; as we are in three variables the rank must be either $2$ or $3$.
First, suppose we have a rank $3$ first syzygy, which after a change of variables is $(x,y,z)$. Considering the quadrics as a syzygy on this, we
see that three quadrics are combinations of the three Koszul syzygies on $(x,y,z)$, which makes them the $2 \times 2$ of a $2 \times 3$ matrix; in particular one
first syzygy of rank $3$ forces a second. Let $J$ be the ideal of the $2 \times 2$ minors, and $Q$ the $4^{th}$ quadric. Since $\codim(J)=2$, this means
$Q$ is a nonzero divisor on $J$, so the resolution is a mapping cone, which means there are only two linear first syzygies, a contradiction.
Thus, all first syzygies have rank $2$, so each linear first syzygy involves only two generators of $I_2$, of the form $L \cdot (l_1,l_2)$. WLOG choose
$L_1\cdot(l_1,l_2)$ for the first two generators, and $L_2\cdot (l_3,l_4)$ as the second two generators, yielding the syzygies
\[
[l_2, -l_1, 0,0]^t, [0,0,l_4, -l_3]^t.
\]
But now there are two additional rank $2$ syzygies among the four generators, forcing the ideal to have the form $\langle ac,ad,bc,bd \rangle$. This has a linear second syzygy, so a top row $(4,4,0)$ is impossible. This completes the analysis for $h_2 =2$.\vskip .1in
\item[{\bf Case }$h_2 =1$] We apply Macaulay's theorem to $I_2$, and find that $b_{23} \in \{5,6\}$. In this case, the GIN must be $\langle x^2,xy,y^2,xz,yz, z^3 \rangle$, which has Betti table
\begin{center}
$ \begin{matrix}
1 & - & - & - & \cr
                                               - & 5& 6 & 2  \cr
                                               - & 1& 2 &1\cr
                                              \cr
               \end{matrix}$. \\
\end{center}
\vskip -.1in
Now, we know $b_{23} \in \{5,6\}$. If $b_{23}=6$, we need to rule out $b_{34}\in \{0, 1\}.$ If $b_{34} = 0$, then $b_{24} = 0$, which means the cubic generator is not in any syzygy, hence a nonzero divisor on $I_2$. Since $\codim(I_2) = 3$, this is impossible. Next, if $b_{34} = 1$, this means there is a single linear second syzygy on $I_2$, necessarily (see earlier remarks) of rank $3$. Let $C$ be the unique cubic, and consider the mapping cone resolution of $I_2+C$. The linear syzygy
represented by $b_{24} =1$ must involve $C$ (otherwise we are in the situation above), so in $I_2:C$ we have one linear form. However, $I_2$ is also
contained in $I_2:C$, so in the mapping cone, the resolution of $S(-3)/I_2:C$ begins   
\[
S(-4)\oplus S(-5)^a \rightarrow S(-3)\rightarrow S(-3)/I_2:C \rightarrow 0, \mbox{ with } a \ge 1.
\]

There can be no cancellation in the top row, and so this yields a resolution for $S/I_2+C$ which is inconsistent with the Betti table. Our final case is when $b_{23} = 5$. We know there is a cancellation of $b_{13}$, so the Betti table is a subset of
\vskip -.3in
\begin{center}
$ \begin{matrix}&&&&\cr
1 & - & - & - & \cr
                                               - & 5& 5 & 2  \cr
                                               - & 0& 2 &1\cr
                                              \cr
               \end{matrix}$. \\
\end{center}
\vskip -.1in
The cases $b_{34}=b_{24} \in \{0,1\}$ occur, so we just need to rule out the case $b_{34} = 2$. As before, the rank of the linear second syzygies must be $3$. Therefore, $L\cdot (x,y,z) \in I_2$, so let $I = \langle Lx,Ly,Lz, Q_1,Q_2\rangle$. As $\codim(L,Q_1,Q_2) = 3$, $\Gamma$ contains a $(2,2)$ CI on the plane $\VV(L)$. 

Therefore the fifth point of $\Gamma$ must be $(0,0,0,1)$, and $Q_1,Q_2 \in \langle x,y,z\rangle$. Writing $Q_1 = ax+by+cz$, we see there is a linear syzygy $aLx+bLy+cLz-LQ_1=0$, and similarly for $Q_2=dx+ey+fz$. Hence the matrix of first syzygies on $I$ is
\begin{align*}
    \begin{bmatrix}
    y& z & 0& a &d\\
    -x& 0 & z& b& e\\
     0 & -x & -y& c &f\\
     0 & 0 & 0& -L& 0\\ 
     0 & 0 & 0& 0&-L \\
    \end{bmatrix}.
\end{align*}
A linear second syzygy involves only the first three columns, so $b_{34}=2$ is impossible.
\end{description}
We close by analyzing the situation when $\Delta H =(1,3,3,1)$. The Hilbert function of $\Gamma$ is $(1,4,7,8,8,\ldots)$, so $\dim (I_\Gamma)_2 =3$. The arguments above show $b_{23} \le 3$; if $b_{23}=3$ then $(I_\Gamma)_2 = L\cdot(x,y,z)$ and there are seven points on a $\Pn^2$, which is precluded by our hypotheses.

When $b_{23}=2$, if there is a rank $3$ first syzygy, then as in the analysis above, $(I_\Gamma)_2$ is the $2 \times 2$ minors of a $2 \times 3$ matrix. If both first syzygies are of rank $2$, then $(I_\Gamma)_2 = \langle l_1l_2, l_1l_3, l_2l_3 \rangle$ and again $(I_\Gamma)_2$ is the $2 \times 2$ minors of a $2 \times 3$ matrix. Both these cases are precluded by the hypothesis that there are no eight points on a determinantal net. So if $\Delta H = (1,3,3,1)$, our hypotheses imply that $I_\Gamma$ is a $(2,2,2)$ CI, which concludes the proof. 
\end{proof}

\section{The quadratic part of the apolar ideal }\label{quadraticideals}

Consider a quartic form $F$ (in four variables, as usual), such that $F^\perp$ contains no linear forms.
We write $Q_F$ for the ideal generated by the quadrics in $F^\perp$;  recall that $A_F$ denotes $S/F^\perp$. In this section, we describe the possible Betti tables of $Q_F$ over all such $F$, and show that they are the ideal of a scheme $Z$ apolar to $F$. In particular, the ideals $Q_F$ are all saturated.
\subsection{Syzygy schemes and preparatory lemmas}

\begin{lemma}\label{lem_QmustSat}
Let $F$ be a quartic form and assume $J\subseteq F^{\perp}$ is an ideal generated by quadrics. If $\reg S/J<4$, then the saturation $J^{\sat}=(J:\mathfrak{m}^{\infty})$ is also in $F^{\perp}$. 
\end{lemma}
\begin{proof}
The $0$-th local cohomology $H^{0}_{\mathfrak{m}}(S/J)$ is isomorphic to $J^{\sat}/J$. By \cite{GeomSyz}*{Theorem 4.3}, $\reg S/J<4$ implies $H^{0}_{\mathfrak{m}}(S/J)$ vanishes at degree $4$, which means $J^{\sat}_4=J_4$.
Suppose that $g \in J^{\sat}$, but not in $J$; since $J^{\sat}_4 = J_4$, $\deg(g) \in \{1,2,3\}$. First consider $\deg(g)=3$. Since $J^{\sat}_4 = J_4$, $\mathfrak{m}\cdot g \in J_4$ so $(\mathfrak{m}\cdot g)\circ F = 0$. If $g \circ F$ is nonzero, then it is a linear form, so $\mathfrak{m}$ cannot annihilate it. Hence we must have $g \circ F = 0$, so $g \in F^\perp$. The cases for $\deg(g) \in \{1,2\}$ work the same way, hence for $g \in J^{\sat}, g \circ F = 0.$ \end{proof}

We say $J$ is saturated at degree $4$, if $J_4^{\sat}=J_4$, as in the lemma.
Now we assume $J=Q_F$ is the quadratic part of $F^{\perp}$, hence $b_{i,i+1}(S/Q_F)=b_{i,i+1}(A_F)$ for all $i$. For the $16$ Gorenstein Betti tables of $A_F$, we now conduct a case-by-case analysis of the possible Betti tables for $S/Q_F$. From \cite{SSY} we know that the numbers $b_{i,i+1}(A_F)$ determine the entire Betti table of $A_F$, and that $0\leq b_{12}\leq 6$; hence if $b_{12} \in \{0,1\}$ the Betti table is fixed. In the upcoming analysis, Lemma~\ref{lem_QmustSat} will allow us to rule out certain possibilities. 
\begin{lemma}\label{ImpossibleBettiTable}
The Betti tables below cannot occur for $S/Q_F$ when $F$ is a quaternary quartic:
\[
    \begin{matrix}1 & - & - & - & - \cr
       - & 3& -& - & - \cr
        - &-& 5& 4 & 1 \cr
    \end{matrix}
    \]
    and
   \[
 \begin{matrix}1 & - & - & - & - &\cr
       - & 4& 2 & - & -& \cr
      - &- & 4& 4 & 1& \cr
      \end{matrix}.
\] 
\end{lemma}
\begin{proof}
The Hilbert polynomial of  $S/Q_F$ for the first Betti table is $2t+2$, and there are two possibilities in the Hilbert scheme: the ideal of two skew lines/double line in a smooth quadric, or the union of a plane conic and a (possibly embedded) point.  
In the second case, any net of quadrics in the ideal would contain a pencil of reducible ones, and hence a linear syzygy. In the first case, after a change of variables, the ideal is either of the form $\langle x_0x_2,x_0x_3,x_1x_2,x_1x_3\rangle$ or $\langle x_0^2,x_1x_1,x_1^2, x_0l_0(x_2,x_3)+x_1l_1(x_2,x_3)\rangle$ with $l_i$ independent. In both cases only quadrics are needed to saturate the ideal; by Lemma~\ref{lem_QmustSat} this cannot occur.

For the second Betti table the Hilbert polynomial of $S/Q_F$ is $t+3$. The only possibility is a line and two (possibly embedded) points, and in this case the saturation is generated by five quadrics, so again by Lemma~\ref{lem_QmustSat} the corresponding Betti table cannot occur for $S/Q_F$.
\end{proof}
In our investigation of $Q_F$ we will make extensive use of the theory of ideals generated by quadrics and how they are determined by their syzygies. Let us recall some definitions and basic results from \cite{vB07a}. 
Let $f$ be a p-th linear syzygy of an ideal $I\subset \operatorname{Sym}(V)$ generated by quadrics, so that $f \in \Tor_{p+1}^S(S/I,\CC)_{p+2}$. We define $I_f$ to be the subideal of $I$ generated by those elements of $I$ which are involved in $f$ (see \cite{vB07a}*{Def.2.2} for the formal definition). Following \cite{vB07a}, we define $G$ to be the vector space of $(p-1)$-st syzygies involved in $f$; note that these syzygies are all linear, and the dimension of $G$ is equal to the rank of $f$. Finally, for a vector space $G$ we define the $k$-th generic syzygy ideal ${\operatorname{Gensyzid}_k(G)} $ of $G$ to be the quadratic ideal generated by the image of the natural inclusion  
$$\wedge^{k+1} G^*\to G^*\otimes \wedge^{k} G^*\subset \operatorname{Sym}_2(G^*\oplus \wedge^{k} G^*)$$ in $\operatorname{Sym}(G^*\oplus \wedge^{k} G^*)$. 

In this notation, if $f$ is a p-th syzygy of rank $p+k+1$ in $I$ and $G$ the space of $(p-1)$-st syzygies involved in $f$, then (cf. the proof of \cite{vB07a}*{Thm 3.4}), the ideal $I_f$ is obtained as the image of the generic k-th syzygy ideal of $G$ under a natural map $\varphi:G^*\oplus \wedge^{k} G^*\to V$, which in particular maps the elements of the $G^*$ component to their corresponding linear forms given by the syzygy $f$. More precisely, we have the following diagram:
\begin{center}
\begin{tikzcd}
{\operatorname{Gensyzid}_k(G)} \arrow[d,hook] \arrow[r]                        & { I_f} \arrow[d,hook] \\
\operatorname{Sym}(G^*\oplus \wedge^{k} G^*) \arrow[r, "\operatorname{Sym}(\varphi)"] &  \operatorname{Sym}(V)
\end{tikzcd}
\end{center}

In this way, the scheme defined by $I_f$ appears as a cone over a linear section of the scheme defined by the generic syzygy ideal. From this point of view, to describe $I_f$, it is good to know what are the generic k-th syzygy ideals for a vector space $G$.
For the following, see \cite{vB07a} and references therein:
\begin{enumerate}
    \item the 0-th generic syzygy ideal of $G$ is the ideal of a union of a hyperplane and a point in $\mathbb{P}(G^*\oplus \mathbb C)$.
    \item the 1-st generic syzygy ideal of $G$ is the ideal of $\mathbb{P}^1\times \mathbb P(G^*)\subset \mathbb P(G^*\oplus \mathbb G^*)$ i.e. it is given by $2\times 2$ minors of a $2\times \dim G$ matrix of coordinates.
    \item the 2-nd generic syzygy ideal of $G$ is the ideal of the union of a Grassmannian $G(G^*\oplus \mathbb C, 2)$ and a linear space $
    \mathbb{P}(\wedge^2 G^*)$ in $
    \mathbb{P}(G^*\oplus\wedge^2 G^*)$.
\end{enumerate}
From the above we deduce the following.
\begin{proposition}\label{syzygy ideals}
Let $I\subset \operatorname{Sym}(V)=S=\CC[x_0,\ldots,x_3]$ be an ideal generated by quadrics. Assume that $I$ admits a $p$-th linear syzygy $f$ of rank $p+k+1$ with $1\leq p\leq 2$. Then $0\leq p\leq p+k\leq 3$ and we have the following possibilities for the syzygy ideal $I_f$:
\begin{enumerate}
    \item $p=1$, $k=0$ and $S/I_f$ has Betti table
\[
 \begin{matrix}1 & - & - & \cr
           - & 2& 1&  \cr
\end{matrix}
\] and $I_f$ is the ideal of a plane and a possibly embedded line.
    \item $p=1$, $k=1$ and $S/I_f$ has Betti table
    \[
 \begin{matrix}1 & - & - & \cr
           - & 3& 2&  \cr
\end{matrix}
\] and $I_f$ is generated by $2\times 2$ minors of $2\times 3$ linear forms and is the ideal of a curve of degree 3 with Hilbert polynomial $3t+1$.
    \item $p=1$, $k=2$ and $I_f$ is generated by four $4\times 4$ Pfaffians of a skew-symmetric $5\times 5$ matrix of linear forms and it is contained in a unique ideal generated by Pfaffians of a skew-symmetric $5\times 5$ matrix of linear forms.
    \item $p=2$, $k=0$
and     $S/I_{f}$ has Betti table
\[
 \begin{matrix}1 & - & - & - & \cr
           - & 3& 3& 1 &  \cr
\end{matrix}
\]
and $I_f$ is the ideal of a plane and a possibly embedded point. 
    \item $p=2$, $k=1$ and we have one of three possibilities
    \begin{enumerate}
        \item $S/I_f$ has Betti table 
 
\[
 \begin{matrix}1 & - & - & - & \cr
           - & 4& 4& 1 &  \cr
\end{matrix}
\]
and $I_f$ is the ideal of two skew lines (or a double line),
\item $S/I_f$ has Betti table 
\[
 \begin{matrix}1 & - & - & - &  \cr
       - & 5& 6& 2 &  \cr
\end{matrix} 
\]
and $I_f$ is the ideal of a line and a (possibly embedded) scheme of length 2,
\item $S/I_f$ has Betti table 
\[
 \begin{matrix}1 & - & - & - &  \cr
       - & 6& 8& 3 &  \cr
\end{matrix} 
\]
and $I_f$ is the ideal of a scheme of length 4.
   \end{enumerate}
\end{enumerate}
\end{proposition}
\begin{proof}
We begin by noting that the rank of a linear syzygy of an ideal in a ring with four variables cannot exceed 4. Furthermore it is well known (see \cite{vB07a}) that we must have $k\geq 0$.   
Now, with the notation and results above, to derive the syzygy ideal from the generic syzygy ideal we need a map
\begin{equation}
    \varphi: G^{*}\oplus \wedge^{k} G^{*}\to V,
\end{equation}
that on the $G^{*}$ component corresponds to associating to an element of $G^{*}$ the corresponding linear form corresponding by $f$ to that syzygy. 
In particular, we may assume by choosing a basis of $G$ that $\varphi$ on the component $G^{*}$ is given by a vector of coordinates. 

If $k=0$, we choose a basis $(g_1,..,g_{p+1},c)$ of   $ G^{*}\oplus \mathbb C$ such that for $i\in\{1,\dots, p+1\}$ we have $\varphi(g_i)=x_i$ and denote $l:=\varphi(c)$. In this case, the generic syzygy ideal $\operatorname{Gensyzid}_0(G)$ is the ideal $\langle g_1c,\dots, g_{p+1}c \rangle$ hence $I_f=\langle x_1l,\dots,x_{p+1}l\rangle$ and is the ideal of a hyperplane defined by $\{l=0\}$ and a codimension ${p+1}$ linear space defined by $\{x_1=\dots=x_{p+1}=0\}$. 

For $p=1$, the ring $S/I_f$ has Betti table 
\[
 \begin{matrix}1 & - & - & \cr
           - & 2& 1&  \cr
\end{matrix},
\]
whereas for $p=2$, the ring $S/I_f$ has Betti table 
\[
 \begin{matrix}1 & - & - & -& \cr
           - & 3& 3& 1&  \cr
\end{matrix}.
\]

If $k=1$, we choose a basis $(g_1,\dots,g_{p+2},g'_1,\dots,g'_{p+2})$ of   $ G^{*}\oplus G^{*}$ such that for $i\in\{1,\dots, p+2\}$ we have $\varphi(g_i)=x_i$ and denote $l_i:=\varphi(g'_i)$. In this case, 
the generic syzygy ideal $\operatorname{Gensyzid}_1(G)$ is the ideal generated by $2\times 2$ minors of the matrix:
\begin{equation}
N=    \begin{pmatrix}
    g_1& \dots & g_{p+2}\\
g'_1& \dots & g'_{p+2}
    \end{pmatrix}.
\end{equation}
Hence $I_f$ is generated by 
$2\times 2$ minors of the matrix:
\begin{equation}
\varphi(N) =   \begin{pmatrix}
    x_1& \dots & x_{p+2}\\
l_1& \dots & l_{p+2}
    \end{pmatrix}.
\end{equation}

Now, if $p=1$, the three minors of $\varphi_N$ generating $I_f$ must be linearly independent and $S/I_f$ has Betti table
\[ \begin{matrix}1 & - & - & \cr
           - & 3& 2&  \cr
\end{matrix}.\]
$I_f$ is the ideal of a possibly non-reduced curve of degree 3 with Hilbert polynomial $3t+1$.

If $p=2$, we define $s$ to be the least number that $\varphi(N)$ can be written in the form such that $l_{i}=0$ for $s+1\leq i\leq 4$ and $l_{s}\neq 0$. A matrix is said to be $1-$generic \cite{GeomSyz} if it cannot be modified by elementary row/column operations to have a zero entry. Note that $\varphi(N)$ cannot be $1$-generic by \cite{GeomSyz}*{Theorem 6.4}, so $s\leq 3$. 

If $s=1$, then $I_{f}=\langle ly,lz,lw\rangle$, so 
$\rank f=3$. However, our hypothesis is $\rank(f)=4$, contradiction.

If $s=2$, then 
\begin{equation}
    \varphi(N)=\begin{pmatrix}
    x& y& z & w\\
    l_{1}&l_{2}&0&0
    \end{pmatrix},
\end{equation}
where $l_{1}$ and $l_{2}$ are not proportional. \\
If $\det \begin{pmatrix}
x & y\\ l_{1} & l_{2}
\end{pmatrix}\in \langle z,w\rangle$, then $S/I_{f}$ has Betti table
\[
 \begin{matrix}1 & - & - & - & \cr
           - & 4& 4& 1 &  \cr
\end{matrix}.
\] 
and is the ideal of two skew lines: $\{z=w=0\}$ and $\{l_1=l_2=0\}$ (or a double line if $l_1,l_2\in \langle z,w\rangle$).
If $\det \begin{pmatrix}
x & y\\ l_{1} & l_{2}
\end{pmatrix}\not\in \langle z,w\rangle$, then $S/I_{f}$ has Betti table
\[
 \begin{matrix}1 & - & - & - & \cr
           - & 5& 6& 2 &  \cr
\end{matrix}.
\] and is the ideal of a line $\{l_1=l_2=0\}$ and a scheme of length $2$ defined by $\{z=t=xl_2-yl_1=0\}$.

Conversely, if $S/I_{f}$ has Betti table
\[
 \begin{matrix}1 & - & - & - & \cr
           - & 5& 6& 2 &  \cr
\end{matrix}
\] 
or
\[
 \begin{matrix}1 & - & - & - & \cr
           - & 4& 4& 1 &  \cr
\end{matrix},
\]
then $s=2$. Reason: If $s=3$ and $b_{12}\leq 5$, then there exists numbers $a_{1},\dots, a_{6}$ such that
\begin{equation}\label{eqn_sneq3}
    (a_{1}l_{1}+a_{2}l_{2}+a_{3}l_{3})w+a_{4}(xl_{2}-yl_{1})+a_{5}(xl_{3}-zl_{1})+a_{6}(yl_{3}-zl_{2})=0.
\end{equation}
This is equivalent to
\begin{equation}\label{eqn_skewrk4}
    \begin{pmatrix}
    x & y& z& w
    \end{pmatrix}\begin{pmatrix}
    0 & -a_{4}& -a_{5} & -a_{1}\\
    a_{4}&0& -a_{6} &-a_{2}\\
    a_{5}& a_{6}& 0& -a_{3}\\
    a_{1}& a_{2}& a_{3}& 0
    \end{pmatrix}\begin{pmatrix}
    l_{1}\\ l_{2}\\ l_{3}\\0
    \end{pmatrix}=0.
\end{equation}
If $w\not\in\langle l_{1},l_{2},l_{3}\rangle$, then after changing coordinates, we may assume $\langle l_{1},l_{2},l_{3}\rangle\subseteq\langle x,y,z\rangle$. Then $(a_{1}l_{1}+a_{2}l_{2}+a_{3}l_{3})w=0$ and $a_{4}(xl_{2}-yl_{1})+a_{5}(xl_{3}-zl_{1})+a_{6}(yl_{3}-zl_{2})=0$, which implies $a_{1}l_{1}+a_{2}l_{2}+a_{3}l_{3}=0$. Therefore $s\leq 2$.\\ 
If $w\in\langle l_{1},l_{2},l_{3}\rangle$, then after changing coordinates, we may assume $l_{3}=w$ and $\langle l_{1},l_{2}\rangle\subseteq\langle x,y,z\rangle$. Making use of Equation~\ref{eqn_sneq3}, we have $a_{3}=0$, $a_{1}l_{2}+a_{2}l_{2}+a_{5}x+a_{6}y=0$ and $a_{4}(xl_{2}-yl_{1})-a_{5}zl_{1}+a_{6}zl_{2}=0$. Therefore, $\det\begin{pmatrix}
a_{4}x+a_{6}z & a_{4}y+a_{5}z\\ l_{1}& l_{2}
\end{pmatrix}=0$. This means $(l_{1},l_{2})=c\cdot (a_{4}x+a_{6}z , a_{4}y+a_{5}z)$ and we conclude $s\leq 2$.

If $s=3$, then $b_{12}(I_{f})=6$, the Betti table of $S/I_{f}$ is
\[
 \begin{matrix}1 & - & - & - &  \cr
       - & 6& 8& 3 &  \cr
\end{matrix}
\]
and $I_f$ is the ideal of a scheme of length 4.

Finally for $k=2$, $p=1$, we choose a basis $(g_1,\dots,g_{4},g'_{12},\dots,g'_{34})$ of   $ G^{*}\oplus \wedge^2 G^{*}$ such that for $i\in\{1,\dots, 4\}$ we have $\varphi(g_i)=x_i$ and for $i,j\in \{1,\dots, 4\}$ with $i<j$ denote $l_{ij}:=\varphi(g'_{ij})$. In this case,
the generic syzygy ideal $\operatorname{Gensyzid}_2(G)$ is the ideal generated by four $4\times 4$ Pfaffians involving the first row of the skew-symmetric matrix 
\begin{equation}
N=    \begin{pmatrix}
    0& g_1& g_2&g_3 & g_{4}\\
-g_1& 0& g'_{12}& g'_{13} & g'_{14}\\
-g_2&-g'_{12} & 0& g'_{23}& g'_{24}\\
-g_3&-g'_{13} & -g'_{23}&0& g'_{34}\\
-g_4&-g'_{14} & -g'_{24}& -g'_{34} & 0
    \end{pmatrix}.
\end{equation}
Hence $I_f$ is generated by four Pfaffians involving the first row of the matrix:
\begin{equation}
\varphi(N)=    \begin{pmatrix}
    0& x_1& x_2&x_3 & x_{4}\\
-x_1& 0& l_{12}& l_{13} & l_{14}\\
-x_2&-l_{12} & 0& l_{23}& l_{24}\\
-x_3&-l_{13} & -l_{23}&0& l_{34}\\
-x_4&-l_{14} & -l_{24}& -l_{34} & 0
    \end{pmatrix}.
\end{equation}
Furthermore the four Pfaffians are linearly independent and since $\langle x_1\dots x_4\rangle$ define an Artinian ideal, the saturation of $I_f$ contains the ideal $\Pf(4,\varphi(N))$. Finally, the first linear syzygies of $I_f$ must be combinations of rows of $\varphi(N)$ with first entry 0, which implies that $f$ is the unique (up to scaling) first linear syzygy of $I_f$. 

To prove the last statement, it is enough to note that the ideal defined by all $4\times 4$ Pfaffians of a $5\times 5$  skew-symmetric matrix of linear forms is saturated, so whenever it contains $I_f $ it must also  contain, and hence be equal to, $\Pf(4,\varphi(N))$.
\end{proof}
\begin{remark}
Note that if in Proposition \ref{syzygy ideals} we omit the assumption $p\leq 2$, we have one more case: $k=0$, $p=3$. In that case, with the same arguments we deduce that $I_f$ is the ideal generated by all quadrics vanishing in a plane. These kinds of ideals cannot appear in our investigation of $Q_F$, so we omit that case in Proposition~\ref{syzygy ideals}.  
\end{remark}

\noindent We now list some additional lemmas concerning higher dimensional spaces of syzygies that will be useful in our investigation of $Q_F$.
\begin{lemma}\label{all syzygies of rank 4}
If $I$ is an ideal generated by four quadrics in $\CC[x_0,\ldots,x_3]$, such that all the first linear syzygies are of rank 4,  then it admits at most one linear syzygy. So the first two rows of the Betti table of $S/I$ are:
\[
 \begin{matrix}1 & - & - &   \cr
       - & 4& 1&   \cr
\end{matrix} 
\]
\end{lemma}
\begin{proof}
The corank $1$ (i.e. rank at most $3$) locus is in this case a hypersurface (defined by the determinant of a matrix with four rows corresponding to variables and four columns corresponding to the number of generators of the ideal) in the space of all first linear syzygies. Hence it can be empty only if the space of first linear syzygies is $1$-dimensional, and in that case we can have no more linear syzygies.  
\end{proof}
  \begin{lemma}\label{rank2syzygies}
If $J$ is an ideal generated by three quadrics with a pencil of syzygies of rank $2$, then $S/J$ has Betti table
\[
 \begin{matrix}1 & - & - & - &  \cr
      - & 3& 3& 1 &  \cr
\end{matrix} 
\]
In this case $J$ is the ideal of the union of a plane and a (possibly embedded) point. 
\end{lemma}
\begin{proof} A rank $2$ syzygy involves two generators of the ideal that have a common linear factor, so we may assume the two generators are $x_0x_1,x_0x_2$. Any other rank 2 syzygy would involve a third generator and a quadric in the span of the two first.  Now if these form a pencil of rank $2$ syzygies the third generator needs to have a common linear factor with every quadric in the span of the two first.  Therefore, the three generators of $J$ have a common linear factor, three linear syzygies of rank $2$ and a Betti table as stated.
\end{proof}

\begin{lemma}\label{lem_4gen3rank2syz}
If $S/Q_F$ has $b_{12}=4$ and at least three independent rank $2$ first linear syzygies, then $Q_F$ has a subideal $J$ with Betti table for $S/J$
\[
 \begin{matrix}1 & - & - &   \cr
      - & 3& 2&   \cr
\end{matrix},
\]
or a subideal $J'$  with Betti table for $S/J'$
\[
 \begin{matrix}1 & - & - & -  \cr
      - & 3& 3&  1 \cr
\end{matrix}.
\]
\end{lemma}
\begin{proof}
Assume $s_{1},s_{2},s_{3}$ are three rank $2$ first linear syzygies. If $s_{1}$ and $s_{2}$ involve all four generators of $Q_F$, then we may assume $s_{1}$ involves $(q_{1},q_{2})=(x_0l_{1},x_0l_{2})$ and $s_{2}$ involves $(q_{3},q_{4})=(l_{0}l_{3},l_{0}l_{4})$. If $l_{0}=x_0$, then $J=\langle q_{1},q_{2},q_{3}\rangle$ and $S/J$ has Betti table
\[
 \begin{matrix}1 & - & - & -  \cr
      - & 3& 3&  1 \cr
\end{matrix}.
\]
If $l_{0}\neq x_0$, then $(q_{3},q_{4})=(x_1l_{3},x_1l_{4})$. Then the third linear syzygy $s_{3}$ must involve some quadric $q_{5}\in \langle q_{1},q_{2}\rangle$ and some quadric $q_{6}\in \langle q_{3},q_{4}\rangle$. We may assume $q_{5}=q_{1}$ and $q_{6}=q_{3}$. Now $J=\langle q_{1},q_{2},q_{3}\rangle=\langle x_0l_{1},x_0l_{2},x_1l_{3}\rangle$. Because there is a linear syzygy involving $x_0l_{1}$ and $x_1l_{3}$, either $l_{1}=x_1$ or $l_{3}=x_0$. If $l_{3}=x_0$, then $S/J$ has Betti table
\[
 \begin{matrix}1 & - & - & -  \cr
      - & 3& 3&  1 \cr
\end{matrix}.
\]
If $l_{1}=x_1$, then $J=\langle x_0x_1,x_0l_{2},x_1l_{3}\rangle$. If $l_{2}= l_{3}$, then $S/J$ has Betti table
\[
 \begin{matrix}1 & - & - & -  \cr
      - & 3& 3&  1 \cr
\end{matrix},
\]
otherwise $S/J$ has Betti table
\[
 \begin{matrix}1 & - & - &   \cr
      - & 3& 2&   \cr
\end{matrix}.
\]
\end{proof}
\subsection{A case by case characterization of the quadratic component}

We now examine the cases $b_{12}(S/Q_F) \in \{2,\ldots 6\}.$ As noted earlier, when $b_{12} \in \{0,1\}$ there is nothing to analyze. 
\vskip .05in
\subsubsection{\noindent {\bf Case $b_{12}=2$}}
\vskip .05in
\begin{proposition}\label{apolar b12=2}
Let $F$ be a quaternary quartic with $b_{12}(A_F)=2$. Then $Q_F$ is the ideal of an apolar scheme $Z$, which differs depending on the two possible Betti tables below.
\begin{enumerate}
    \item 
   If $F$ is of type $[200]$, then $S/Q_F$ has Betti table 
    \[
 \begin{matrix}1 & - & - & \cr
      - & 2& -& \cr
       - &-& 1&  \cr
 \end{matrix} 
\] and $Q_F$ is a $(2,2)$ complete intersection.
    \item
   If $F$ is of type $[210]$, then $S/Q_F$ has Betti table 
                \[
 \begin{matrix}1 & - & - & \cr
      - & 2& 1& \cr
 \end{matrix}, 
\]and $Q_F$  is the ideal of the union of a plane and a, possibly embedded, line.
   \end{enumerate}
 
 The set of ideals $Q_F$ in the Hilbert scheme, when $F$ is of type $[200]$ or $[210]$, is irreducible.
\end{proposition}
\begin{proof} If $Q_F$ is not a complete intersection, then it has a linear syzygy, so we are in the case of Proposition \ref{syzygy ideals} case $(1)$. 
Complete intersections form an open set in the Grassmannian of pencils in the space of quadrics, while the set of unions of a plane and a line is parameterized by an open set in the product of the space of planes and the space of lines in $\Pn^3$.  The irreducibility for the set of ideals $Q_F$ follows.
\end{proof}
\vskip .05in
\subsubsection{\noindent {\bf Case $b_{12}=3$}}
\vskip .05in
\begin{proposition}\label{QF0fCGKK4}
Let $F$ be a quaternary quartic of  type $[300]$. Then either 
\begin{enumerate}

         \item $S/Q_F$ has Betti table $[300a]$ or $[300b]$
\[ \begin{matrix}1 & - & - & - &  \cr
       - & 3& -& - &  \cr
        - &-& 3& - &  \cr
        - &-& -& 1 &  \cr
        \end{matrix}
        \]
  and $Q_F$ is a complete intersection $(2,2,2)$, or
   \item  $S/Q_F$ has Betti table $[300c]$
   \[  \begin{matrix}1 & - & - & - &  \cr
       - & 3& -& - &  \cr
        - &-& 4& 2 &  \cr
        \end{matrix}
        \]
        and $Q_F$ is the ideal of a union of a line and a (possibly embedded) length $4$ scheme.
    \end{enumerate}
    In either case, the set of ideals $Q_F$ in the Hilbert scheme is irreducible.
\end{proposition}
\begin{proof}  The two cases follow from Lemma \ref{lem_2LS300} below. The set of $(2,2,2)$ complete intersections coincides with an open set in the Grassmannian of nets of quadrics in $\Pn^3$, so it is irreducible.  The set of ideals of type $[300c]$ maps dominantly to the Grassmannian of lines, and to the Hilbert scheme of length $4$ subschemes.  The latter two are irreducible, so the set of ideals of type $[300c]$ is also irreducible. 
\end{proof}

\begin{proposition}\label{apolar b12=3}
Let $F$ be a quaternary quartic with $b_{12}(A_F)=3$, and assume that $F$ is not of type $[300]$. Then $Q_F$ is the ideal of an apolar scheme $Z$, varying with the Betti table as described below.
\begin{enumerate}
   \item 
 If $F$ is of type $[310]$, then $S/Q_F$ has Betti table
 \[   \begin{matrix}
 1 & - & -& - &  \cr
                                               - & 3& 1& - &  \cr
                                               - &-& 2&1&  \cr
               \end{matrix}
    \]
   and $Q_F$ is the ideal of the union $Z$ of a plane conic and a (possibly embedded)  scheme of length $2$.
    \item 
 If $F$ is of type $[320]$, then $S/Q_F$ has Betti table
 \[
 \begin{matrix}
 1 & - & -& - &  \cr
                                               - & 3& 2& -&  \cr
               \end{matrix}
               \]
   and $Q_F$ is the ideal a curve $Z$ of degree $3$ defined by a determinantal net of quadrics.
    \item 
If $F$ is of type $[331]$, then $S/Q_F$ has Betti table
   \[ \begin{matrix}
 1 & - & -& - &  \cr
                                               - & 3& 3& 1 &  \cr
               \end{matrix}
               \]
   and $Q_F$ is the ideal of the union $Z$ of a plane and a (possibly embedded) point.
    \end{enumerate}

The set of ideals $Q_F$ when $F$ is of type $[310]$, type $[320]$, or type $[331]$, form irreducible components of the corresponding Hilbert schemes.
\end{proposition}

\begin{proof}  The first two cases follows from lemmas \ref{lembetti310} and \ref{lembetti320} below, while in the third case $Q_F$ is the syzygy ideal of the second order linear syzygy described in Proposition \ref{syzygy ideals} case (4).  In each case, the set of schemes $Z$ form irreducible components of the Hilbert scheme, so the proposition follows.
\end{proof}

\begin{lemma}\label{lem_2LS300}
If the Betti table of $S/Q_F$ has second row $(3,0,0)$, then the Betti table of $S/Q_F$ is either 
\[
        \begin{matrix}1 & - & - & - &  \cr
       - & 3& -& - &  \cr
        - &-& 4& 2 &  \cr
        \end{matrix}
        \]
and $\VV(Q_F)$ is the union of a line and a possibly embedded scheme of length $4$, or
\[
 \begin{matrix}1 & - & - & - &  \cr
       - & 3& -& - &  \cr
        - &-& 3& - &  \cr
     - &-& -& 1 &  \cr
\end{matrix}
\]
and $\VV(Q_F)$ is a complete intersection.
\end{lemma}
\begin{proof}
Since there is no linear syzygy, no two of the forms share a common factor, hence $\codim(Q_F)=2$. So $Q_F$ contains a subideal $J=\langle q_1,q_2\rangle$ which is a $(2,2)$ complete intersection, with $Q_F=J+\langle q_3\rangle$. As a complete intersection $S/J$ is Cohen-Macaulay of degree $4$, and we have the short exact sequence
\begin{equation}\label{sesQuadrics}
    0 \longrightarrow S(-2)/(J:q_3) \stackrel{\cdot q_3}{\longrightarrow} S/J \longrightarrow S/Q_F \longrightarrow 0.
\end{equation}
\vskip .05in
It follows from \cite{CAwvtAG}*{Lemma 3.6} that the associated primes of $(J:q_3)$ are all codimension $2$. Hence $\deg(J:q_3) \in \{1,\ldots, 4\}$. If $\deg(J:q_3)=1$, then $(J:q_3)$ must be reduced, which would imply that $(J:q_3)$ has a linear generator. This generator would yield a linear first syzygy in the mapping cone resolution of $Q_F$, contradicting the hypothesis that there is no linear first syzygy on $Q_F$. So $\deg (J:q_3) \in \{2,3,4\}$, and all associated primes are of codimension $2$. Note also that from the long exact sequence in $\Ext$ of the sequence in Equation~\ref{sesQuadrics}, $\Ext^3(S/(J:q_3),S) \simeq \Ext^4(S/Q_F,S)$, so $(J:q_3)$ is Cohen-Macaulay iff $Q_F$ is saturated.
\begin{description}
\item[{\bf The degree of $(J:q_3)=4$}] In this case $q_3$ is a nonzero divisor on $J$ and therefore  $Q_F$ is a $(2,2,2)$ complete intersection, with Betti table of $S/Q_F$ of the form
    \[
    \begin{matrix}1 & - & - & - &  \cr
       - & 3& -& - &  \cr
        - &-& 3& - &  \cr
     - &-& -& 1 &  \cr
    \end{matrix}.
    \]
   \item[{\bf The degree of $(J:q_3)=3$}] In this case, we claim the Betti table of $S/(J:q_3)$ is 
    \[
    \begin{matrix}1 & - & - &\cr
       - & 3& 2 &\cr
    \end{matrix}.
    \]
    Since $\deg(Q_F)=1$, the ideal $Q_F$ is contained, after a change of variables, in the ideal $\langle x_0,x_1 \rangle$ of a line $L$. As $J \subset Q_F\subset \langle x_0,x_1 \rangle$, we infer $J$ can be generated by the pair of $2\times 2$ minors involving the last column of the matrix  
\[
 \begin{pmatrix}l_1 & l_2 & x_0  \cr
       l_3 & l_4& x_1  \cr
\end{pmatrix},
\]
for an appropriate choice of linear forms $l_i$.
  The three $2\times 2$ minors of the matrix generate the ideal $I$ of a cubic curve, residual to the line $L$ in the complete intersection $\VV(J)$.  So both $(l_1l_4-l_2l_3)x_0$ and $(l_1l_4-l_2l_3)x_1$ lie in $J$.  But $q_3=l_5x_0+l_6x_1$ for some linear forms $l_5,l_6$, so $(l_1l_4-l_2l_3)q_3\in J$, which means 
$(l_1l_4-l_2l_3)\in (J:q_3)$ and therefore $I\subset (J:q_3)$.  Both $I$ and $(J:q_3)$ have degree $3$, and since $I$ generates the ideal of a curve, equality holds.
 
 With the resolution of $(J:q_3)$ in hand, it follows that $S/Q_F$ has Betti table
    \[
        \begin{matrix}1 & - & - & - &  \cr
       - & 3& -& - &  \cr
        - &-& 4& 2 &  \cr
        \end{matrix}.
        \]
    In this case
    $\VV(q_3)$ intersects $\VV(J:q_3)$ in a scheme of length $6$, and a subscheme of length $2$ is on $L$. Therefore, $\VV(Q_F)$ is the union of a line and zero-scheme of degree $4$.
    \vskip .05in
    \item[{\bf The degree of $(J:q_3)=2$}] In this case, $(J:q_3)$ is the ideal of two skew lines or a double structure on a line. The set of double structures on a codimension $2$ ideal of degree $1$ is determined in \cite{Engheta}. After a change of variables, the ideal is either $\langle x_0,x_1^2\rangle$ or $\langle x_0^2,x_0x_1,x_1^2, x_0l_0+x_1l_1 \rangle$, with $l_i$ homogeneous and $\codim(x_0,x_1,l_0,l_1) = 4$. From this, it follows that for both possibilities the resolutions have Betti table
    \[
    \begin{matrix}1 & - & -& - \cr
       - & 4& 4 & 1 \cr
    \end{matrix}.
    \]
    Hence from the mapping cone construction, $S/Q_F$ would have Betti table
    \[
    \begin{matrix}1 & - & - & - & - \cr
       - & 3& -& - & - \cr
        - &-& 5& 4 & 1 \cr
    \end{matrix}.
    \]
    By Lemma~\ref{ImpossibleBettiTable} this cannot occur. 
\end{description}
In particular, the only two Betti tables possible for second row $(3,0,0)$ are as claimed.
\end{proof}

\begin{lemma}\label{lembetti310}
If the Betti table of $S/Q_F$ has second row $(3,1,0)$, then $S/Q_F$ has Betti table
\[
 \begin{matrix}1 & - & - & - &  \cr
      - & 3& 1& - &  \cr
      - &- & 2& 1 &  \cr
\end{matrix}.
\]
In this case $Q_F$ is the ideal of the union of a plane conic and a (possibly embedded) scheme of length $2$. 
\end{lemma}
\begin{proof}
If the first syzygy has rank $3$, then by Proposition \ref{syzygy ideals}, $Q_F$ has a subideal $J$ generated by $2\times 2$ minors of a $2\times 3$ matrix, which has $b_{23}=2$, contradiction. Therefore, the rank of the first syzygy is $2$. This means $Q_F=J+\langle q_3\rangle$, where $J=\langle l_1l_3,l_2l_3\rangle$ and $(J:q_3)= J$. Hence the Betti table of $S/Q_F$ is
\[
 \begin{matrix}1 & - & - & - &  \cr
      - & 3& 1& - &  \cr
      - &- & 2& 1 &  \cr
\end{matrix}.
\]
In this case the scheme defined by $Q_F$ is either a plane conic and two points, for example $\langle x_0x_1,x_0x_2,q\rangle$, or a plane conic with two embedded points, for example $\langle x_0^2,x_0x_1,q\rangle$. 
\end{proof}

\begin{lemma}\label{lembetti320}
If the Betti table of $S/Q_F$ has second row $(3,2,0)$, then $S/Q_F$ has Betti table
\begin{equation*}
\begin{matrix}1 & - & - &   \cr
      - & 3& 2&   \cr
\end{matrix}.
\end{equation*}
In this case, $Q_F$ is generated by $2\times 2$ minors of a $2\times 3$ matrix of linear forms and is the ideal of a curve of degree $3$ that spans $\Pn^3$.
\end{lemma}
\begin{proof}
Since there are only two linear syzygies, there is by Lemma \ref{rank2syzygies} at least one first linear syzygy of $Q_F$ having rank $3$.  Therefore, by Proposition \ref{syzygy ideals}, $Q_F$ is generated by the $2\times 2$ minors of a $2\times 3$ matrix and is the ideal of a curve with Hilbert polynomial $3t+1$.
\end{proof}

\vskip .05in
\subsubsection{\noindent {\bf Case $b_{12}=4$}}
\vskip .05in
When $F^\perp$ is a complete intersection $(2,2,2,2)$, then $Q_F=F^\perp$.
\begin{proposition}\label{qCGKK3}
Let $F$ be a quaternary quartic, then the following are equivalent
\begin{enumerate}
    \item $Q_F=F^\perp$
    \item $S/Q_F$ is Artinian.
    \item $Q_F$ contains a complete intersection $(2,2,2,2)$
\end{enumerate}
Furthermore, the set of forms $F$ satisfying these conditions is irreducible.
\end{proposition}
\begin{proof}
Clearly $(1)$ implies $(2)$.  Furthermore, if $S/Q_F$ is Artinian, then a general subset of four quadrics in $Q_F$ forms a complete intersection, so $(2)$ implies $(3)$.  Finally, if $Q_F$ contains a complete intersection $I$, then $I(4)\subset S_4$ has codimension $1$, so coincides with $F^\perp$ for a quartic form $F$. Therefore $Q_F\subset I=F^\perp$. 

The complete intersections form a dense open set in the Grassmannian $Gr(4,S_2)$, so an irreducible set.
\end{proof}
\begin{proposition}\label{apolar b12=4}
Let $F$ be a quaternary quartic with $b_{12}(A_F)=4$, and assume $F$ is not of type $[400]$. Then $Q_F$ is saturated and is the ideal of an apolar scheme $Z$.
\begin{enumerate}
    \item If $F$ is of type $[420]$, then $S/Q_F$ has Betti table
    \[
    \begin{matrix}&&&&\cr
 1 & - & -& - &  \cr
                                               - & 4& 2& - &  \cr
                                               - &-& 3&2&  \cr
               \end{matrix}
    \]
   and $Q_F$ is the ideal of a scheme  $Z$ of length $6$, the intersection of a curve of degree $3$ defined by a determinantal net of quadrics, and another quadric. 
     \item If $F$ is of type $[430]$, then $S/Q_F$ has Betti table
     \[\begin{matrix}&&&&\cr
 1 & - & -& - &  \cr
                                               - & 4& 3& - &  \cr
                                               - &-& 1&1&  \cr
               \end{matrix}
               \]
     and $Q_F$ is the ideal of the union $Z$ of a line and a scheme of length $3$ that span $\Pn^3$.
    \item If $F$ is of type $[441]$, then $S/Q_F$ has Betti table
 \[   \begin{matrix}&&&&\cr
 1 & - & -& - &  \cr
                                               - & 4& 4& 1 &  \cr
               \end{matrix}.
    \]
    If the second order syzygy has rank $3$, then  $Q_F$ is the ideal of the union of a plane conic and a (possibly embedded) point. Designate this as type $[441a]$.
  
  If the second order syzygy has rank $4$, then  $Q_F$ is the ideal of the union $Z$ of two skew lines or a double line on a smooth quadric surface. Designate this as type $[441b]$.
      \end{enumerate}
      The set of ideals $Q_F$, when $F$ is of type $[420]$ or $[430]$, form irreducible components of the Hilbert scheme, while when $F$ is of type $[441]$ the corresponding ideals $Q_F$ form two irreducible components of the Hilbert scheme.
\end{proposition}
 
\begin{proof}   The three cases follow from lemmas \ref{lem420}, \ref{lem430} and \ref{lem441} below.  In the first two cases a general form $F$ apolar to $Z$ is of type $[420]$ and type $[430]$ respectively, and the schemes $Z$ form irreducible components of the Hilbert scheme. In the third case, the two subcases for $Z$ each correspond to irreducible components of the Hilbert scheme, so the proposition follows.
\end{proof}
We say that a quaternary quartic $F$ is of type $[441a]$ when the second order syzygy of $Q_F$ has rank $3$, and that a quaternary quartic $F$ is of type $[441b]$ when the second order syzygy of $Q_F$ has rank $4$.

Note that if the Betti table of $S/Q_F$ has second row $(4,0,0)$, then $Q_F$ is a $(2,2,2,2)$ complete intersection. 
\begin{lemma}\label{lem420}
If the Betti table of $S/Q_F$ has second row $(4,2,0)$, then $S/Q_F$ has Betti table 
\[
 \begin{matrix}1 & - & - & - &  \cr
       - & 4& 2 & - &  \cr
      - &- & 3& 2 &  \cr
\end{matrix} 
\] 
and the scheme defined by $Q_F$ is a scheme of length $6$, the intersection of a determinantal curve of degree $3$ and another quadric. 

\end{lemma}
\begin{proof}
By Lemma \ref{all syzygies of rank 4} not all first syzygies of $Q_F$ have rank $4$.
If there is a linear first syzygy of rank $3$, then by Proposition \ref{syzygy ideals}, there is a subideal $J\subset Q_F$ such that $J$ is $2\times 2$ minors of a $2\times 3$ matrix of linear forms, and $Q_F=J+\langle q_4\rangle$. Since $q_4$ is a non-zero divisor on $J$ iff $(J:q_4) =J$, the result is immediate in this case, so suppose $q_4$ is a zero divisor, and consider the short exact sequence of Equation~\ref{sesQuadrics}, with $q_3$ replaced by $q_4$. Note that $(J:q_4)$ has no linear generator, for such a generator would yield an extra linear syzygy in the resolution for $Q_F$, and the two linear syzygies from $J$ preclude this. Since $\deg(J)=3$, $\deg(J:q_4) \in \{ 1,2 \}$. Since the associated primes of $(J:q_4)$ are a subset of the associated primes of $J$, all associated primes of $(J:q_4)$ are of codimension $2$. In particular, if $\deg(J:q_4)$ is $1$, then it contains a linear form, a contradiction, and if $\deg(J:q_4)$ is $2$ and $(J:q_4)$ is prime, it is generated by a linear form and irreducible quadric, again a contradiction. Finally, if $\deg(J:q_4)$ is $2$ but it is not prime, it is either a pair of reduced lines, or a double structure on a line. In the first case, the lines cannot meet (or there is a linear form in the idea), so the lines are skew, with resolution below. 
\[
    \begin{matrix}1 & - & -& - \cr
       - & 4& 4 & 1 \cr
    \end{matrix}.
    \]If $(J:q_4)$ is a double structure on a line, then as in the proof of Lemma~\ref{lem_2LS300}, $(J:q_4)$ is either $\langle x_0,x_1^2\rangle$ or $\langle x_0^2,x_0x_1,x_1^2,x_0l_0+x_1l_1\rangle$. The first case is impossible due to the existence of a linear generator. The second case has resolution as above. Both possibilities yield a mapping cone resolution for $S/Q_F$ of the form 
\[
 \begin{matrix}1 & - & - & - & - \cr
       - & 4& 2 & - & - \cr
      - &- & 4& 4 & 1 \cr
\end{matrix},
\]
which is precluded by Lemma~\ref{ImpossibleBettiTable}. So a rank 3 linear syzygy implies that $S/Q_F$ has Betti table
\[
 \begin{matrix}1 & - & - & - &  \cr
       - & 4& 2 & - &  \cr
      - &- & 3& 2 &  \cr
\end{matrix}.
\]
If there is no linear first syzygy of rank $3$, then $Q_F$ has two rank $2$ linear syzygies involving different pairs of generators, so is of the form $\langle x_0l_{1},x_1l_{1},x_2l_{2},x_3l_{2}\rangle$, which has Betti table
\[
 \begin{matrix}1 & - & - & - & - \cr
       - & 4& 2 & - & - \cr
      - &- & 4& 4 & 1 \cr
\end{matrix}.
\]
This is precluded by Lemma~\ref{ImpossibleBettiTable}.
\end{proof}

\begin{lemma}\label{lem430}
If the Betti table of $S/Q_F$ has second row $(4,3,0)$, then $S/Q_F$ has Betti table
\[
 \begin{matrix}1 & - & - & - &  \cr
           - & 4& 3& - &  \cr
           - & -& 1& 1 &  \cr
\end{matrix}.
\]
In this case $Q_F$ is the ideal of a a line and a (possibly embedded) scheme of length $3$.
\end{lemma}
\begin{proof} We analyze by the rank of linear first syzygies. By Lemma \ref{all syzygies of rank 4}, not all first linear syzygies of $Q_F$ have rank $4$.
\vskip .1in
\begin{description}
    \item[{\bf There is a rank 3 linear first syzygy on $Q_F$}] Then by Proposition \ref{syzygy ideals}, there must be a subideal $J\subseteq Q_F$ such that $S/J$ has Betti table
\[
 \begin{matrix}1 & - & - &   \cr
      - & 3& 2&   \cr
\end{matrix}.
\]
So $J$ is Cohen-Macaulay of degree $3$ and codimension $2$. Assume $Q_F=J+\langle q_4\rangle$.  Then $Q_F$ has a linear syzygy involving $q_4$, and $(J:q_4)$ is the ideal of a curve with one linear generator. 
If $\deg (J:q_4)=3$, then $(J:q_4)=J$, so it cannot have a linear generator, contradiction. If $\deg (J:q_4)=1$, then $(J:q_4)$ has two linear generators, again a contradiction. So $\deg (J:q_4)=2$ and $(J:q_4)$ is a $(1,2)$ complete intersection.  Therefore  $S/(J:q_4)$  has Betti table
\[
 \begin{matrix}1 & 1 & - &   \cr
      - & 1& 1&   \cr
\end{matrix},
\]
and, from the mapping cone, $S/Q_F$ has Betti table
\[
 \begin{matrix}1 & - & - & - &  \cr
           - & 4& 3& - &  \cr
           - & -& 1& 1 &  \cr
\end{matrix}.
\]
The quadric $q_4$ intersects $\VV(J)$ in a line, and the complete intersection  $(J:q_4)$ in a scheme of length $4$.  But the line meets the complete intersection in a point, so $Q_F$ is the ideal of the union $Z$ of a line and a possibly embedded scheme of length $3$.
\vskip .1in
\item[{\bf There is no linear first syzygy of rank $\ge 3$ on $Q_F$}] By Lemma \ref{lem_4gen3rank2syz}, $Q_F$ contains either a subideal $J_{1}$ with Betti table of $S/J_1$
\[
 \begin{matrix}1 & - & - &   \cr
      - & 3& 2&   \cr
\end{matrix}
\]
or a subideal $J_{2}$ with Betti table of $S/J_2$
\[
 \begin{matrix}1 & - & - & - & \cr
      - & 3& 3& 1 & \cr
\end{matrix}.
\]
If $J_{1}\subseteq Q_F$, then by the same reasoning as the case with rank 3 syzygy above, $S/Q_F$ has Betti table
\[
 \begin{matrix}1 & - & - & - &  \cr
           - & 4& 3& - &  \cr
           - & -& 1& 1 &  \cr
\end{matrix}.
\]
On the other hand, $J_{2}$ cannot be in $Q_F$, because the linear second syzygy on $J_2$ would persist in the resolution of $Q_F$. 
\end{description}
Therefore, the resolution of $Q_F$ with top row $(4,3,0)$ is as claimed.
\end{proof}

\begin{lemma}\label{scheme in a plane and a point}
If $Q_F$ admits a second linear syzygy $f$ of order 3, then $I_f$ is the ideal of a plane and a (possibly embedded) point $P$, such that all elements of $Q_F$ vanish at $P$. 
\end{lemma}
\begin{proof} The description of $I_f$ follows from Proposition \ref{syzygy ideals}. For the last statement,
change coordinates so that $P=(1:0:0:0)$ and $I_f$ is one of the following:
\begin{enumerate}
    \item $I=\langle x_0x_1,x_0x_2,x_0x_3 \rangle$
    \item $I=\langle x_1^2,x_1x_2,x_1x_3 \rangle$
\end{enumerate}
Let $q$ be a quadric $\in Q_F\setminus I$, and suppose $q$ does not vanish at $P$. Then after modifying $q$ using elements from $I$ we can write $q$ as 
\[
q=x_0^2+f(x_1,x_2,x_3), \mbox{ with }f \mbox{ quadratic}.
\]
Now if $I=\langle x_0x_1,x_0x_2,x_0x_3 \rangle$ we see that $x_0g\in I\subset Q_F$ for any polynomial $g$ in the variables $x_1,x_2,x_3$, and therefore $x_0^3=qx_0-fx_0 \in Q_F$. 

It follows that all cubics and in consequence all quartics vanishing on the plane $x_0=0$ are in $Q_F$. But this is impossible if $F$ is non-degenerate. 
Similarly, if $I=\langle x_1^2,x_1x_2,x_1x_3 \rangle$, then after reducing modulo $I$, we may assume that $q=x_0^2+ax_0x_1+x_0(x_2,x_3)$. Then 
\[
x_0^2x_1=qx_1-ax_0x_1^2+x_0x_1l(x_2,x_3)\in Q_F,
\]
but $\langle x_0^2x_1 \rangle+ I$ generates all cubics vanishing on the plane $x_1=0$, and again we have a contradiction.
\end{proof}
\begin{lemma}\label{lem441}
If the Betti table of $S/Q_F$ has second row $(4,4,1)$, then the Betti table of $S/Q_F$ is
\[
 \begin{matrix}1 & - & - & - & \cr
           - & 4& 4& 1 &  \cr
\end{matrix}.
\]
 If the second order syzygy has rank $4$, then $Q_F$ is the ideal of two skew lines or a double line in a smooth quadric surface.
If the second order syzygy has rank $3$, then $Q_F$ is the ideal of  a plane conic and a, possibly embedded, point.

\end{lemma}
\begin{proof}
Let $f$ denote the second linear syzygy of $Q_F$. By Proposition \ref{syzygy ideals} we have the following two possibilities:
\begin{description}
 \item[{\bf $f$ is rank $4$}] Then $I_f$ is the ideal of either two skew lines, or of a double line in a smooth quadric surface, such that $S/I_f$  in both cases has Betti table 
    \[
 \begin{matrix}1 & - & - & - & \cr
           - & 4& 4& 1 &  \cr
\end{matrix}.
\]
In that case we must have $Q_F=I_f$.
\vskip .05in
\item[{\bf $f$ is rank $3$}] Then $I_f$ is the ideal of a plane and a (possibly embedded) point $P$, and  $S/I_f$ has Betti table 
    \[
 \begin{matrix}1 & - & - & - & \cr
           - & 3& 3& 1 &  \cr
\end{matrix}.
\] In that case $Q_F=I_f+\langle q\rangle$ for some quadric $q$ vanishing at $P$ as seen in Lemma \ref{scheme in a plane and a point}. It follows that $Q_F$ is the ideal of a plane conic and a (possibly embedded) point.
\end{description}
Thus the possibilities are as claimed.
\end{proof}
\vskip .05in
\subsubsection{\noindent {\bf Case $b_{12}=5$}}
\vskip .05in
\begin{proposition}\label{apolar b12=5}
Let $F$ be a quaternary quartic with $b_{12}(A_F)=5$. Then the quadratic ideal $Q_F$ is saturated and is the ideal of an apolar scheme $Z$.
\begin{enumerate}
    \item If $F$ is of type $[550]$, then $S/Q_F$ has Betti table
 \[   \begin{matrix}&&&&\cr
 1 & - & -& - &  \cr
                                               - & 5& 5& - &  \cr
                                               - &-& - &1&  \cr
               \end{matrix}
               \]
   and $Q_F$ is the ideal of a scheme $Z$ of length $5$ with no subscheme of length $4$ in a plane. 
     \item If $F$ is of type $[551]$, then $S/Q_F$ has the Betti table
    \[ \begin{matrix}&&&&\cr
 1 & - & -& - &  \cr
                                               - & 5& 5& 1 &  \cr
                                               - &-& 1&1&  \cr
               \end{matrix}
               \]
               and $Q_F$ is the ideal of a scheme $Z$ of length $5$ with a subscheme of length $4$ in a plane and no subscheme of length $3$ in a line.
          \item If $F$ is of type $[562]$, then $S/Q_F$ has the Betti table
\[
        \begin{matrix}&&&&\cr
 1 & - & -& - &  \cr
                                               - & 5& 6& 2 &  \cr
               \end{matrix},
         \]
   and $Q_F$ is the ideal of the union $Z$ of a line and a (possibly embedded) scheme of length $2$ such that that $Z$ spans $\Pn^3$.
      \end{enumerate}
For each of the three types, the corresponding set of ideals $Q_F$ form an irreducible component of the Hilbert scheme. 
\end{proposition}
\begin{proof}  The three cases follow from the lemmas \ref{lem551and562} and \ref{lembetti550} below. In each case, the schemes $Z$ corresponding to $Q_F$ form an irreducible component of the Hilbert scheme. 
\end{proof}
We analyze $Q_F$ in terms of its second linear syzygies. 

\begin{lemma}\label{lem551and562} If $Q_F$ has $b_{12}=5$ and admits a second linear syzygy then $Q_F$ is one of the following:
\begin{enumerate}
\item $Q_F$ is the ideal of a scheme $Z$ of length $5$, with a complete intersection subscheme of length $4$ in a plane.
Then $S/Q_F$ has Betti table
    \[ \begin{matrix}&&&&\cr
 1 & - & -& - &  \cr
                                               - & 5& 5& 1 &  \cr
                                               - &-& 1&1&  \cr
               \end{matrix}.
               \]
               Furthermore, no subscheme of length 3 in $Z$ is contained in a line.
\item $Q_F$ is the ideal of a line and a (possibly embedded) subscheme of length $2$, and $S/Q_F$ has Betti table 
\[
 \begin{matrix}1 & - & - & - & \cr
           - & 5& 6& 2 &  \cr
\end{matrix}.
\]
\end{enumerate}

\end{lemma}
\begin{proof}
By Proposition \ref{syzygy ideals} and Lemma \ref{all syzygies of rank 4} we have two possibilities.
$b_{34}=1$ and the second linear syzygy $f$ of $Q_F$ has rank $3$ or $4$. If $f$ has rank $3$, then 
 $I_f$ is a subideal of $Q_F$ which is the defining ideal of a plane and a point $P$. By Lemma \ref{scheme in a plane and a point} the remaining quadrics must pass through the point $P$ and when restricted to the plane they vanish on either a complete intersection scheme of length 4  or on a line and a point. In both cases the ideal will be saturated and as in the assertion. Note that in the case where $Q_F$ defines a finite scheme $Z$, since it is defined by quadrics, this finite scheme cannot contain any subscheme of length $3$ contained in a line.  
\vskip .05in
 If $b_{34}=1$ and the unique (up to scaling) second linear syzygy $f$ is of rank 4,  then by Proposition \ref{syzygy ideals}, the ideal $Q_F$ contains the ideal of two skew lines. Adding one more quadric we obtain either an ideal defining four points or the ideal of a line and two points. 

In the first case, consider a subideal $Q'$ of $Q_F$ consisting of quadrics passing through two general additional points coplanar with three of the four points. The subideal $Q'$ is generated by three quadrics and vanishes on a plane and an additional point. But $Q'$ has a second linear syzygy of rank $3$ which is a contradiction with $b_{34}=1$ and $f$ of rank 4.

In the second case the Betti table of $S/Q_F$ is  
\[
 \begin{matrix}1 & - & - & - & \cr
           - & 5& 6& 2 &  \cr
\end{matrix}.
\]
This contradicts $b_{34}=1$, hence the case $b_{34}=1$ with the second linear syzygy of rank 4 never occurs. 
\end{proof}

\begin{lemma}\label{lembetti550}
If $F$ is a quaternary quartic such that $A_F$ is of type $[550]$, then 
$Q_F$ is the ideal of a scheme of length $5$, with no subscheme of length $4$ lying in a plane.  
\end{lemma}
\begin{proof} 
The Betti table of $A_F$ is 
\[
        \begin{matrix}1 & - & - & - & -& \cr
       - & 5& 5& - & -& \cr
        - &1& -& 1 & -& \cr
         - &-& 5& 5 & -& \cr
          - & -& -& - & 1& \cr
        \end{matrix}.
        \]
        Hence $F^\perp$ is generated by $Q_F$ and a cubic form.
Notice that the second syzygy of degree 5 must be a syzygy only on the quadrics. 
Furthermore, the first syzygies of degree $5$ are all Koszul syzygies between the cubic form and the quadrics.  Therefore $S/Q_F$ has Betti table 
\[
        \begin{matrix}1 & - & - & -  \cr
       - & 5& 5& -  \cr
        - &-& -& 1  \cr
        \end{matrix}
        \]
and $I$ is the ideal of a finite scheme $Z$ of length 5.
  
  If $Z$ contains a subscheme of length $4$ in a plane, then at least a $3$-dimensional space of quadrics in $Q_F$ would vanish on the plane of this subscheme.  But then these quadrics have a second order linear syzygy, contradicting the fact that $Q_F$ does not.
   \end{proof}
\subsubsection{\noindent {\bf Case $b_{12}=6$}}
\vskip .05in
\begin{proposition}\label{apolar b12=6}
Let $F$ be a quaternary quartic with $b_{12}(A_F)=6$. Then $F$ is of type $[683]$ and the quadratic ideal $Q_F$ is saturated. The Betti table of $S/Q_F$ is
\[
        \begin{matrix}1 & - & - & -  \cr
       - & 6& 8& 3  \cr
        \end{matrix}
        \]
and $Q_F$ is the ideal of a finite apolar scheme $Z$ of length $4$.
When $F$ is of type $[683]$, the corresponding set of ideals $Q_F$ form an irreducible component of the Hilbert scheme.
\end{proposition}
\begin{proof} It follows from the proof of Theorem 2.2 of \cite{SSY} that when $b_{12}=6$ the top two rows of the Betti table of $S/Q_F$ are
\[
        \begin{matrix}1 & - & - & -  &-\cr
       - & 6& 8& 3 &- \cr
        \end{matrix}.
        \]
We wish to show that this is the entire Betti table for $S/Q_F$.
  There are a number of ways to do
this.  Here we use a generic initial ideal argument.  After a general change of coordinates, in graded reverse lex order the initial ideal, $\gin(Q_F)$, will be a strongly stable ideal with 6 quadratic generators (see Appendix \ref{Appendix2} for the terminology and basic results used here).  There are two possibilities: 
\begin{enumerate}
    \item The quadradic part of $\gin(Q_F)$ is $(x_0, x_1, x_2)^2$.
    \item $\gin(Q_F)$ contains a monomial $x_i x_3$. 
    \end{enumerate} 
    In the second case, the saturation of $Q_F$ is $(Q_F : x_3)$, which will contain a linear form $L$, and therefore $Q_F$ contains $(x_0, x_1, x_2, x_3) \cdot L$, which contradicts the known quadratic linear strand of the Betti table of $S/Q_F$ above.

Therefore, we must be in the setting of the first case, where the degree $2$ component of $\gin(Q_F)$ is $(x_0, x_1, x_2)^2$.
To compute a Gr\"obner basis for $Q_F$, notice that there are eight minimal S-pairs on the six generators of $Q_F$ (in these general coordinates), which all occur in degree $3$.  Each of these must reduce to 0, since there are eight linear syzygies on the generators of $Q_F$.  Therefore the six generators of $Q_F$ form a Gr\"obner basis. As the Betti table of $S/Q_F$ is pointwise at most the Betti table of $S/\gin(Q_F)$, this forces the Betti table of $S/Q_F$ to be the above table.

Therefore, $Q_F$ defines a degree 4 zero scheme in $\Pn^3$, and this scheme spans $\Pn^3$.  Schemes of length $4$ that span $\Pn^3$ form an open dense subset of the Hilbert scheme $\Hilb^4(\Pn^3)$, and the proposition follows.
\end{proof}

\section{Rank and powersum presentations of quartics}\label{section_VSP}
In this section we determine the rank $r=\rr(F)$ and the variety of powersum presentations $VSP(F,r)$ of the general form $F$ of each Betti table type.   

Theorem 2.2 of \cite{SSY} proves that a graded Artin Gorenstein algebra $A$ of regularity and codimenson $4$ has one of $16$ Betti tables; these Betti tables are reproduced as Table \ref{tableCGKK} 
and Table \ref{tableremaining} in Appendix \ref{Appendix}. Recall from Definition~\ref{FBdefinition} that ${\mathcal F}_B\subset \Pn^{34}$ is the set of forms $F$ for which $A_F$ has Betti table $B$, labelled as in the Appendix \ref{Appendix}. We will analyze the sets ${\mathcal F}_B$ via the set of quadratic ideals $\{Q_F|F\in {\mathcal F}_B\}$.  These sets of quadratic ideals were identified in Section \ref{quadraticideals} via their Betti tables, and by Definition~\ref{GbettiTable}, ${\mathcal G}_B$ denotes the set of ideals $Q_F$ for quaternary quartics $F$ such that $S/Q_F$ has Betti table $B$.

In Section \ref{quadraticideals} we 
showed that the sets ${\mathcal G}_B$ are irreducible, except for one Betti table.  The set
${\mathcal G}_{[441]}$ has two components that we denote ${\mathcal G}_{[441a]}$ and ${\mathcal G}_{[441b]}$, which are distinguished by the rank of a second linear syzygy.
With these distinctions the quadratic ideals form $16$ different irreducible components. 
We begin by discussing the relation between the sets 
${\mathcal F}_B$ and ${\mathcal G}_{B'}$, when the latter is an irreducible component of the set of quadratic ideals that contains an ideal  $Q_F\in {\mathcal G}_{B'}$ for some quartic form $F$ for which $A_F$ has Betti table $B$. 

For two sets ${\mathcal F}_B$ we distinguish some subsets. Recall that the cactus rank $\operatorname{cr}(F)$ of a form $F$ is the shortest length of a subscheme apolar to $F$.
\begin{definition}
We set 
$${\mathcal F}_{[300a]}=\{F |Q_F\in {\mathcal G}_{[300ab]}, \operatorname{cr}(F)=8\},
{\mathcal F}_{[300b]}=\{F |Q_F\in {\mathcal G}_{[300ab]}, \operatorname{cr}(F)=7\},$$
$${\rm and}\;\;{\mathcal F}_{[300c]}=\{F |Q_F\in {\mathcal G}_{[300c]}\},$$


and we  set 

$${\mathcal F}_{[441a]}=\{F |Q_F\in {\mathcal G}_{[441a]}\}\;\; {\rm  and}\;\;
{\mathcal F}_{[441b]}=\{F |Q_F\in {\mathcal G}_{[441b]}\}.$$
\end{definition}
\begin{lemma}\label{specialsubsets} 
$${\mathcal F}_{[300]}={\mathcal F}_{[300a]}\sqcup {\mathcal F}_{[300b]}\sqcup{\mathcal F}_{[300c]}$$
and 
$${\mathcal F}_{[441]}={\mathcal F}_{[441a]}\sqcup {\mathcal F}_{[441b]}.$$
\end{lemma}
\begin{proof}
The last part follows from the distinction of the sets of quadratic ideals ${\mathcal G}_{[441]}={\mathcal G}_{[441a]}\sqcup {\mathcal G}_{[441b]}$. For the first part, the ideals in ${\mathcal G}_{[300ab]}$ are the complete intersections $(2,2,2)$. 
The forms $F$ with $Q_F$ a complete intersection $(2,2,2)$ and cactus rank $7$, are, by apolarity, precisely those containing an ideal $J\subset F^\perp$ with $J\in \operatorname{Hilb}_7$ and no ideal $J'\subset F^\perp$ with $J'\in \operatorname{Hilb}_k, \;\; k<7$.  Any ideal $J'\in \operatorname{Hilb}_k, \;\; k<7$ would have ${\rm dim}J'(2)\geq 4$, so $J'\subset F^\perp$ would imply $b_{12}(A_F)\geq 4$ against the assumption.
So $${\mathcal F}_{[300b]}=\{F| Q_F\in {\mathcal G}_{[300ab]},  F^\perp\supset J\;{\rm with}\; J\in \operatorname{Hilb}_7\}.$$

 Clearly,
 any form $F$ in the complement of ${\mathcal F}_{[300b]}$ in the set 
 $$\{F|Q_F\in {\mathcal G}_{[300ab]}\}$$
 has cactus rank $8$. In fact, $Q_F\in \operatorname{Hilb}_8$, and any $J'\subset F^\perp$ with $J'\in \operatorname{Hilb}_k, \;\; k<8$ would contain the complete intersection $Q_F$.  Therefore, this complement coincides with 
 ${\mathcal F}_{[300a]}$ and the lemma follows.
\end{proof}

\begin{proposition}\label{irreducibleF_B} 
The following $19$ sets of quaternary quartics are all irreducible:
\begin{enumerate}
\item The $14$ sets ${\mathcal F}_B$ for a Betti table $B$ different from $[300]$ and $[441]$.
\item  The three sets 
${\mathcal F}_{[300a]},{\mathcal F}_{[300b]},{\mathcal F}_{[300c]}$.
\item The two sets ${\mathcal F}_{[441a]},{\mathcal F}_{[441b]}$.
\end{enumerate}


\end{proposition}
\begin{proof}
We first set up the natural algebraic incidence between ideals $Q_F$ and forms $F$.

\begin{lemma}\label{incidence} Let $k>0$ and consider the Grassmannian $\Gr(k,S_{2})$ of $k$
dimensional spaces of quadrics in $S$. The set $X=\{(Q,[F])\in \Gr(k,S_{2})\times \PP^{34}\mid q\circ F=0 \mbox{ for all }q\in Q\}\subset \Gr(k,S_{2})\times \PP^{34}$ is closed.
\end{lemma}
\begin{proof}
Let $\Gamma\subseteq S_{2}$ be a $(10-k)$-subspace of $S_{2}$. Assume $U_{\Gamma}\subset \Gr(k,S_{2})$ is the subset of $k$-planes which only intersect $\Gamma$ at $0$. Then $U_{\Gamma}\simeq\CC^{k(n-k)}$ is an open subset in $\Gr(k,S_{2})$. Assume the local coordinates of $(U_{\Gamma})$ is $\{a_{ij}\}_{1\leq i\leq k, 1\leq j\leq n-k}$, that is, a $k$-plane in $U_{\Gamma}$ can be written as the row space of
\begin{equation}
    \begin{bmatrix}
    1& & &a_{1,1}& \dots&a_{1,n-k} &\\
     &\ddots & &\vdots& &\vdots &\\
    & & 1&a_{k,1}& \dots&a_{k,n-k} &
    \end{bmatrix}.
\end{equation}
Therefore $X\cap (U_{\Gamma}\times \PP^{34})$ is given by zero locus of polynomials in $a_{i,j}$ and the coordinates of $\PP^{34}$, hence closed on $(U_{\Gamma}\times \PP^{34})$. Note that this argument works on every patch of an open affine cover $\{U_{\Gamma}\}$ of $\Gr(k,S_{2})$, so $X$ is closed.
\end{proof}

In the notation of Lemma \ref{incidence} we may consider the projection maps $\pi_{1}: X\to \Gr(k,S_{2})$ and $\pi_{2}: X\to \PP^{34}$. Then $\pi_{2}(X)=\{F\in\PP^{34}\mid b_{12}(S/Q_{F})\geq k\}$. Assume $F\in\pi_{2}(X)$. Then the fiber of $\pi_{2}$ over $F$ is $\{Q\in \Gr(k,S_{2})\mid Q\subseteq Q_{F}\}$. In particular, when $b_{12}(S/Q_F)=k$, the fiber has only one point. On the other hand, the fiber of $\pi_{1}$ over a $Q\in \Gr(k,S_{2})$ is the set $\{[F']\in\PP^{34}\mid Q\subseteq Q_{F'}\}$.
Thus there is, for each $b_{12}=b_{12}(S/F^\perp)=b_{12}(S/Q_F)$ a regular map 
$$\varphi: \{F\in \Pn^{34}|b_{12}(S/F^\perp)=k\}\to G(k,S_2),\quad [F]\mapsto [Q_F].$$


We now let ${\mathcal G}$ be an irreducible set of saturated quadratic ideals in $S$, all with the same Betti table. Let $I\in {\mathcal G}$ and let $B'$ be the Betti table of $S/I$. 
Let $I_4\subset S_4$ be the space of quartic forms in $I$.  Each quaternary quartic form $F\in R_4$ defines a hyperplane $F^\perp\subset S_4$.  
If $I_4\subset F^\perp$, i.e.~$F\in (I_4)^\perp\subset S_4^*=R_4$,  then $I\subset Q_F$ since both are saturated.  If the latter inclusion is strict, the inclusion $I_4\subset(Q_F)_4$ is strict, therefore $I=Q_F$ for the general $F\in (I_4)^\perp$.  In particular this is an irreducible set of dimension depending only on the common Betti table $B'$ of $S/I$ for ideals $I\in {\mathcal G}$.  
 The restriction of $\varphi$ is therefore a map 
$$\{F|Q_F\in \mathcal{G}_{B'}\}\to \mathcal{G}_{B'}\quad F\mapsto Q_F$$
with irreducible fibers of the same dimension and the proposition follows.

Finally, the ideals in ${\mathcal G}_{[300ab]}$ are the complete intersections $(2,2,2)$ and form an irreducible set.  Therefore the set of forms $F$ with $Q_F$ a complete intersection is also irreducible.
Now, the set of ideals $J$ of subschemes of length $7$ is irreducible and those ideals $J$ that additionally contain a complete intersection $(2,2,2)$ form an dense subset ${\mathcal G'}\subset \operatorname{Hilb}_7$.  In particular, as in the above proof for irreducible sets of quadratic ideals, the set of forms ${\mathcal F}_{[300b]}\subset \Pn^{34}$, i.e.~the set of forms $F$ with an ideal $J\subset F^\perp$ and $J\in {\mathcal G'}$, is irreducible.


 By Lemma \ref{specialsubsets}, the set
 ${\mathcal F}_{[300a]}$ is 
 the complement of ${\mathcal F}_{[300b]}$ in the set of forms $F$ with $Q_F$ a complete intersection $(2,2,2)$. The latter set is irreducible, while ${\mathcal F}_{[300a]}$ clearly is a dense subset, so ${\mathcal F}_{[300a]}$ is also irreducible. 
\end{proof}

For each of the irreducible subsets of $\Pn^{34}$ of  Proposition \ref{irreducibleF_B}, we compute the dimension of ${\mathcal F}_B$, and
for a general quartic $F\in {\mathcal F}_B$ we determine the rank $\rr(F)$ and $VSP(F,r)$, and characterize as in Section \ref{bettitables}, the general set of points computing $\rr(F)$. 

\begin{proposition}\label{vsp proposition}
Let $F$ be a quaternary quartic that is not a cone.  Then $F$ belongs to one of $19$ disjoint irreducible sets listed in tables \ref{tableCGKK} and \ref{tableremaining} with the given dimension.  For each set the table gives, for a general form $F$ in the set,  the rank $r=r(F)$, the variety $VSP(F,r)$, and the Betti table for a general set $\Gamma$ of $r$ points apolar to $F$.  
In each case, the set of points $\Gamma$ impose independant conditions on quartics.
\end{proposition}

\begin{proof}
By Theorem 2.2 of \cite{SSY} we know that any quartic $F$ which is not a cone has one of $16$ Betti tables, and hence belong to one of the $19$ sets of Proposition \ref{irreducibleF_B}.

The rank of a form $F$ has the following lower bound:
\begin{lemma}\label{rankq}
Let $b_{12}(F)$ be the dimension of the space of quadratic forms in $F^\perp$, i.e.~the space of  generators of $Q_F$, and let $r(F)$ be the rank of $F$.  Then
$r(F)\geq 10-b_{12}(F)$.
\end{lemma}
\begin{proof} If $\Gamma$ is a set of $r$ points apolar to $F$, then the ideal $I_\Gamma$ of $\Gamma$ is contained in $F^\perp$.  In particular the quadratic part of $I_\Gamma$ must be contained in $Q_F$.  But $r$ points impose at most $r$ conditions on quadrics, so
$10-r\leq {\rm dim}I_\Gamma(2)\leq b_{12}(F)$, which implies $$r\geq 10-b_{12}(F).$$
\end{proof}
For each of the $19$ irreducible sets, we now argue one by one and start with forms with Betti table in Table \ref{tableCGKK} of Section \ref{Appendix}.
%
\begin{table}[H]
\begin{tabular}{|c|c|c|c|c|}
\hline Betti table B & $\rr(F)$ & VSP(F,r) & $\Gamma$ from Table \ref{tablepoints}& $\dim({\mathcal F}_B)$ \\
\hline [683] & $4$       & one point & four points, not in a plane  & $15$\\
\hline [550] & $5$       & one point & five points, no four in a plane & $19$\\
\hline [400] & $8$       & $\Pn^3$ & a $(2,2,2)$ CI  & $24$\\
\hline [320] & $7$       & $\Pn^1$ &seven points on a twisted cubic  & $24$\\
\hline [300a] & $8$       & one point & a $(2,2,2)$ CI  & $28$\\
\hline [300b]& $7$       & one point & seven points in a $(2,2,2)$ CI & $27$\\
\hline [300c] & $7$       &$\Pn^1$  & seven points, three on a line   & $24$\\
\hline [200] & $8$       & two points & eight points in a $(2,2)$ CI & $31$\\
\hline [100] & $9$       & K3 surface &  nine points in a quadric & $33$\\
\hline [000] & $10$       & Fivefold & ten points, not on quadric & $34$\\
\hline
\end{tabular}
\vskip .1in
\caption{VSP for Betti tables from Table~\ref{tableCGKK}}\label{tablevspcgkk}
\vskip -.15in
\end{table}
\begin{description}

\item[{\bf Type $[683]$}]  The quadratic ideal $Q_F$ is, by Proposition \ref{apolar b12=6},  the ideal of a scheme of length $4$, i.e.~four points for a general $F$.
So the rank of $F$ is at most $4$ and therefore by Lemma \ref{rankq} equal to $4$.  Furthermore, the ideal of any four points computing the rank of $F$ contains a $6$-dimensional space of quadrics in $F^{\bot}$, i.e.~all of $Q_F$,  so $VSP(F,4)$ is one point.
The dimension of the family of $4$-tuples of points in $\Pn^3$ is $4 \times 3=12$ and each $4$-tuple of points span a $\Pn^3$ in $\Pn^{34}$,  so there is an irreducible $12+3=15$-dimensional family of forms $F$  with this Betti table in $\Pn^{34}$.

\vskip .1in
\item[{\bf Type $[550]$}] The quadratic ideal $Q_F$ is, by Proposition \ref{apolar b12=5}, the ideal of a scheme of length $5$, i.e.~five points for a general $F$.
 So $r(F)\leq 5$ and the ideal generated by these quadrics is the only ideal of five points contained in $F^{\bot}.$ 
 
By Lemma \ref{rankq} the rank is least $5$, so $r(F)=5$ and $VSP(F,5)$ is a point.
The dimension of the family of $5$-tuples of points in $\Pn^3$ is $5 \times 3=15$, each $5$-tuple of points span a $\Pn^4$ in $\Pn^{34}$,  so there is an irreducible $15+4=19$-dimensional family of forms $F$  with this Betti table in $\Pn^{34}$.
\vskip .1in
\item[{\bf Type $[400]$}] $F^{\bot}$ is a complete intersection of quadrics $(2,2,2,2)$. Any $3$-dimensional subspace of quadrics in $F^{\bot}$ generate the ideal of a scheme of length $8$, 
while no ideal of fewer points are contained in the complete intersection.
Therefore $F$ has rank $8$ and $VSP(F,8)= \Pn^3$.

The Grassmannian $G(4,10)$ of $4$-dimensional spaces of quadrics in $\Pn^3$ has dimension $4 \times 6 = 24$ and parameterizes $(2,2,2,2)$ complete interesections in $\Pn^3$, yielding an irreducible $24$-dimensional family of forms $F$  with this Betti table in $\Pn^{34}$.
\vskip .1in
\item[{\bf Type $[320]$}] The quadratic ideal $Q_F$ is, by Proposition \ref{apolar b12=3},   a determinantal net, i.e.~for a general $F\in {\mathcal F}_{[320]}$, the ideal a twisted cubic curve $C$ that is apolar to $F$. By Lemma \ref{curves} there is a pencil of subschemes  of length $7$ on $C$ that is apolar to $F$.  The ideal of any seven points contain a $3$-dimensional space of quadrics, so if the ideal is contained $F^{\bot}$ the seven points are on the twisted cubic curve.  By Lemma \ref{rankq}, $F$ has rank $7$ and  $VSP(F,7)=\Pn^1$.
The dimension of the family of  twisted cubic curves in $\Pn^3$ is $12$.  Each twisted cubic $C$ is mapped by $v_4$ to a rational normal curve $v_4(C)$ of degree $12$ that spans  a $\Pn^{12}$ in $\Pn^{34}$,  so there is an irreducible $12+12=24$-dimensional family of forms $F$  with this Betti table in $\Pn^{34}$. In Theorem 2.7 of \cite{Boij}, Boij computes the dimension of components of $\Gor(H)$ determined by $s$ points on a rational normal curve in $\Pn^n$, but types $[320]$ fall outside the range where his result applies. 
\vskip .1in
\item[{\bf Type $[300a]$}]  The quadratic ideal $Q_F$ define, by Proposition \ref{apolar b12=3}, a complete intersection $Z$ of length $8$.
When $Z$ consists of eight distinct points and no cubic in $F^{\bot}$ vanishes in any seven points, but not all eight in $Z$, then $r=8$ and $Z$ is the only set of eight points apolar to $F$, so  $VSP(F,8)=\{1 \; point\}$.  The dimension of the family of  complete intersections $(2,2,2)$ in $\Pn^3$ is $3 \times 7=21$, which is the dimension of $G(3,10)$ the Grassmannian of nets of quadrics in $\Pn^3$. Each complete intersection spans a $\Pn^7$ in $\Pn^{34}$,  so there is an irreducible $21+7=28$-dimensional family of forms $F$ of rank $8$  with this Betti table in $\Pn^{34}$.
\vskip .1in
\item[{\bf Type $[300b]$}]  The quadratic ideal $Q_F$ define, by Proposition \ref{apolar b12=3},  a complete intersection $Z$ of length $8$.
One of the cubic generators in $F^{\bot}$ vanishes on a subscheme $Z'$ of $Z$ of length $7$, but not on $Z$.
When $Z'$ consists of seven distinct points, $r(F)=7$ and $VSP(F,7)=\{\mbox{one point}\}$.
In fact, the rank is at most $7$, while, by Lemma \ref{rankq},  the rank is at least $7$, so equality follows. 
If $F$ is in the span of two subset of seven among the eight points $Z$, then it lies in the span of the common six points of the two subset and the rank is at most $6$.  So there is only one point in $VSP(F,7)$.

The dimension of the family of $7$-tuples of points in $\Pn^3$ is $7 \times 3=21$ and each $7$-tuple of points span a $\Pn^6$ in $\Pn^{34}$,  so there is an irreducible $21+6=27$-dimensional family of forms $F$ of rank $7$ with this Betti table in $\Pn^{34}$.
\vskip .1in
\item[{\bf Type $[300c]$}] \label{lineandfourpoints}
Included in boundary of the family $[300b]$ is the family $[300c]$ of quartic forms of rank $7$ that are apolar to the set of  three collinear points and four points in general position, see Proposition \ref{apolar b12=3}.  
In this case $Q_F$ is a not a complete intersection; it ~is the ideal of a line and a scheme  of length $4$, i.e.~for a general $F\in {\mathcal F}_{[300c]}$ a line $L$ and a set $Z$ of four general points.   Forms $F$ of this type has a unique decomposition $F=F_1+F_2$, where $F_1$ is apolar to $L$ 
while $F_2$ is apolar to $Z$:  
$$[F]\in \langle v_4(L\cup Z)\rangle\subset \Pn^{34},$$ so there are unique forms $$[F_1]\in \langle v_4(L)\rangle\;\;{\rm and} \;\;[F_2]\in \langle v_4(Z)\rangle$$ such that $$[F]\in \langle [F_1],[F_2]\rangle.$$

We may assume that $F_1$ is a general form apolar to $L$, so 
by Lemma \ref{curves} $F_1$ has rank $3$ and  $VSP(F_1,3)=\Pn^1$, while $F_2$ is of type $[683]$ so $VSP(F_2,4)=\{point\}$.   Therefore also $VSP(F,7)=\Pn^1$.
The dimension of the family of lines is $4$, so the set of $5$-tuples (line + four points) is $4+3+3+3+3=16$-dimensional. Each $5$-tuple spans a $\Pn^{8}$ in $\Pn^{34}$,  so there is an irreducible $16+8=24$-dimensional family of forms $F$ with this Betti table in $\Pn^{34}$.
\vskip .1in

\item[{\bf Type $[200]$}]   The quadratic ideal $Q_F$ is, by Proposition \ref{apolar b12=2}, a complete intersection $2,2)$, so for a general $F\in {\mathcal F}_{[200]}$ it is the ideal of an elliptic quartic curve $C$ that is apolar to $F$. By Lemma \ref{curves}, $F$ is apolar to two sets of eight points on $C$.
By Lemma \ref{rankq}, the rank of $r(F)=8$ and $VSP(F,8)=\{\mbox{two points}\}$.   The family of elliptic quartics in $\Pn^3$ has dimension $16$.  Each elliptic quartic curve $C$  is mapped by $v_4$ to an elliptic normal curve $v_4(C)$ of degree $16$ spanning a $\Pn^{15}$ in $\Pn^{34}$, yielding an irreducible $16+15=31$-dimensional family of forms $F$  with this Betti table in $\Pn^{34}$.
\vskip .1in
\item[{\bf Type $[100]$}]   
Let $Q$ be the unique quadric in $F^{\bot}$. By Lemma \ref{rankq} the rank of $F$ is at least $9$.
Now,  any set of nine points apolar to $F$ is necessarily contained in $Q$. 
The $4$-uple embedding $v_4(Q)$ of $Q$ spans a $\Pn^{24}\subset \Pn^{34}$. There is a $18$-dimensional family of $9$-tuples of points on the quadric that each span a $\Pn^8$ in $\Pn^{24}$, so one expects that $F$ is apolar to a surface of $9$-tuples of points.  We show that $VSP(F,9)$ is a K3 surface for a general form $F$ of type $[100]$, in Section \ref{vsp9}.

Since $F$ is apolar to a quadric, the middle catalecticant matrix has rank $9$.  A rank $1$ catalecticant matrix is the catalecticant of a rank $1$ quartic form, so the determinant of the catalecticant  defines the closure of the set of forms $F$ of rank $9$, i.e.~forms with a Betti table of type $[100]$.  Since the determinant is irreducible we conclude that there is an irreducible  $33$-dimensional family of forms with this Betti table in $\Pn^{34}$. 
\vskip .1in
\item[{\bf Type $[000]$}] If $F^{\bot}$ contains no quadric, the rank is at least $10$ by Lemma \ref{rankq}.  The family of forms of rank $9$ is dense in a hypersurface, so the general $F\in \Pn^{34}$ must have rank $10$. By dimension count $VSP(F,10)$ is a fivefold for a general $F$, and the rank is computed by ten points that are not in any quadric surface.  
One can prove that a double quadric have rank $10$, the following question arise.
\begin{question}
Describe VSP(2Q,10) with $Q\subset \mathbb{P}^3$ a quadric.
\end{question}
Note that in this case there is an open orbit of OG(4). Recall that the case of variety of sums of powers of the double quadric curve give rise to the Mukai-Umemura example.
\end{description}

              
\vskip .2cm
Next, we consider the Betti tables in 
\cite{SSY}*{Table 2}, (see Table \ref{tableremaining} in Appendix \ref{Appendix}).
\renewcommand{\arraystretch}{1.8}
\begin{center}
\begin{table}[H]
\begin{tabular}{|c|c|c|c|c|}
\hline Betti table B & $r$ & VSP(F,r) & $\Gamma$ from Table \ref{tablepoints}& $\dim({\mathcal F}_B)$ \\
\hline [562] & $5$       & $\Pn^1$ & five points, three collinear & $16$\\
\hline [551] & $5$       & one point & \stackanchor[1pt]{five points, four coplanar}{and no three collinear}  & $18$\\
\hline [441a] & $6$       & $\Pn^1$ & six points, five coplanar  & $20$\\
\hline [441b] & $6$       & $\Pn^1 \times \Pn^1$ & \stackanchor[1pt]{six points, three on}{each of two skew lines} & $17$ \\
\hline [430] & $6$       & $\Pn^1$ & six points, three collinear  & $20$\\
\hline [420] & $6$       & one point & six points, no five coplanar  & $23$\\
\hline [331] & $7$       & Fano threefold & seven points, six coplanar  & $21$\\
\hline [310] & $7$       & $\Pn^1$ & seven points, five coplanar  & $24$\\
\hline [210] & $8$       & Fano threefold $\cup\; \Pn^1\times\Pn^1$ & \stackanchor[1pt]{eight points, six coplanar or}{five coplanar, three collinear} & $25$\\
\hline
\end{tabular}
\vskip .1in
\caption{VSP for Betti tables from Table~\ref{tableremaining}}\label{tablesremaining}
\end{table}
\vskip -.2in
\end{center}
\begin{description}
\item[{\bf Type $[562]$}]\label{vsptype28}  For a general $F\in {\mathcal F}_{[562]}$ the quadratic ideal $Q_F$ is, by Proposition \ref{apolar b12=5},  the ideal of $\Gamma$: the union of a  line $L$ and two points $p_1,p_2$ that altogether span $\Pn^3$.  So this $\Gamma$ is apolar to $F$.  In the $4$-uple embedding of $v_4:\Pn^3\to\Pn^{34}$, this $\Gamma$ is reembedded as the union of a rational normal quartic curve $v_4(L)$ and two points $v_4(p_1),v_4(p_2)$, and by apolarity $[F]$ is a point in the span of this union.   

In fact, it lies in the span of the two points and a unique point $[F_1]$ in the span of the quartic curve $v_4(L)$.  The form $F_1$ is apolar to the line $L$, so by Lemma \ref{curves}, $F_1$ has rank $2$, or it has rank $3$ and $VSP(F_1,3)=\Pn^1$.  If $F_1$ has rank $2$, then $F$ would have rank $4$, which is impossible by Lemma \ref{rankq}, so $F_1$ has rank $3$.
The ideal of any five points apolar to $F$ is contained in at least five quadrics, i.e.~in $\Gamma$, so we conclude that also $VSP(F,5)=\Pn^1$.
Now, the image $v_4(\Gamma)$ in $\Pn^{34}$ spans a $\Pn^6$, so for each choice of a line and two points in $\Pn^3$ there is a $6$-dimensional family of forms $F$ of this type.  

Notice the $\Gamma$ is unique for a general $F$ apolar to it.  Counting the dimension of the family of schemes $\Gamma$ we get $4$ for the line, and $3$ for each of the points, yielding an irreducible $6+4+3+3=16$ dimensional family of forms $F$ of this type in $\Pn^{34}$.

\vskip .1in
\item[{\bf Type $[551]$}] For a general $F\in {\mathcal F}_{[551]}$ the quadratic ideal $Q_F$ is, by Proposition \ref{apolar b12=5}, the ideal of five points, of which four lie in a plane.   In the embedding $v_4:\Pn^3\to\Pn^{34}$ these five points span a $\Pn^4$ that contain the point $[F]$. A general form $[F']$ in that $\Pn^4$ will have the same Betti table for its apolar ring $A_{F'}$.  The ideal of any set of five points apolar to $F$ is generated by at least five quadrics, i.e.~by the five quadrics in $F^{\bot}$, so  $VSP(F,5)=\{\mbox{one point}\}$.
 Counting the dimension of the family of $5$-tuples of points with four in a plane, we get $3$ for the plane, $8$ for the four points in the plane, and $3$ for the point outside, so 
altogether there is an irreducible $4+3+8+3=18$-dimensional family of forms $F$ of this type in $\Pn^{34}$.
\vskip .1in
\item[{\bf Type $[441a]$}]  For a general $F\in {\mathcal F}_{[441a]}$ the quadratic ideal $Q_F$ is, by Proposition \ref{apolar b12=4},  the ideal of $\Gamma$:  The union of a plane conic curve $C$ and a point $p$ not in the plane of the conic. The image $v_4(C)$ in $\Pn^{34}$ is a rational normal curve of degree $8$.  By apolarity  $[F]$ lies in the span of the image of the point $p$  and a unique point $[F_1]$ in the span of $v_4(C)$.  The quartic form $F_1$ is apolar to $C$, so by Lemma \ref{curves}, it has rank $5$, and $VSP(F_1,5)=\Pn^1$.  

The ideal of any six points apolar to $F$ is generated by at least four quadrics, so is contained in $\Gamma$, so also $VSP(F,5)=\Pn^1$. 
Counting the dimension of the family of pairs  of a conic and a point, we get $8$ for the conic and $3$ for the point.  The span of their image in $\Pn^{34}$
is $9$-dimensional, so there is an irreducible $8+3+9=20$-dimensional family of forms $F$ of this type in $\Pn^{34}$.
\vskip .1in
\item[{\bf Type $[441b]$}] For a general $F\in {\mathcal F}_{[441b]}$ the quadratic ideal $Q_F$ is, by Proposition \ref{apolar b12=4},  the ideal of $\Gamma$: The union of two skew lines $L_1$ and $L_2$.  The form $F$ has a unique decomposition $F=F_1+F_2$, where  $F_i$ is apolar to the line $L_i$ for $i=1,2$. By  Lemma \ref{curves}, each $F_i$ has rank $3$ and $VSP(F_i,3)=\Pn^1$.  
The ideal of any six points apolar to $F$ is generated by at least four quadrics, so is contained in $L_1\cup L_2$ and consists of two sets of three points, apolar to $F_1$ and $F_2$ respectively, so $VSP(F,6)=\Pn^1\times \Pn^1$.
Counting the dimension of the family of pairs  of lines, we get $8$.  The span of their image in $\Pn^{34}$
is $9$-dimensional, so there is an irreducible $8+9=17$-dimensional family of forms $F$ of this type in $\Pn^{34}$.

\vskip .1in
\item[{\bf Type $[430]$}]   For a general $F\in {\mathcal F}_{[310]}$ the quadratic ideal $Q_F$ is, by Proposition \ref{apolar b12=4}, the ideal of a line and three points. 
The ideal of any set of six points apolar to $F$ has at least four quadric generators, i.e.~the quadrics in in $F^{\bot}$ and a cubic generator.  
So $F$ is a sum of three fixed rank $1$ quartics and a unique quartic $F_1$ that is apolar to a line.  If $F_1$ is apolar to two points on the line, $F$  would be apolar to five points and $F^{\bot}$ would contain more quadrics, so $F_1$  is not apolar to two points on the line. Therefore by Lemma \ref{curves}, the form  $F_1$ has rank $3$ and $VSP(F_1,3)=\Pn^1$, and thus also $VSP(F,6)=\Pn^1$.
The dimension of the family of triples of points is $9$, and  the family of lines is $4$-dimensional.  The image of the line and the three independent  points in $\Pn^{34}$
span a $7$-space, so there is an irreducible $9+4+7=20$-dimensional family of forms $F$  of this type in $\Pn^{34}$.
\vskip .1in
\item[{\bf Type $[420]$}]  For a general $F\in {\mathcal F}_{[420]}$ the quadratic ideal $Q_F$ is, by Proposition \ref{apolar b12=4}, the ideal of six points. By Castelnuovo's Lemma \cite{EH}*{Theorem 1}, there is a unique twisted cubic curve through six points in $\Pn^3$ as soon as no subset of four lie in a plane,  so the six points is the intersection of the curve with a quadric.  
The ideal of any set of six points is generated by at least four quadrics, i.e.~by the quadrics in $F^{\bot}$.  Therefore $VSP(F,6)=\{1 \; point\}$.
The dimension of the family of $6$-tuples of points is $18$, and $6$ independent  points in $\Pn^{34}$
span a $5$-space, so there is a $18+5=23$-dimensional family of forms $F$  of this type in $\Pn^{34}$.
\vskip .1in
\item[{\bf Type $[331]$}]    For a general $F\in {\mathcal F}_{[331]}$ the quadratic ideal $Q_F$ is, by Proposition \ref{apolar b12=3}, the ideal of a plane $P$ and a point outside the plane.   So in the space $\Pn^{34}$  of quartics  the point $[F]$ lies in the span of the point  and a unique point $[F_1]$ in the span of the image $v_4(P)$ of the plane.   So $F_1$ is a quartic form  apolar to the plane.  Since there is a $3$-dimensional space of quadrics in $F^{\bot}$, the rank of $F$ is at least $7$, so $F_1$ has rank at least $6$.
Therefore $F_1$ is not apolar to any quadric in the plane, its rank is $6$ and $VSP(F_1,6)$ is a Fano threefold by \cite{Mukai}. 
The ideal of any apolar set of seven points contains a  $3$-dimensional space of quadrics, i.e.~all the quadrics in $F^{\bot}$. Hence any apolar set of seven points contains the point outside the plane, and $VSP(F,7)=VSP(F_1,6)$ i.e.~a Fano threefold.
Counting dimensions we get $3$ for the point and $3$ for the plane, while the span of the point and the image of the plane in $\Pn^{34}$  is $15$ -dimensional, yielding an irreducible $3+3+15=21$ dimensional family of forms $F$ of this type in $\Pn^{34}$.

\vskip .1in

          \item[{\bf Type $[310]$}]  For a general $F\in {\mathcal F}_{[310]}$ the quadratic ideal $Q_F$ is, by Proposition \ref{apolar b12=3}, is the ideal of the union of a conic $C$ and two points $p_1,p_2$ outside the plane of the conic.
The ideal of any set of seven points apolar to $F$ has at least $3$ quadric generators, i.e.~ the quadrics in $F^{\bot}$.   So any such set consists of the two points outside the conic and five points on the conic $C$.
The image $v_4(C)$ in $\Pn^{34}$ is a rational normal curve of degree $8$. By apolarity  $[F]$ lies in the span of the image of the two points $p_1,p_2$ and $v_4(C)$,  so $F$ has a unique decomposition $F=F_1+F_2$ where $F_1$  is apolar to the two points $p_1,p_2$ and has rank $2$, while $F_2$ is apolar to the conic $C$ and has rank $5$.   By Lemma \ref{curves},  $VSP(F_2,5)=\Pn^1$, 
therefore also $VSP(F,7)=\Pn^1$.
The dimension of the family of conics in a plane in $\Pn^3$ is $3+5=8$, so the set of triples (conic + two points) is $8+3+3=14$-dimensional. Each triple spans a $\Pn^{10}$ in $\Pn^{34}$,  so there is an irreducible $14+10=24$-dimensional family of forms $F$  of this type in $\Pn^{34}$.

\vskip .1in

           \item[{\bf Type $[210]$}] For a general $F\in {\mathcal F}_{[210]}$ the quadratic ideal $Q_F$ is, by Proposition \ref{apolar b12=2}, the ideal of a plane $P$ and a line $L$ in $\Pn^3$.   In the space $\Pn^{34}$ of quartics, the span of $v_4(P)$ and 
    $F$ intersects the span of $v_4(L)$ in a line that passes through the intersection point $v_4(P\cap L)$. So there is a linear pencil of lines through $F$ that intersect both the span of $v_4(P)$ and the span of $v_4(L)$. The line through $F$ and the point $v_4(P\cap L)$ is in this pencil, so there is a unique line in the pencil that intersects the span of $v_4(P)$ in a form of rank $5$, and similarly a unique line in the pencil that intersect the span of $v_4(L)$ in a form of rank $2$.
    The first line defines a decomposition $F=F_1+F_2$, where $F_1$ is apolar to the plane and has rank $5$, while $F_2$ is apolar to the line and has rank $3$.  
    The second line defines a decomposition $F=F'_1+F'_2$, where $F'_1$ is apolar to the plane and has rank $6$ and $F'_2$ is apolar to the line and has rank $2$.

The ideal of any set of eight points contains at least two quadratic generators, so any such set that is apolar to $F$ must be contained the union of the line and the plane. And so any set of eight points defines a decomposition $F=F_1+F_2$, where $F_1$ is apolar to the line and $F_2$ is apolar to the plane.

By the above, there are exactly two such decompositions where the sum of the ranks of the two summands is $8$.
When $F_1$ has rank $3$, then $F_2$ must have rank $5$ and is apolar to a conic in the plane.   By Lemma \ref{curves}, $VSP(F_1,3)=\Pn^1$ and $VSP(F_2,5)=\Pn^1$, so $VSP(F,8)$ contains a component $\Pn^1\times\Pn^1$ from the first decomposition.  

In the second decomposition $F'_1$ has rank $2$, then $F'_2$ is apolar to the plane and have rank $6$. $VSP(F'_1,2)=\{\mbox{one point}\}$ and 
$VSP(F'_2,6)$ is a Fano threefold, \cite{Mukai}.  
So $VSP(F,8)$ contains the Fano threefold $VSP(F'_2,6)$ from the second decomposition.
We conclude that  $VSP(F,8)$ is the disjoint union of a $\Pn^1\times\Pn^1$ and a Fano threefold.

In $\Pn^{34}$ the span of $v_4(P\cup L)$ is a $\Pn^{18}$, and the family of pairs of planes $P$ and lines $L$ defined by $Q_F$ is $3+4=7$-dimensional, so there is an 
irreducible $18+7=25$-dimensional family of forms $F$ 
with Betti table of this type in $\Pn^{34}$.
          
             \end{description}
 \end{proof}

        \subsection{$VSP(F,9)$ for a general quartic form of rank $9$}\label{vsp9}
             If $F$ has rank $9$ then $F^\perp$ contains a unique quadric form.  Here we assume that the quadric $Q$ defined by this form is smooth, which is certainly the case for a general form of rank $9$.  Since any set $\Gamma$ of nine points in $\Pn^3$ lies on a quadric, any  set $\Gamma$ that is apolar to $F$ is contained in $Q$. In particular the variety $VSP(F,9)$ is contained in the Hilbert scheme $\operatorname{Hilb}_9(Q)$ of schemes of length $9$ on $Q$.  Therefore we restrict our attention to $Q$. 

   \subsubsection{\bf Apolarity on $Q$.}
                  On $Q=\Pn^1\times\Pn^1$, the divisors are defined by bihomogeneous forms.  Let $$S_Q=\CC[u_0,u_1,v_0,v_1]=\oplus_{(a,b)}H^0({\mathcal O}_Q(a,b)$$ be the bigraded ring of forms on $Q$. 
             
              Then there is a natural restriction map
$S\to S_Q$, with kernel generated by the homogeneous ideal $I_Q$ of $Q$.  
Given $F\in R_4$, let $(F^{\bot})_Q$ be the ideal generated by the image of $F^{\bot}$. 
When $I_Q\subset  F^{\bot}\subset S$, 
then $(F^{\bot})_Q(4)$ has codimension $1$ in $S_Q(4,4)$, so is defined by a linear form $F_Q\in  S_Q(4,4)^*$.  
 We call {\em $F_Q$ the restriction of $F$ to $Q$}.
 
 Conversely, the natural map  
 \[
 S_Q(4,4)^*\to S_4^*=R_4
 \]
 is injective, and associates a unique form $F\in R_4$ to a form $F_Q\in S_Q(4,4)^*$.

Now, apolarity extends to  bihomogeneous forms.    Any form $G$ of bidegree $(a,b)$ in $S_Q$ may be interpreted as a differential operator on $F_Q$:
$$G(F_Q)\in S_Q(4-a,4-b)^*;\quad G(F_Q):G'\mapsto F_Q(G\cdot G').$$  
We define the apolar ideal  $F_Q^{\bot}\subset S_Q$ as the ideal generated by the saturation 
$$F_Q^{\bot}=\langle \cup_{(a,b)} \{G\in S_Q(a,b)|S_Q(4-a,4-b)\cdot G\subset {\rm ker}F_Q\}\rangle.$$ 
Then $F^{\bot}|_Q\subset F_Q^{\bot}$ and $F^\perp$ can be recovered from $F_Q^{\bot}$ as the biggest ideal in $S$ whose restriction $F^{\bot}|_Q$ to $Q$ is contained in $F_Q^{\bot}$.
  
  A scheme $\Gamma$ of length $9$ is apolar to $F$ if $I_\Gamma\subset F^\perp$.  Any such $\Gamma$ lies in $Q$,
  so $I_\Gamma|_Q\subset F^\perp|_Q\subset F_Q^\perp$. Since the ideal $I_\Gamma$ is saturated in $S$, and $F_Q^\perp$ is saturated in $S_Q$, the saturation of $I_\Gamma|_Q$ in $S_Q$ is contained in $F_Q^\perp$, which means that $\Gamma\in VSP(F_Q,9)\subset \operatorname{Hilb}_9(Q)$. 
  
  Conversely, let  $\Gamma\in VSP(F_Q,9)$ with $I_\Gamma\subset F_Q^\perp,$ and let $I'_\Gamma\subset S$ be the ideal of $\Gamma\subset\Pn^3$. Then $I'_\Gamma\subset F_Q^\perp$, and so $\Gamma\in VSP(F,9).$  Since $\operatorname{Hilb}_9(Q)$ is closed in the Hilbert scheme of length $9$ subschemes in $\Pn^3$, we have shown
  \begin{lemma} Let $F$ be a quaternary quartic of rank $9$, apolar to a smooth quadric $Q$, and let $F_Q$ be the restriction of $F$ to $Q$. Then 
  $$VSP(F,9)\cong VSP(F_Q,9).$$

  \end{lemma}
        
        \begin{lemma}\label{nonspecial44} 
Let $\Gamma\subset Q$ be a scheme of length $9$ with $h^0(I_\Gamma(4,4))\geq 17$.  Then the forms in $H^0(I_\Gamma(4,4))$ have a fixed component of bidegree $(0,1)$ or $(1,0)$ that contains  a subscheme of length at least $6$ on $\Gamma$.

Furthermore, if $F_Q$ is a $(4,4)$-form and the minimal length of a scheme  apolar to $F_Q$ is $9$, then $h^0(I_\Gamma(4,4))=16$ for any scheme $\Gamma$ of length $9$ that  is apolar to $F_Q$.
\end{lemma}
\begin{proof}
 Let $\Gamma$ be a scheme of length $9$ that impose dependent conditions on $(4,4)$-forms, i.e.~$h^0(I_\Gamma(4,4))\geq 17$. Then the restriction of the forms in $I_\Gamma(4,4)$ to any irreducible curve $C$ defined by a form in $I_\Gamma(4,4))$ would be special.  But any such curve $C$ has geometric genus at most $9$, so has no special linear system of dimension $\geq 16$.
  The forms in $I_\Gamma(4,4)$, must therefore have a fixed component, for dimension reasons, of bidegree $(0,1)$ or $(1,0)$.   Furthermore, if this component has bidegree $(1,0)$ and contains a subscheme $\Gamma'$  of $\Gamma$, and $\Gamma''=\Gamma\setminus\Gamma$ is the residual scheme, then 
  $h^0(I_{\Gamma''}(3,4))\geq 17$, which means the $\Gamma''$ has length at most $3$.  Hence, $\Gamma'$ has length at least $6$.

   By Lemma \ref{curves}, any scheme of minimal length that is apolar to a $(4,4)$-form  contains a subscheme of length at most $3$ on a curve of bidegree $(0,1)$ or $(1,0)$.  So the moving part of the forms  in $I_\Gamma(4,4)$ are $(3,4)$ or $(4,3)$ forms in the ideal of a scheme $\Gamma_0$ of length at least $6$ and defines a linear system of dimension at least $17$.  Again a general such form is irreducible and defines a curve $C$ of genus at most $6$.  But the restriction of the linear system to $C$ would be special, which is impossible when the dimension is at least $16$.    
\end{proof}
\begin{proposition}. Let $F_Q$ be a general $(4,4)$-form on a smooth quadric surface $Q$.  Then 
$VSP(F_Q,9)\subset \operatorname{Hilb}_9(Q)$ is an irreducible surface.
\end{proposition}
\begin{proof} Consider the incidence $$I_9=\{(F_Q, I_{\Gamma}| I_\Gamma\subset F_Q^\perp\}\subset 
\Pn(T(4,4))\times \operatorname{Hilb}_9(Q).$$
Let $H_0\subset \operatorname{Hilb}_9(Q)$ be the locus of ideals $J$, with ${\rm dim} J(4,4)=16$, it is a dense subset, so its inverse image $I_{9,0}$ in $I_9$ is irreducible.  Hence, the general fiber of the projection $I_{9,0}\to \Pn (T(4,4))$ is also irreducible.  Let $F_Q$ be a general $(4,4)$-form.  Then, by Lemma \ref{nonspecial44}, the variety $VSP(F_Q,9)$ is the closure in $\operatorname{Hilb}_9(Q)$ of the fiber over $F_q$ in  $I_{9,0}$, and the lemma follows. 
\end{proof}

     \vskip .5cm 
        
         Now, let $\Gamma\subset Q$ be a set of nine points on $Q$.   Then, by Proposition \ref{prop_BettiGamma}, the space of cubic generators in the ideal of $\Gamma$ modulo $Q$ is $7$-dimensional, i.e.~if $I_\Gamma\subset S_Q$ is the ideal of $\Gamma$ on $Q$, then $h^0(I_\Gamma(3,3))=7$. The forms of bidegree $(3,3)$ in $I_\Gamma$ are, however, not minimal generators.  In fact,  $h^0(I_\Gamma(3,2))=h^0(I_\Gamma(2,3))=3$ for a general $\Gamma$.          
 On the other hand,  
 ${\rm dim}_k F_Q^{\bot}(2,3)={\rm dim}_k F_Q^{\bot}(3,2)=6$, so if a set  $\Gamma$ of nine points is apolar to $F$, then 
 $H^0(I_\Gamma(3,2))\subset F_Q^{\bot}(3,2)$ and $H^0(I_\Gamma(2,3))\subset F_Q^{\bot}(2,3)$.
 
 For a general set $\Gamma$ of nine points on $Q$ the vanishing locus of both the forms of bidegree $(2,3)$ and those of bidegree $(3,2)$ in the ideal is $\Gamma$, so we may conclude that $VSP(F_Q,9)$ has a birational model in the two Grassmannians $G(3,  F_Q^{\bot}(2,3))$ 
and  $G(3,  F_Q^{\bot}(3,2))$.
In particular, by semicontinuity, ${\rm dim} H^0(I_\Gamma(2,3))\cap F_Q^{\bot}(2,3)\geq 3$ and 
$${\rm dim} H^0(I_\Gamma(3,2))\cap F_Q^{\bot}(3,2)\geq 3$$ 
for any $\Gamma\in VSP(F_Q,9)$.

We now characterize the nets of $(2,3)$-forms that vanish on nine general points on $Q$.
\begin{lemma}\label{linsyzygy} Let $\langle G_1,G_2,G_3\rangle $ be a net of $(2,3)$-forms on $Q$ whose vanishing locus is a scheme $\Gamma$ of length $9$, and assume that the ideal $I_\Gamma$ contains no form of bidegree $(1,3)$.  Then the net has a syzygy of bidegree $(1,1)$ and generates the space of $2\times 2$ minors of a matrix
$$
\begin{pmatrix} A_1,A_2,A_3\\
B_1,B_2, B_3
\end{pmatrix}, \qquad  \mbox{ with deg }A_i=(1,1), \mbox{ deg }B_i=(1,2).
$$
In particular $I_\Gamma(2,3)=\langle G_1,G_2,G_3\rangle $.
\end{lemma}
\begin{proof} Given the scheme $\Gamma$ of length $9$ and a general pencil of forms of bidegree $(2,3)$ in $I_\Gamma$.  Then $\Gamma$ is linked in this pencil to a scheme $\Gamma_0$ on $Q$ of length  $3$.  Then $\Gamma_0$ spans a line in $Q$ or a plane.  
If $\Gamma_0$ spans a line $L$, this line must have bidegree $(1,0)$, and one of the forms in the pencil must vanish on $L$.  But then that form has a factor of bidegree $(1,3)$ that vanishes on $\Gamma$, contrary to the assumption.
  So $\Gamma_0$ spans a plane. It is  then a complete intersection of forms of bidegree $(1,1)$ and $(1,2)$.
The ideal of $\Gamma$ therefore contains the $2$-minors of a $2\times 3$ matrix as stated in the lemma. The first row is then a syzygy among the $2$-minors.  

 It remains to show that the space of minors coincide with  original net of $(2,3)$ forms.  If not, $H^0(I_\Gamma(2,3))$ is at least $4$-dimensional, or equivalently, $h^1(I_\Gamma(2,3))>0$.
Let $C\subset Q$ be a curve defined by a general form $G\in I_\Gamma(2,3)$.  Then $C$ is irreducible, and of genus  $\leq 2$.  From the cohomology of the exact sequence
$$0\longrightarrow {\mathcal O}_Q \stackrel{\cdot G}{\longrightarrow} I_\Gamma(2,3)\longrightarrow I_{C,\Gamma}(2,3)\longrightarrow 0$$
we get $h^1(I_{C,\Gamma}(2,3))>0$ and $h^0(I_{C,\Gamma}(2,3))>2$.  So  the linear system $H^0(I_{C,\Gamma}(2,3))$ on $C$ is special.  But no curve of genus $\leq 2$ has a special linear system of dimension $>2$.  
\end{proof}
 
 We let $V_{F_Q}(2,3)\subset G(3, F_Q^{\bot}(2,3))$ be the locus of nets of $(2,3)$-forms in $F_Q^{\bot}(2,3))$ with a syzygy of bidegree $(1,1)$.  The variety $V_{F_Q}(3,2)\subset G(3, F_Q^{\bot}(3,2))$ is defined similarly.
 
 We will show that for a general $(4,4)$-form $F_Q$ the loci 
 $V_{F_Q}(2,3)$ and $V_{F_Q}(3,2)$ are surfaces with two smooth components, 
one birational to $VSP(F_Q,9)$ and the other consisting of nets of forms whose common zeros is a scheme of length $8$, in fact a complete intersection $((1,3),(2,2))$ and $((3,1),(2,2))$ respectively.  

To make this statement precise we need some lemmas on generality assumptions.
   
   For the $(4,4)$-form $F_Q$, consider the space of partials of order $(2,1)$:  $$P_{2,3}=\{G(F_Q) | G\in H^0({\mathcal O}_Q(2,1))\}\subset H^0({\mathcal O}_Q(2,3)).$$  Then  
   $F_Q^{\bot}(2,3)=P_{2,3}^{\bot}$  defines a rational map 
  $$
   \begin{matrix}
   &&&\Pn( H^0({\mathcal O}_Q(2,3))\\
   &v_{2,3}\hskip -0.4cm&\nearrow&\downarrow\\
   q_{2,3}: & Q&\to& \Pn( F_Q^{\bot}(2,3)^*)\\
   \end{matrix},
   $$
   which may be interpreted as the composition of the map
    $v_{2,3}:Q\to \Pn( H^0({\mathcal O}_Q(2,3))$ defined by $H^0({\mathcal O}_Q(2,3)^*$ and the projection from the space of partials  $P_{2,3}\subset \Pn( H^0({\mathcal O}_Q(2,3))$.
   In this interpretation, the following lemma is a consequence:
   \begin{lemma}\label{2,3} 
   The map  $q_{2,3}$ is a morphism if and only if $F_Q$ has no partials of order $(2,1)$ and rank $\leq 1$.
The map $q_{2,3}$ is an embedding if and only if $F_Q$ has no partials of order $(2,1)$ and cactus rank $\leq 2$.
\end{lemma}
\begin{proof} A point $p=[G(F)]$ in the intersection $P_{2,3}\cap v_{2,3}(Q)$ is a form whose ideal $I_p\in S_Q$ is contained in $G(F)^\perp$.  Thus, $G(F)$ is apolar to a point $p$, for some partial $G\in S_Q$ of order $(2,1)$ if and only if $q_{2,3}$ is not defined at $p$.
Similarly, a point $s=G(F)$ lies in the span of a length $2$ scheme $Z\subset Q$, if and only if $Z$ is apolar to $G(F)$, i.e.~$I_Z\subset P(F)^\perp$. Thus, $G(F)$ is apolar to a scheme $Z$ of length $2$, for some partial $G\in S_Q$ of order $(2,1)$ if and only if $q_{2,3}$ is not an embedding. 
\end{proof}
Of course, there is an analogous lemma for $q_{3,2}$.

    \begin{lemma}\label{3,4rank6} In the space of $(4,4)$-forms, the following loci are at most divisors:
    \begin{enumerate}
    \item the locus of forms apolar to the union of $(0,1)$-curve or a $(1,0)$-curve and a scheme of length at most $6$.
    \item the locus of forms apolar to some $(1,3)$ or $(2,2)$-form.
     \item the locus of forms with a partial of order $(2,1)$ or $(1,2)$  that is apolar to a scheme of length at most $2$.
    \item the locus of forms that is apolar to some scheme $\Gamma$ of length $9$ that impose dependent conditions on $(4,4)$ -forms.
    \end{enumerate}
    \end{lemma} 
    \begin{proof} 
    For the first case, let $Z$ be the union of a $(0,1)$-curve $L$ and a finite scheme $Z_0$ of length at most $5$ disjoint from $L$.   If a $(4,4)$-form $F_Q$ is apolar to $Z$, then $F_Q$ decomposes into $F_Q=F_{Q,L}+F_{Q,0}$, where $F_{Q,L}$ is apolar to $L$ and $F_{Q,0}$ is apolar to $Z_0$.  Now, since $v_{4,4}(L)$ spans a $\Pn^4$ and there is a pencil of $(0,1)$ curves, the forms apolar to a $(0,1)$-curve has dimension $4+1=5.$  Similarly the set of forms apolar to $Z_0$ is $5$-dimensional and the Hilbert scheme $\operatorname{Hilb}_k(Q)$ is $2k$-dimensional, so the set of forms apolar to $Z_0$ is at most $17$-dimensional.  The set of forms $F_Q$ with a decomposition  $F_Q=F_{Q,L}+F_{Q,0}$ is therefore at most $23$-dimensional, i.e.~at most a divisor in the space of $(4,4)$-forms.
    
    Similarly the set of $(4,4)$-forms apolar to a fixed $(2,2)$ curve or a $(1,3)$-curve is $15$-dimensional, respectively $16$-dimensional, while the families of such curves are $8$, respectively $7$- dimensional.  So the set of forms apolar to some $(2,2)$-curve or some $(1,3)$-curves  is at most $23$-dimensional, i.e.~at most a divisor in the space of $(4,4)$-forms.
    
    The ideal of a general scheme $Z$ of length $2$ on $Q$ is generated by two $(1,1)$-forms say $A_1$ and $A_2$.  A partial $G(F_Q)$ of  degree $(2,3)$ is apolar to $Z$ only if $A_1(G(F_Q))=A_2(G(F_Q))=0$, i.e.~if both $A_1(F_Q)$ and $A_2(F_Q)$ are annihilated by the $(2,1)$ form $G$.
   
    Any $(1,2)$ form $G\in S_Q(1,2)$ defines a linear map differentiation $$Cat_G: S_Q(3,3)^*\to S_Q(1,2)^*;\quad H\mapsto G(H).$$
    Consider its restriction to the space of degree $(3,3)$-partials
    $$P_{3,3}(F)=\{A(F_Q)|A\in S_Q(1,1)\}\subset S_Q(3,3)^*.$$ 
    Now, $G\in H^\perp$ only if $Cat_G$ has corank at least $1$. 
    Similarly, $G$ annihilates a pencil of forms in $P_{3,3}(F)$ only if $Cat_G$ restricted to $P_{3,3}(F)$ has corank at least $2$.
   Since ${\rm dim}P_{3,3}(F)=4$ and ${\rm dim}S_Q(1,2)^*=6$, the corank $2$ locus has expected dimension $8$, i.e.~the set of  $[G]\in \Pn(S_Q(2,1))$, such that $Cat_G$ is expected to be empty for general $F$.  This is also easily confirmed with the example 
   $$F_Q=x_0^4y_0^4+x_1^4y_0^4+x_0^3x_1y_0y_1^3+x_0^2x_1^2y_0^2y_1^2+x_1^4y_1^4.$$
   So the set of $(4,4)$-forms $F_Q$ for which a partial $G(F_Q)$ of  degree $(2,3)$ is apolar to length $2$ subscheme $Z$, defined by two $(1,1)$ forms, is at most a divisor in the space of $(4,4)$-forms.  If a length $2$ scheme $Z$ 
   is contained in a $(0,1)$-curve or a $(1,0)$-curve the similar argument applies with the same conclusion.
     
  By Lemma \ref{nonspecial44},  a $(4,4)$-form $F_Q$ of cactus rank $9$, is not part of the fourth locus, so this latter locus is at most a divisor. 
    \end{proof}
   
       We say that a $(4,4)$-form is {\em sufficiently general} if it is not in any of the loci of Lemma \ref{3,4rank6}.  
       Note that any apolar scheme of length at most $8$ has a $(2,2)$-form in its ideal, so a sufficiently general $(4,4)$ -form is not apolar to any scheme of length $\leq 8$.

    \begin{proposition}\label{11syz} Let $F_Q$ be a sufficiently general $(4,4)$ form on a quadric $Q$, and let 
  $\langle G_1,G_2,G_3\rangle$ be a net of forms in $F_Q^{\bot}(2,3)$ with a syzygy of bidegree $(1,1)$, i.e.~a net that  belongs to $V_{F_Q}(2,3)$.
  Let  
$$\Gamma=\VV(G_1,G_2,G_3).$$  Then $\Gamma$ is a subscheme of length $9$ apolar to $F_Q$, i.e.~$\Gamma\in VSP(F_Q,9)$,
or $\Gamma$ is a complete intersection $((1,3),(2,2))$ of length $8$.  In the latter case the $(1,1)$-forms of the syzygy vanish in a common point on $Q$.
 \end{proposition}
               
\begin{proof} 
Assume $\langle G_1,G_2,G_3\rangle$ is a net of $(2,3)$-forms that are apolar to $F_Q$ and have a syzygy 
$A_1G_1+A_2G_2+A_3G_3=0$ where the $A_i$ have bidegree $(1,1)$.
We analyze case by case:
\begin{enumerate}
\item ${\rm rank }\langle A_1,A_2,A_3\rangle=2$
\item ${\rm rank }\langle A_1,A_2,A_3\rangle=3$ and the forms $G_1,G_2,G_3$ have a common component.
\item ${\rm rank }\langle A_1,A_2,A_3\rangle=3$ and the vanishing  locus $\VV(G_1,G_2,G_3)$ is finite. 
\end{enumerate}

In case $(1)$ we may assume $A_3=0$ and  $A_1G_1=A_2G_2$.  If $A_1$ or $A_2$ is irreducible, then there is a $(1,2)$-form  $G_0$ such that $G_1=A_2G_0$ and $G_2=A_1G_0$.
The form $G_0(F_Q)$ has bidegree $(3,2)$, and $A_1( G_0(F_Q))=A_2( G_0(F_Q))=0$.  But $(A_1,A_2)$ generates the ideal of a scheme of length $2$, so  $G_0(F_Q)$ is apolar to a subscheme of length $2$. This is excluded by the generality assumptions of $F_Q$. If $A_1$ and $A_2$ are reducible with a common factor of bidegree $(0,1)$, then we may assume that $B_1G_1=B_2G_2$ for $(1,0)$-forms $B_i$ such that 
$G_1=B_2G_0$ and $G_2=B_1G_0$ .  Then $G_0$ is a $(1,3)$-form and $G_0(F_Q)$ is a $(3,1)$-form that is annihilated by every $(0,1)$ form.  But that means $G_0$ is apolar to $F_Q$, again contrary to the assumption on $F_Q$.  If  $A_1$ and $A_2$ are reducible with a common factor of bidegree $(1,0)$, then 
a similar argument implies that there is a $(2,2)$ form apolar to $F_Q$, against the generality assumption.

For cases $(2)$ and $(3)$, we first notice that 
when $F_Q$ is sufficiently general, then,  by Lemmas \ref{2,3} and \ref{3,4rank6}, the map $\phi_{2,3}:Q\to \Pn^5$ defined by $F_Q^{\bot}(2,3)$ is an embedding.  Therefore, the only curves on $Q$ that are mapped into a plane by $\phi_{2,3}$ are the curves of bidegree $(0,1)$.

In case $(2)$ we therefore may assume that the common factor $B$ of the $G_i$ has bidegree $(0,1)$.  So $G_i=BH_i$ where the $H_i$ have bidegree $(2,2)$ and 
$A_1H_1+A_2H_2+A_3H_3=0$.   The $A_i$ have at most one common zero, otherwise we would be in case $(1)$ again.  If the $A_i$ have one common zero, then the $H_i$ vanish on a common scheme $Z$ of length $5$, which would be the complete intersection of a $(2,1)$-form and a $(1,2)$-form.  The three $H_i$ would span a net of forms in the $4$-dimensional space of $(2,2)$-forms in the ideal of $Z$. Still, one may compute that the three $(2,3)$ forms $BH_i$ generate the space of $(4,4)$-forms in the  ideal of $Z\cup \{B=0\}$.  This means that if the $BH_i$ lie in $F_Q^\perp$, then $Z\cup \{B=0\}$ is apolar to $F_Q$, contrary to the generality assumption.

If the $(1,1)$-forms have no common zeros, then the $(2,2)$-forms $H_i$ generate the ideal of a scheme $Z$ of length $6$, and the union of this scheme and the curve $\{B=0\}$ is apolar to $F$, again contrary to the generality assumption on $F_Q$.  

In case $(3)$ we may assume that the $G_i$ are all irreducible.   Since $A_1G_1+A_2G_2+A_3G_3=0$, the $A_iG_i$ form a pencil. 
In the baselocus $\Gamma_P$ of this pencil of degree $24$, each $A_i$ contains a length $7$ subscheme, and the intersection of any two $A_i$ has length $2$.
If the three $A_i$ have no common zeros the intersection of the three $G_i$ is a subscheme of length $24-3\times 7+3\times 2=9$ of the baselocus $\Gamma_P$.   In this case the  vanishing locus $\{G_1=G_2=G_3=0\}$ is a scheme $\Gamma$ of length $9$.  

If the common zeros of the three $A_i$ is a reduced point, then the intersection of the three $G_i$ is a subscheme of length $24-3\times 7+3\times 2-1=8$ of the baselocus $\Gamma_P$.   In this case the  vanishing locus $\{G_1=G_2=G_3=0\}$ is a scheme $\Gamma$ of length $8$.  

If the common zeros of the three $A_i$ is scheme of length $2$, then the $A_i$ form a pencil, which is case (1).

Assume, first that the vanishing locus $\{G_1=G_2=G_3=0\}$ is a scheme $\Gamma$ of length $9$, and consider a general pencil in the net.  In this pencil $\Gamma$ is linked to a scheme $\Gamma_0$ of length $3$.
Then $\Gamma_0$ is a complete intersection $(1,0),(1,3)$  or $(1,1),(1,2)$.
Correspondingly the net is the space of $2\times 2$-minors  of a matrix
$$
\begin{pmatrix} A_1,A_2,A_3\\
B_1,B_2, B_3
\end{pmatrix}, 
$$
where ${\rm deg}A_i=(1,0), {\rm deg}B_i=(1,3)$ or ${\rm deg}A_i=(1,1), {\rm deg}B_i=(1,2)$ respectively.
 In the first case there is a syzygy of bidegree $(1,0)$ which generates syzygies of bidegree $(1,1)$ of rank $2$.  This is excluded in case $(1)$ above.   So only the second case is possible, as stated in the lemma.

To show that $\Gamma$ is apolar to $F_Q$ we need  to show that 
$I_\Gamma(a,b)\subset F_Q^{\bot}(a,b)$ for any $(a,b)$.  By Lemma \ref{linsyzygy} we know that this is the case for $(a,b)=(2,3)$.    
The space of forms of degree $(4,4)$ generated by $\langle G_1,G_2,G_3\rangle$
has a dimension that may be computed from the matrix representation:  The $(2,3)$-forms generate $3\cdot 6=18$ form of bidegree $(4,4)$, with two syzygies generated by the syzygy of degree $(1,1)$, so $16$ linearly independent $(4,4)$-forms.   Since the vanishing locus of the $(2,3)$-forms is $\Gamma$, the general $(4,4)$-form in this space defines an  irreducible curve $C$ of genus at most $9$.  The restriction of the $16$-dimensional space of $(4,4)$-forms to $C$ is $15$-dimensional, so that space defines nonspecial linear system on $C$.  We conclude that the $16$-dimensional space of $(4,4)$-forms coincides with $I_\Gamma(4,4)$.
In particular, $I_\Gamma(4,4)\subset F_Q^{\bot}(4,4)$.   By apolarity we infer:   For any form $G\in I_\Gamma$ of bidegree $(a,b)$, with $a\leq 4, b\leq 4$, the form 
$G(F_Q)$  has bidegree $(4-a,4-b)$ and is annihilated by any form of bidegree $(4-a,4-b)$.  Therefore $G(F_Q)=0$, and $G\in F_Q^{\bot}(a,b)$.
We conclude that $I_\Gamma(a,b)\subset F_Q^{\bot}(a,b)$ for any $(a,b)$.

Assume, next, that the vanishing locus $\{G_1=G_2=G_3=0\}$ is a scheme $\Gamma$ of length $8$, and consider a general pencil in the net.  In this pencil $\Gamma$ is linked to a scheme $\Gamma_0$ of length $4$.
In the relation $A_1G_1+A_2G_2+A_3G_3=0$, when the common zeros of the $A_i$ is a reduced point, the three forms $A_1,G_2,G_3$ must vanish in scheme of length $4$.  So the common zeros $\Gamma$ of the $G_i$ is linked in $(G_2,G_3)$ to a scheme that is the common zeros of $A_1$ and a $(2,2)$-form $B_1$, or $A_1$ and a $(1,3)$-form $B_2$.  From the first case there is a $(1,3)$-form $H_1$ in the ideal of $\Gamma$, while from the second case there is a $(2,2)$-form $H_2$ in the ideal of $\Gamma$.  For degree reasons, the ideal of $\Gamma$ is a complete intersection $(H_1,H_2)$.

The multiples of $H_1$ and $H_2$ generate the $4$-dimensional space $I_Z(2,3)$ of forms of bidegree $(2,3)$, with a $4$-dimensional space of syzygies of bidegree $(1,1)$.  In fact, any $3$-dimensional subspace in $I_Z(2,3)$ has a syzygy of bidegree $(1,1)$, such that the $(1,1)$ forms have a common zero on $Q$.
This can be seen explicitly:  As generators of a net in $I_Z(2,3)$ we may choose one form
$A_1H_1$, $B_1H_2$ and $A_2H_1+B_2H_2$, the $A_1$ and $A_2$ are $(1,0$-forms, while $B_1$ and $B_2$ are $(0,1)$-forms.  Then the syzygy of bidegree $(1,1)$ among the three is
$$A_2B_1\cdot A_1H_1-A_1B_1\cdot (A_2H_1+B_2H_2)+A_1B_2\cdot B_1H_2=0.$$
Notice that the $(1,1)$ forms vanish at the point $p=\VV(A_1,B_1)$.  When $A_2H_1+B_2H_2$ vanishes at this point, then the three $(2,3)$ forms in the ideal of $Z$ actually vanishes in a scheme of length $9$.  In this case $Z\cup p$ is apolar to $F_Q^\perp$.  But when $A_2H_1+B_2H_2$ does not vanish at $p$, then the three $(2,3)$-forms in $I_Z(2,3)$ does not vanish on an apolar scheme, even if the net lies in $F_Q^\perp$.

Thus the net of $(2,3)$ forms in $F_Q^\perp(2,3)$ with a syzygy in bidegree $(1,1)$ consists both of nets that define apolar schemes of length $9$ to $F$, and schemes of length $8$ that are not apolar to $F$.

 \end{proof}
 There is, of course, a completely analogous proposition for $V_{F_Q}(3,2)$.
 
 Let $V^8_{F_Q}(2,3)$ and  $V^9_{F_Q}(2,3)$ be the sets of nets $\langle G_1,G_2,G_3 \rangle\in G(3,F^{\perp}(2,3))$ with a $(1,1)$-syzygy, such that $\VV(G_1,G_2,G_3)$ has length $8$ and length $9$, respectively.  And similarly $V^8_{F_Q}(3,1)$ and  $V^9_{F_Q}(3,2)$ in $G(3,F^{\perp}(3,2))$.
 
 Clearly, both  $V^8_{F_Q}(2,3)$ and  $V^9_{F_Q}(2,3)$ are closed, 
  $V^8_{F_Q}(2,3)$ since it is the locus where the $(1,1)$-syzygy vanishes on $Q$, and $V^9_{F_Q}(2,3)$ since it the locus where the net of $(2,3)$ curves vanish on a scheme of length $9$.
  Also, since any net in closed variety $V^9_{F_Q}(2,3)$ defines an apolar scheme of length $9$ to $F$, the  natural map 
   $$V^9_{F_Q}(2,3)\to VSP(F_Q,9)$$
   is by Proposition \ref{11syz} both surjective and injective.
   In particular every point in $VSP(F_Q,9)$ is apolar to $F$.   
 \begin{proposition}\label{vsp9smooth} When $F_Q$ is a general $(4,4)$-form, then $VSP(F_Q,9)$ is a smooth surface isomorphic to $V^9_{F_Q}(2,3)$.
 \end{proposition}
 \begin{proof} The last isomorphism follows from the above bijection, as soon as $VSP(F_Q,9)$ is smooth.
 
 Consider the subset $U\subset \operatorname{Hilb}_9Q$ of schemes $\Gamma$ such that $I_\Gamma(4,4)$ is $16$-dimensional.  By Lemma \ref{nonspecial44}, $U$ is an open set, that, by Proposition \ref{11syz}, contains $VSP(F_Q,9)$.  Let $E_U$ be the rank $16$ vector bundle on $U$ whose fiber over $\Gamma\in U$ is $I_\Gamma(4,4)$, then any $(4,4)$-form $F_Q$ defines a section of $E_U^*$.   The vector bundle is clearly generated by these sections, and the zeros of a section is the set of $\Gamma$ such that $I_\Gamma(4,4)$ is in the kernel of $F_Q$.  In particular this zero locus is $VSP(F_Q,9)\cap U$.  Finally, $U$ is smooth (cf. \cite{Fogarty}),
and $VSP(F_Q,9)\subset  U$ for a general $F_Q$, so by the Bertini theorem, $VSP(F_Q,9)$ is also smooth.
 \end{proof}

\begin{theorem}
 Let $F$ be a general quaternary quartic form of rank $9$, then the variety $VSP(F,9)$ is a K3 surface isomorphic to a smooth quartic surface in $\mathbb{P}^3$.  
\end{theorem}
\begin{proof}
We may assume that $F$ is apolar to a smooth quadric $Q$, and that the restriction $F_Q$ is a sufficiently general $(4,4)$-form as described by Lemma \ref{3,4rank6}. Then by Lemma \ref{linsyzygy} and the proof of Proposition \ref{11syz} to any $\Gamma\in VSP(F_Q,9)$ we can associate a unique net of $(1,1)$-forms representing the $(1,1)$-syzygy among the generators of $I_{\Gamma}(2,3)$.
We hence have a map $\kappa: VSP(F_Q,9)\to \mathbb{P}(S_Q(1,1)^{\vee})$. Clearly the image of $\kappa$ is contained in the following locus:
$$D=\{<l_1,l_2,l_3>\in \mathbb{P}(S_Q(1,1)^{\vee})\quad |\quad  \exists \; q_1,q_2,q_3\in F^{\perp}(2,3) \text{ such that } q_1l_1+q_2l_2+q_3l_3=0 \}.$$
On the other hand $D$ is described as the rank $\leq 17$ locus of the composition $$\rho\colon \mathcal U_{\mathbb{P}(S_Q(1,1)^{\vee})} \otimes F^{\perp}(2,3)\to S_Q(1,1)\otimes F^{\perp}(2,3) \otimes \mathcal O_{\mathbb{P}(S_Q(1,1)^{\vee})} \to F^{\perp}(3,4),$$
where $\mathcal U_{\mathbb{P}(S_Q(1,1)^{\vee})}$ is the universal bundle of 3-spaces in $S_Q(1,1)$, the first map comes from the universal bundle sequence, whereas the second is given by multiplication in $S_Q$. Since $-c_1(\mathcal U_{\mathbb{P}(S_Q(1,1)^{\vee})})$ is the class of the hyperplane we conclude that $D$ is a hypersurface of degree 6 in $\mathbb{P}(S_Q(1,1)^{\vee})$. Furthermore observe that 
for $F$ sufficiently general we have $Q\subset D$.  Indeed, if $<l_1,l_2,l_3>\subset S_Q(1,1)$ are the space of $(1,1)$ forms that vanish at 
a point $p$ in $Q$, then the image of the map $\rho$ is contained in  $F^{\perp}(3,4)\cap I_p$ which has dimension 17 and hence $\rho$ is degenerate over that point. Note however that the quadric $Q$ is not the image of $VSP(F_Q,9)$ by $\kappa$.

It remains to see on an example computed in Macaulay2, cf. \cite{M2}*{Package QuaternaryQuartics}, that for suitably general $F$ the divisor $D$ is in general a union of the quadric $Q$ with a smooth quartic surface $X$ meeting  $Q$ transversally and the rank $16$-locus of $\rho$ is empty:

For the $(4,4)$-form 
$$F_Q=x_0^4y_0^4+x_1^4y_0^4+x_0^3x_1y_0y_1^3+x_0^2x_1^2y_0^2y_1^2+x_1^4y_1^4.$$

the $VSP(F_Q,9)$ is isomorphic to the following nonsingular quartic (of type $[000]$):
\begin{align*}
    X(F_Q)=&949968z_1^4-6713280z_2^4-540372420z_2^3z_3+481656096z_2^2z_3^2\\
    &-214742080z_2z_3^3+47738880z_3^4
    -100117184265z_1z_2^2z_4\\
    &+4449690134z_1z_2z_3z_4-82880z_1z_3^2z_4-34412740z_1^2z_4^2+949968z_4^4.
\end{align*}
Here $(z_1,z_2,z_3,z_4)$ and $(x_0y_0,x_0y_1,x_1y_0,x_1y_1)$ are dual basis.

We conclude that for a general quartic $F$ of type $[100]$ the $VSP(F,9)$ is a smooth quartic of type $[000]$, 
in particular it is a $K3$-surface.
\end{proof}

\section{Stratifications of the space of quaternary quartics}\label{stratification}

The open set $U_{34}\subset \Pn^{34}$ of forms that are non-degenerate, that is, those forms $F$ with no linear form in $F^{\perp}$, is the disjoint union of the $19$ irreducible sets  
of Proposition \ref{irreducibleF_B} and Proposition \ref{vsp proposition}.
In this section we show the closure relations between these sets. 

We first we show how they fit in the natural stratification of $\Pn^{34}$ in terms of border rank. The {\em border rank} (or {\em secant}) {\em stratification} of  $\Pn^{34}$ is given by the secant varieties to the variety $V_4=v_4(\Pn^3)$ consisting of forms of rank $1$, the image of the Veronese map of Equation~\ref{VeroneseMap}. The secant varieties to $V_4$, appearing in Equation~\ref{SecantToVeronese}, form a sequence of closed subvarieties 
$$V_4\subset \Sec_2(V_4)\subset...\subset \Sec_9(V_4)\subset \Pn^{34}.$$
The relative complements $\Sec_k(V_4)\setminus \Sec_{k-1}(V_4)$, for $k=2,...,10$ form a stratification of  $\Pn^{34}$. Consider the $10 \times 10$ catalecticant matrix $\Cat(F)$ of a quaternary quartic $F$, which appears in Example~\ref{Cat224}. Let 
$$\Cat_k=\{[F]| \rk(\Cat(F))\leq k\} \subset \Pn^{34}.$$
Thus $$\Cat_k \cap U_{34}=\cup_{b_{12}\geq 10-k} {\mathcal F}_{{[b_{12}b_{23}b_{34}]}}.$$
The sets $\Cat_k\setminus \Cat_{k-1}, k=2,...,10$ form the {\em catalecticant stratification} of the space $\Pn^{34}$ of quaternary quartics, with the obvious closure relation $\Cat_k\supset \Cat_{k-1}$.  The relation between subsets ${\mathcal F}_{{[400]}}$ and ${\mathcal F}_{{[300a]}}$ explains the difference between the secant stratification and the catalecticant stratification:
 
 \begin{proposition}\label{SecantStrat} Let  $\Sec_k(V_4)\subset\Pn^{34}$ be the $k$-th secant variety of $V_4$, i.e.~the closure of the union of spans of $k$ points on $V_4$. Then 
 
 \begin{enumerate}
     \item $\Sec_k(V_4)=\Cat_k$ if $k\leq 5$,
      \item $\Sec_6(V_4)=\Cat_6\setminus {\mathcal F}_{{[400]}} $,
       \item $\Sec_7(V_4)=\Cat_7\setminus ({\mathcal F}_{{[300a]}}\cup {\mathcal F}_{{[400]}})$, 
       \item $\Sec_8(V_4)=\Cat_8$,
         \item $\Sec_9(V_4)=\Cat_9$.
        \end{enumerate}
  \end{proposition}
\begin{proof}  We consider the quadratic part $Q_F$ of the apolar ideal $F^\perp$ of $F$.
As we observed following Example~\ref{Cat224}, 
\[
\rk(\Cat(F))=10-{\dim}(F^{\bot})_2. 
\]
The secant varieties ${\rm Sec}_k(V_4)$ are irreducible, so for each Betti table $B$, by comparing the rank of a general $F\in {\mathcal F}_B$ with $\rk(\Cat(F))$, we get that any quaternary quartic form $F$ of border rank different from $6$ and $7$ is the limit of forms $F_t$ of rank $r(F_t)=\rk(\Cat(F))$.

 To complete the proof we consider the cases when the quadrics in $F^{\bot}_2$ form a complete intersection $(2,2,2)$ or $(2,2,2,2)$.  The forms $F$ with a complete intersection $(2,2,2,2)$ in $F^{\bot}_2$ are those of type $[400]$, and form an open set in $\Cat_6$, whose general member has rank $8$, so ${\mathcal F}_{{[400]}}\subset \Sec_8(V_4)\setminus  \Sec_7(V_4)$.

The forms $F$ such that $F^{\bot}_2$ is a complete intersection $(2,2,2)$ are those of type $[300]$, and form an open set in $\Cat_7$. The general member is of type $[300b]$ and has rank $8$, so $${\mathcal F}_{{[300b]}}\subset \Sec_8(V_4)\setminus  \Sec_7(V_4).$$  The forms of rank $7$ with a complete intersection $(2,2,2)$ in $F^{\bot}_2$ are those of type $[300a]$, and form an irreducible closed subset. So
$${\mathcal F}_{{[300a]}}\subset \Sec_7(V_4)\setminus  \Sec_6(V_4),$$
as we needed to show.
\end{proof}

\subsection{The Betti stratification}\label{bstratification}
We now show the closure relations between the $19$ irreducible sets of Proposition \ref{vsp proposition} whose disjoint union is $U_{34}$, the set of nondegenerate quartics.  
We call the $19$ sets the {\em Betti strata} of $U_{34}$.
 
For the closure relation between the Betti strata, we start with the following simple lemma:
\begin{lemma}\label{closurerelation1} Let $\mathcal{F}_{B}\subset\Pn(R_4)=\Pn^{34}$ be one of the $19$ Betti strata and assume $F \in \mathcal{F}_B$.  If $Q_{F}\subset Q_{F'}$, then $F'\in\overline{\mathcal{F}_B}$.
\end{lemma}
\begin{proof}
If $Q_{F}\subset Q_{F'}$, then $(Q_{F'})_4^\perp\subset (Q_F)_4^\perp \subset R_4$, and so the lemma follows from the inclusion $\Pn((Q_F)_4^\perp)\subset \overline{\mathcal{F}_B}\subset \Pn(R_4)$.
\end{proof}

 First we treat two closure relations between sets $\mathcal{F}_{B}$ with an ad hoc argument:
\begin{proposition}\label{CGKK4a,b,c}
$${\mathcal F}_{[300c]}\subset \overline{ {\mathcal F}_{[300b]}}\subset \overline{{\mathcal F}_{[300a]}}.$$
\end{proposition}
\begin{proof} 
The set of forms $F$ with $Q_F$ a $(2,2,2)$ complete intersection is irreducible, and by Proposition \ref{irreducibleF_B}, the set ${\mathcal F}_{[300a]}$ is a dense subset with complement ${\mathcal F}_{[300b]}$, so 
 ${\mathcal F}_{[300b]}\subset \overline{{\mathcal F}_{[300a]}}$

Let $Z$ be a scheme of length $7$ whose ideal $J$ contains a complete intersection $(2,2,2)$.  Then $Z$ contains no subscheme of length $3$ contained in a line.  But $Z$ may degenerate in a flat family to a scheme $Z'$ with a subscheme of length $3$ in a line.  On the other hand, if  $Q_F=I$ is the ideal of the union of a line and a subscheme of length $4$, then $F$ has a decomposition $F_1+F_2$, where $F_1$ is apolar to the line, and $F_2$ is apolar to the scheme of length $4$.  The form $F_1$ is then apolar to a pencil of subschemes of length $3$ on the line.  Then $F$ is apolar to a pencil of schemes of length $7$ (cf \ref{lineandfourpoints}), with a subscheme of length $3$ in a line.
Any such scheme deforms flatly to a complete intersection.  Therefore 
${\mathcal F}_{[300c]}\subset \overline{{\mathcal F}_{[300b]}}$.
\end{proof}

We further deduce closure relations between sets $\mathcal{F}_{B}$ from their relation to closure relations between sets $\mathcal{G}_{B}$ (recall Definition \ref{GbettiTable}).

For every Betti table $B$ for quadratic ideals there is  an embedding
$$\mathcal{G}_{B}\to \Gr(b_{12},S_2),\quad I\mapsto [I_2].$$
This embedding provides one of two natural closure relations between irreducible sets of quadratic ideals.
\begin{definition}
We write  
${\mathcal G}\prec {\mathcal G'}$ if ${\mathcal G}\subset\overline{ {\mathcal G'}}$ and $b_{12}(S/I)$ is the same for all $I$ in either set.
We write ${\mathcal G}\sqsubset {\mathcal G'}$ if a general ideal $J\in {\mathcal G}$ contains some ideal $I\in {\mathcal G'}$.
\end{definition}

\begin{proposition}\label{precG}
The irreducible sets of quadratic ideals $Q_F$ for quaternary quartics satisfy the following relations of the kind ($\prec$), where $b_{12}=b_{12}(S/Q_F)$:
\begin{enumerate}
    \item ($b_{12}=5$) :  $\mathcal{G}_{[562]}$ (a line and two points) $\prec$  $\mathcal{G}_{[551]}$ (five points, four in a plane) $\prec$  $\mathcal{G}_{[550]}$ (five general points).
    \item ($b_{12}=4$) : $\mathcal{G}_{[441a]}$  (a conic and a point) $\prec$ $\mathcal{G}_{[420]}$ (six general), 
    \item ($b_{12}=4$): $\mathcal{G}_{[441b]}$ (two skew lines) $\prec$ $\mathcal{G}_{[430]}$ (a line and $3$ points)$\prec$ $\mathcal{G}_{[420]}$ (six general points).
     \item ($b_{12}=3$) : $\mathcal{G}_{[320]}$ (a twisted cubic)$\prec$ $\mathcal{G}_{[300ab]}$ (CI $(2,2,2)$), 
    \item ($b_{12}=3$) :  $\mathcal{G}_{[300c]}$ (a line and four points)$\prec$ $\mathcal{G}_{[300ab]}$ (CI $(2,2,2)$), 
    \item ($b_{12}=3$): $\mathcal{G}_{[331]}$ (a plane and a point) $\prec$ $\mathcal{G}_{[310]}$ (a conic and $2$ points) $\prec$$\mathcal{G}_{[300ab]}$ (CI $(2,2,2)$).
    \item ($b_{12}=2$) : $\mathcal{G}_{[210]}$ (a plane and a line) $\prec$ $\mathcal{G}_{[200]}$ (CI $(2,2)$), 
\end{enumerate}

\end{proposition}
\begin{proof}  

Interpreted in $\Gr(b_{12},S_2)$ the closure relation is straightforward in each of the seven items in the proposition.
\end{proof}

\begin{proposition}\label{sqsubsetG}
The irreducible sets of quadratic ideals $Q_F$ for quaternary quartics satisfy the following relations of the kind $\sqsubset$:
\begin{enumerate}
    \item $\mathcal{G}_{[683]}$ $\sqsubset$ $\mathcal{G}_{[550]}$ $\sqsubset$ $\mathcal{G}_{[420]}$ $\sqsubset$ $\mathcal{G}_{[300ab]}$ $\sqsubset$ $\mathcal{G}_{[200]}$ $\sqsubset$ $\mathcal{G}_{[100]}$.
    \item $\mathcal{G}_{[683]}$(four general) $\sqsubset$ $\mathcal{G}_{[551]}$ (five points, four on a plane)
    \item $\mathcal{G}_{[550]}$(five general) $\sqsubset$ $\mathcal{G}_{[430]}$ (a line and three points)
    \item $\mathcal{G}_{[683]}$(four general) $\sqsubset$ $\mathcal{G}_{[562]}$ (a line and two points) $\sqsubset$ $\mathcal{G}_{[430]}$ (a line and three points) $\sqsubset$ $\mathcal{G}_{[300c]}$ (a line and four points)
    \item $\mathcal{G}_{[562]}$ (a line and two points) $\sqsubset$ $\mathcal{G}_{[441b]}$ (two skew lines)
    \item $\mathcal{G}_{[551]}$ (five points, four on a plane) $\sqsubset$ $\mathcal{G}_{[441a]}$ (a conic and a point) $\sqsubset$ $\mathcal{G}_{[310]}$ (a conic and two points) $\sqsubset$ $\mathcal{G}_{[210]}$ (a plane and a line)
    \item $\mathcal{G}_{[441a]}$ (a conic and a point) $\sqsubset$ $\mathcal{G}_{331]}$ (a plane and a point) $\sqsubset$ $\mathcal{G}_{[210]}$ (a plane and a line)
\end{enumerate}
\end{proposition}
\begin{proof} For each Betti table $B$ and each  component ${\mathcal G}\subset {\mathcal G}_B$, an ideal $I\in {\mathcal G}$ is the ideal of a scheme $Z$.  By choosing a point $p\in \Pn^3$ appropriately, the ideal $J$ of $Z\cup \{p\}$ belongs to ${\mathcal G}_{B'}$.    In case (1) the point $p$ is general.  In case (2) it lies in the plane spanned by a subscheme of length $3$ in $Z$. The remaining cases are similar.  
\end{proof}

\begin{lemma} \label{closurerelationGF} Assume ${\mathcal G}_i$ is an irreducible component of the locus of quadratic ideals $Q_F$ for forms $F\in {\mathcal F}_{B_i}$, where $i \in \{1,2\}$. 
If ${\mathcal G}_1\sqsubset {\mathcal G}_2$ or ${\mathcal G}_1\prec {\mathcal G}_2$, then
${\mathcal F}_{B_1}\subset\overline{ {\mathcal F}_{B_2}}$.
\end{lemma}
\begin{proof}
Assume first that ${\mathcal G}_1\sqsubset {\mathcal G}_2$. 
Recall that  $\mathcal{F}_{B_1}$ is an irreducible constructible subset in $\mathbb{P}^{34}$. It follows that it is enough to prove that for a general $F\in \mathcal{F}_{B_1}$ we have $F\in \overline{\mathcal{F}_{B_2}}$.
Let us hence fix $F_1\in \mathcal{F}_{B_1}$ general. Since the map $\mathcal F_{B_1}\to {\mathcal G}_1$ given by $F\mapsto Q_F$  is algebraic $Q_{F_1}$ is general in ${\mathcal G}_1$ and hence by assumption there exists $F_2\in \mathcal {F}_{B_2}$ such that $Q_{F_2}\subset Q_{F_1}$. Let $L\subset \mathbb{P}^{34}$ be the space of polynomials of degree 4 which are annihilated by $Q_{F_2}$.
Clearly $L$ is a linear subspace such that $F_1\in L$. Furthermore, by Lemma \ref{closurerelation1} we have $L\subset \overline{\mathcal{F}_{B_2}}$. In particular $F_1\in \overline{\mathcal{F}_{B_2}}$.

Let us now consider the relation ${\mathcal G}_1\prec {\mathcal G}_2$. This part is slightly more delicate. Note that all the relations of type ${\mathcal G}_1\prec {\mathcal G}_2$ are listed in Proposition \ref{precG}. Recall that we have already proven that $\mathcal{F}_{[300b]} \subset \overline{\mathcal{F}_{[300a]}}$.
Furthermore, in Proposition \ref{vsp proposition} we have a description of all strata $\mathcal{F}_B$ in terms of descriptions of configurations of points computing the rank of a general element of  $\mathcal{F}_B$. 
Observe that for each relation ${\mathcal G}_1\prec {\mathcal G}_2$ listed in Proposition \ref{precG} and not involving $\mathcal{G}_{[300ab]}$ the ranks of the general elements of ${\mathcal F}_{B_1}$ and ${\mathcal F}_{B_2}$ are equal. 
We verify that, in these cases, the configurations of points computing the rank of elements of ${\mathcal F}_{B_1}$ are specialisations of configurations of points computing ranks of elements of ${\mathcal F}_{B_2}$. 

More precisely, let $F\in {\mathcal F}_{B_1}$ be general and let  $\Gamma_F$ be a configuration of points computing its rank. Let $\langle \Gamma_F\rangle_4\subset \Pn^{34}$ be the linear span of $\Gamma_F$ in its embedding in the space of quartics. Then there exists a sequence of configurations of points $(\Gamma_n)_{n\in \mathbb N}$ such that their limit is $\Gamma_F$ and $\Gamma_n$ are configurations of points computing the ranks of some forms $F_n\in {\mathcal F}_{B_2}$. In fact, in this case,  for each $n$ the general quartic apolar to $\Gamma_n$ is in ${\mathcal F}_{B_2}$. It follows that $\langle \Gamma_n\rangle_4\subset \overline{{\mathcal F}_{B_2}}$. 

It is now enough to observe that since all considered configurations of points give independent conditions on quartics and consist of the same number of points then we have $\dim (\langle \Gamma_n\rangle_4)=\dim (\langle \Gamma_F\rangle_4)$. We infer that $F$ is a limit of some $F_n\in \langle \Gamma_n\rangle_4\subset \overline{{\mathcal F}_{B_2}}$, and therefore $F\in \overline{{\mathcal F}_{B_2}}$.
\end{proof}

We combine Propositions \ref{precG} and \ref{sqsubsetG} with  Lemma \ref{closurerelationGF} to obtain closure relations between the Betti strata ${\mathcal F}_B$. Type $[000]$ is the general case, hence the closure of $\mathcal{F}_{[000]}$ is $\PP^{34}$.

\begin{proposition}\label{precF}
The irreducible Betti strata for quaternary quartics satisfy the closure relations below. In each case, we describe the locus for a general apolar set $\Gamma$.
\begin{enumerate}
    \item   ${\mathcal F}_{[562]}$ (five points, three on a line) $\subset\overline{ \mathcal{F}_{[551]}}$ (five points, four in a plane) $\subset\overline{ \mathcal{F}_{[550]}}$ (five general points).
    \item  ${\mathcal F}_{[441a]}$  (six points, five in a plane) $\subset\overline{ \mathcal{F}_{[420]}}$ (six general points), 
    \item  ${\mathcal F}_{[441b]}$ (six points, three on each of two skew lines) $\subset\overline{ \mathcal{F}_{[430]}}$ (six points, three in a line) $\subset\overline{ \mathcal{F}_{[420]}}$ (six general points).
     \item   $\mathcal{F}_{[320]}$ (seven points on twisted cubic) $\subset\overline{ \mathcal{F}_{[300b]}}$ (seven general points), 
    \item    $\mathcal{F}_{[300c]}$ (seven points, three in a line) $\subset\overline{ \mathcal{F}_{[300b]}}$ (seven general points), 
    \item  $\mathcal{F}_{[331]}$ (seven points, six in a plane) $\subset\overline{ \mathcal{F}_{[310]}}$ (seven points, five in a plane) $\subset\overline{ \mathcal{F}_{[300b]}}$ (seven points in a CI $(2,2,2)$).
    \item  $\mathcal{F}_{[210]}$ (eight points, six in a plane) $\subset\overline{ \mathcal{F}_{[200]}}$ (eight general points), 
    \item $\mathcal{F}_{[683]}$ $\subset\overline{ \mathcal{F}_{[550]}}$ $\subset\overline{ \mathcal{F}_{[420]}}$ $\subset\overline{ \mathcal{F}_{[300b]}}$ $\subset\overline{ \mathcal{F}_{[300a]}}$ $\subset\overline{ \mathcal{F}_{[200]}}$ $\subset\overline{ \mathcal{F}_{[100]}}$.
    \item $\mathcal{F}_{[683]}$(four general points) $\subset\overline{ \mathcal{F}_{[551]}}$ (five points, four on a plane)
    \item $\mathcal{F}_{[550]}$(five general points) $\subset\overline{ \mathcal{F}_{[430]}}$ (six points, three in a line)
    \item $\mathcal{F}_{[683]}$(four general points) $\subset\overline{\mathcal{F}_{[562]}}$ (five points, three on a line) $\subset\overline{ \mathcal{F}_{[430]}}$ (six points, three on a line) $\subset\overline{ \mathcal{F}_{[300c]}}$ (seven points, three on a line)
    \item $\mathcal{F}_{[562]}$ (five points, three on a line) $\subset\overline{ \mathcal{F}_{[441b]}}$ (six points, three on each of two skew lines)
    \item $\mathcal{F}_{[552]}$ (five points, four on a plane) $\subset\overline{\mathcal{F}_{[441a]}}$ (six points, five in a plane) $\subset\overline{ \mathcal{F}_{[310]}}$ (seven points, five in a plane)$\subset\overline{ \mathcal{F}_{[021]}}$ (eight points, six in a plane)
    \item ${\mathcal{F}_{[441a]}}$ (six points, five in a plane) $\subset \overline{\mathcal{F}_{[331]}}$ (seven points, six in a plane) $\subset\overline{ \mathcal{F}_{[210]}}$ (eight points, six in a plane)
\end{enumerate}
\end{proposition}

The complete intersection case of type $[400]$ requires additional investigation.
\begin{proposition}\label{ci}
$\mathcal{F}_{[420]}$ (six general points) $\subset\overline{ \mathcal{F}_{[400]}}\subset\overline{  \mathcal{F}_{[300a]}}$ (eight points, a complete intersection), while $\mathcal{F}_{[400]}$ does not lie in the closure of $\mathcal{F}_{[300b]}$ (seven points in a complete intersection).
\end{proposition}
\begin{proof}
Assume first that a form $F_0$ with Betti table $\mathcal{F}_{[400]}$ is a limit of forms $F_t$ in $\mathcal{F}_{[300b]}$, then for each $t\not=0$ there is a net $P_t$ of quadrics in $F_t^{\perp}$.  Let $P_0$ be the limit net.
Then $P_0$ is a net inside a web $L_0$, the web of quadrics in $F_0^{\perp}$.
Notice that both $P_0$ and $P_t$ define a complete intersection; i.e.~a scheme $Z_t$ of length $8$. For each $t\not=0$, the quartic $F_t$ is apolar to a distinguished subscheme $Z'_t\subset Z_t$ of length $7$, which is linked to a point in $Z_t$.  In the limit, $Z_0$ has a distinguished subscheme $Z'_0$ of length $7$.  The image $v_4(Z'_t)$ in $\Pn^{34}$ spans a $\Pn^6$ for each $t$.  By apolarity, for each $t\not=0$ the quartic form $F_t$ lies in the span of $v_4(Z'_t)$, therefore $F_0$ also lies in the span of $v_4(Z'_0)$, i.e.~$F_0$ is apolar to a scheme of length $7$.   But $F_0$ is in $\mathcal{F}_{[400]}$, so the rank is $8$, a contradiction.

Next, assume $F_0$ is of type $[420]$, and consider a net $P_0$ of quadrics in the ideal $F_0^{\perp}$ that defines a complete intersection:  A general net in the web $L_0$ of quadrics in $F_0^{\perp}$ suffices.  Next, let $L_t$ be a family of webs of quadrics that all contain $P_0$, such that $L_t$ defines a complete intersection when $t\not=0$.  Then $L_t$ generates $F_t^{\perp}$ for a quartic $F_t$ in $\mathcal{F}_{[400]}$ when $t\not=0$, and $F_0$ is the limit of $F_t$.
Hence, $\mathcal{F}_{[420]}$ lies in the closure of $\mathcal{F}_{[400]}$. 

Finally, ${\mathcal G}_{[400]}\sqsubset {\mathcal G}_{[300ab]}$, so by Lemma \ref{closurerelationGF} and  Proposition~\ref{irreducibleF_B}, $\mathcal{F}_{[400]}\subset\overline{  \mathcal{F}_{[300a]}}.$
\end{proof}

 Now we go in the opposite direction, and prove that certain containments cannot occur.
 For this we need upper semi-continuity of some graded Betti numbers:
\begin{lemma}\label{lem_2linearBetti}
Fix a Betti table $B=(\beta_{ij})$, and suppose $I$ is a nondegenerate ideal such that $I\in\overline{\mathcal{G}_{B}}$. Then the Betti table of $S/I$ satisfies the inequality $b_{i,i+1}(S/I) \geq \beta_{i,i+1}$.
\end{lemma}
\begin{proof}
Note that $b_{i,i+1}=\dim H_{i}((S/I)\otimes\mathcal{K}_{\bullet})_{i+1}$, where $\mathcal{K}_{\bullet}$ is the Koszul complex. Let $V=\mathfrak{m}/\mathfrak{m}^{2}$ where $\mathfrak{m}$ is the irrelevant ideal of $S$. Consider the complex of vector spaces
\begin{equation*}
    (S/I)_{0}\otimes_{\kk}\wedge^{i+1}V\xrightarrow{\delta_{i+1,1}}(S/I)_{1}\otimes_{\kk}\wedge^{i}V\xrightarrow{\delta_{i,1}}(S/I)_{2}\otimes_{\kk}\wedge^{i-1}V
\end{equation*}
It is clear that
\begin{equation*}
    b_{i,i+1}=\dim \ker\delta_{i,1}-\dim \mathrm{im}~\delta_{i+1,1}.
\end{equation*}
Since $I$ contains no linear form, $\delta_{i+1,1}$ does not depend on $I$, so $\dim \mathrm{im}~\delta_{i+1,1}$ is constant. On the other hand, $\dim \ker\delta_{i,1}=\dim (S/F^{\perp})_{1}\otimes_{\kk}\wedge^{i}V-\rank\delta_{i,1}$. As $\rank \delta_{i,1}$ is lower semi-continuous over the parameter space $G(\beta_{1,2},S_2)$, we conclude that $b_{i,i+1}$ is upper semi-continuous.
\end{proof}

\begin{proposition}\label{noncontainment}
We have the following non-containments
\begin{enumerate}
    \item $\mathcal{F}_{[441b]}\not\subset\overline{\mathcal{F}_{[550]}}$
    \item $\mathcal{F}_{[550]}\not\subset\overline{\mathcal{F}_{[441a]}}$
    \item $\mathcal{F}_{[331]}\not\subset\overline{\mathcal{F}_{[420]}}$
    \item $\mathcal{F}_{[210]}\not\subset\overline{\mathcal{F}_{[300a]}}$
    \item $\mathcal{F}_{[331]}\not\subset\overline{\mathcal{F}_{[320]}}$
\end{enumerate}
\end{proposition}

\begin{proof} The map $${\mathcal F}_B\to {\mathcal G}_{B'}, \quad F\to Q_F$$ is surjective for each irreducible stratum 
${\mathcal F}_B$, when $B'$ is the Betti table of $S/Q_F$ for some $F\in {\mathcal F}_B$.
Therefore the first four cases follow immediately from Lemma \ref{lem_2linearBetti}.

The last case is more delicate.
Let $F_{0}=w^{4}+G_{0}(x,y,z)$. We would like to show $F_{0}\not\in\overline{\mathcal{F}_{[320]}}$. \\
Assume $F_{0}$ is a limit of $F_{t}$ in $\mathcal{F}_{[320]}$. For each $t\neq 0$, there is a unique ideal $C_{t}\subseteq F_{t}^{\perp}$ such that $C_{t}$ is a twisted cubic. Therefore, $C_{t}$ is in $\Hilb^{3m+1}(\PP^{3})$. The limit $C_{0}$ must be a subideal of $F_{0}^{\perp}$. By \cite{PS85}, $\Hilb^{3m+1}(\PP^{3})$ has two components $H$ and $H'$ where $H$ is the closure of $H_{0}$, the space of smooth twisted cubics, and $H'$ is the closure of 
\begin{equation*}
    H_{0}'=\{C'\mid~C'=\mbox{a plane, smooth cubic curve in $\PP^{3}$ union a point in $\PP^{3}$ not on the curve}\}.
\end{equation*}
All $C_{t}$ with $t\neq 0$ is in $H$, therefore $C_{0}$ must also be in $H$. On the other hand, note that the second row of of type $[331]$ is $(3,3,1)$, so $F_{0}^{\perp}$ does not have a subideal in $H \setminus H\cap H'$. Hence $C_{0}\in H\cap H'$, which means it is projectively equivalent to $\langle xz,yz,z^2,c(x,y,w)\rangle$ for some $c(x,y,w)\in S_{3}$. This is a contradiction because $\langle xz,yz,z^2\rangle$ is not projectively equivalent to $\langle xw,yw,zw\rangle=F_{0}^{\perp}(2)$. Therefore, the hypothesis that $F_{0}$ is a limit of $F_{t}$ is false. 
\end{proof}
\noindent Putting everything together yields the stratification appearing in Figure 1 on the next page. We include in each box the name of the Betti table $B$, the dimension of the locus $\mathcal{F}_{B} \subseteq \Pn^{34}$, and a set of points $\Gamma$ such that $\Gamma$ is a minimal set of points apolar to a general $F\in\mathcal{F}_{B} $. 

\begin{remark} In the next sections we will consider liftings of the Artinian rings $A_F$ to AG varieties $X$. The Hilbert polynomial of $X$ then depend only on the dimension and on the Hilbert function of $A_F$.  Therefore, if $X$ and $X'$ are liftings of quaternary quartic forms and have the same Hilbert polynomial, then $X$ deforms to $X'$ only if the same holds for Artinian reductions of $X$ and $X'$.  And the latter can be read off of the Betti stratification in Figure 1, since over a locus $\mathcal H$ for which the apolar ideals have a fixed Hilbert function the family $\{S/F^{\perp}\}_{F\in \mathcal H}$ is flat.
\end{remark}
\begin{remark}
Note that in this section we do NOT claim we have a full stratification of $U_{34}$ by Betti tables. Denote by $P$ a set of subsets of $U_{34}$. By full stratification we mean that $P$ satisfies the following properties:
\begin{enumerate}
    \item $U_{34}=\coprod_{\mathcal{F}\in P} \mathcal{F}$,
    \item $\overline{\mathcal{F}}=\coprod_{\mathcal{F'}\subseteq \overline{\mathcal{F}}}\mathcal{F}'$.
\end{enumerate}
The second property above is equivalent to the so-called \emph{frontier property}, which says if $f\in \mathcal{F}$ and $f\in\overline{\mathcal{F}'}$, then $\mathcal{F}\subseteq\overline{\mathcal{F}'}$. Iarrobino and Kanev provide in \cite{IK}*{Section 5} a full stratification of the punctual Hilbert scheme $\Hilb^{s}(\PP^{n})$. In particular, that stratification satisfies the frontier property.

The frontier property does not hold in our case. For example, since twisted cubics degenerate to a union of a line and a plane conic, $\mathcal{F}_{[320]}$ contains elements with apolar sets in which five of the seven points are on a plane conic. These elements are in the closure of $\mathcal{F}_{[310]}$, while the general element in $\mathcal{F}_{[320]}$ is not, so $\mathcal{F}_{[320]}\cap \overline{\mathcal{F}_{[310]}}\not= \emptyset$ and $\mathcal{F}_{[320]}\not\subset\overline{\mathcal{F}_{[310]}}$.\\
We may obtain a stratification satisfying the frontier property by applying the following procedure: Take $\cF_{ij}=\cF_{i}\cap \overline{\cF_{j}}$ if $\cF_{i}\cap \overline{\cF_{j}}$ is neither $\cF_{i}$ nor $\emptyset$. Then take $\cF_{i}'=\cF_{i}-\cup_{j}\cF_{ij}$ and consider the new collection $P'=\{\cF_{i}'\}\cup\{\cF_{ij}\}$. By irreducibility of $\cF_{i}$, $\overline{\cF_{i}'}=\overline{\cF_{i}}$. We claim $P'$ satisfies the frontier property: Assume $f\in\overline{\cF_{i}}$ and $f\not\in\cF_{i}'$. Then either $f\in\cF_{i}$, in which case $f\in\cF_{ij}$ for some $j$, or $f\not\in \cF_{i}$, in which case $f\in\cF_{ki}$ for some $k$.
\end{remark}

\begin{center}
\begin{tikzcd}
                             & {\begin{tabular}{|c|}\hline [000] $\dim 34$\\ten points\\\hline \end{tabular}}                        &                                        \\
                             & {\begin{tabular}{|c|}\hline [100] $\dim 33$\\nine points\\\hline \end{tabular}} \arrow[u] &                                        \\
                             & {\begin{tabular}{|c|}\hline [200] $\dim 31$\\eight points\\\hline \end{tabular}}\arrow[u]           &                                        \\
                             & {\begin{tabular}{|c|}\hline [300a] $\dim 28$\\eight points, CI\\\hline \end{tabular}}\arrow[u]              &           \\
                             & {\begin{tabular}{|c|}\hline [300b] $\dim 27$\\seven points\\\hline \end{tabular}}\arrow[u]              &   {\begin{tabular}{|c|}\hline [210] $\dim 25$\\eight points, six on $\PP^{2}$\\\hline \end{tabular}}\arrow[luu]                                      \\
{\begin{tabular}{|c|}\hline [400] $\dim 24$\\eight points, CI\\\hline \end{tabular}}\arrow[ruu]   & {\begin{tabular}{|c|}\hline [320] $\dim 24$\\seven points, on a twisted cubic\\\hline \end{tabular}} \arrow[u]           & {\begin{tabular}{|c|}\hline Type\; [310] $\dim 24$\\seven points, five on $\PP^{2}$\\\hline \end{tabular}}\arrow[u] \arrow[lu] \\
         {\begin{tabular}{|c|}\hline [300c] $\dim 24$\\seven points, three on $\PP^{1}$\\\hline \end{tabular}} \arrow[ruu]                   & {\begin{tabular}{|c|}\hline [420] $\dim 23$\\six points\\\hline \end{tabular}}  \arrow[u] \arrow[lu]            & {\begin{tabular}{|c|}\hline [331] $\dim 21$\\seven points, six on $\PP^{2}$\\\hline \end{tabular}}\arrow[u]            \\
{\begin{tabular}{|c|}\hline [430] $\dim 20$\\six points, three on $\PP^{1}$\\\hline \end{tabular}} \arrow[u] \arrow[ru]  &                                        & {\begin{tabular}{|c|}\hline [441a] $\dim 20$\\six points, five on $\PP^{2}$\\\hline \end{tabular}} \arrow[u] \arrow[lu] \\
                             & {\begin{tabular}{|c|}\hline [550] $\dim 19$\\five points\\\hline \end{tabular}} \arrow[lu]              &                                        \\
{\begin{tabular}{|c|}\hline [441b] $\dim 17$\\six points, $3+3$ on $\PP^{1}$\\\hline \end{tabular}}\arrow[uu] &                                        & {\begin{tabular}{|c|}\hline [551] $\dim 18$\\five points, four on $\PP^2$\\\hline \end{tabular}} \arrow[lu] \arrow[uu] \\
                             & {\begin{tabular}{|c|}\hline [562] $\dim 16$\\five points, three on $\PP^{1}$\\\hline \end{tabular}} \arrow[ru] \arrow[lu] &                                        \\
                             & {\begin{tabular}{|c|}\hline [683] $\dim 15$\\four points\\\hline \end{tabular}}  \arrow[u]               &                                       
\end{tikzcd}
\mbox{FIGURE 1. Betti table stratification}
\end{center}

\subsection{Relations to Noether-Lefschetz varieties}\label{non-noether}
Each Betti stratum admit the Fermat quartics in its closure, therefore the general quartic in each stratum is smooth, i.e.~defines a smooth K3 surface.  Furthermore, each stratum is invariant with respect to the action of $PGL(3)$, so the smooth quartics in each stratum form a subset in the moduli space of K3 surfaces.
We therefore compare the Betti strata with the Noether-Lefschetz locus in the moduli space of K3 surfaces.

We consider, for each of the Betti strata, the set of smooth quartics in the moduli space $\mathcal M_4$ of K3 surfaces with polarisation of degree 4. In the latter moduli space we distinguish the Noether-Lefschetz locus of K3 surfaces with Picard number greater than 1. It is a union of an infinite number of divisors $D_K$ called Noether-Lefschetz divisors, indexed by primitive rank 2 sublattices of the K3 lattice containing the polarisation of degree 4. 

For such a lattice $K$ the divisor $D_K$ is the closure in the moduli space of the locus of all K3 surfaces having $K$ as Picard lattice. 
\begin{definition}
For primitive sublattices $K$ of rank $r>2$, $V_K$ is a {\em Noether-Lefschetz variety} of codimension $r-1$ in the moduli space $\mathcal M_4$. 
\end{definition}
We have the following results.

\begin{proposition} The following Betti strata of quaternary quartics form Noether-Lefschetz  varieties.
\begin{itemize}
    \item Quartics of type $[683]$ are Fermat surfaces of Picard number 20 and form a Noether-Lefschetz variety consisting of a point.
    \vskip .1in
    \item Quartics of type $[441b]$ are double covers of Kummer surfaces of the product of two elliptic curves and have Picard rank $18$ and form a $2$-dimensional Noether-Lefschetz variety. 
\end{itemize}
\end{proposition}

\begin{proof} For Fermat quartics the result is classical, see for example \cite{Shioda}.
H. Inose \cite{Inose} gave the double cover description of quartics of type $[441b]$, and their Picard number were computed by Shioda-Inose \cite{Shioda-Inose}, see also \cite{Kuwata}. 
\end{proof}

\begin{corollary} The Betti strata of quaternary quartics of types
\begin{equation}\label{NoetherBad}
\{ [562], [550], [551], [441a], [331], [310], [210]\}
\end{equation} 
do not form Noether-Lefschetz varieties.

\end{corollary}
\begin{proof}
By \cite{Kuwata}*{Thm. 1.3}, a general quartic of type $[562]$ has Picard number 18. Indeed, a quartic of type $[562]$ is a Kummer surface associated to a product of elliptic curves, one of which is the double cover of $\mathbb{P}^1$ branched in the points $(1:\zeta_i)$ with $\zeta_i$ being 4-th roots of unity. 

Since the other elliptic curve can be chosen randomly, these are not isogenous and hence the associated Kummer surface has Picard number 18, by \cite{Kuwata}*{Thm. 1.3}.  Now, since the stratum of type $[441b]$ is irreducible, the stratum of quartics of type $[562]$ cannot fill a Noether-Lefschetz variety.  Consequently any strata containing $\mathcal{F}_{[562]}$ but not $\mathcal{F}_{[441b]}$ cannot form a Noether-Lefschetz variety. This includes the strata $\mathcal{F}_{B}$ of quartics of the types in Equation~\ref{NoetherBad}.
\end{proof}

The proposition above shows in particular that type $[331]$ is not a Noether-Lefschetz variety. However, in this case we can be more precise. For instance, the periods of the associated K3 surfaces for type $[331]$ are well understood: quartics of this type are Galois quadruple covers of $\mathbb{P}^2$ branched along a quartic curve. These were carefully studied by Kondo in \cites{Kondogenus3, Kondoplanequartics}.

\begin{lemma}\label{NLType25}The Betti stratum $\mathcal F_{[331]}$ of quaternary quartics of type $[331]$ is contained in the family of quartics  that are  double covers of del Pezzo surfaces of degree 2. The latter form a Noether-Lefschetz variety $V_L$ corresponding to the Picard Lattice $L:=\langle2\rangle+\langle -2\rangle^7$. Furthermore, if $N$ is a Noether-Lefschetz variety such that $\mathcal F_{[331]}\subseteq N$ then $V_L\subseteq N$.
\end{lemma}
\begin{proof}
S. Kondo, loc. cit.,  gives a precise description of the periods of the associated K3 surfaces. In particular, a general K3 surface of type $[331]$, which is a general Galois quadruple cover of $\mathbb{P}^2$, has Picard Lattice $L=\langle2\rangle+\langle -2\rangle^7$. We thus have a $6$-dimensional family of K3 surfaces contained in the $12$-dimensional Noether-Lefschetz variety $V_{L}$ in the moduli space of polarised K3 surfaces of degree 4. These do not fill the whole Noether-Lefschetz variety $V_L$, but also are not contained in any smaller Noether-Lefschetz variety. 
\end{proof}

We can now deduce information about all Betti strata containing the stratum of type $[331]$.
\begin{corollary}
A Betti stratum $\overline{\mathcal F_B}$ of quaternary quartics containing $\mathcal F_{[331]}$ cannot be a Noether-Lefschetz variety. In particular, the Betti stratum $\overline{\mathcal F_{[100]}}$ is not a Noether-Lefschetz divisor. In other words, the K3 surface defined by a very general quartic of rank $9$ has Picard number $1$.
\end{corollary}
\begin{proof}
Let $q$ be general quartic of the form 
$$(x^2+f_2(y,z,t))^2+f_4(y,z,t), $$
for $f_2$, $f_4$ polynomials of degree 2 and 4 respectively depending only on the variables $y,z,t$ in $\mathbb{P}(x,y,z,t)$. Then $q$ defines an element in $V_L$ from the proof of Lemma \ref{NLType25} and its apolar ring can be checked with Macaulay2, \cite{M2}*{Package QuaternaryQuartics}, on the example: 
\[
(x^2+y^2+z^2+t^2)^2+y^4+z^4+t^4
\]
to have Betti table of type $[000]$. But by Lemma \ref{NLType25} any Noether-Lefschetz variety containing the stratum type $[331]$ must contain the whole variety $V_L$. 
\end{proof}

\begin{remark} The fact that the locus of quartics of rank 9 is not a Noether-Lefschetz divisor can be proven alternatively in the following way. 
Observe that each Noether-Lefschetz divisor contains a non-empty locus of hyperelliptic K3's that are covers of a smooth quadric.
It follows that we can find a flat 1-parameter family of quartics in each Noether-Lefschetz divisor converging to a double quadric (see \cite{Shah-Trans}).
It remains to observe (compute the Betti table of the apolar ring with Macaulay2, \cite{M2}*{Package QuaternaryQuartics} ) that a double quadric is of type $[000]$.
\end{remark}
\begin{remark} Let us discuss the types of some other known K3 surfaces.
The Vinberg most singular K3 surface $X_4$ given by the equation 
\[
x^4=-yzt(y+z+t)
\]
is of type $[331]$.
A general element of the Dwork pencil
\[
x^4+y^4+z^4+t^4-4txyzt=0
\]
is of type [000].
The K3 quartics $S_t\subset \mathbb{P}^5$ with coordinates $x_1,\dots,x_5$ having icosahedral symmetry given by equations 
\[
x_1^4 +\dots+x_5^4-t(x_1^2 +\dots+x_5^2)^2=x_1 +\dots+x_5=0
\]
for general $t$ are of type [000], however special elements in the pencil are in smaller strata. For example $S_0$ is of type $[550]$.  
\end{remark}

\section{Codimension three varieties in quadrics}\label{Pfaffian}
We now tackle the question of lifting quartic forms to positive dimensional AG-varieties.
Our goal in this section is to describe AG varieties $X$ of codimension $3$ in smooth quadric hypersurfaces, in particular those whose homogeneous coordinate ring has regularity $4$.
 According to \cite{EPW}, AG-varieties of codimension $3$ 
 in a smooth quadric hypersurface that satisfies certain cohomological conditions 
 appear as Lagrangian degeneracy loci of a map between two Lagrangian subbundles of an orthogonal bundle. 
 Furthermore, the Lagrangian degeneracy loci may have a simpler description as Pfaffian loci. We start with a description of the Artinian Gorenstein rings.

\subsection{The Artinian case} Most of the considered Artinian Gorenstein rings up to degree 19 can be seen explicitly as quotients of the homogeneous coordinate ring of a smooth quadric surface $Q\subset \mathbb{P}^3$.
More precisely let $I$ be a homogeneous ideal  in the coordinate ring $S=\mathbb C[x_0,x_1,x_2,x_3]$ of $\mathbb{P}^3$. Assume that in the ideal there is a smooth quadric $Q$ that we interpret as $\mathbb{P}^1\times \mathbb{P}^1$. We can refine the grading on $S$ so that $\deg x_0=\deg x_1=(1,0)$ and $\deg x_2=\deg x_3=(0,1)$, and assume $Q$ has bidegree $(1,1)$. Let $T=\oplus_{k,l}T(k,l)$ be the bigraded homogeneous coordinate ring of $Q$.
We have an exact sequence:
$$0\rightarrow S(-1,-1)\stackrel{\cdot Q}{\rightarrow}  S\xrightarrow{r} T,$$
with the latter map being the restriction map.  The image of $r$ is the subring $\oplus_kT(k,k)$. 
\begin{remark}\label{fully apolar}
For a fixed ideal Gorenstein ideal $I\subset S$ there might be many different ideals $J\subset T$ such that $r^{-1}(J)= I$. In particular, if $F$ is the dual socle generator of $I$, so that $I=F^{\bot}$ and $F_Q$ the restriction of $F$ to $Q$, then $r^{-1}(F_Q^{\bot})=F^{\bot}$. 
The ideal $F_Q^{\bot}$ is the biggest ideal with the property $r^{-1}(F_Q^{\bot})=F^{\bot}$. 
In fact, there might be smaller ideals $J\subset F_Q^{\bot}$ such that $r^{-1}(J)=F^{\bot}$. From the point of view of apolarity on $T$, note that the condition $r^{-1}(J)=F^{\bot}$ implies $J(4,4)=F_Q^{\perp}(4,4)$, however the latter condition is weaker.
\end{remark}
Remark \ref{fully apolar} justifies the introduction of the following definition that is needed in our investigation.
\begin{definition}
Let $F_Q$ be a form of degree $(a,b)$ in $T^{\vee}$. We say that an ideal $J\subset T$ is fully apolar to $F_Q$ if $J(a,b)=F_Q^{\perp}(a,b)$.
\end{definition}
Our aim in this section is to relate the Artinian rings $S/F^{\bot}$ 
to grade 3 Gorenstein rings obtained as quotients of $T$. The quotient ring $T/F_Q^{\perp}$ is a ring of grade 4. Since the restriction map $r:S\to T\to T/F_Q^{\perp}$ is not surjective, there may still be ideals $J$ of grade 3 such that $r^{-1}(J)=F^{\bot}$.  However, this is not always the case. For that reason we consider the weaker condition of full apolarity to $F_Q$. Before we formulate the proposition let us recall some further terminology.
 
 A Pfaffian ideal $J\subset T$ is an ideal with a resolution of the form
 $$0\to T(k,l) \to E^{\vee}\xrightarrow{M} E\to J\to 0,$$
 where $E$ is a free $T$-module and $M$ is an $(2n+1)\times (2n+1)$-dimensional skew symmetric matrix of bihomogeneous forms.
 The map $E\to J$ is defined by the $2n\times 2n$ Pfaffians of $M$.
 If $n=3$, $J$ is the complete intersection ideal generated by the three nonzero entries of $M$.

 \begin{proposition}\label{Artinian pfaffian}
 Let $F$ be a quaternary quartic apolar to s smooth quadric $Q$, and let $F_Q$ be the restriction of $F$ to $Q$.  Assume $F$ is general in its type. Then there is 
 an  ideal $J$ fully apolar to $F_Q$ with a Pfaffian minimal resolution associated to a skew-symmetric map $\varphi: E^{\vee}\to E$. The bundle $E$ and the zero blocks of $\varphi$ are listed in the table below, for each possible type.  
\begin{center}
\begin{table}[H]
\begin{tabular}{|c|c|c|}
\hline Type & $E$ & Zero blocks of $\varphi$ \\
\hline $[683]$ & $T(2,1)\oplus T(2,1)\oplus T(-1,1)$  & \\
\hline $[550]$ & $T(2,1)\oplus T(1,2)\oplus T$ & \\
\hline $[400]$ & $T(1,1)\oplus T(1,1)\oplus T(1,1)$ &\\
\hline $[300a]$ & $T(1,0)\oplus T(0,1) \oplus 2 T(1,1)\oplus T  $ & \\
\hline $[300b]$ &  $T(1,0)\oplus T(0,1) \oplus 2T(1,1)\oplus T  $ & $T(0,1)\oplus T$ \\
\hline $[300c]$ & $T(1,0)\oplus T(0,1) \oplus 2T(1,1)\oplus T  $ &  $T(0,1)\oplus T$, $T(1,0)\oplus T(0,1)$\\
\hline $[310]$  &$T(1,0)\oplus T(0,1) \oplus 2T(1,1)\oplus T  $ &  $T(0,1)\oplus T$,  $T(1,0) \oplus T$ \\
\hline $[320]$ & $T(1,0)\oplus 2T(0,1) \oplus T(2,1)\oplus T  $ & \\
\hline $[200]$ & $2T(1,0)\oplus 2T(0,1) \oplus T(1,1)$ & \\ 
\hline $[100]$ & $3T(1,0)\oplus 3T(0,1) \oplus T$ & \\ 
\hline $[420]$ & $3T(1,1)\oplus 2T$ & \\
\hline $[430]$ & $T(1,1)\oplus T(2,1)\oplus T(0,1)\oplus 2T$ &  $T\oplus T(0,1)$ \\
\hline $[441a]$ & $T(1,2)\oplus T(2,1)\oplus T(0,1)\oplus T(1,0)\oplus T(-1,-1)$ & \\
\hline $[441b]$ & $T(1,0)\oplus T(1,0)\oplus T(1,3)$ & \\
\hline $[551]$ & $T(1,2)\oplus T(2,1)\oplus T(1,1)\oplus T(-1,-1)\oplus T$ & \\
\hline $[562]$ & $T(1,2)\oplus T(2,1)\oplus T(0,1)\oplus T(0,-1)\oplus T$ & \\
\hline
\end{tabular}
\vskip .1in
\caption{The bundle $E$}\label{PfaffTable}
\end{table}
\end{center}
\vskip -.4in
 \end{proposition}
 \begin{remark}
 Note that in Proposition \ref{Artinian pfaffian} we consider minimal resolutions of the considered shape.  This in particular implies that the matrix associated to the map $\varphi$ has no constant nonzero entries. So in all the cases in which, for degree reasons, constant entries may appear, these constant entries are $0$.

 \end{remark}
 \begin{remark}
 Note that the same 
 quartic form $F$ may admit several fully apolar ideals $J$ on a quadric $Q$ apolar to $F$. This includes possible ideals $J$ whose resolution is different from those listed in Proposition \ref{Artinian pfaffian}. Furthermore, notice that in all cases treated above only type $[100]$ are contained in a unique quadric. In the remaining cases we have a choice of quadric and the description depends on the choice.  This means that on some quadrics, possibly even general ones, the restriction $F_Q$ might not be fully apolar to a $J$ of the form listed in Proposition \ref{Artinian pfaffian}.
 \end{remark}

\begin{proof} 
Observe that from the Pfaffian resolution of the ideal $J$  we obtain that the space $J(4,4)$ has codimension $1$ in $T(4,4)$. It follows that $J(4,4)=F\perp_Q(4,4)$ for some $(4,4)$-form $F_Q$. In particular, $J$ is fully apolar to $F_Q$. 

To determine the type of the corresponding quartic form $F$ on $R$ in most (all) cases (except $[300]$ which is trivial) it is enough to identify a suitable configuration $\Gamma$ of points of  on $Q$  defined by elements of $J$ such that $F$ is apolar to $\Gamma$. 
We can then conclude if we prove that a general quartic $F_Q$ apolar to $\Gamma$ can be obtained by means of our construction.  
For a $(4,4)$-form $F_J$, we say it is of type $[b_{12},b_{23},b_{34}]$ if there is an $F$ of this type with restriction $F_Q=F_J$.

Let us be more precise in each case.
\begin{description}
\item[{\bf Type $[683]$}] In this case $J$ is a complete intersection ideal of type $(2,1),(2,1), (2,4)$ generated by $p_1,p_2,p_3$ of the corresponding bidegrees. Since it is a complete intersection we easily check that $\langle p_1,p_2,p_3 \rangle(4,4)$ is of dimension 24 and hence of codimension 1 in $T(4,4)$. Let $F_J$ be the $(4,4)$-form apolar to $J$, unique up to constant multiple. We claim that $F_J$ is of type $[683]$. Indeed, $\langle p_1,p_2 \rangle$ defines a set $\Gamma$  of 4 general points on $Q$. Furthermore $\langle p_1,p_2 \rangle(4,4)$ is a subspace of codimension 4 in $T(4,4)$ and hence is the space of $(4,4)$-forms vanishing on $\Gamma$. It follows that $F_J$ is apolar to $\Gamma$ and hence is of type $[683]$.

For the generality part, consider a general quartic $F$ of rank 4 and let $\Gamma$ be its apolar fourtuple of points, which are in general position. Clearly $\Gamma$ lies on a smooth quadric surface and appears there as a complete intersection of two forms $p_1$, $p_2$ of degree $(2,1)$. Furthermore, by dimension count there must exist a form $p_3$ of type $(2,4)$ annihilating $F_Q$. Moreover if $F_Q$ is general among $(4,4)$ forms apolar to $\Gamma$ then $\langle p_1, p_2, p_3 \rangle$ is a complete intersection ideal in T. 
\vskip .05in
\item[{\bf Type $[550]$}] A set of five general points $\Gamma$ lies on a smooth quadric $Q$ as a complete intersection of two forms $p_1$,$p_2$ of respective bidegrees $(1,2), (2,1)$. Then, $\dim_{\CC} \langle p_1,p_2 \rangle(2,2)=4$, $\dim_{\CC} \langle p_1,p_2 \rangle(3,3)=11$ and $\langle p_1,p_2 \rangle(4,4)=20$. We deduce that any form apolar to $\langle p_1,p_2 \rangle$ is apolar to $\Gamma$ the set of five points (i.e. to the whole ideal of $\Gamma$) defined by $p_1,p_2$. Furthermore if $p_3$ is a general element of $T(3,3)$ then $\dim_{\CC} \langle p_1,p_2,p_3 \rangle(4,4)=24$, hence there is a form $F_Q$ apolar to $\Gamma$ to which $\langle p_1,p_2,p_3 \rangle$ is fully apolar in $T$. 

For generality,  we observe that for  a general $(4,4)$ form $F_Q$  apolar to a set of five general points $\Gamma$ defined by $p_1$,$p_2$ of bidegrees $(2,1)$ and $(1,2)$ we have $\dim_{\CC} F_Q^{\perp}(3,3)=12$ and $\dim_{\CC} \langle p_1,p_2 \rangle(3,3)=11$. It follows that $F_Q$ admit an additional apolar $(3,3)$ form $p_3$. Now since $F_Q$ was general $\langle p_1,p_2,p_3 \rangle$ is a complete intersection.

\vskip .05in
\item[{\bf Type $[400]$}] This case is trivial, since if $J$ is a complete intersection of three quadrics on a quadric $Q$ and F is the unique up to constant $(4,4)$ form apolar to all the quadrics then the stronger condition $r^{-1}(J)=F^{\bot}$ is clearly fulfilled. 
\vskip .05in
\item[{\bf Type $[300a]$}] If $J$ is a the Pfaffian ideal of a general $\varphi$ for this type in Table \ref{PfaffTable}, then one computes  $\dim_{\CC} J(4,4)=24$ and hence there is a $(4,4)$-form $F_Q$ to which $J$ is apolar. Now, $F_Q$ is apolar to the set $\Gamma$ of eight points defined by the pencil of elements in $J(2,2)$. Indeed $\dim_{\CC} \langle J(2,2) \rangle(4,4)=25-8$ which is the dimension of the space of $(4,4)$ forms passing through 8 points. To prove that $F_Q$ corresponds to the family $[300a]$ it is enough to check that it is a general $(4,4)$ form apolar to $\Gamma$.

Let $\Gamma$ be a complete intersection of three quadrics, consisting of eight points, and assume $F$ is a general quartic apolar to  $\Gamma$. Let $Q$ be a smooth quadric that contains $\Gamma$ and let $q_1, q_2\in T(2,2)$ define $\Gamma$ on $Q$. Consider the $(4,4)$-form  $F_Q$, the restriction of $F$ to $Q$. Then $\dim_{\CC} F_Q^{\bot}(2,2)=2$,  $\dim_{\CC} F_Q^{\bot}(3,2)=\dim_{\CC} F_Q^{\bot}(2,3)=6$ and $\dim_{\CC} F_Q^{\bot}(3,3)=12$. Let $p_1$ be a general form in $F_Q^{\bot}(3,2)$. Then $\langle q_1,q_2,p_1 \rangle$ is a complete intersection ideal such that $\dim_{\CC} \langle q_1,q_2,p_1 \rangle(4,4)=23$. We furthermore choose $p_2\in F_Q^{\bot}(2,3)$ such that $\langle q_1,q_2,p_1,p_2 \rangle$ admits one syzygy in $T(3,3)$. This is possible since $\dim_{\CC} \langle q_1,q_2,p_1 \rangle(3,3)=10$,  $\dim_{\CC} F_Q^{\bot}(3,3)=12$ and every form $p_2\in F_Q^{\bot}(2,3)$ generate a pencil of forms  $F_Q^{\bot}(3,3)$, so the set of forms $p_2$ such that this pencil intersects $\langle q_1,q_2,p_1 \rangle(3,3)$ in $F_Q^{\bot}(3,3)$ has codimension at most $1$. 
To see that there is a $p_2$ such that $\langle q_1,q_2,p_1,p_2 \rangle$ admits only one syzygy in $T(3,3)$ we note that $\langle p_2 \rangle(3,3)\subset \langle q_1,q_2,p_1 \rangle(3,3)$ implies $p_2\in \langle q_1,q_2 \rangle$.
Note moreover, that by possibly replacing $q_1,q_2$  by suitable linear combinations and replacing $p_1, p_2$ by adding multiples of $q_1$ and $q_2$  we can assume that our syzygy in $T(3,3)$ is of the form $q_2l+p_1s_1+p_2s_2=0$ with $l,s_1,s_2$ of respective bidegrees $(1,1)$,$(0,1)$,$(1,0)$. Furthermore, by generality, we may assume $l,s_1,s_2$ to be a complete intersection.  In that case $q_2,p_1,p_2$ are the $2\times 2$ minors of a matrix 
$$N=\begin{pmatrix} l&s_1&s_2\\
m_1&m_2&m_3\\
\end{pmatrix}$$ 
representing a map $T(0,0)\oplus T(-1,-1)\to T(1,0)\oplus T(0,1)\oplus T(1,1)$. The bidegrees of the entries $ m_1,m_2,m_3$ are $(2,2),(1,2),(2,1)$ respectively.
Then the only syzygy of $\langle q_1,q_2,p_1,p_2 \rangle$ in bidegree $(4,4)$ that involves $q_1$ is the Koszul syzygy with $q_2$, so one computes  $\dim_{\CC}\langle q_1,q_2,p_1,p_2 \rangle(4,4)=24$.   Thus $\langle q_1,q_2,p_1,p_2 \rangle$ is fully apolar to $F_Q$.
We may express $q_1$ as a quadric in the ideal generated by the first row of $N$, i.e. $q_1=a_1l+a_2s_1+a_3s_2$ where the $a_i$ have bidegrees $(1,1),(2,1),(1,2)$ respectively.
Then the skew-symmetric matrix 
$$M=\begin{pmatrix}0&0&l&s_1&s_2\\
0&0&m_1&m_2&m_3\\ 
-l&-m_1&0&-a_3&a_2\\
-s_1&-m_2&a_3&0&a_1\\ 
-s_2&-m_3&-a_2&-a_1&0\\
 \end{pmatrix}$$ 
representing a map from $E^{\vee}\to E$, have $q_1,q_2,p_1,p_2$ as four of its pfaffians. 
The remaining pfaffian $p_3:=a_1m_1+a_2m_2+a_3m_3$ has bidegree $(3,3)$. 

There are three syzygies describing  $p_3l, p_3 s_1,p_3s_2$ in terms of $q_1,q_2,p_1,p_2$. Since $\langle l,s_1,s_2 \rangle(1,1)=T(1,1)$, it follows that $\langle q_1,q_2,p_1,p_2,p_3 \rangle(4,4)=\langle q_1,q_2,p_1,p_2 \rangle(4,4)$ i.e. $\langle q_1,q_2,p_1,p_2,p_3 \rangle$ is a pfaffian ideal of the desired shape that is fully apolar to $F_Q$.
\vskip .05in
\item[{\bf Type $[300b]$}] 
Let $\phi$ be a general skew map for this type as described in Table \ref{PfaffTable}. 
Among the pfaffians generating the Pfaffian ideal $J$ we recover the $2\times 2$ minors of a general map $T\oplus T(-1,0)\to 2T(1,1)\oplus T(0,1)$. The latter defines a set of seven points $\Gamma$ such that $I_{\Gamma}(4,4)$ is generated by these minors.  Notice that a general  set of seven points in a complete intersection $((2,2),(2,2))$ on $Q$ are defined this way.
Furthermore, we compute $\dim_{\CC} J(4,4)=24$, hence $J$ is fully apolar to some $(4,4)$ form $F_Q$. This form $F_Q$ and hence also a corresponding $F$ is apolar to $\Gamma$. 

Assume $F$ is a general quartic apolar to a general set of seven general points $\Gamma$ and consider a smooth quadric $Q$ passing through $\Gamma$. On $Q$ the set $\Gamma$ is contained in two $(2,2)$ forms $q_1$, $q_2$ and a $(3,2)$ form $p_1$. Moreover $\Gamma$, by syzygy computation, can be described as a degeneracy locus of a map $T\oplus T(-1,0)\to 2T(1,1)\oplus T(0,1)$ represented by a matrix $N$. Note that $F^{\bot}$ contains a $(3,3)$ form $p$ that does not pass through any point of $\Gamma$. Since $N$ is general we can write $p_2$ in terms of the three forms $l_1$, $l_2$, $l_3$ of bidegrees $(2,1)$, $(2,1)$, $(1,1)$ forming the second row (i.e. $T(-1,0)\to 2T(1,1)\oplus T(0,1)$) of $M$. Therefore, as in the previous case, we can construct a $5\times 5$ skew symmetric matrix $M$. The first two rows are given by $N$ completed by a zero diagonal $2\times 2$ block, and  $q_1,q_2,p_1,p_2$ appear as four of its pfaffians. The remaining pfaffian $p_3$ of $M$ has syzygies describing 
$l_1p_3$,$l_2p_3$,$l_3p_3$ in terms of $q_1,q_2,p_1,p_2$. In particular, $\langle p_3 \rangle(4,4)\subset \langle q_1,q_2,p_1,p_2 \rangle(4,4)$. Hence $\langle q_1,q_2,p_1,p_2,p_3 \rangle$ is apolar and in consequence fully apolar to $F$.
\vskip .05in
\item[Other cases of size $5\times 5$] The remaining arguments are similar, we just need to describe the set of points $\Gamma$ computing the rank of $F$ in terms of $2\times 2$ minors of a $2\times 3$ matrix and complete the matrix to a skew symmetric matrix using one appropriate additional generator in $F_Q^{\perp}$ and its syzgies as done above. 
\vskip .05in
\item[{\bf Type $[100]$}] The only case when the corresponding matrix is of size $7\times 7$ is type $[100]$. In this case, let $J$ be the ideal generated by the $6\times 6$ Pfaffians of the matrix.  From the resolution we compute $\dim_{\CC} J(4,4)=24$, hence $J$ is fully apolar to some $(4,4)$-form $F_Q$.

Then to prove that a general quartic is obtained in this way  we recall from section \ref{section_VSP} that for a general quartic $F$ from this family and $Q$ its apolar quadric,  we have $\dim_{\CC} F_Q^{\bot}(3,2)=\dim_{\CC} F_Q^{\bot}(2,3)=6$. Furthermore, we can choose in $F_Q^{\bot}(3,2)$ three elements $p_1$, $p_2$, $p_3$ defining a set $\Gamma$ of 9 general points on $Q$ which are apolar to $F$.  These must have a syzygy in $T(4,3)$ and are given as three of the minors of a $3\times 4$ matrix $N$ corresponding to a general map:
$3T(0,-1)\to T\oplus 3T(1,0)$. 

Consider furthermore a general triple $(q_1,q_1,q_3)$ of elements in  $F_Q^{\bot}(2,3)$ which admit a syzygy in $T(3,4)$, such that $p_1,p_2,p_3,q_1,q_1,q_3$ admit a syzygy in $T(3,3)$. 
This can be done for dimension reasons: Given a pair of subschemes $\Gamma,\Gamma'\in VSP(F_Q,9)$ apolar to $F_Q$.  
The subspace of $I_{\Gamma}(3,3)$ generated by  $I_{\Gamma}(3,2)$ and the subspace of $I_{\Gamma'}(3,3)$ 
generated by $I_{\Gamma'}(2,3)$ 
are both $6$-dimensional subspaces of the $12$-dimensional 
space $F_Q^\perp(3,3)$.  
Since $VSP(F,9)$ is $2$-dimensional, there are pairs $\Gamma,\Gamma'$ such that these two subspaces intersect, which means that there is a syzygy in $T(3,3)$ among the generators in bidegree $(3,2)$ and $(2,3)$ respectively for the ideals of  $\Gamma$ and $\Gamma'$.
By possibly replacing $p_1,p_2,p_3$ $q_1,q_1,q_3$ and making corresponding row and column operations on the matrix $N$ we can assume:
\begin{enumerate}
    \item $q_1 x+q_2 y=p_1 z+p_2 t$, where $x,y$ are $(1,0)$-forms and $z,t$ are $(0,1)$-forms.
    \item 
    $N=\left[\begin{array}{cccc}
        x &  s_1 & s_2 &s_3 \\
        y & u_1 & u_2 &u_3 \\
        0 & l_1& l_2 &l_3\\
    \end{array}
    \right]$
    for some $s_i,u_i,l_i \in T(1,1)$.
    \item The $p_i$ are the corresponding $3\times 3$ signed minors of $N$ of bidegree $(3,2)$.
    \item $q_1 s_3+q_2 u_3+q_3 l_3=0$ for  $s_3,u_3, l_3$ from $N$ - indeed by possibly changing $q_1,q_2,q_3$ to their linear combinations we can assume that for the syzygy in $T(3,4)$ that they admit the coefficient at $q_3$ is $l_3$, then by column operations on $N$ involving the first column and row operations on the first two rows we can arrange $s_3$, $u_3$ to be the remaining coefficients in the $T(3,4)$ syzygy between $q_1,q_2,q_3$.
    \item $l_1,l_2,l_3$ are irreducible and form a complete intersection and similarly $s_3,u_3,l_3$ -- this follows from genericity of $F$ and $\Gamma$
\end{enumerate}
Let $W=\left[\begin{array}{ccc}
        s_1 & s_2 &s_3 \\
        u_1 & u_2 &u_3 \\
        l_1& l_2 &l_3\\
        z&t&0
    \end{array}
    \right]$
and $w_i$ be the minors of $W$ obtained by deleting the $i$-th row.    
Under the above assumptions we infer from (1) and descriptions of $p_i$ in terms of minors of $N$ that:
$(q_1-w_1)x+(q_2-w_2)y=0.$
It follows that there exists $k\in T(1,3)$ such that 
$q_1-w_1=yk$ and $q_2-w_2=-xk$.
Now from (4) we have 
$(s_3 y+u_3 x)k+ (q_3-w_3)l_3=0,$
but since $l_3$ is irreducible and has no syzygy in $T(2,1)$ with $s_3$, $u_3$ we infer $k=l_3 h$ for $h\in T(0,2)$. Finally note that by adding combination of the first column of $N$ to the second and third column we can modify $s_1,s_2,u_1,u_2$ without changing $p_i$ but replacing $w_1$, $w_2$ with $ w_1+yl_3(z a+tb)$ and $w_2-yl_3(z a+tb)$ for any chosen $a,b\in T(0,1)$. We conclude that after suitable choice of $ a, b$ for the corresponding modification of $N$ we can assume $h=0$ and hence
$q_1=w_1$, $q_2=w_2$, $q_3=w_3$. If we now consider the skew symmetric matrix 
$$\left[\begin{array}{cccccc}
    0&0&    x&s_1 & s_2 &s_3 \\
    \quad &0&    y&u_1 & u_2 &u_3 \\
    \quad &\quad &0&l_1& l_2 &l_3\\
    \quad &\quad &\quad&     z&t&0\\
    \quad &\quad &\quad&\quad&0&0\\
    \quad &\quad &\quad&\quad&\quad&0\\
    \end{array}\right],$$
    its pfaffian ideal $J$ has the property 
    $\langle p_1,p_2,p_3,q_1,q_2,q_3 \rangle\subset J$
    and 
    $$J(4,4)=\langle p_1,p_2,p_3,q_1,q_2,q_3 \rangle(4,4)=F_Q^{\bot}(4,4)$$ i.e. $J$ is fully apolar to $F$.

Here $x,y,z,t$ are coordinates in $T$ of respective bidegrees $(1,0)$, $(1,0)$, $(0,1)$, $(0,1)$.

\end{description}
\end{proof}

 \begin{remark}
Note that 
if we take
$E=T(2,1)\oplus T(2,1)\oplus T(-1,1)$, then we get a complete intersection ideal $J$ of grade 3 in $T$ with a quartic dual socle generator $F_Q$ which is a restriction of $F$ from the family $[683]$. Moreover, a general $F$ in $[683]$ can be obtained in this way. 
However, in this case $r^{-1}(J)\neq F^{\bot}$ as $\dim_{\CC} J(2,2)=4$

 \end{remark}

\subsection{AG-liftings to curves}
 We describe liftings  to curves in smooth quadric fourfolds of $A_F$ for quaternary quartics $F$ of the types $[100],[200],[300],[400]$.
Let $C\subset Q \subset \Pn^5$ be an Arithmetic Gorenstein scheme of dimension 1 contained in a smooth quadric fourfold $Q$. Recall that $Q\simeq G(2,4)$. Let furthermore $H$ be the class of a hyperplane section on $Q$. 

Suppose the curve $C\subset Q$ is a Pfaffian subvariety on $Q$ defined by the 
submaximal minors of a skew symmetric map between vector bundles $E_C^{\vee}\to E_C$ of odd rank and $det(E_C)=eH$. 
From \cite{Okonek}*{p.~427} the ideal sheaf of $C$ has a resolution 
$$0\to \mathcal{O}(-2eH) \to E_C^{\vee}(-eH)\to E_C(-eH)\to \mathcal{I}_{C|Q} \to 0 ,$$
and the canonical line bundle on $C$ is $K_C=(2e-4)H|_C$.

Consider the exact sequence of ideal sheaves
$$0\to \mathcal{O}(-2) \to \mathcal{I}_C \to \mathcal{I}_{C|Q} \to 0.$$
From the associated long exct sequence in cohomology,  we deduce that $C\subset \mathbb{P}^5$ is ACM iff
$C\subset Q$ is ACM.
For the Pfaffian varieties
constructed from a bundle $E$ we infer the following:
\begin{lemma}\label{ACMvars} If $H^1(E_C(k))=H^2(E_C^{\vee}(k))=0$ for $k\geq 0$ then the Pfaffian varieties on the quadric $Q$ which are associated to $E$ are arithmetically Gorenstein.
\end{lemma}
\begin{proof} By splitting the Pfaffian sequence twisted by $k$ we deduce that the vanishings assumed in the lemma imply $H^1(I_{C|Q}(k))=0$ for $k\geq 0$. 
By the relative ideal sequence this implies that $H^1(I_{C}(k))=0$ for $k\geq 0$ and $X$ is ACM.
\end{proof}

We are now ready to provide constructions of ACM Gorenstein curves with given Betti tables. 
See section \ref{Tom} for the notation  on spinor bundles.
\begin{proposition} \label{pfaffians in G(2,4)} Let be $S_1,S_2$ be spinor bundles on $Q_4$.
Pfaffian varieties associated to general skew symmetric maps $E\to E^{\vee}$ have Betti table: 
\begin{enumerate}
    \item Type $[100]$ for $E=3\mathcal S_1\oplus \mathcal O$, $E=2\mathcal S_1\oplus \mathcal S_2\oplus \mathcal O$
    \item Type $[200]$ for $E=2\mathcal S_1\oplus \mathcal O(1)$, $E=\mathcal S_1\oplus \mathcal S_2\oplus \mathcal O(1)$
    \item Type $[300]$ for $E=\mathcal S_1\oplus 2\mathcal O(1)\oplus \mathcal O$
    \item Type $[400]$ for $E=3 \mathcal{O}(1)$.
\end{enumerate}

\end{proposition}

\begin{proof}
 The following formulas relating the degree of $C$ and the Chern classes of $E_C$ hold:
$$c_1(E_C)=3H \text{ and } \deg C=\deg (3Hc_2(E_C)-c_3(E_C))$$
where $H\in H^2(Q,\mathbb Z)$ is the class of the  hyperplane section of $Q$  and $c_i(E_C)\in H^{2i}(Q, \mathbb Z)$ are the Chern classes of $E$.

We deduce that the associated Pfaffian subvarieties of $Q_4$ are smooth ACM half canonical curves of the correct degree. Furthermore, observe that for a general codimension 2 linear section $L$ of $Q_4$ in each case the map $H^0(Q_4,\wedge^2 E)\to H^0( L,\wedge^2 E|_L)$ is surjective and the Betti table of the section of the half canonical curves by $L$ can be checked via Macaulay2, cf. \cite{M2}*{Package QuaternaryQuartics}.
\end{proof}
\begin{remark}
Note that the bundle $3\mathcal{O}(1)$ defines a family of complete intersections of four quadrics.
Moreover, the varieties in the family $[320]$ are contained in a cubic scroll which is generated by quadrics of rank $4$ independent of their dimension. Hence $[320]$ does not admit Pfaffian descriptions in $Q_4$. 
\end{remark}
\begin{remark}
In the list of ACM bundles above we can exchange $S_1$ with $S_2$ obtaining a complete list of possible ACM bundles on $Q_4$ giving rise to ACM Gorenstein half-canonical curves.
\end{remark}

Some of these constructions admit further liftings to a smooth 6-dimensional quadric $Q_6$. However, we shall need to supplement  our constructions to include  Lagrangian degeneracy loci.

\subsection{AG-liftings to Calabi-Yau threefolds; Lagrangian degeneracy loci }\label{CGKK10}
We construct Calabi-Yau threefolds in a smooth quadric sixfold $Q_6\subset\Pn^7$ of degrees $17, 18$ and $19$ that are liftings of $A_F$ for quaternary quartics $F$ of type $[300],[200],[100]$ respectively.  

For the degree $18$ case, we consider the Pfaffian variety of a general skew symmetric map $E\to E^{\vee}$  where $E=S_1\oplus \mathcal{O}(1)$ and $S_1$ is a spinor bundle (see Section \ref{Tom} for the notation).
Indeed, to lift the Pfaffians associated to $\mathcal S_1\oplus \mathcal S_2\oplus \mathcal{O}(1)$ on $Q_4$ we simply note that the bundle $S_1\oplus S_2$ on $Q_4$ is the restriction of the spinor bundle on $Q_6$.

In order to construct Calabi-Yau threefolds of degrees $17$ and $19$ the situation is more complicated.
  We discuss a Lagrangian construction and treat in some detail the degree $19$ case.  We comment briefly on the degree $17$ case in Remark \ref{liftto17}.

Recall the Lagrangian construction studied in \cites{EPW,EPW1}:  Let $V$ be a $8$-dimensional vector space with a nondegenerate quadratic form $q:V\to\CC$ and let $Q_6=\{q=0\}$ be the $6$-dimensional quadric defined by $q$ in $\Pn(V)$.
Consider a Lagrangian subbundle $\mathcal{E}\subset V\otimes{\mathcal O}_{Q_6}$ with Pfaffian line bundle $L$, so that $L^{\otimes 2}={\rm det}\mathcal{E}$.  Let $W\subset V$ be a Lagrangian subspace such that $W\cap \mathcal{E}(x) $ is odd-dimensional for all $x\in Q_6$. Then, by \cite{EPW1}*{Thm 2.1},  for sufficiently general $W$,  $$X=\{x\in Q_6|\ {\rm dim}(W\cap \mathcal{E}(x))=3\}$$
is a smooth subvariety of $Q_6$ of codimension $3$ 
whose sheaf of ideals 
has symmetric quasi-isomorphic locally free resolutions 
\begin{equation}\label{EPW}
\xymatrix{
0 \ar[r] & L^{\otimes 2}  \ar[r]^{} \ar[d] & \mathcal{E}(L) \ar[r] \ar[d]^{\alpha} & W^{\ast}\otimes \mathcal{O}_{Q_6}(L) \ar[r] \ar[d]& \ar[r] \mathcal{I}_{X|{Q_6}} & 0\\
0 \ar[r] & L^{\otimes 2}  \ar[r]^{} & W \otimes \mathcal{O}_{Q_6}(L) \ar[r]^{\beta}& {\mathcal{E}}^{\ast}(L)    \ar[r] & \ar[r] \mathcal{I}_{X|{Q_6}} & 0
}.
\end{equation}
This is a commutative diagram such that $\alpha\beta$ is a skew symmetric map.
Moreover, $L^{\otimes 2}=det(\mathcal{E})$
such that $\omega_X=(\omega_{Q_6}\otimes det(\mathcal{E}^{\ast}))|_{X}$.

The class of $X$ in ${Q_6}$ is computed by the Pragacz-Fulton formula \cite{FP}
as $\frac{1}{4}(c_1(\mathcal{E}^{\ast})c_2(\mathcal{E}^{\ast})-2c_3(\mathcal{E}^{\ast}))$.
In particular, if $\mathcal{E}$ is a spinor bundle, $\mathcal{S}_1$ or $\mathcal{S}_2$, on $Q_6$ the subvariety $X$ will be a $\mathbb{P}^3\subset Q_6$ that do not have Pfaffian resolution (see \cite{EPW1}*{\S 4}).
 
 We have the following interpretation of the Tom
 and Jerry construction of \cite{Reid}.
\begin{proposition}\label{TomandJerry} The del Pezzo Fano threefolds
$$\mathbb{P}^1\times \mathbb{P}^1\times \mathbb{P}^1 \quad (Jerry)\quad{\rm and}\quad \mathbb{P}^2\times \mathbb{P}^2\cap \mathbb{P}^7\quad (Tom)\quad \subset Q_6$$ are  subcanonical subvarieties of codimension $3$ defined by the Lagrangian construction with 
the bundle $\mathcal{E}=2\mathcal{S}_1$ for Jerry and the bundle $\mathcal{E}=\mathcal{S}_1 \oplus \mathcal{S}_2$ for Tom, where $\mathcal{S}_1,\mathcal{S}_2$ are the spinor bundles on $Q_6$.  
\end{proposition} 
\begin{proof}
Observe that $\mathcal{S}_1\oplus \mathcal{S}_2$ is the restriction of the spinor bundle on $Q_7$. The Lagrangian construction with that bundle defines on $Q_7$ a del Pezzo fourfold of degree $6$ so the Segre embedding $\mathbb{P}^2\times \mathbb{P}^2\subset Q_7$ from the classification of del Pezzo fourfolds.

It is enough to show the other example is different. We compute the restrictions of those degeneracy
 loci to the two families of isotropic $\mathbb{P}^3\subset Q_6$.
 From \cite{Ot}*{Thm.\ 2.6} we infer the restriction of $\mathcal{S}_1$ to one of those $\mathbb{P}^3$ is $\Omega^2(2)\oplus \mathcal{O}$ and the other $\Omega^2(2)\oplus \Omega^3(3)$ and $\mathcal{S}_2$ in the opposite way.
 Using again the Pragacz-Fulton formula we infer $\mathcal{S}_1 \oplus \mathcal{S}_2$ gives a threefold of bidegree $(3,3)$ and $2\mathcal{S}_1$ of bidegree $(2,4)$.
 
 In order to conclude we observe that a hyperplane section of $\mathbb{P}^2\times \mathbb{P}^2$ contained in a smooth quadric $Q_6$ lifts to $\mathbb{P}^2\times \mathbb{P}^2$ contained in a smooth quadric $Q_7$.
 Since in $Q_7$ there is one family of $\mathbb{P}^3$ we deduce that Tom have always bidegree $(3,3)$ in a smooth quadric. Therefore $2\mathcal{S}_1$ defines Jerry.
 \end{proof}
By a spinor analogue of bilinkage we construct Calabi-Yau threefolds in the smooth $6$-dimensional quadric $Q_6$. 
\begin{proposition}\label{degree 19 Lagrangian}
Let $X\subset Q_6$ be the  subvariety of codimension $3$ defined by the Lagrangian construction with 
 $$\mathcal{E}=3\mathcal{S}_1\quad {\rm or}\quad \mathcal{E}= \mathcal{S}_1\oplus 2\mathcal{S}_2$$ in diagram {\rm (\ref{EPW})}, where $\mathcal{S}_1$ and $\mathcal{S}_2$ are spinor bundles on $Q_6$, then 
 $X$ is a Calabi-Yau threefold of degree $19$ that is  Arithmetically Gorenstein (AG) in $\mathbb{P}^7$.  Furthermore, the two choices of bundle $\mathcal{E}$ give distinct families of polarized Calabi-Yau threefolds. \end{proposition}
\begin{proof}
The canonical divisor is computed from the fact that $det(3\mathcal{S}_i)=6h$.
In order to find the degree we use the Fulton-Pragacz formula
knowing that $c_1(\mathcal{S}_i)=2h$, $c_2(\mathcal{S}_i)=2h^2$
and $c_3(\mathcal{S}_i)=(2,0)$.

To prove it is AG it is enough to use the long cohomology exact sequence corresponding to the first sequence in diagram
\ref{EPW}. The cohomologies of the twists of $\mathcal{S}_i$ are
described in \cite{Ot}*{Thm.~4.3}.

In order to show that the two bundles above gives different polarised Calabi-Yau threefolds we compute as in the proof of Lemma \ref{TomandJerry}  that they have bidegrees $(9,10)$ and $(7,12)$ in the class of cycle of codimension $3$.  
\end{proof}
\begin{remark}
As in the Tom and Jerry case we see that a general AG surface of degree $19$ in $\Pn^6$ with Betti table of type $[100]$ lifts to distinct families of AG Calabi-Yau threefolds in $\Pn^7$ explicitly described on smooth quadrics. 
\end{remark}
 Notice that Proposition \ref{degree 19 Lagrangian-intro} follows immediately from Proposition \ref{degree 19 Lagrangian}.

We expect that the Pfaffian Calabi-Yau threefolds of degree $19$ defined above specialize to families constructed by unprojection in \cite{CGKK}.
Recall that in \cite{CGKK} such examples were constructed in quadrics of rank $6$.

\begin{remark}\label{LagrangianPfaffian degree 19}

Let $X$ be a Calabi-Yau threefold in $Q_6$ constructed in Proposition \ref{degree 19 Lagrangian}. Let $Y$ be a general codimension 2 linear section of $X$ contained in a quadric $Q_4$. Then $Y$ is described by a Pfaffian resolution related to any of the bundles in item (1) of Proposition \ref{pfaffians in G(2,4)}.   Since both $S_1$ and $S_2$ on $Q_6$ restricts to $\mathcal{S}_1\oplus \mathcal{S}_2$ on $Q_4$, we consider the restriction of the diagram of sheaves on $Q_6$ to $Q_4$, after inserting an additional component $\mathcal O$ :

\begin{equation}\label{diagram degree 19 G(2,4)}
\xymatrix{
0 \ar[r] & \mathcal{O}(-6)  \ar[r]^{} \ar[d] & (3\mathcal{S}_1\oplus 3\mathcal{S}_2\oplus \mathcal O) (-3) \ar[r] \ar[d] & 13 \mathcal{O}(-3) \ar[r] \ar[d]& \ar[r] \mathcal{I}_{Y|Q_4} & 0\\
0 \ar[r] & \mathcal{O}(-6)  \ar[r]^{} & 13 \mathcal{O}(-3) \ar[r]^{\beta}& (3{\mathcal{S}_1}^{\ast}\oplus 3{\mathcal{S}_2}^{\ast}\oplus \mathcal O)(-3)    \ar[r] & \ar[r] \mathcal{I}_{Y|Q_4} & 0
},
\end{equation}
The added component allows for a nonminimal Lagrangian resolution, so that the composition $$3\mathcal{S}_1\oplus 3\mathcal{S}_2\oplus \mathcal O\to 3{\mathcal{S}_1}^{\ast}\oplus 3{\mathcal{S}_2}^{\ast}\oplus \mathcal O$$
becomes a skew symmetric map between odd-dimensional vector bundles.

Note that under the generality assumption the map $(3\mathcal{S}_1\oplus 3\mathcal{S}_2\oplus \mathcal O) (-3)\to 13 \mathcal{O}(-3)$ induces an embedding $3\mathcal{S}_1(-3)\to 13 \mathcal O(-3)$. Indeed, note that since $\mathcal S_1^*$ has 4 sections the image of $3\mathcal{S}_1(-3)$ in $26\mathcal{O}(-3)$ is contained in some non-isotropic subbundle $12\mathcal O(-3)$. By Kleiman transversality, a chosen general isotropic subspace of dimension 3 will hence avoid the image of $3\mathcal{S}_1(-3)$ and then the composition  $3\mathcal{S}_1(-3)\to 13 \mathcal O(-3)$ will be an embedding.

But then, the resolutions of $I_Y$ from the diagram are not minimal. Indeed, whenever we have a resolution 
$$0\to A\to B\oplus C  \to D\to I\to 0,$$
such that the map 
$B\to D$ is an embedding with cokernel $E$, i.e. inducing an exact sequence
$$0\to B\to D\to E\to 0,$$
then by snake lemma we also have a resolution
$$0\to A\to  C  \to E\to I\to 0,$$

We apply this argument, using the universal sequence 
$$0\to S_2\to 4\mathcal O\to S_1^*\to 0$$
on $Q_4=G(2,4)$.  
The diagram \ref{diagram degree 19 G(2,4)} after simplification becomes:
\begin{equation}
\xymatrix{
0 \ar[r] & \mathcal{O}(-6)  \ar[r]^{} \ar[d] & (3\mathcal{S}_1\oplus \mathcal O) (-3) \ar[r] \ar[d] & (3\mathcal S_1^{*}\oplus \mathcal{O})(-3) \ar[r] \ar[d]& \ar[r] \mathcal{I}_{Y|Q} & 0\\
0 \ar[r] & \mathcal{O}(-6)  \ar[r]^{} & (3 S_1\oplus \mathcal{O})(-3) \ar[r]^{}& (3{\mathcal{S}_1}^{\ast}\oplus \mathcal O)(-3)    \ar[r] & \ar[r] \mathcal{I}_{Y|Q} & 0
},
\end{equation}
which means that $\mathcal{I}_{Y|Q}$ admits a Pfaffian resolution.
\end{remark}

\begin{remark}\label{liftto17}
We can also construct the AG Calabi--Yau threefolds of Type $[300a]$ of degree $17$ in a smooth quadric. Indeed, a Calabi-Yau threefold of Type [300a] can be constructed by bilinkage of $\mathbb{P}^3$ in a complete intersection of three quadrics (cf. \cite{CGKK} and Section \ref{section_irreducible_liftings}). The resolution of this bilinkage on a fixed quadric can be found tracing the resolution through the bilinkage construction by means of the mapping cone (see \cite{PS74}) starting from diagram (\ref{EPW}) with $\mathcal{E}=\mathcal{S}_1$. We obtain a resolution of the shape 
\begin{equation}\label{EPW2}
\xymatrix{
0 \ar[r] & \mathcal{O}(-6) \ar[r]^{} \ar[d] & \mathcal{S}_1(-3)\oplus 2\mathcal{O}(-4) \ar[r] \ar[d] & 4\mathcal{O}(-3)\oplus 2\mathcal{O}(-2) \ar[r] \ar[d]& \ar[r] \mathcal{I}_{X|Q} & 0\\
0 \ar[r] & \mathcal{O}(-6) \ar[r] & 4\mathcal{O}(-3)\oplus 2\mathcal{O}(-4)   \ar[r]  &   \mathcal{S}_1^{\ast}(-3)\oplus 2\mathcal{O}(-2)        \ar[r] & \ar[r] \mathcal{I}_{X|Q} & 0
},
\end{equation}
where $\mathcal{S}_1\oplus 2\mathcal{O}(-1)\subset 8\mathcal{O}\oplus 2\mathcal{O}(-1)\oplus 2\mathcal{O}(1)$
is embedded as a Lagrangian subbundle of a orthogonal bundle (with natural orthogonal structure).
\end{remark}

 \subsection{Degenerations of liftings}
We considered in the previous sections stratification
of the space of quartics and the associated Artinian rings. It is natural to expect we obtain degenerations of liftings of such Artinian rings. Let us discuss some examples.
 \begin{proposition}
 The complete intersection of four quadrics in $\mathbb{P}^7$ degenerates to the example with resolution of type $[420]$. 
 \end{proposition}
 \begin{proof}
 First $[400]$ is defined by the Pfaffians of a skewsymmetric map with the bundle $3\mathcal{O}(1)$ on $Q_6$.
 But we can consider also the bundle $3\mathcal{O}(1)\oplus 2\mathcal{O}$ on $Q_6$. It is enough to observe that (2.4) is defined
 by a special skew map for $3\mathcal O(-1)\oplus 2\mathcal O\to 3\mathcal{O}(1)\oplus 2\mathcal{O}$ such that 
 the map $2\mathcal{O} \to 2\mathcal{O}$ is the zero map.
 \end{proof}
 One could expect that degenerations of Artinian rings could 
 lead to natural degenerations of Calabi-Yau threefolds.
 However the situation is more complicated.
Recall that Calabi-Yau threefolds with different Hodge number cannot be deformation equivalent \cite{Voisinhodge}*{Prop. 9.20}. We can however prove the following:
 \begin{proposition}
 Artinian rings with Betti table $[300a]$ degenerate to general rings with Betti table $[320]$. The corresponding Calabi-Yau threefolds do not degenerate, but have a common degeneration to a singular Calabi-Yau threefold  $[320]$.
  \end{proposition}
  \begin{proof}
  Recall that there can be  is no smooth family containing elements of both families since these have different Hodge numbers. However, the common degenerations to singular Calabi-Yau threefolds can be described as follows.
  The intersection of three quadrics degenerates to  a degree $8$ fourfold being a divisor in the cubic scroll (being a cone over $\mathbb{P}^1\times \mathbb{P}^2$).
  Consider the Calabi-Yau fourfold is bilinked to a 
  $\mathbb{P}^3$ contained in the fourfold of degree $8$. In fact we can choose the $\mathbb{P}^3$ above in two ways in the rulings of the cubic scroll obtaining two families of degenerated Calabi-Yau threefolds of degree $17$, being degenerations of elements in both families.
  \end{proof}

\section{Irreducible liftings} 
\label{section_irreducible_liftings}
In this section and the next we investigate liftings of the Artinian Gorenstein rings $A_F$ of quaternary quartics. In the previous section, some such liftings were found via Pfaffian and Lagrangian constructions on smooth quadrics.
In this section we treat systematically liftings of rings $A_F$ with Betti table in \ref{Appendix}, Table~\ref{tableCGKK}, leaving liftings of rings $A_F$ with Betti table in \ref{Appendix}, Table~\ref{tableremaining} to the next section.

In Corollary \ref{liftofCGKK} we shall see that a generic Artinian Gorenstein ring with a given Betti table considered in this section lifts to a curve; this lifting is obtained by the doubling construction. 
Next we attempt to give an explicit geometric description of a general element of each family with given Betti table. We are especially interested in a description of the lifting as divisors in some lift of the configuration of points computing the rank. Some of the cases follow from the classification contained in \cite{CGKK}, but we repeat the argument here for completeness.
In many cases the construction involves bilinkage, cf.\cites{PS74,Har}, as used in \cite{CGKK}:  Let the variety $Z\subset\Pn^n$ be arithmetic Cohen-Macaulay, let  $X\subset Z$ be a codimension $1$ arithmetic Gorenstein variety and $H$ a hyperplane section of $Z$.  Then any generalised \cite{Har} divisor $Y\in |X+aH|,\; a\in \ZZ$ is arithmetic Gorenstein and bilinked to $X$ in $Z$.

\begin{remark} \label{lem linkage} 
Bilinkage in complete intersections give uniqueness of lifting as follows. 
Assume  $X,Y,Z$ are as above and $Z$ is a complete intersection.
Assume that there exists  a smooth variety $\mathcal X$ such that  $X=\mathcal X\cap \mathcal H$, where $\mathcal H$ is a hyperplane. Then  there exists a complete intersection $\mathcal Z$ such that $Z=\mathcal Z\cap \mathcal H$,  
and a variety $\mathcal Y$ bilinked to $\mathcal X$  such that $Y=\mathcal Y\cap H$
and $\mathcal X=\mathcal Y-aH_{\mathcal Z}$.  
Indeed,  the existence and shape of bilinkage in complete intersections  follows from the Betti tables which do not change via liftings. This combined with the fact that the restriction of a bilinkage is a bilinkage proves our claim.
\end{remark}
\noindent We now describe the liftings of the quartics $F$ with Betti tables from Table \ref{tableCGKK} in Appendix \ref{Appendix}.
\begin{description}
    \item [{Type $[683]$}] From the Betti table we deduce that the space of quadrics in $F^{\perp}$ defines a finite scheme that is a proper linear section of a variety $Z$ of minimal degree in codimension $3$, i.e. a rational scroll of degree $4$, if dim$(Z)>2$.  A codimension $1$ AG variety $X$ on $Z$ that lifts $A_F$ is a divisor of degree $14$ whose canonical divisor is a multiple of the hyperplane section.  If dim$(Z)=c$, then the canonical divisor on $Z$ is
$-cH+2F$ where $H$ is a hyperplane section and $F$ is the class of a ruling in the scroll. So by adjunction, $X$ is a Calabi-Yau threefold if $c=4$ and $X$ is a divisor  of type $4H-2F$ on $Z$.  It is bilinked on $Z$ to a quadric threefold, a divisor of type $H-2F$.  A quadric threefold is arithmetic Gorenstein, so therefore also, by bilinkage, the threefold $X$ is.
\vskip .1in
\item [{Type $[550]$}] From the Betti table we deduce that the space of quadrics in $F^{\perp}$ defines a finite AG scheme which is a proper linear section of $G(2,5)$. The ideal $F^{\perp}$ is generated by the quadrics and an additional cubic form, so the $A_F$ lifts to the intersection of a cone over $G(2,5)$ with a cubic hypersurface. 
\vskip .1in
\item [{Type $[400]$}] From the Betti table we deduce that $F^\perp$ is a complete intersection of quadrics and hence lifts to any dimension as a smooth intersection of four quadrics.
\vskip .1in
\item [{Type $[300]$}] There are three types of arithmetic Gorenstein varieties $X$ of degree $17$ in a complete intersection of three quadrics with this Betti table.  They are distinguished by the rank of the apolar quartic of a generic Artinian reduction $A_F$. 
\vskip .1in
\item [{Type $[300a]$}]  In the case where the quartic $F$ is of rank $8$, the ideal $F^{\perp}$ is bilinked in a complete intersection $Z$ of three quadrics to an ideal of colength $1$.  This ideal lifts to a linear space in any dimension. The bilinkage to $P+2H$ of the linear space $P$ in a complete intersection $Z$ of three quadrics provides liftings of general schemes of type $A_F$ to AG-varieties. In fact, this family of liftings is unique by Remark \ref {lem linkage}.
\vskip .1in
\item [{Type $[300b]$}] When $Q_F$ is a complete intersection $F$ has rank $7$ the ideal $F^{\perp}$ contains the ideal of seven among the eight points defined by the complete intersection $Q_F$ of three quadrics $F^{\perp}$. These seven points can be lifted to a variety of degree $7$ obtained as the residual component $Y$ of the intersection of quadrics containing a linear space $P$ of codimension $3$. The intersection $X_0=Y\cap P$ is a cubic hypersurface in $P$ of codimension 4. A lifting of $F$ may be found as a variety $X\in|X_0+2H|_Y|$ obtained by bilinkage of $X_0$ in $Y$.

A general $F$ lifts smoothly to a surface $X$ by this construction.  The fourfold variety $Y$ in $\Pn^7$ has isolated singularities, among them six on $X_0$ and  a threefold lifting of $F$ has isolated singularities.
\vskip .1in
\item [{Type $[300c]$}] When the quadric ideal $Q_F$ in $F^\perp$ defines a line and four points, then a lifting of $A_F$ may be found as a reducible arithmetic Gorenstein threefold  $X$ of degree $17$  in $\mathbb P^7$. The threefold $X$ lies in three quadrics, a lifting of $Q_F$,  that define a $\Pn^5$ and a rational fourfold scroll $Y_4$ of degree $4$, and have Betti table of type $[300]$. 
A lifting of a set $\Gamma$ of seven points apolar to $F$, to a fourfold,  is the union of $Y_4$ and a cubic fourfold $Y_3$ in $\Pn^5$.  The fourfolds $Y_4$ and $Y_3$ intersect in a threefold that spans $\Pn^5$, so a Segre cubic threefold scroll $X_3$. A reducible threefold $X$ of degree $17$ may now be constructed by a doubling of $Y_4\cup Y_3$ as in Example \ref{Eg_300c}:
Consider $Y_{16}$ a general complete intersection $(2,2,4)$ that contains $Y_3\cup Y_4.$ Then $Y_{16}$ is Gorenstein, with trivial canonical bundle, and $X=(Y_3\cup Y_4)\cap Y_9$, where $Y_3\cup Y_4\cup Y_9=Y_{16}$. 
In this case $X=X_6\cup X_{11}$, where $X_6=X\cap Y_3$ is a determinantal threefold of degree $6$, linked $(3,3)$ to the cubic scroll $X_3$ in $\Pn^5$, while $X_{11}=X\cap Y_4$ is a threefold of degree $11$, residual to a plane in the intersection of $Y_4$ with a cubic hypersurface. 

A particular phenomenon appears in $X$:
The intersection $X_6\cap X_{11}$ is an anticanonical divisor on both $X_6$ and on $X_{11}$.  A general anticanonical divisor on $X_6$ is a complete intersection $K3$ surface of degree $8$.  The intersection $X_6\cap X_{11}$, however, is not. It coincides with the $X_6\cap X_3$ and is an anticanoncial divisor also on $X_3$ ; it is a $K3$ surface of degree $8$ that is not a complete intersection. 

Note that $[300c]$ cannot be constructed by bilaison of height 2 from a (possibly reducible) cubic threefold on $Y_4\cup Y_3$ as in $[300b]$. Indeed, such a biliaison would give an effective divisor in $|-K_{Y_3}-2H|$, but this system is empty. However, we can still find an effective divisor in $|-K_{Y_3}-H|$ (for example using the construction in Example \ref{doubling and component of complete intersection}) and perform bilinkage from a degree 10 threefold being an effective divisor in that system.
\vskip .1in
\item [{Type $[320]$}] When $F$ is a general form of type $[320]$, then it is apolar to a pencil of sets $\Gamma$ of seven points on a twisted cubic curve $C$.  On one hand, this means that $F^\perp$ is the sum of ideals of two such sets $\Gamma$.  
On the other hand, we show that $F^\perp$ is also generated by the $2\times 2$ minors of a $3\times 3$ matrix, with two rows of linear forms and one row of quadratic forms. 

Indeed, we may assume the twisted cubic curve $C$ is smooth and that its ideal $I_C$ is generated by the $2\times 2$ minors of the top $2\times 3$ submatrix of the $3\times 3$ matrix $$M=\begin{pmatrix}x_0&x_1&x_2\\ x_1&x_2&x_3\\ q_1&q_2&q_3\end{pmatrix}, $$ where the $q_i$ are quadratic forms. Let $p_1=\VV(x_0,x_1,x_2)$, $p_2=\VV(x_1,x_2,x_3)$, and let $I_{ij}$  be the ideal generated by the $2\times 2$ minors of the submatrix consisting of the columns $i$ and $j$ in $M$.  Consider the two ideals $J_{12}=(I_C+I_{12}):I_{p_1}$ and $J_{23}=(I_C+I_{23}):I_{p_2}$. If none of the quadrics $q_1,q_2,q_3$ vanish in $p_1$ or $p_2$, then both $J_{12}$ and $J_{23}$ define a scheme of length $7$ on $C$, and if these schemes are disjoint, the ideal $J_{12}+J_{23}$ coincides with the ideal $I_M$ generated by the $2\times 2$ minors of $M$. Varying  $q_1$ and $q_2$, the ideal $J_{12}$ moves in a $7$-dimensional family, while fixing $q_2$ and  varying $q_3$ the ideal $J_{23}$ moves in a $6$-dimensional family.  So if $F$ is a general quartic apolar to $C$, then we may choose the $q_i$ such that $\VV(J_{13})$ and $\VV(J_{23})$ generate the pencil of schemes of length $7$  on $C$ that are apolar to $F$.  But by the first characterization,  $J_{12}+J_{23}=F^\perp$ when $F$ is of type $[320]$, on the other hand $J_{12}+J_{23}=I_M$, so $F^\perp=I_M$.  

In the first characterization, the ideal $F^{\perp}$ lifts  to a complete intersection of divisors of type $(3,2),(1,3)$ in  $\mathbb{P}^1\times \mathbb{P}^2$ in $\Pn^5$. 
Each of the divisors $(3,2)$ and $(1,3)$ are surfaces $S$ of degree $7$.  For a general $F$ of type $[320]$, there is a surface $S$ of type $(3,2)$, and a surface $S'$ of type $(1,3)$ such that $A_F$ is the Artinian reduction of a curve in the system $|2H-K_{S}|$ respectively $|2H-K_{S'}|$. 

In the second characterization, the ideal $F^\perp$ clearly lifts to the ideal generated by the $2\times 2$ minors of a $3\times 3$ matrix with two rows of linear forms and one row of quadratic forms in $\Pn^7$  (cf. \cite{CGKK}*{\S 4.5}) .   
\vskip .1in
\item [{Type $[200]$}] The ideal $F^{\perp}$ is bilinked in a complete intersection of type $(2,2,3)$ to an Artinian Gorenstein ideal of length $6$, that lifts to any del Pezzo variety $Y$ of degree $6$. The higher dimensional del Pezzo variety of degree $6$ split in two kinds: either to $\mathbb{P}^1\times \mathbb{P}^1\times \mathbb{P}^1\subset \mathbb P^7$ (called Jerry) or to $\mathbb{P}^2\times \mathbb P^2\subset \mathbb P^8$ (called Tom) each having the same Betti table. By projective normality the bilinkage extends to these dimensions producing a lifting $X$ of $F$.
\begin{proposition} A general Artinian Gorenstein ring $A_F$ for a quaternary quartic $F$ of rank $8$ lifts to a halfcanonical curve $C_F$ in $\Pn^5$ such that $C_F$ is an element of the system $|2H-K_X|$ for a surface $X$ of degree 8 in $\mathbb{P}^5$.
\end{proposition}
\begin{proof} By assumption $F^\perp$ contains the ideal $I_\Gamma$ of a scheme $\Gamma$ consisting of eight general points lying on a pencil of quadric surfaces. Let us choose one of these surfaces and denote it by $Q$. Then the ideal of $\Gamma$ is generated by five bihomogeneous forms of bidegrees $(2,2),(3,2),(3,2),(2,3),(2,3)$ admitting four syzygies in degree $(3,3)$. Hence $\Gamma$  arises as the corank 1 locus of a map $4\mathcal O_Q\to 2\mathcal O_Q(1,0)\oplus 2\mathcal O_Q(0,1)\oplus \mathcal O_Q(1,1)$. The latter extends to a map $4\mathcal O_G\to \mathcal S\oplus \mathcal O_G(1)$ degenerating over a surface $Y_2$ of degree 8 on the Grassmannian $G=G(2,4)$ in its Pl\"ucker embedding equipped with the universal quotient bundle $\mathcal Q$. 
The surface $Y_2$ is described in \cite{Gross}*{Thm. 4.1 (g)} as $\mathbb{P}^2$ blown up in ten points embedded via the system of proper transforms of sextic curves passing through six of the points with multiplicity $2$ and the four remaining ones with multiplicity $1$. 

Consider a curve $C$ in the system $|2H-K_{Y_2}|$ with $H$ the class of the hyperplane section on $Y_2$. Note that $C$ is then an arithmetic Gorenstein curve of degree $18$. The resolution of its ideal is a doubling of the resolution of the ideal of $Y_2$ and the Artinian reductions of $C$ are of type $[200]$. To prove that $A_F$ is an Artinian reduction of $C$ it is enough to prove that the restriction map:
$$H^0({\mathcal O_{Y_2}(2H-K_{Y_2})})\to H^0( {\mathcal O_{\Gamma}(2H-K_{Y_2}) })$$
is surjective. 
We proceed step by step. Let $Y_1$ be a hyperplane section of $Y_2$ containing $\Gamma$. Then $Y_1$ is a curve of genus $4$ and the elements of the system $|2H-K_{Y_1}|$ are of degree $10$ hence $H^1(\mathcal O(2H-K_{Y_1}))=H^0(2K_{Y_1}-2H)=0$. This together with the sequence
$$0\to {\mathcal O_{Y_1}(2H-K_{Y_1})}\to {\mathcal O_{Y_1}(2H-K_{Y_2})}\to {\mathcal O_{\Gamma}(2H-K_{Y_1}) }\to, 0$$
yields surjectivity of the restriction 
$$H^0({\mathcal O_{Y_1}(2H-K_{Y_2})})\to H^0( {\mathcal O_{\Gamma}(2H-K_{Y_2}) }).$$
The remaining step is the surjectivity of the map 
$$H^0({\mathcal O_{Y_2}(2H-K_{Y_2})})\to H^0( {\mathcal O_{Y_1}(2H-K_{Y_2}) }),$$
which is implied by the vanishing of $H^1({\mathcal O_{Y_2}(H-K_{Y_2})})$. The latter system is the system of strict transforms of curve of  degree $9$ with six triple points and four double points in the ten blown up points of $\mathbb{P}^2$. This has vanishing first cohomology for a general configuration of points, which concludes the proof.
\end{proof}
 \vskip .1in
\item [{Type $[100]$}]  Two families of Calabi-Yau threefold $X$ of degree $19$ with this Betti table are constructed as a Lagrangian degeneracy locus of a map between two Lagrangian subbundles of an orthogonal bundle on a smooth quadric of dimension 6 in Proposition \ref{degree 19 Lagrangian} in Section \ref{CGKK10}.
\begin{proposition} A general Artinian Gorenstein ring $A_F$ for a quaternary quartic $F$ of rank $9$ lifts to a halfcanonical curve $C_F$ in $\Pn^5$, a deformation of any smooth $\Pn^5$ section of an AG Calabi-Yau threefold $Y$ of degree $19$ in $\Pn^7$. Moreover, $C_F$ is an element of the system $|2H-K_X|$ for a surface $X$ of degree 9 in $\mathbb{P}^5$.
\end{proposition}
\begin{proof}
 To show that a general $A_F$ has a lifting to a halfcanonical curve in $\Pn^5$, we go step by step.  
 By assumption $F^\perp$ contains the ideal $I_\Gamma$ of a scheme $\Gamma$ consisting of nine general points lying on a quadric surface $Q$. The space of $(3,2)$ forms as well as that of $(2,3)$ forms on $Q$ that vanish on $\Gamma$ are $3$-dimensional, whereas the space of $(3,3)$ forms vanishing on $\Gamma$ is 7 dimensional. It follows that there are five syzygies in degree $(3,3)$ between six generators of degree $(3,2)$ and $(2,3)$. We conclude that $\Gamma$ is the corank 1 locus of a map $5\mathcal O_Q\to 3\mathcal O_Q(0,1)\oplus 3 \mathcal O_Q(1,0)$. Now $Q$ can be seen as a codimension $2$ linear section of a quadric fourfold. The latter can be interpreted as a Grassmannian $G=G(2,4)$ in its Pl\"ucker embedding. Clearly the map $5\mathcal O_Q\to 3\mathcal O_Q(0,1)\oplus 3 \mathcal O_Q(1,0)$ extends to a map $5\mathcal O_G\to 3 \mathcal Q_G$. The corank 1 locus of the latter map is a surface of degree 9 that was described by Gross in \cite{Gross}*{Thm. 4.2} and provides a lifting of $\Gamma$.  
 
Denote the surface by $Y_2$. 
It is rational, the embedding of $\Pn^2$ blown up in a set of ten points $\Delta\subset \Pn^2$ embedded by the linear system of strict transforms of plane curves of degree $7$ with double points at $\Delta$.  ( loc.cit.).  
A halfcanonical curve $C$ on $Y_2$ in the system $|2H-K_{Y_2}|$ is the strict transform of a curve of degree $17$ with multiplicity $5$ along $\Delta$. 

First, on a hyperplane section $Y_1$ of $Y_2$ there exists $C_1\in |3H-K_{Y_1}|$ which is a lifting of $A_F$ if on $Y_1$ the global sections of the line bundle ${\mathcal O_{Y_1}(3H-K_{Y_1})}$ map surjectively to ${\mathcal O_{Y_0}(|3H-K_{Y_1})}$, where $Y_0$ is a $\Pn^3$ section of $Y$:  
The global sections $H^0({\mathcal O_{Y_0}(3H-K_{Y_1})})$ is naturally identified with the space of quaternary forms apolar to $Y_0$.
Consider now,  the cohomology of the sequence
$$0\to {\mathcal O}_{Y_1}(2H-K_{Y_1})\to {\mathcal O_{Y_1}(3H-K_{Y_1})}\to {\mathcal O_{Y_0}(3H-K_{Y_1}) }\to 0.$$
 The curve $Y_1$ has genus $5$ and the divisor $2H-K_{Y_1}$ on $Y_1$ has degree $10$, so is nonspecial.  Therefore the above restriction map on global sections is surjective and $A_F$ lifts to points on $Y_1$.

Similarly the cohomology of the  exact sequence
$$0\to {\mathcal O}_{Y_2}(H-K_{Y_2})\to {\mathcal O_{Y_2}(2H-K_{Y_2})}\to {\mathcal O_{Y_1}(2H-K_{Y_2}) }\to 0.$$
gives a lifting of $C_1$ to $C_2\in |2H-K_{Y_2}| $.  In fact $H-K_{Y_2}$ is the class of the strict transform of a plane curve of degree $10$ with multiplicity $3$ along $\Delta$.   For general $\Delta$ the cohomology $h^1({\mathcal O}_{Y_2}(H-K_{Y_2}))=0$.  So again the lifting follows.
\end{proof}
\vskip .1in
\item [{Type $[000]$}] There are several families of Calabi-Yau threefolds of degreee $20$ whose general Artinian reduction are quartic forms $F$ of rank $10$. 
We first describe briefly a well known complete family of smooth Calabi-Yau threefolds of degree  20, and then construct another family of nodal non-smoothable Calabi-Yau threefolds of degree 20.  The latter are rationally fibred in elliptic curves over $\Pn^2$.

For the first family, consider the 
 $3\times 3$ minors of a general $4\times 4$ matrix of linear forms in  $\mathbb P^7$.
 These minors define a Calabi-Yau threefold of degree $20$.  We call these determinantal threefolds.  They form a complete family \cite{CGKK}.  But their Artinian reductions are Artinian Gorenstein  rings of regularity $4$ that do not form a complete family.  The corresponding family of quartic forms has codimension $1$ in the space of quartic forms. 
 
The other family consists of threefolds obtained as anticanonical divisors in fourfolds of degree $10$.

The $3\times 3$ minors of a general $3\times 5$ matrix $A$ of linear forms in $\mathbb P^7$ generate the ideal of a fourfold $Y$ of degree 10. 
This fourfold is a $\Pn^2$-scroll over $\Pn^2$.  An effective anticanonical divisor $X$ on $Y$ would be a Calabi-Yau threefold of degree $20$ with Betti table of type $[000]$.  
The class of the anticanonical divisor is $3H-2F$ where $H$ is a hyperplane section and $F$ is the pullback of a line from $\Pn^2$.  So any anticanonical divisor would be residual in a cubic hypersurface section to a scroll in $Y$, the pullback in $Y$ of a conic in $\Pn^2$.
If the conic degenerates to two lines, the scroll is the union of two scrolls over lines in $Y$.  Each such scroll over a line is the rank $1$ locus of a pencil of rows in $A$. 
It is now an easy {\tt Macaulay2}-computation, cf. \cite{M2}*{Package QuaternaryQuartics}, to check that a scroll over a conic is not in any cubic hypersurface section.  In fact the space of cubic forms in the ideal of such a scroll is exactly $10$-dimensional, as expected by Riemann Roch. So these cubic forms coincide with the $3\times 3$ minors of $A$. Therefore, the anticanonical divisor of $Y$ is not effective.

If, however $A$ has a rank $1$ point $p$ in $\Pn^7$, then $Y$ is singular, and any $\Pn^2$-scroll over a conic would have a non-normal double point at $p$.   By Riemann-Roch it is then expected to lie in $11$ cubic hypersurfaces in $\Pn^7$, and hence $Y$ has an anticanonical divisor in this case.
In computations with {\tt Macaulay2} (\cite{M2}*{Package QuaternaryQuartics} ) we find $Y$ with a single rank $1$ point, and a unique anticanonical divisor $X$  with two quadratic singularities, one at $p$ and one at some other point $q$, a smooth point on $Y$.  A general $\Pn^2$-scroll $Z$ over a line in $Y$ passes through  $p$.  
The surface of intersection $Z\cap X$ has a quadratic singularity at $p$ if it does not pass through $q$, and is smooth at $p$ and has a quadratic singularity at $q$ if it passes through $q$.

We conclude with the following.
\begin{proposition}\label{CY20}  Let $Y\subset \Pn^7$ be a fourfold defined by the maximal minors of a $3\times 5$ matrix $A$ of linear forms.  If $A$ has rank $\geq 2$ at every point in $\Pn^7$, then the anticanonical divisor of $Y$ is not effective.  If $A$ has rank $1$ at some point, then $Y$ has an effective anticanonical divisor.  For a general $A$ with a rank $1$ point, the anticanonical system consists of a unique divisor $X$ that is a $2$-nodal AG Calabi-Yau threefold of degree $20$. 
Furthermore, $X$ is rationally fibred in plane elliptic curves, and $X$ is not smoothable.

A general Artinian Gorenstein ring $A_F$ of a quaternary quartic lifts to a halfcanonical curve in $\Pn^5$, a deformation of any smooth $\Pn^5$ section of the nodal AG Calabi-Yau threefold of degree 20 in $\Pn^7$.
\end{proposition}

\begin{proof}
We observe that the net of rows in $A$ defines $Y$ as a  $\Pn^2$ scroll over $\Pn^2$.  The rank $1$ point is a common point of a pencil of planes in the scroll, and the planes intersect the anticanonical divisor in cubic curves, so $X$ has a rational ellipic fibration over $\Pn^2$.
Furthermore, if $Z$ is a general scroll over a line that passes through $q$, then the surface of intersection $S=X\cap Z$ is smooth at the rank $1$ point and has a quadratic singularity at $q$, so it is not Cartier as a divisor at the rank $1$ point on $X$.  It is Cartier elsewhere, so it follows by a result by R. Friedman \cite{Friedman}*{Rmk. 4.5}, that $X$ is not smoothable.  

It remains then to check liftings of $A_F$ for general quartics $F$. 
An Artinian reduction of the nodal threefold $X$ is $A_F$ for a quartic form $F$ that is apolar to a determinantal scheme of length $10$ defined by the cubic minors of a $3\times 5$ matrix of linear forms.  The locus of these schemes in the Hilbert scheme has dimension $27$.  There is no quartic singular along a general length $10$ scheme of this kind, so by Terracini the  set of quartic forms that are apolar to such a scheme is dense in the space of quartics.  In particular a general quartic is apolar to a surface of $10$-tuples of points that are defined by the cubic minors of a  $3\times 5$ matrix of linear forms.

Let $Y_0$ be a $\Pn^3$ section of $Y$, then $H^0({\mathcal O_{Y_0}(X)})$ is naturally identified with the space of quaternary quartic forms apolar to $Y_0$. The Artinian reduction of $X$ as a divisor in $Y$ is then the image of the section $s\in H^0({\mathcal O_{Y}(X)})$, that defines $X$, by the restriction map
$$H^0({\mathcal O_{Y}(X)})\to H^0({\mathcal O_{Y_0}(X)}).$$

Now any determinantal scheme of length $10$ defined by the cubic minors of a $3\times 5$ matrix is a $\Pn^3$ section of a determinantal variety $Y\subset \Pn^N$ for any $N>3$.  
Therefore, to show that a general $A_F$  has a lifting to a halfcanonical curve, we show that the restriction map
$$H^0({\mathcal O_{Y_2}(C)})\to H^0({\mathcal O_{Y_0}(C)})$$
is surjective for a general determinantal surface $Y_2$, a smooth intersection of $Y$ with a $\Pn^5$ that contains $Y_0$, where $C\subset Y_2$ is a halfcanonical curve, i.e. $C\equiv 2H-K_{Y_2}$ on $Y_2$.

The surjectivity of the restriction map is best analyzed step by step.
If $Y_1\subset Y_2$ is a smooth hyperplane section that contains $Y_0$, then $H^0({\mathcal O_{Y_2}(C)})\to H^0({\mathcal O_{Y_0}(C)})$ is surjective, if both 
$H^0({\mathcal O_{Y_2}(C)})\to H^0({\mathcal O_{Y_1}(C)})$
and
$H^0({\mathcal O_{Y_1}(C)})\to H^0({\mathcal O_{Y_0}(C)})$
are surjective.

By the cohomology of the exact sequences 

$$0\to {\mathcal O}_{Y_2}(C-H)\to {\mathcal O_{Y_2}(C)}\to {\mathcal O_{Y_1}(C) }\to 0$$
and  
$$0\to {\mathcal O}_{Y_1}(C-H)\to {\mathcal O_{Y_1}(C)}\to {\mathcal O_{Y_0}(C) }\to 0,$$
the surjectivity then follows if both
$$H^1({\mathcal O_{Y_2}(C-H)})=0\quad {\rm and }\quad H^1({\mathcal O_{Y_1}(C-H)})=0.$$

 Now, the surface $Y_2$ is the rank $2$ locus of a $3\times 5$ matrix with entries in the $6$-dimensional vector space $W$ of linear forms on $\Pn^5$.  Interpreting each point on $Y_2$ as map $\phi:\CC^3\to \CC^5$, there is a birational morphism $\phi\mapsto {\rm ker}(\phi)$ to $\Pn^2$. The inverse of this morphism
 is the map  on $\Pn^2$  defined by the $5\times 5$ minors of a $5\times 6$ matrix with linear forms. The map extends to a morphism on $\Pn^2$ blown up in the fifteen points $\Gamma\subset \Pn^2$, the common zeros of the $5\times 5$ minors.
 
A halfcanonical curve $C$ on $Y_2$ is the strict transform of a plane curve of degree $13$ with multiplicity $3$ along $\Gamma$. Therefore a section $H^0({\mathcal O_{Y_2}(C-H)})$ defines the strict transform of a plane curve of degree $8$ with multiplicity $2$ along $\Gamma$.
But, by Riemann Roch, such a curve exists if and only if $H^1({\mathcal O_{Y_2}(C-H)})\not=0$.  The existence of such a curve is therefore a closed condition on $\Gamma$. A computation in {\tt Macaulay2}, \cite{M2}*{Package QuaternaryQuartics},  shows that there are no plane curve of degree $8$ with multiplicity $2$ along a random determinantal $\Gamma$, so we conclude that
$H^1({\mathcal O_{Y_2}(C-H)})=0$ for a general $\Gamma$.

It remains to show that
 $H^1({\mathcal O_{Y_1}(C-H)})=0$ for a general hyperplane section $Y_1$ of $Y_2$.
 The curve $Y_1$ has genus $6$ and the divisor $C-H$ on $Y_1$ has degree $10$, so $H^1({\mathcal O}_{Y_1}(C-H))\not=0$ if and only if ${\mathcal O}_{Y_1}(C-H)$ is the canonical sheaf on $Y_1$. 
 But $Y_1$ is the strict transform on $Y_2$ of a smooth plane quintic $\bar Y_1$ curve that contains $\Gamma$.  So a the canonical sheaf  on $Y_1$ is the restriction of ${\mathcal O}_{\Pn^2}(2L)$ to $Y_1$, where $L$ is the class of a line in $\Pn^2$.
 In particular, ${\mathcal O}_{Y_1}(C-H)$ is canonical if and only if ${\mathcal O}_{Y_1}(C-H-2L)={\mathcal O}_{Y_1}$.  Curves in the class of
 $C-H-2L$ are strict transforms on $Y_2$ of plane curves of degree $6$ with multiplicity $2$ along $\Gamma$. Their restriction to $\bar Y_1$ has multiplicity $2$ at the points of $\Gamma$, so the line bundle ${\mathcal O}_{Y_1}(C-H-2L) $ is trivial, i.e. has a nonzero section if and only if there is a plane curve of degree $6$ tangent to $\bar Y_1$ along $\Gamma$.  This occurs, only if the $5\times 6$ matrix defining the map into $\Pn^5$ has a symmetric $5\times 5$ submatrix that defines $Y_1$, or equivalently ${\mathcal O}_{Y_1}(3L-\Gamma)$ is a $2$-torsion element in the Picard group of $Y_1$. 
For a general, hyperplane section $Y_1$ this is not the case.  
We may conclude that a general $A_F$ lifts to points on $Y_1$, and hence also to a halfcanonical curve on $Y_2.$
\end{proof}
\end{description}
 Note that Calabi-Yau threefolds  defined by $3\times 3$ minors of a $4\times 4$ matrix with linear forms are lifting of special AG rings with Betti table $[000]$.
 \begin{question}
 Describe the quartics corresponding to AG rings of type $[000]$ that lift to smooth threefolds.
 \end{question}

In this section we have considered liftings of apolar rings $A_F$ to smooth AG varieties, in particular to smooth threefolds.  In some cases, the liftings extends to higher dimensions.  We first summarize  the possible liftings of $A_F$ to curves $C_F$:
\begin{proposition}\label{liftofCGKK}
Let $F$ be a  quartic form of a type appearing in Appendix \ref{Appendix}, Table~\ref{tableCGKK}. Assume that $F$ is general in its family. Then there exists a configuration of distinct points $\Gamma$ apolar to $F$ with the following property:  The set of points $\Gamma$ admits a lifting to a surface $S_{\Gamma}$ that contains a curve $C_F\subset S_{\Gamma}$ such that $A_F$ is an Artinian reduction of $C_F$. Moreover, we have two possibilities:
\begin{enumerate}
    \item If $F$ is not of type $[300a]$ then $C_F\in |2H-K_{S_{\Gamma}}|$.
    \item If $F$ is of type $[300a]$, then $S_{\Gamma}$ is a K3 surface complete intersection of three quadrics containing a line $l$ and 
    $C_F\in |2H-K_{S_{\Gamma}}+l|=|2H+l|$.
\end{enumerate}
\end{proposition}
\begin{proof} We let $\Gamma$ be a finite set of points that computes the rank of $F$ as listed in Table \ref{tablevspcgkk}. The Betti table of $S/\Gamma$, is then listed in Table \ref{tablepoints}, from which a lifting of $\Gamma$ to a surface $S_\Gamma$ is straightforward and may be summarized as follows: \begin{description} 
   \item [{Type $[683]$}] $S_{\Gamma}$ is a rational scroll of degree $4$. 
    \item [{Type $[550]$}]  $S_{\Gamma}$ is an anticanonical del Pezzo surface of degree $5$.
    \item [{Type $[400]$}] $S_{\Gamma}$ is a complete intersection of three quadrics.
    \item [{Type $[300a]$}]  $S_{\Gamma}$ is a complete intersection of three quadrics containing a line.
    \item [{Type $[300b]$}] $S_{\Gamma}$ is a scroll of degree $7$ obtained as the residual component of the intersection of three quadrics contianing a $\mathbb{P}^2$.
    \item [{Type $[320]$}] $S_{\Gamma}$ is a surface of degree $7$ described as a divisor of type $(3,2)$ on $\Pn^1\times \Pn^2$ in its Segre embedding. 
    \item [{Type $[200]$}] $S_{\Gamma}$ is a rational surface of degree $8$ described as the image of $\mathbb{P}^2$ through the map induced by the system of sextics with six fixed double points and passing through additional four fixed points.
    \item [{Type $[100]$}] $S_{\Gamma}$ is a rational surface of degree $9$ described as the image of $\mathbb{P}^2$ through the map induced by the system of septics with $10$ fixed double points.
     \item [{Type $[000]$}] $S_{\Gamma}$ is a surface defined as the corank $1$ locus of a $5\times 3$ matrix of linear forms on $\mathbb{P}^5$. It is a rational surface of degree $10$ described as the image of $\mathbb{P}^2$ through the map induced by the system of $5\times 5$ minors of a $5\times 6$ matrix of linear forms on $\mathbb{P}^2$.
    \end{description}
In each of these cases the system $|2H-K_{S_{\Gamma}}|$ is base point free and its general elements are arithmetic Gorenstein manifolds whose resolution is a doubling of the resolution of $S_{\Gamma}$. The fact that a general quartic in the families are obtained follows from the surjectivity of the restriction maps as in the cases type $[100]$ and type $[200]$ computed explicitly above. 

In the case of type $[300b]$ we check by the adjunction formula that $C_F$ is an AG curve and since it is bilinked to $l$ in a complete intersection of quadrics it is of type $[300b]$. To see that in this way we obtain a general $F$ in the family it is enough to observe that the bilinkage construction of $A_F$ can be lifted to a bilinkage construction of $C_F$.
\end{proof}

The lifting of $A_F$ to smooth varieties of higher dimensions is summarized in the following proposition.

\begin{corollary}\label{cor:lifting}
Let $F$ be a  quartic form of type $[683]$, $[550]$, $[320]$, $[300a]$, $[300b]$ or $[200]$. Assume that $F$ is general in its family. Then there exists a smooth AG variety $X$ of dimension $d$ such that $A_F$ is an Artinian reduction of $X_F$, where $d=3,6,4,3,2,4$ respectively.
\end{corollary}
\begin{proof}
We consider the constructions of surfaces $S_\Gamma$ and curves $C_F$ of Proposition \ref{liftofCGKK}. In cases $[683]$, $[550]$ and $[320]$  
the surface $S_\Gamma$ and the curve $C_F$ are either determinantal varieties or divisors in determinantal varieties. The obstruction to smooth liftings to higher dimensions is the singular locus of the determinantal varieties.  
In  cases $[683]$, $[550]$, $[320]$, the singularities of the determinantal variety allow smooth liftings of $A_F$ to dimensions $3,6,4$, respectively.

In the cases $[300a]$ and $[300b]$ we may lift $\Gamma$ further to a complete intersection $Y_8$ of type $(2,2,2)$ and  a codimension $3$ variety $Y_7$ linked to a codimension $3$ linear space $P$ in a complete intersection $(2,2,2)$, respectively.
A lifting of $A_F$ is obtained by bilinkage of a codimension $4$ linear space $P'$ in $Y_8$, and the variety $Y_7\cap P$, respectively.
Obstructions to a smooth lifting, are the singularities of $Y_8$ in $P'$ and of $Y_7$ in $P$, respectively.
Smooth liftings to a threefold of $F$ of type $[300a]$ was found in \cite{CGKK}, while a lifting of $F$ of type $[300b]$ to a threefold has isolated singularities, as seen above. 

Similarly, in case $[200]$ a general $A_F$ lifts to a Fano fourfold bilinked to  $\mathbb P^2\times \mathbb P^2$ in a complete intersection $(2,2,3)$ and the same types of obstructions to further lifting arise.
\end{proof}

\section{Reducible liftings 
}\label{section_reducible_liftings}
\noindent

Patience Ablett, cf. \cite{Ablett}, recently found liftings to curves in $\Pn^5$ of quartics $F$ whose Betti table is in Table \ref{tableremaining} in Appendix \ref{Appendix}. We recover these liftings in the following constructions of liftings to AG threefolds in $\Pn^7$.  

Any lifting of a ring $A_F$ of type $[300c]$ is reducible; this case was treated in the previous section. For liftings of rings $A_F$ considered in this section we have:
\begin{proposition}
Every Betti table appearing in Table~\ref{tableremaining} can be realized as the Betti table of an AG threefold in $\mathbb P^7$.  The threefold is irreducible but singular for the type $[420]$, in all other cases it is reducible.
\end{proposition}
\begin{proof}
In Theorem 2.3 of \cite{SSY}, it is shown that every Betti table in Table~\ref{tableremaining} of Appendix \ref{Appendix} (except for type $[420]$) cannot be the Betti table of an irreducible threefold. We construct the liftings case by case, and confirm the constructions by checking a random example using {\tt Macaulay2}, cf. \cite{M2}*{Package QuaternaryQuartics} .
\begin{description}
    \item [{\bf Type $[210]$}]  There exists an AG variety $X$ of codimension 4 in $\mathbb P^7$ that lifts $A_F$ for quartic forms $F$ with Betti table of type $[210]$. The variety 
    $X$ is reducible with three components, $X_1,X_6$ and $X_{11}$ of degrees $1,6$ and $11$ respectively.  The variety $X_{11}$ is linked to a threefold $X'_7$ in a complete intersection of degrees $(1,2,3,3)$, and $X'_7$ is linked to a $\mathbb P^3$ in a complete intersection of degrees $(1,2,2,2)$.
The threefold $X_6$ is a complete intersection $(1,1,2,3)$ and $X_1$ is a $\mathbb P^3$.  
The space $X_1$ intersects both $X_6$ and $X_{11}$ in a quadric surface, while $X_{11}$ intersect $X_6$ in a complete intersection surface of degrees $(1,1,1,2,2)$.

Here is a more detailed description of a lifting of $A_F$ to an AG curve.
A general $\Pn^5$ intersects the three components of the threefold $X$ in curves $L,C_6$ and $C_{11}$ respectively. The curve $C_{11}$ is linked $(3,3)$ on a quadric threefold $Q$ to a curve of degree $7$ which again is linked $(2,2)$ on $Q$ to a line. The curve $C_{11}$ lie in five cubics on $Q$.
The curve $C_6$ is cubic section of a quadric surface $S_2$. The curves $C_{11}$ and $C_6$ intersect in four coplanar points.  The line $L$ lies in this plane $S_1\subset Q$, and intersects $S_2$ in a conic.
The line $L$ is a secant line to both $C_{11}$ and $C_6$. 
There is a pencil of cubic sections of $Q$ that contain $C_{11}\cup S_1$.
These sections are reducible, in a quintic surface and the plane $S_1$.  Let $S_5$ be one of the quintic surfaces.
Then $C_{11}\subset S_5$, $L\subset S_1$ and $C_6\subset S_2$.
Furthermore, $S_5\cap S_1=C$ is a conic and $S_1\cap S_2=C'$ is a conic.
The curve $C_{11}$ is linearly equivalent to $2H-K_{S_5}-C$ on $S_5$.
The line $L$ is linearly equivalent to $2H-K_{S_1}-C-C'$ on $S_1$.
The curve $C_{6}$ is linearly equivalent to $2H-K_{S_2}-C'$ on $S_2$.
And $S_\Gamma=S_5\cup S_1\cup S_2$ is an arithmetic Cohen-Macaulay surface of degree $8$ that contain the AG curve  $C_F=C_{11}\cup L\cup C_6$.
The curve $C_F$ has Artinian reduction $A_F$ for a quartic form $F$ of type $[210]$.
\vskip .1in
 \item [{\bf Type $[310]$}]  There exists an AG variety $X$ of codimension 4 in $\mathbb P^7$ that lifts $A_F$ for quartics $F$ with Betti table  of type $[310]$. The variety $X$ is reducible with two components, $X_6$ and $X_{11}$ of degrees $6$ and $11$ respectively.  The variety $X_{11}$ is linked to a threefold $Y_7$ in a complete intersection of degrees $(1,2,3,3)$, and $Y_7$ is linked to a $\mathbb P^3$ in a complete intersection of degrees $(1,2,2,2)$.
The threefold $X_6$ is a complete intersection $(1,1,2,3)$ and intersects $X_{11}$ in a hyperplane section of $X_6$.

Here is a more detailed description of a lifting of $A_F$ to an AG curve.
Consider a rational surface $S_5$ of degree $5$ in a $\Pn^4$ linked to a plane in the complete intersection $(2,3)$ and a quadric surface $S_2$ in a $\Pn^3$, so that $S_5\cup S_2$ spans a $\Pn^5$, and $S_5\cap S_2$ is a conic section $C$.  Assume furthermore that the conic $C$ moves in a pencil on $S_5$. 
Let $C_{11}$ be a curve linearly equivalent to  $2H-K_{S_5}-C$ on $S_5$ that intersects $C$ properly, i.e. in a scheme of length $6$. Let $C_6$ be a curve linearly equivalent to $2H-K_{S_2}-C$ such that  $C_6\cap C=C_{11}\cap C$.  Then $S_\Gamma=S_5\cup S_2$ is arithmetic Cohen-Macaulay and $C_F=C_{11}\cup C_6\subset S_\Gamma$ is arithmetic Gorenstein and has Artinian reduction $A_F$ for a quartic form $F$ of type $[310]$.

 \begin{proposition}\label{ReidNoAnswer}
 Type $[310]$ gives a negative answer to Reid's question from \cite{ReidGeneral}: ``For odd $k$, are hypersurfaces in a codimension 3 Gorenstein variety the only cases?'' 
 \end{proposition}
 \begin{proof}
 If these were the only cases, then the hypersurface $\VV(f)$ would have to have $f$ a nonzero divisor of degree $2$ or $3$. 
 
 In the former case, the free resolution of $I \setminus f$ has Betti table of the form below, where $*$ denotes an unknown entry:
 \[
 \begin{matrix}1 & - & - & - & -& \cr
       - & 2& 1 & - & - &\cr
       - &5 & * &* & *& \cr
       - &- & *& * & * &\cr
       - &- & *& * & *&  \cr
\end{matrix}. 
\] 
The linear syzygy on the two remaining quadrics must persist, for if not, this would imply that $f$ is a zero divisor. Since the Pfaffians are unmixed of codimension $3$, quotienting by a zero divisor $f$ cannot yield a codimension $4$ ideal. Taking the mapping cone would then force $b_{25}\ge 5$, which is inconsistent with the Betti table of type $[310]$. 
Similarly, if $f$ is a cubic, the Betti table of $I \setminus f$ is
\[
 \begin{matrix}1 & - & - & - & -& \cr
       - & 3& 1 & - & - &\cr
       - &4 & * &* & * &\cr
       - &- & *& * & * &\cr
       - &- & *& * & *&  \cr
\end{matrix} .
\] 
The mapping cone construction yields $b_{25}\ge 4$, which is again inconsistent with a Betti table of type $[310]$. For an irreducible threefold providing a negative answer to Reid's question, reason as above with type $[200]$. 
\end{proof}
 \item [{\bf Type $[430]$}] There exists an AG variety $X$ of codimension 4 in $\mathbb P^7$ that lifts $A_F$ for quartic forms $F$ with Betti table  of  type $[430]$. The variety $X$ is reducible with three components, $X_1,X_6$ and $X_{9}$ of degrees $1,6$ and $9$ respectively.  The union $X_1\cup X_{9}$ is linked to a quadric threefold $Y_2$ in a complete intersection of degrees $(1,2,2,3)$, such that $X_1\cap Y_2$ is a quadric surface. The threefold $X_6$ is a complete intersection $(1,1,2,3)$ and intersects $X_1\cup X_{9}$ in a hyperplane section of $X_6$.
 
 Here is a more detailed description of a lifting of $A_F$ to an AG curve.
 Consider a plane $S_1$ and a $\Pn^4$ and a $\Pn^3$ in a $\Pn^5$ that intersect in $S_1$.
 Let $S_3\subset \Pn^4$ be a cubic  surface scroll that intersect $S_1$ in a conic $C$.
 Let $C_9\subset S_3$ the linked to $C\cup L$ in a cubic section of $S_3\cup S_1$.
 Let $S_2\subset\Pn^3$ be a quadric surface, and let $C'=S_1\cap S_2$.  
 Then $C'\cap L\cup C_9$ is a complete intersection $(2,3)$ in $S_1$. Let $S_6$ be a general cubic section of $S_2$ that contains $C'\cap L\cup C_9$.  Then $S_1\cap (S_3\cup S_2)=C\cup C'$, $S_3\cap (S_1\cup S_2)=C$ and $S_2\cap (S_3\cup S_1)= C'$, and therefore
 $C_1$ is linearly equivalent to $2H-K-C-C'$ on $S_1$,
 $C_9$ is linearly equivalent to $2H-K-C$ on $S_3$, and
 $C_1$ is linearly equivalent to $2H-K-C'$ on $S_2$.
 
 Furthermore $S_\Gamma=S_3 \cup S_1\cup S_2$ is arithmetic Cohen-Macaulay and $C_F=C_{9}\cup L\cup C_6\subset S_\Gamma$ is AG and has Artinian reduction $A_F$ for a quartic form $F$ of this type.
\vskip .1in
 \item [{\bf Type $[420]$}]  There exists an AG variety $X$ of codimension 4 in $\mathbb P^7$ that lifts $A_F$ for quartic forms $F$ with Betti table  of type $[420]$. The variety $X$ is irreducible of degree $16$.  It is linked to a quadric threefold in a complete intersection $(2,3)$  on a fivefold cubic scroll.  $X$ is singular in its intersection with the vertex of the scroll.
 
 Here is a more detailed description of a lifting of $A_F$ to an AG curve.
 Let $S_6$ be a general quadric hypersurface section  of $\Pn^1\times \Pn^2$ in its Segre embedding in $\Pn^5$.  It is a conic bundle over $\Pn^1$.  Let $C\subset S_6$  be a conic, and let $C_{16}$ be the residual curve to $C$ in a general cubic section of $S_6$ that contains $C$.
 Then
 $C_{16}$ is linearly equivalent to $2H-K_{S_6}$ on $S_6$.
 Furthermore $S_6$ is arithmetic Cohen-Macaulay and $C_{16}\subset S_6$ is arithmetic Gorenstein and has Artinian reduction $A_F$ for a quartic form $F$ of this type.
 \vskip .1in
 \item [{\bf Type $[331]$}]  There exists an AG variety $X$ of codimension 4 in $\mathbb P^7$ that lifts $A_F$ for quartic forms $F$ with Betti table of type $[331]$. The variety $X$ is reducible with two components, $X_4$ and $X_{13}$ of degrees $4$ and $13$ respectively.  The threefold $X_{13}$ is linked to a $\mathbb P^3$ in the cubic pfaffians of a $7\times 7$ skew matrix, and $X_4$ is a quartic threefold in a $\mathbb P^4$.  The two components of $X$ intersect in a hyperplane section of $X_4$.
 
 Here is a more detailed description of a lifting of $A_F$ to an AG curve.
 Let $C_{13}\subset \Pn^4$ be a smooth curve of degree $13$ residual to a line  $L$ in the zero locus of the cubic pfaffians of a $7\times 7$ skew symmetric matrix $M$ with linear forms. 
 Then  $C_{13}$ is contained in a surface $S_6$ of degree $6$ defined by the $3\times 3$ minors of a $3\times 4$ submatrix of $M$.
Furthermore, $L\cap C_{13}$ is a scheme of length $4$.  Let $C_4$ be a plane quartic curve that
 contains $L\cap C_{13}$, and let $S_1$ be the plane of $C_4$.   Then $C_F=C_{13}\cup C_4$ is a curve in the reducible surface $S_\Gamma=S_6\cup S_1$ that spans $\Pn^5$.

 Then
 $C_{13}$ is linearly equivalent to $2H-K_{S_6}-L$ on $S_6$, while $C_4$ is linearly equivalent to $2H-K_{S_1}-L$ on $S_1$.
 Furthermore $S_\Gamma$ is arithmetic Cohen-Macaulay and $C_F\subset S_\Gamma$ is arithmetic Gorenstein and has Artinian reduction $A_F$ for a quartic form $F$ of this type.
\vskip .1in
\item  [{\bf Type $[441a]$}] There exists an AG variety $X$ of codimension 4 in $\mathbb P^7$ that lifts $A_F$ for quartic forms $F$ with Betti table of type $[441a]$.
The quadrics in $F^{\perp}$ define a conic and a point.

The threefold $X$ is reducible with two components, $X_4$ and $X_{12}$ of degrees $4$ and $12$ respectively.  The threefold $X_{12}$ is the rank $1$ locus of a $2\times 4$ matrix $M$ with two columns of linear forms and two columns of quadratic forms in a $\mathbb P^6$. The quadric minor defines a quadric hypersurface that reduces to a conic for $F$.  The component $X_4$ is a quartic threefold in a $\mathbb P^4$ that intersects $X_{12}$ in a hyperplane section of $X_4$.  The $\mathbb P^4$ reduces to a point for $F$.
 We check using {\tt Macaulay2}, \cite{M2}*{Package QuaternaryQuartics},  on a random example that the Betti table of $X$ is of this type.
 
 Restricted to a general $\Pn^5$, a submatrix of $M$ consisting of two columns of linear forms and one column of quadratic forms has rank $1$ on a surface $S_5$ of degree $5$ in a $\Pn^4$.  The surface $S_5$ is a conic bundle with some reducible conics.  The $\Pn^4$ of the quartic threefold $X_4$
 intersects $S_5$ in a line $L$, and the $\Pn^5$ in a plane $S_1$.  The restriction of $X_{12}$ to $S_5$ is a curve $C_{12}$ linearly equivalent to $2H-K_{S_5}-L$ on $S_5$, while the restriction of $X_4$ to $S_1$ is a curve $C_4$, linearly equivalent to $2H-K_{S_1}-L$ on $S_1$.
 \vskip .1in
\item  [{\bf Type $[441b]$}] There exists an AG variety $X$ of codimension 4 in $\mathbb P^7$ that lifts $A_F$ for quartic forms $F$ with Betti table of type $[441b]$.  The quadrics in $F^{\perp}$ define two skew lines. 

 The threefold $X$ is reducible with two components, $X_8$ and $Y_{8}$  both of  degrees $8$.  Both components are linked to a common $\mathbb P^3$ in complete intersections of degrees $(1,1,3,3)$, and they intersect along a quartic surface  in the  common $\mathbb P^3$.  The two linear forms define $\Pn^5$'s for each component  that reduces to the two skew lines for $F$.
 
 Restricted to a general $\Pn^5$, each of the two components $X_8$ and $Y_8$, restrict to curves $C_8$ and $C'_8$ respectively of degree $8$, each in a $\Pn^3$.  The two $\Pn^3$'s intersect in a line $L$, and each of the two curves lie in a pencil of cubic surfaces that contain that line.  Let $S_3$ and $S'_3$  be two such surfaces, such that
 $C_F=C_8\cup C'_8$ is contained in $S_F=S_3\cup S'_3$. Then $S_3\cap S'_3=L$
 and $C_8$ is linearly equivalent to $2H-K_{S_3}-L$ on $S_3$, while
$C'_8$ is linearly equivalent to $2H-K_{S'_3}-L$ on $S'_3$.
  \vskip .1in
  \item [{\bf Type $[551]$}]  There exists an AG variety $X$ of codimension $4$ in $\mathbb P^n$ that lifts $A_F$ for quartic forms $F$ with Betti table  of this type.

Consider the variety $Y$ that is the intersection of five quadrics and decomposes as the union of a linear space $P$ of codimension $3$ and a complete intersection $Q$ of type $(2,2)$ in a hyperplane $H$, such that  $Q\cap P$ is the linear space $H\cap P$ of codimension $4$. Take a general cubic hypersurface $C$ containing $P$. Then $Z=\overline{Y\cap C\setminus P}$ is a variety of codimension $4$ and degree $11$. Let finally $D$ be a quartic hypersurface containing $Z$, then $X=Y\cap C\cap D$ is a variety of degree $15$ decomposing as the union of a quartic hypersurface $D\cap P$ and the variety $Z$ of degree $11$. 

A general $\Pn^5$ intersects the fourfold $Q$ in a surface $S_4$ in a $\Pn^4$ and  $P$ in a plane $S_1$, such that $S_1\cap S_4$ is a line $L$.   In a cubic section containing $L$, the residual curve $C_{11}$ on $S_4$ has degree $11$.  The quartic hypersurface  $D$ intersects. $S_1$ in a curve $C_{4}$ of degree $4$.
 Thus $C_F=C_{11}\cup C_4$ is contained in $S_F=S_4\cup S_1$
 and $C_{11}$ is linearly equivalent to $2H-K_{S_4}-L$ on $S_4$, while
$C_4$ is linearly equivalent to $2H-K_{S_1}-L$ on $S_1$.
\vskip .1in
\item[{\bf Type $[562]$}] There exist two families (I and II) of AG varieties $X$ of codimension $4$ in $\mathbb P^7$ that lift $A_F$ for  quartic forms $F$ with Betti table  of type $[562]$. Note that the quadrics in $F^{\perp}$ define a line and two points that altogether span $\Pn^3$, and that $F$ is apolar to the union of a set of three points on the line and the two points outside.  Therefore the forms of this family are projectively equivalent.  So the different AG-varieties $X$ all have the same Artinian reduction.
  
  In this case a general Artinian algebra $A_F$ with this Betti table lift to AG-point sets in both families.

The general member of the first family (I) is a reducible threefold  
$$X=X_7\cup X_8$$ 
where the components $X_7$ and $X_{8}$ have degree $7$ and $8$ respectively.  The threefold $X_{8}$ is linked to a 
$\mathbb P^3$ in a complete intersection $(1,1,3,3)$, while $X_7$ is linked to a 
$\mathbb P^3$ in a complete intersection $(1,1,2,4)$. 

The two components intersect in a quartic surface in the intersection of their spans. 

 The general member of the second family (II) is a reducible threefold  
 $$X=X_7\cup X_4\cup X'_4,$$ 
 where the  components $X_7, X_{4}$ and $X'_4 $ have degrees $7,4$ and $4$  respectively.  The threefold $X_{7}$ is linked to a 
reducible quadric $Y_2=Y_1\cup Y'_1$ in a complete intersection $(1,1,3,3)$, while $X_4$ is a quartic threefold in a $\Pn^4$ that contains the quartic surface $(X_7\cap Y_1)\cup (Y_1\cap Y'_1)$, and $X'_4$ is a quartic threefold in a $\Pn^4$ and contains the quartic surface $(X_7\cap Y'_1)\cup (Y_1\cap Y'_1)$. 
The span of $X_7$ is a $\Pn^5$, 
while the two $\Pn^4$'s of $X_4$ and $X'_4$ each intersect $\Pn^5$ in a $\Pn^3$ and each other in a $\Pn^2$.  The threefold $X$, of course, spans $\mathbb P^7$.

It is instructive to see the two distinct types of liftings (I and II) of type $[562]$ to AG sets of points $X\subset \Pn^4$. 
A general member of the first family (I) restricts to a general $\Pn^4$ as the union $\Gamma=\Gamma_1\cup \Gamma_2$ of two planar sets of points. The set $\Gamma_1$ in one plane is eight points that are linked to the intersection point of the two planes in a complete intersection of two cubics.   The set $\Gamma_2$ in the other plane is seven points that are linked to the intersection point in a conic and a quartic.  
In particular $\Gamma$ lies in the union of a plane cubic curve and a plane conic that meet at a point, the eight points $\Gamma_1$ on the cubic, the seven points $\Gamma_2$ on the conic.  The eight points together with the intersection point is the complete intersection with another cubic.
 
 The general member of the other family (II) restricts to a general $\Pn^4$ as a set of fifteen points $\Gamma=\Gamma_1\cup \Gamma_2\cup \Gamma_3$.  The set $\Gamma_1$ is seven points that lie in a plane.  Each of the two sets $\Gamma_2$ and $ \Gamma_3$ consists of four points that lie in a line.  The two lines are skew and meet the plane of $\Gamma_1$ in two points that are linked in the plane to $\Gamma_1$ in a complete intersection of two cubics.  
 
 In particular $\Gamma$ lies in the union of a plane cubic curve and two skew lines that both intersect the plane cubic curve.  The seven points on the cubic curve together with the intersection points with the lines is the complete intersection with another cubic.
 
 A {\tt Macaulay2} computation, cf. \cite{M2}*{Package QuaternaryQuartics}, shows that the two families of liftings have a common degeneration where two of the fifteen points come together as a scheme of length $2$ supported at a point.  This corresponds to the quintic curves, in the two liftings above, degenerating to a plane cubic and two lines that meet at a point on the plane cubic. 
 \end{description}
 \end{proof}

	\vfill\eject
	\appendix
\section{Betti tables for the inverse system of a quarternary quartic}\label{Appendix}

The $16$ possible Betti tables for a minimal resolution of a nondegenerate graded Artinian Gorenstein quotient $A_F$ of $S=\CC[x_0,\ldots,x_3]$ of regularity $4$ is given in \cite {SSY}. For convenience we list the set of Betti tables here.  They are given in {\tt macaulay2} notation, as in \cite{GeomSyz}: if the minimal resolution is
$$ 0 \leftarrow A_F  \leftarrow F_0 \leftarrow F_1 \leftarrow \ldots  \leftarrow F_3 \leftarrow  0$$
with $F_i = \bigoplus_{j \in \ZZ} b_{ij}S(-j)$ and $b_i=\sum_j b_{ij} $,  then the Betti table of an Artinian Gorenstein ring of codimension $4$ and regularity $4$ is symmetric with the shape
\[
\begin{matrix}
& 1 & b_{1} & b_{2} & b_{3}& 1 &\cr
0:& 1 & . & . & .& . &\cr
1:&. & b_{12} & b_{23}  & b_{34}& . &\cr
2:&. & b_{13} & b_{24}  & b_{35}& .& \cr
3:&. & b_{14} & b_{25} &b_{36}& . &\cr
4:&. & . & . &.& 1 &\cr
\end{matrix}.
\]

The first eight tables are the Betti tables of arithmetic Gorenstein Calabi-Yau threefolds in $\Pn^7$ studied by 
Coughlan-Golebiowski-Kapustka-Kapustka

\begin{figure}[htbp]
\begin{center}
\includegraphics[width=260pt, angle=270]{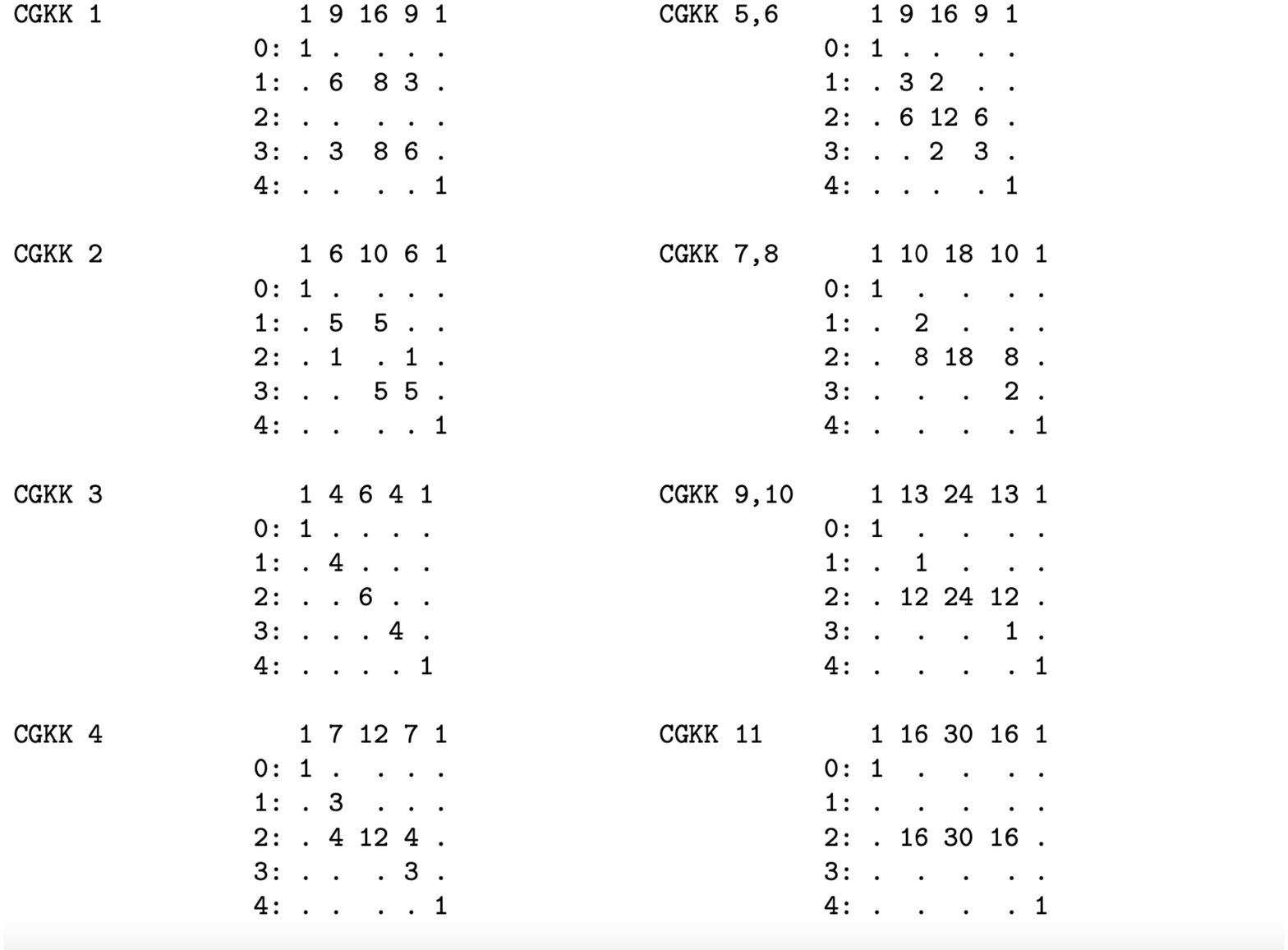} 

\vskip 0.4cm
\captionof{table}{Betti tables of AG Calabi-Yau threefolds in \cite{CGKK}}
\label{tableCGKK}
\end{center}
\end{figure}
\vskip 2.5cm
\begin{center}
\begin{figure}[htbp]
\includegraphics[width=260pt, angle=90]{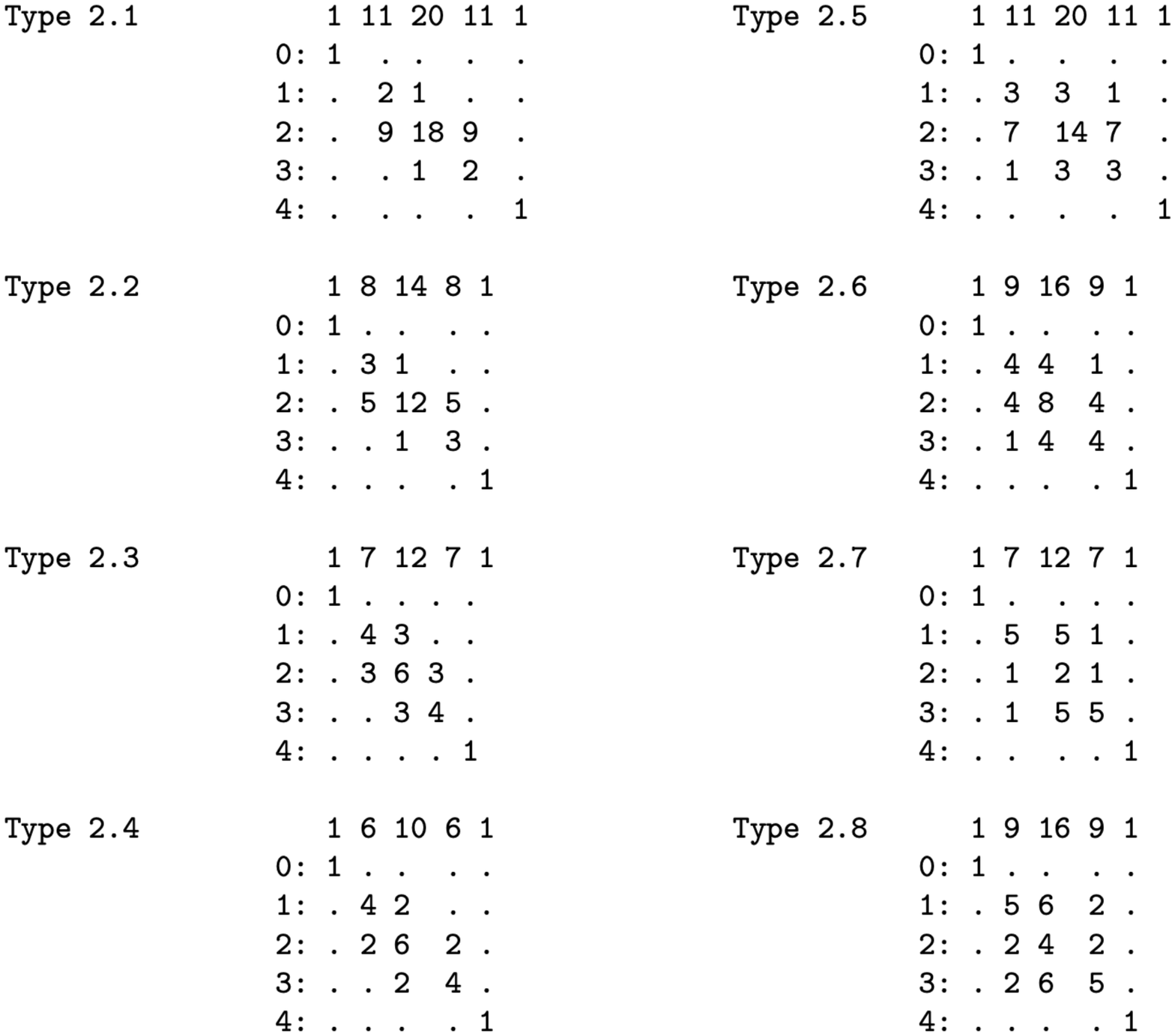} 
\vskip 0.4cm
\captionof{table}{Betti tables for the remaining eight Artinian Gorenstein algebras in \cite {SSY}}
\label{tableremaining}
\end{figure}
\end{center}

\section{Betti table subschemes of the Hilbert scheme}\label{Appendix2}

In this appendix, we investigate the Betti table strata of the
Hilbert scheme $\Hilb = \Hilb^d(\PP^3)$ whose points parametrize subschemes $\VV(I) \subset \PP^3$ of length $d$. We assume familiarity
with term orders and Groebner bases at the level of \cite{CLO}.
If $B$ is a Betti table the locus we wish to consider is
\[ \Hilb_B :=  \{ [I] \in \Hilb \mid BT(S/I) = B \}. 
\]
Similarly, if $J$ is a monomial ideal, we can consider the subset
\[
\Hilb_J := \{ [I] \in \Hilb \mid \operatorname{in}_{\text{grevlex}}(I) = J \}. 
\]
Both of these sets have natural structures as locally closed subschemes of $\Hilb$, and
for a Betti table $B$ and monomial ideal $J$, we let 
\[
\Hilb_{B,J} := \Hilb_J \cap \Hilb_B
\]
We show how to use generic initial ideals, Schreyer resolutions,
Groebner strata and {\it Macaulay2} to find the components of this subscheme of $\Hilb$.  
In
the range we are interested in ($d \le 9$), this locus can have more than one
component, but each component turns out to be rational for these degrees.  The computations involve ideals
in lots of variables (basically, affine coordinates of open subsets of
$\Hilb$), so we sometimes need to use special features to get computations
to finish.

In this paper, we are interested in the $\Hilb_B$, where $B$ is the
Betti table of an ideal with $4 \le d \le 9$ points, and every ideal
with the given Betti table has $\reg(S/I) \le 2$.  In
this appendix, we will explain how we investigate the components of these $\Hilb_B$,
focusing on the case $d = 6$.

We are interested in the case when the characteristic of the base
field $\kk$ is zero.  Most of our computations will be done over a
finite field, though, but with large enough characteristic so that the
results will hold in characteristic zero.

\subsection{Strongly stable ideals with a given Hilbert polynomial}

A {\em strongly stable monomial ideal} is a monomial ideal $J \subset S = \kk[x_0, \ldots, x_n]$
such that if $x_i m$ is in $J$,
then $x_j m$ is in $J$, for all $j < i$.

In characteristic zero (or more generally, if the characteristic is
greater than the regularity of $J$), this coincides with the notion of
an ideal being fixed under the Borel action of upper triangular
matrices.

There are only finitely many saturated strongly stable monomial ideals of
a given Hilbert polynomial (this is folklore, but D. Maclagan has a
nice general paper \cite{maclagan} which proves a sweeping generalization of this).
Alyson Reeves \cite{alyson-thesis} gave an algorithm in her Cornell thesis to compute all
of the strongly stable ideals, with a given Hilbert polynomial.  Moore and
Nagel \cite{moore-nagel}, and separately Albarelli-Lella \cite{alb-lella} have papers (and the latter has a {\tt Macaulay2} implementation improving this construction).

We compute using Macaulay2, the case when $d=6$.
We use the {\tt StronglyStableIdeals} Macaulay2 package written by Albarelli and Lella.

{\small\begin{verbatim}
i2 : needsPackage "QuartenaryQuartics";
i3 : kk = ZZ/32003;
i4 : S = kk[a..d];
i5 : B6 = stronglyStableIdeals(6, S);

i6 : netList B6

     +----------------------------------+
     |              6                   |
o6 = |ideal (b, a, c )                  |
     +----------------------------------+
     |                2   5             |
     |ideal (a, b*c, b , c )            |
     +----------------------------------+
     |           2     2   4            |
     |ideal (a, b , b*c , c )           |
     +----------------------------------+
     |           3     2   2    3       |
     |ideal (a, c , b*c , b c, b )      |
     +----------------------------------+
     |                  2        2   4  |
     |ideal (b*c, a*c, b , a*b, a , c ) |
     +----------------------------------+
     |             2        2   3     2 |
     |ideal (a*c, b , a*b, a , c , b*c )|
     +----------------------------------+

i7 : netList pack(3, B6/minimalBetti)

     +--------------+--------------+--------------+
     |       0 1 2 3|       0 1 2 3|       0 1 2 3|
o7 = |total: 1 3 3 1|total: 1 4 5 2|total: 1 4 5 2|
     |    0: 1 2 1 .|    0: 1 1 . .|    0: 1 1 . .|
     |    1: . . . .|    1: . 2 3 1|    1: . 1 1 .|
     |    2: . . . .|    2: . . . .|    2: . 1 2 1|
     |    3: . . . .|    3: . . . .|    3: . 1 2 1|
     |    4: . . . .|    4: . 1 2 1|              |
     |    5: . 1 2 1|              |              |
     +--------------+--------------+--------------+
     |       0 1 2 3|       0 1 2 3|       0 1 2 3|
     |total: 1 5 7 3|total: 1 6 8 3|total: 1 6 8 3|
     |    0: 1 1 . .|    0: 1 . . .|    0: 1 . . .|
     |    1: . . . .|    1: . 5 6 2|    1: . 4 4 1|
     |    2: . 4 7 3|    2: . . . .|    2: . 2 4 2|
     |              |    3: . 1 2 1|              |
     +--------------+--------------+--------------+

\end{verbatim}
}

\medskip
Thus, there are six saturated strongly stable monomial ideals, although four of these
are degenerate: they lie on a hyperplane. Of the remaining two, only one has regularity  $\reg(S/J) = 2$.

\subsection{Generic initial ideals}

There are two basic theorems (see Eisenbud's commutative algebra book \cite{CAwvtAG}, for an
      exposition) that we will use:
  
      \begin{theorem}[Galligo \cite{Ga74}, Bayer-Stillman \cite{BSDuke}]
        Fix a term order $>$, and let $\kk$ be an infinite field.  Let $I
        \subset T = \kk[x_0, \ldots, x_n]$ be a homogeneous ideal.
        After a general change of coordinates $g \in GL(n+1, \kk)$,
        the initial ideal $J = in_>(g \cdot I)$ in this term order is Borel
        fixed.
      \end{theorem}

      If the term order is
      the graded reverse lexicographic order, we call this monomial ideal the {\em generic initial ideal} of $I$ (or $\gin(I)$).
      
      \begin{theorem}[Bayer-Stillman \cite{BS87}]
      The generic initial ideal $J = \gin(I)$
      has the same depth, projective dimension, and regularity as $I$.
      \end{theorem}

      In particular, after a general, or sufficiently random, change of coordinates,
      the initial ideal $J$ is saturated if $I$ is, it is Borel fixed, and has the same regularity and projective
      dimension as $I$.  

      In characteristic zero (or if the characteristic is greater than the largest
      degree generator of $J$), $J$ is strongly stable, and the regularity of $I$ is precisely the maximal
      degree of a generator of $J$.
      
      In the range of degrees we are discussing, for characteristic
      larger than, say, 100, strongly stable ideals and Borel fixed
      ideals coincide, so we do not encounter the subtleties that
      occur in finite characteristic.

      Thus, we can partition the set of ideals $I$ with $[I] \in H$ by their generic initial
      ideal $J = \gin(I)$.

      In our running example,
      there is exactly one strongly stable saturated ideal $J$ with no linear forms in 
      the ideal, and $\operatorname{reg}(S/J) \le 2$.
      
{\small\begin{verbatim}
i8 : select(B6, J -> first degree J_0 != 1 and regularity (S^1/J) <= 2)

                   2        2   3     2
o8 = {ideal (a*c, b , a*b, a , c , b*c )}

i9 : J = ideal(a^2, a*b, b^2, a*c, b*c^2, c^3)

             2        2          2   3
o9 = ideal (a , a*b, b , a*c, b*c , c )

i10 : betti res J

             0 1 2 3
o10 = total: 1 6 8 3
          0: 1 . . .
          1: . 4 4 1
          2: . 2 4 2
\end{verbatim}
}

We compare this to the resolution of the ideal $I$ of 6 general (random) 
points in $\PP^3$.  A general fact (which follows from upper-semicontinuity) is that the
Betti table of $I$ is entry by entry at most the Betti table of its initial ideal $J$.

{\small\begin{verbatim}
i11 : I = pointsIdeal randomPoints(S, 6);

o11 : Ideal of S

i12 : J == ideal leadTerm I

o12 = true

i13 : (betti res I, betti res J)

              0 1 2 3         0 1 2 3
o13 = (total: 1 4 5 2, total: 1 6 8 3)
           0: 1 . . .      0: 1 . . .
           1: . 4 2 .      1: . 4 4 1
           2: . . 3 2      2: . 2 4 2
\end{verbatim}
}
%

\subsection{The Groebner family and the Groebner stratum of an initial ideal}

      Given a monomial ideal $J$, and a term order (in this paper, we 
      always take the graded reverse lex order), we can form a
      parameter space of all ideals having initial ideal $J$.  
      This gives rise to a locally closed subscheme $\Hilb_J \subset \Hilb$ of the Hilbert scheme,
      whose points correspond to ideals having initial ideal $J$.  See 
      Lella-Roggero \cite{lella-roggero} for additional details.
      
      We now sketch the construction.
      For each
      minimal generator $x^\alpha$ of $J$, consider the set $x^\gamma$
      of monomials of the same degree, not in $J$.  For each such
      $\alpha$, form a polynomial
      
      \[      
          F_\alpha = x^\alpha + \sum\limits_{\gamma < \alpha} t_{\alpha, \gamma} x^\gamma
      \]
      
      Let $U = T[x_0, \ldots, x_n]$, where $T = \kk[t_{\alpha, \gamma}]$
      is the polynomial ring generated by all variables of the form $t_{\alpha, \gamma}$ as above.
      Let $F$ be the ideal in $U$ generated by the $F_\alpha$.
      The ideal $F = (F_\alpha)$ is called the {\em Groebner family of $J$}.
      
      Any term order $>$ on the monomials of $S$, when restricted to a finite set $M$ of monomials, is given by
      a weight vector $w \in \mathbb{Z}^{n+1}$: for $x^\alpha, x^\beta \in M$, 
      $x^\alpha > x^\beta$ if and only if $w \cdot \alpha > w \cdot \beta$.  If we let 
      the weight of $t_{\alpha, \gamma}$ be $w \cdot (\alpha - \gamma)$, then with this grading,
      $F_\alpha$ is homogeneous of weight $w \cdot \alpha$.  This grading is useful for 
      performance reasons: computing with homogeneous ideals is much more efficient than with
      arbitrary ideals.  Additionally, we can order the variables $t_{\alpha, \gamma}$ by refining
      the weight degree by any term order on $T$: if the weight of $t_{\alpha, \gamma}$ is
      greater than $t_{\alpha', \gamma'}$ then $t_{\alpha, \gamma} > t_{\alpha', \gamma'}$ in this term order.
      This insures that if a homogeneous polynomial in $T$ has a term which is a variable, then the lead term will be a variable.


We want the ideal $L \subset T$ in the parameter variables that encodes the constraints that insure $F$ is a Groebner basis.
Let $L$ be
the ideal generated by all of the coefficients in $T$ of the reductions
of the S-polynomials on the $F_\alpha$.  The ideal $L$ parametrizes all ideals which have S-pairs reducing to zero.
 The scheme 
$\VV(L)$ with ideal $L$ is called the {\em Groebner stratum of $J$}.

Over $\VV(L)$, $F$ defines a flat family all of whose fibers
have initial ideal $J$, and consequently defines via representability
of the Hilbert scheme, a locally closed subscheme $Hilb_J := \VV(L)$ of $\Hilb$. These families,
over all saturated monomial ideals, cover the Hilbert scheme. For $p \in \VV(L)$, let
$F_p \subset S$ be the ideal corresponding to that point.

If we restrict to generic initial ideals, then every ideal is equivalent,
under a change of coordinates, to an ideal corresponding to a point in one of these subsets of $\Hilb$.

\medskip
In our running example above, this Groebner stratum is an open affine subset
of the Hilbert scheme.  As we saw earlier (and also follows directly), $\dim \Hilb = 18$, and is known to be irreducible and rational.
We check the dimension and irreducibility of $\Hilb_J$, which is an open subset of $\Hilb$.

{\small\begin{verbatim}
i14 : F = groebnerFamily J;
i15 : U = ring F;
i16 : T = coefficientRing U;

i17 : netList F_*

      +--------------------------------------------------------------+
      | 2                      2                      2              |
o17 = |a  + t b*c + t a*d + t c  + t b*d + t c*d + t d               |
      |      1       3       2      4       5       6                |
      +--------------------------------------------------------------+
      |                         2                         2          |
      |a*b + t b*c + t a*d + t c  + t  b*d + t  c*d + t  d           |
      |       7       9       8      10       11       12            |
      +--------------------------------------------------------------+
      | 2                         2                         2        |
      |b  + t  b*c + t  a*d + t  c  + t  b*d + t  c*d + t  d         |
      |      13       15       14      16       17       18          |
      +--------------------------------------------------------------+
      |                            2                         2       |
      |a*c + t  b*c + t  a*d + t  c  + t  b*d + t  c*d + t  d        |
      |       19       21       20      22       23       24         |
      +--------------------------------------------------------------+
      |   2                    2       2          2         2       3|
      |b*c  + t  b*c*d + t  a*d  + t  c d + t  b*d  + t  c*d  + t  d |
      |        25         27        26       28        29        30  |
      +--------------------------------------------------------------+
      | 3                    2       2          2         2       3  |
      |c  + t  b*c*d + t  a*d  + t  c d + t  b*d  + t  c*d  + t  d   |
      |      31         33        32       34        35        36    |
      +--------------------------------------------------------------+

i18 : L = trim groebnerStratum F;

o18 : Ideal of T

i19 : dim L == 18

o19 = true

i20 : isPrime L

o20 = true
\end{verbatim}
}

\subsection{The Schreyer resolution and minimal Betti numbers}

Schreyer's algorithm for computing free resolutions \cite{Schreyer80} can be used
to construct equations for $\Hilb_{B,J}$ from the Groebner family and stratum.  

For any Betti table $B$, and saturated monomial ideal $J$, we obtain the
locally closed subscheme $\Hilb_{B,J}$, consisting of all $p \in \VV(L)$ such that
$F_p$ has initial ideal $J$, and $S/F_p$ has Betti table $B$.
The steps for this construction are as follows.
\begin{enumerate}
    \item A free resolution $C$ of the ideal $F$ over $U/L$ can be computed via Schreyer's algorithm \cites{Schreyer80, lascala}.
    \item This resolution restricts to a minimal free resolution of the initial ideal $J$ if $J$ is strongly stable.
    \item We get a number of degree 0 maps in this resolution, say $\phi_{i,d}$, with entries in the ring $T/L$.  We consider the loci of these maps with prescribed ranks.  The ranks of these maps describe where cancellation in the resolution can
    occur, and consequently describe the Betti table.  For fixed size of minors of each of these matrices, we get a subscheme of $\VV(L)$ which is the union of all of the
    $\Hilb_{B,J}$, over all Betti tables entrywise greater than the one corresponding to the rank conditions.
    \item To be a bit more precise, for any point
    $p \in \VV(L) \subset Hilb$, we get a nonminimal resolution $C$ for $S/F_p$.  The minimal Betti table of this ideal can be computed from $C$ by 
    \[b_{id}(S/F_p) = \rank(C_i) - \rank(\phi_{i,d})(p) - \rank(\phi_{i-1,d})(p).\]
    Similarly, given a Betti table $B$ which is entrywise $\le B_J$, there exists a
    unique assignment of ranks which correspond to this Betti table.
    \item Thus, given a strongly stable ideal $J$ and a Betti table $B$, we can consider the locus $\overline{\Hilb}_{B,J} \subset \Hilb$ in $\VV(L)$, defined by the minors of these $\phi_{i,d}$ of the corresponding size.  This is a closed subscheme
    of $\VV(L)$ containing the union of all $\Hilb_{B',J}$, with $B' \ge B$.
    \end{enumerate}

It turns out for small numbers of points, we can compute all of these loci! Rather, we can compute enough about them to
identify the irreducible components and to show that they are rational.

A complication with doing this in Macaulay2 arises, but it is a small issue: we can lift a free resolution
of $J$ to one of $F$, but it is not always the minimal resolution of $J$.  Still,
we can use this to determine the equations of the $\overline\Hilb_{B,J}$ locally on $\VV(L)$, and consequently we can determine the component structure of the $\Hilb_{B,J}$

\medskip
We now illustrate the above technique in our running example.  We have written a function in Macaulay2, in the package {\tt ParameterSchemes},
which computes this lifted resolution, and the various degree 0 maps whose ranks determine the
Betti table of the minimal free resolution (at any point of $\VV(L)$).

{\small\begin{verbatim}
i21 : (CF, H) = nonminimalMaps F;
i22 : U = ring CF;

i23 : CF

       1      6      10      6      1
o23 = U  <-- U  <-- U   <-- U  <-- U
                                    
      0      1      2       3      4

o23 : ChainComplex

i24 : betti(CF, Weights=>{1}) -- in this case it is indeed a non-minimal resolution

             0 1  2 3 4
o24 = total: 1 6 10 6 1
          0: 1 .  . . .
          1: . 4  4 2 .
          2: . 2  5 3 1
          3: . .  1 1 .

i25 : isHomogeneous CF -- but it is homogeneous, as it needs to be.
o25 = true
\end{verbatim}
}
Notice that the resulting complex is not minimal when restricted to the origin (i.e. the point $J$).
The following 4 matrices determine the minimal Betti numbers of points on $\Hilb$.
In fact, the last 2 matrices always have full rank (as they must), and the $5 \times 2$ matrix $M_2$
always has rank at least 1.

\medskip

{\small\begin{verbatim}
i26 : keys H -- these are the maps of scalars in the resolution.  
             -- Key is (homological degree, internal degree).
o26 = {(3, 4), (3, 5), (4, 6), (2, 3)}

o26 : List

i27 :  M1 = H#(2,3) -- rank is 0, 1, or 2.

o27 = {3} | -t_8-t_14t_19      t_7t_14-t_20t_14+t_14t_19t_13            
      {3} | -t_7+t_20-t_19t_13 -t_8-t_14t_19+t_7t_13-t_20t_13+t_19t_13^2
      ---------------------------------------------------------------------
      -t_2-t_20^2+t_14t_19^2    -t_8t_20+t_1t_14+t_7t_14t_19             |
      -t_1-2t_20t_19+t_19^2t_13 -t_2-t_7t_20-t_8t_19+t_1t_13+t_7t_19t_13 |

              2       4
o27 : Matrix T  <--- T

i28 : M2 = H#(3,4) -- rank is 1 or 2

o28 = {4} | -t_14                                   
      {4} | -1                                      
      {4} | t_8+t_14t_19-t_7t_13+t_20t_13-t_19t_13^2
      {4} | -t_7+t_20-t_19t_13                      
      {4} | 0                                       
      ---------------------------------------------------------------------
      -t_8                                    |
      t_19                                    |
      t_2+t_7t_20+t_8t_19-t_1t_13-t_7t_19t_13 |
      -t_1-2t_20t_19+t_19^2t_13               |
      t_7-t_20+t_19t_13                       |

              5       2
o28 : Matrix T  <--- T

i29 : M3 = H#(3,5) -- maximal rank, can ignore

o29 = {5} | -1 t_19 -t_20 |

              1       3
o29 : Matrix T  <--- T

i30 : M4 = H#(4,6) -- maximal rank, can ignore

o30 = {6} | -1 |

              1       1
o30 : Matrix T  <--- T

i31 : minors(2, M2) == minors(1, M1)

o31 = true
\end{verbatim}
}

Since the ideal of maximal minors of $M_2$ is equal to the ideal of 
the maximal minors of $M_1$ (which are the entries), we obtain 3 possible
rank conditions 
\[ (\rank(M_1), \rank(M_2)) \in \{ (0,1), (1,2), (2,2) \}. \]
These correspond to three Betti table loci $\Hilb_{B,J}$.
The generic case (ranks are maximal) corresponds to the Betti table of
6 general points in $\PP^3$.  The minimal rank case corresponds to the Betti table of $J$.  All three Betti tables occur as cases (5, 6, 7) in Table~\ref{tablepoints}.

\subsection{The components of the Betti table loci of degree $6$}

We will next compute the components of the $\overline{\Hilb}_{B,J}$.  We can compute the loci in $\VV(L)$ that correspond to these three rank
conditions.  This will lead us to the following loci:
\begin{enumerate}
\item 6 general points.  This locus is dense in $\Hilb$, which is rational of dimension 18, and has Betti table type (5) in Table~\ref{tablepoints}.
\item  3 points on a line, and 3 general points.  
The closure of the locus
of such points is rational and irreducible, and has dimension 16,
and has Betti table type (6) in Table~\ref{tablepoints}.
\item 5 points in a plane and one point off the plane. The closure of the locus of such points is rational and irreducible, and has dimension 16,
and has Betti table type (7) in Table~\ref{tablepoints}.
\item 3 points on each of 2 skew lines.  The closure of the locus
of such points is rational and irreducible, and has dimension 14,
and has Betti table type (7) in Table~\ref{tablepoints}.
\end{enumerate}

We now explain how to determine this information.
Above, we have already handled case (1): the ideal $L$ itself is prime of dimension 18, and is dense in the Hilbert scheme $\Hilb$ parametrizing length 6 subschemes of $\PP^3$.  

We compute the ideals $L_{441}$ and $L_{430}$
defining the rank loci $\overline{\Hilb}_{B,J}$.  Recall that the corresponding Betti tables are:

\[\small
\begin{matrix}
&0&1&2&3\\
\text{total:}&1&6&8&3\\
\text{0:}&1&\text{.}&\text{.}&\text{.}\\
\text{1:}&\text{.}&4&4&1\\
\text{2:}&\text{.}&2&4&2\\
\end{matrix}
\qquad
\begin{matrix}
&0&1&2&3\\
\text{total:}&1&5&6&2\\
\text{0:}&1&\text{.}&\text{.}&\text{.}\\
\text{1:}&\text{.}&4&3&\text{.}\\
\text{2:}&\text{.}&1&3&2\\
\end{matrix}
\]

The first step is to compute the ideals $L_{441}$ and $L_{430}$ which define the rank loci, and decompose these
loci into irreducible components.

{\small\begin{verbatim}
i32 : L441 = trim(minors(1, M1) + L);
o32 : Ideal of T
i33 : L430 = trim(minors(2, M1) + L);
o33 : Ideal of T
i34 : compsL441 = decompose L441;
i35 : compsL430 = decompose L430;
i36 : (#compsL441, #compsL430)

o36 = (2, 2)

i37 : compsL441/dim

o37 = {16, 14}

i38 : compsL430/dim

o38 = {16, 16}

i39 : compsL441_0 == compsL430_0
o39 = true

i40 : isSubset(compsL430_1, compsL441_1)
o40 = true
\end{verbatim}
}

The last 2 lines confirm that the rank locus $\VV(L_{430})$
contains the rank locus $\VV(L_{441})$.  We see that there
are 2 irreducible components of $\Hilb_{B_{441}, J}$
of dimensions 14 and 16, and that there is one irreducible component of $\Hilb_{B_{430}, J}$, it has dimension 16.

Next, we check that the 2 components for the $441$ locus 
are both rational.  The function {\tt randomPointOnRationalVariety} returns a $\kk$-valued point, if certain sufficient conditions for the variety to be rational are satisfied.  This holds for both components of $L_{441}$, so they are rational, of dimensions 14 and 16.

We now take the random point on each (rational) component that this function returns, 
and compute the corresponding ideal in $S$.  By considering
its irreducible decomposition, we obtain a candidate
for the general element on each component.  In these cases,
a dimension count matches the dimension we computed for the
component.

We implement this with the code below.  The dimension 14
component has general element 3 points on one line, and 3 points on another skew line.  The dimension 16 component
has general element 5 points in a plane, and one point off
the plane.

{\small\begin{verbatim}
i42 : pta = randomPointOnRationalVariety compsL441_0

o42 = | 4818 14868 14185 11709 -12680 -3804 11646 9165 1329 -11430 -14770
      ---------------------------------------------------------------------
      -11245 -7240 13564 8800 -14581 12746 -792 15900 4173 15035 1560 -5149
      ---------------------------------------------------------------------
      12775 4571 7802 1805 13264 -10075 7936 -805 1321 -2281 9132 10444
      ---------------------------------------------------------------------
      3905 |

               1        36
o42 : Matrix kk  <--- kk

i43 : ptb = randomPointOnRationalVariety compsL441_1

o43 = | -10329 -120 -8822 7364 -103 11106 -435 -15249 -10586 -3570 -12139
      ---------------------------------------------------------------------
      6555 -10648 -8190 5354 1013 -5960 -67 13303 11931 -5834 13202 10214
      ---------------------------------------------------------------------
      10956 -2813 -12260 14265 8927 12815 -7346 11947 -23 0 11453 -5155
      ---------------------------------------------------------------------
      -3006 |

               1        36
o43 : Matrix kk  <--- kk

i44 : I441a = sub(F, (vars S) | pta)

              2                   2                                  
o44 = ideal (a  + 8800b*c - 11430c  - 11245a*d - 3804b*d + ..., ...)

o44 : Ideal of S

i45 : I441b = sub(F, (vars S) | ptb)

              2                  2                                 
o45 = ideal (a  + 5354b*c - 3570c  + 6555a*d + ..., ...)

o45 : Ideal of S

i46 : betti res I441a

             0 1 2 3
o46 = total: 1 6 8 3
          0: 1 . . .
          1: . 4 4 1
          2: . 2 4 2

o46 : BettiTally

i47 : betti res I441b

             0 1 2 3
o47 = total: 1 6 8 3
          0: 1 . . .
          1: . 4 4 1
          2: . 2 4 2

o47 : BettiTally

i48 : decompose I441a

                                                                           
o48 = {ideal (c - 10075d, b + 10116d, a + 15005d), ideal (a + 1321b + 7802c
      ---------------------------------------------------------------------
                 2                    2                             2   3  
      - 12657d, b  + 10444b*c + 13264c  + 6479b*d - 11887c*d + 7735d , c  +
      ---------------------------------------------------------------------
                      2            2          2        3     2            
      3905b*c*d - 805c d + 14534b*d  - 9978c*d  - 1947d , b*c  + 7936b*c*d
      ---------------------------------------------------------------------
             2           2           2        3
      - 5149c d - 3437b*d  + 15943c*d  + 7452d )}

o48 : List

i49 : (decompose I441a)/degree

o49 = {1, 5}
\end{verbatim}
}

Finally, the ideal $L_{430}$ does not meet the
sufficient conditions to determine that it is rational.
The workaround is to find a rational family of ideals which each have Betti table $B_{430}$, and form a family of dimension 16.  Since $\VV(L_{430})$ is irreducible of dimension 16
and must contain this family, they must coincide, and
therefore that component is rational of dimension 16.
The general element is a set of 3 points on a line, together with 3 general points in space.

The computation below yields the Betti table $B_{430}$ for
a random such ideal.

{\small\begin{verbatim}
i50 : I430 = pointsIdeal randomBlockMatrix({S^2, S^2}, {S^3, S^3}, 
          {{random, random}, 
           {0, random}})
                                            2   2                      
o50 = ideal (a*d - 9045b*d - 9460c*d + 2639d , c  + ..., ...)

o50 : Ideal of S

i51 : betti res I430

             0 1 2 3
o51 = total: 1 5 6 2
          0: 1 . . .
          1: . 4 3 .
          2: . 1 3 2
\end{verbatim}
}

This vignette illustrates the techniques to analyze the case of 6 points.  For degrees $4 \le d \le 9$, a similar
analysis
allows us to produce all of the components and show that they are rational.

\end{document}